\documentclass[reqno,12pt,a4paper]{article}

\usepackage[height=23.5cm,left=2.5cm,right=2.5cm]{geometry}

\usepackage{eucal,enumerate,mathrsfs,xspace}
\usepackage{amsmath,amssymb,epsfig,bbm,amsthm}
\usepackage[toc,page]{appendix}
\usepackage[usenames]{color}

\usepackage{ifthen}
\numberwithin{equation}{section}

\newtheorem{theorem}{Theorem}[section]

\newtheorem{problem}[theorem]{Problem}
\newtheorem{corollary}[theorem]{Corollary}
\newtheorem{lemma}[theorem]{Lemma}
\newtheorem{proposition}[theorem]{Proposition}
\newtheorem{definition}[theorem]{Definition}

\newtheorem{example}[theorem]{Example}

\newtheorem{remarkx}[theorem]{Remark}
\newenvironment{remark}
  {\pushQED{\qed}\remarkx}
  {\popQED\endremarkx}

\usepackage[normalem]{ulem}
\usepackage{pdfsync}
\usepackage{hyperref}

\newcommand{\A}{\mathbb{A}}

\newcommand{\C}{\mathbb{C}}

\newcommand{\K}{\mathbb{K}}
\renewcommand{\L}{\mathbb{L}}
\newcommand{\M}{\mathbb{M}}
\newcommand{\N}{\mathbb{N}}

\newcommand{\R}{\mathbb{R}}
\renewcommand{\S}{\mathbb{S}}

\newcommand{\CC}{\mathscr{C}}
\newcommand{\DD}{\mathscr{D}}
\newcommand{\EE}{\mathscr{E}}
\newcommand{\FF}{\mathscr{F}}
\newcommand{\GG}{\mathscr{G}}
\newcommand{\HH}{\mathscr{H}}

\newcommand{\LL}{\mathscr{L}}

\newcommand{\PP}{\mathscr{P}}
\newcommand{\QQ}{\mathscr{Q}}
\newcommand{\RR}{\mathscr{R}}

\newcommand{\cB}{{\ensuremath{\mathcal B}}}
\newcommand{\calC}{{\ensuremath{\mathcal C}}}

\newcommand{\cF}{{\ensuremath{\mathcal F}}}

\newcommand{\cH}{{\ensuremath{\mathcal H}}}
\newcommand{\cK}{{\ensuremath{\mathcal K}}}

\newcommand{\cM}{{\ensuremath{\mathcal M}}}

\newcommand{\cP}{{\ensuremath{\mathcal P}}}

\renewcommand{\tt}{{\mbox{\boldmath$t$}}}
\newcommand{\uu}{{\mbox{\boldmath$u$}}}
\newcommand{\vv}{{\mbox{\boldmath$v$}}}

\newcommand{\yy}{{\mbox{\boldmath$y$}}}
\newcommand{\zz}{{\mbox{\boldmath$z$}}}

\newcommand{\soo}{{\mbox{\scriptsize\boldmath$o$}}}

\newcommand{\szz}{{\mbox{\scriptsize\boldmath$z$}}}

\newcommand{\dD}{{\mbox{\boldmath$D$}}}

\newcommand{\xX}{{\mbox{\boldmath$X$}}}
\newcommand{\yY}{{\mbox{\boldmath$Y$}}}

\newcommand{\sdD}{{\mbox{\scriptsize\boldmath$D$}}}

\newcommand{\sxX}{{\mbox{\scriptsize\boldmath$X$}}}

\newcommand{\syY}{{\mbox{\scriptsize\boldmath$Y$}}}

\newcommand{\aalpha}{{\mbox{\boldmath$\alpha$}}}
\newcommand{\bbeta}{{\mbox{\boldmath$\beta$}}}
\newcommand{\ggamma}{{\mbox{\boldmath$\gamma$}}}

\newcommand{\eeta}{{\mbox{\boldmath$\eta$}}}

\newcommand{\ppi}{{\mbox{\boldmath$\pi$}}}

\newcommand{\ssigma}{{\mbox{\boldmath$\sigma$}}}

\newcommand{\vvarphi}{{\mbox{\boldmath$\varphi$}}}
\newcommand{\ppsi}{{\mbox{\boldmath$\psi$}}}
\newcommand{\pPsi}{{\mbox{\boldmath$\Psi$}}}
\newcommand{\ppPsi}{{\mbox{\scriptsize\boldmath$\Psi$}}}
\newcommand{\pPhi}{{\mbox{\boldmath$\Phi$}}}
\newcommand{\ppPhi}{{\mbox{\scriptsize\boldmath$\Phi$}}}

\newcommand{\xxi}{{\mbox{\boldmath$ \xi$}}}

\newcommand{\saalpha}{{\mbox{\scriptsize\boldmath$\alpha$}}}

\newcommand{\sggamma}{{\mbox{\scriptsize\boldmath$\gamma$}}}

\newcommand{\sssigma}{{\mbox{\scriptsize\boldmath$\sigma$}}}

\newcommand{\svvarphi}{{\mbox{\scriptsize\boldmath$\varphi$}}}
\newcommand{\sppsi}{{\mbox{\scriptsize\boldmath$\psi$}}}

\newcommand{\sfc}{{\sf c}}
\newcommand{\sfd}{{\sf d}}
\newcommand{\sfe}{{\sf e}}
\newcommand{\sfg}{{\sf g}}

\newcommand{\sfr}{{\sf r}}
\newcommand{\sfs}{{\sf s}}

\newcommand{\sfx}{{\sf x}}
\newcommand{\sfy}{{\sf y}}

\newcommand{\sfD}{{\sf D}}

\newcommand{\sfH}{{\sf H}}

\newcommand{\sfR}{{\sf R}}

\newcommand{\sfT}{{\sf T}}

\newcommand{\sfW}{{\sf W}}

\newcommand{\frf}{{\frak f}}
\newcommand{\frg}{{\frak g}}
\newcommand{\frh}{{\frak h}}

\newcommand{\fro}{{\frak o}}
\newcommand{\frp}{{\frak p}}
\newcommand{\frq}{{\frak q}}

\newcommand{\frs}{{\frak s}}

\newcommand{\fry}{{\frak y}}
\newcommand{\frz}{{\frak z}}

\newcommand{\frC}{{\frak C}}

\newcommand{\frF}{{\frak F}}

\newcommand{\frH}{{\frak H}}

\newcommand{\frK}{{\frak K}}

\newcommand{\frR}{{\frak R}}

\newcommand{\rmc}{{\mathrm c}}

\newcommand{\rme}{{\mathrm e}}

\newcommand{\rmh}{{\mathrm h}}

\newcommand{\rmr}{{\mathrm r}}

\newcommand{\rmx}{{\mathrm x}}
\newcommand{\rmy}{{\mathrm y}}
\newcommand{\rmz}{{\mathrm z}}

\newcommand{\rmB}{{\mathrm B}}
\newcommand{\rmC}{{\mathrm C}}
\newcommand{\rmD}{{\mathrm D}}

\newcommand{\rmI}{{\mathrm I}}
\newcommand{\rmL}{{\mathrm L}}
\newcommand{\rmM}{{\mathrm M}}

\newcommand{\Kliminf}{K\kern-3pt-\kern-2pt\mathop{\rm
lim\,inf}\limits}  
\newcommand{\Klimsup}{K\kern-3pt-\kern-2pt\mathop{\rm lim\,sup}\limits}  
\newcommand{\supp}{\mathop{\rm supp}\nolimits}   

\newcommand{\Lip}{\mathop{\rm Lip}\nolimits}          
\newcommand{\interior}{\mathop{\rm int}\nolimits}   
\renewcommand{\d}{{\mathrm d}}
\newcommand{\dt}{{\d t}}

\newcommand{\restr}[1]{\lower3pt\hbox{$|_{#1}$}}
\newcommand{\topref}[2]{\stackrel{\eqref{#1}}#2}
\newcommand{\Leb}[1]{{\mathscr L}^{#1}}      

\newcommand{\down}{\downarrow}              
\newcommand{\up}{\uparrow}

\newcommand{\eps}{\varepsilon}  
\newcommand{\nchi}{{\raise.3ex\hbox{$\chi$}}}
\newcommand{\Rd}{{\R^d}}

\newcommand{\forevery}{\text{for every }}

\def\qed{\ifmmode 
  \else \leavevmode\unskip\penalty9999 \hbox{}\nobreak\hfill
  \fi               
    \qquad           \hbox{\hskip.5em $\square$
                \hskip.1em}}

\newcommand{\nc}{\normalcolor}

\newcommand{\EEE}{\normalcolor\normalsize} 

\newcommand{\PART}[1]{\vspace{0.5em}%
\section*{\Large\bfseries #1}%
\addcontentsline{toc}{part}{\bfseries #1}}

\newcommand{\dom}[1]{\mathrm{Dom}(#1)}
\newcommand{\HK}{\mathsf{H\kern-3pt K}}
\newcommand{\BL}{\mathsf{B\kern-.5pt L}}
\newcommand{\GHK}{\mathsf{G\kern-2pt H\kern-3pt K}}
\newcommand{\Hell}{\mathsf{Hell}}
\newcommand{\LET}{\mathsf{L\kern-2pt E\kern-2pt T}}
\newcommand{\ETint}{\mathscr{E}}
\newcommand{\ET}{\mathsf{E\kern-1pt T}}
\newcommand{\OptET}{\mathrm{Opt}_{\ET}}
\newcommand{\OptHK}{\mathrm{Opt}_{\HK}}
\newcommand{\OptLET}{\mathrm{Opt}_{\LET}}
\newcommand{\Fstar}{F^*}
\newcommand{\Gstar}{F^\circ}

\newcommand{\LE}
{L\kern-3pt E}
\newcommand{\rec}[1]{{#1'_\infty}}
\newcommand{\asympt}[1]{{\mathrm{aff} {#1}_\infty}}
\newcommand{\derzero}[1]{{#1_0'}}
\newcommand{\mass}[1]{{m_{#1}}}
\newcommand{\res}{\mathop{\hbox{\vrule height 7pt width .5pt depth 0pt
\vrule height .5pt width 6pt depth 0pt}}\nolimits}
\newcommand{\MP}{H} 

\newcommand{\BorelSets}[1]{\cB(#1)}

\newcommand{\LSC}{\mathrm{LSC}}
\newcommand{\USC}{\mathrm{USC}}

\newcommand{\Cdual}[2]{\ifthenelse{\equal{#1}{\rmL}}{\tilde\Sigma}{\Sigma}}

\DeclareMathOperator*{\infp}{inf\vphantom{p}}

\newcommand{\FH}{R}
\newcommand{\FHstar}{R^*}
\newcommand{\FHH}{\RR}
\newcommand{\PE}{U}

\newcommand{\PEstar}{U^*}
\newcommand{\Cpsi}[1]{\pPsi}%
\newcommand{\cCpsi}[1]{\ppPsi}%
\newcommand{\Cphi}[1]{\pPhi}%

\newcommand{\cCphi}[1]{\ppPhi}%
\newcommand{\alc}[2]{|\rmD_{#1} #2|_{a}}
\newcommand{\alcs}[2]{|\rmD_{#1} #2|^2_{a}}
\newcommand{\rp}{r}
\newcommand{\pY}{Y}
\newcommand{\tY}{\frC}
\newcommand{\cball}[1]{\frC[#1]}
\newcommand{\Xext}{X} 
\newcommand{\xext}{\bar x} 
\newcommand{\Yext}{\pY} 
\newcommand{\tyY}{\boldsymbol{\mathfrak C}} 
\renewcommand{\frq}{\sfy}
\newcommand{\frqext}{\sfy}
\newcommand{\ty}{\fry} 
\newcommand{\py}{y}
\newcommand{\tyy}{\boldsymbol\ty}
\newcommand{\taalpha}{\aalpha}
\newcommand{\paalpha}{\bar\aalpha}
\newcommand{\frpd}{\boldsymbol\frp}
\newcommand{\frqd}{\boldsymbol\sfy}

\newcommand{\tnu}{\nu}
\newcommand{\pnu}{\bar\nu}
\newcommand{\kp}{\pi/2}
\newcommand{\tG}{\boldsymbol{\mathfrak G}} 
\newcommand{\elli}{g}
\newcommand{\sfdg}{\sfg}
\newcommand{\dil}[2]{\mathrm{dil}_{#1}(#2)}
\newcommand{\dilp}[3]{\mathrm{dil}_{#1,#2}(#3)}

\newcommand{\sfdp}{\sfd_\pi}
\newcommand{\sfdpt}{\sfd_{\pi/2}}
\newcommand{\sfdc}{\sfd_\frC}
\newcommand{\asfdc}[1]{\sfd_{#1,\frC}}
\newcommand{\tsfdc}{\sfd_{\pi/2,\frC}}
\newcommand{\chm}[2]{\frh^{
    \ifthenelse{\equal{#1}{}}{}{#1}}
_{#2}}
\newcommand{\hm}[2]{\rmh^{
    \ifthenelse{\equal{#1}{}}{}{#1}}
_{#2}}
\newcommand{\cHM}[3]{\frH^{
    \ifthenelse{\equal{#1}{}}{1}{\kern0.5pt #1} 
  }_{=} (#2,#3)}
\newcommand{\cHMle}[3]{\frH^{
    \ifthenelse{\equal{#1}{}}{1}{\kern0.5pt #1} 
  }_\le(#2,#3)}
\newcommand{\HM}[3]{\cH^{
    \ifthenelse{\equal{#1}{}}{1}{\kern0.5pt #1} 
  }_{=} (#2,#3)}
\newcommand{\HMle}[3]{\cH^{
    \ifthenelse{\equal{#1}{}}{1}{\kern0.5pt #1} 
  }_\le(#2,#3)}
\newcommand{\tMPc}[3]{\tilde\MP
    \ifthenelse
    {
      \equal{#3}{}
    }
    {}
    {_{#3}}
    (#1,#2)
  }
\newcommand{\MPc}[3]{\MP
    \ifthenelse
    {
      \equal{#3}{}
    }
    {}
    {_{#3}}
    (#1,#2)
  }
\newcommand{\MPH}[4]{\MP
  (#1,#2;#3,#4)
  }

\newcommand{\cMp}[1]{\cM_{#1}}
\newcommand{\cPp}[1]{\cP_{#1}}
\newcommand{\Wc}{\sfW_{\kern-1pt \sfdc}}
\newcommand{\Wd}{\sfW_{\kern-1pt \sfd}}
\newcommand{\tens}[2]{{#1}^{\otimes #2}}
\newcommand{\pdyY}[1]{\tens \frC{#1}}
\newcommand{\timesc}{\times_\tY}

\newcommand{\measu}[1]{{#1}_I}
\newcommand{\hatmeasu}[1]{\hat{#1}_I}
\newcommand{\eval}{\sfe}
\newcommand{\HJ}[2]{\PP_{#2}}
\newcommand{\HJC}[2]{\QQ_{#2}}

\renewcommand{\r}{s}
\newcommand{\s}{r}

\newcommand{\ldom}[1]{\r^-_{#1}}
\newcommand{\rdom}[1]{\r^+_{#1}}

\renewcommand{\soo}{\fro}
\newcommand{\opz}{\oplus_o}
\newcommand{\pz}{+_o}

\newcommand{\AC}{{\mathrm {AC}}}
\newcommand{\tAC}{\widetilde{\mathrm {AC}}}
\newcommand{\tACp}[1]{\widetilde{\mathrm {AC}}\vphantom{\rmC}^{#1}}
\newcommand{\trmC}{\widetilde\rmC}
\newcommand{\Geo}[1]{\mathrm{Geo}(#1)}
\newcommand{\sLE}{{\mathsf{L\!E}}}

\newcommand{\AAA}{}
\newcommand{\WWW}{} 

\renewcommand{\EEE}{\normalcolor}

\title{Optimal Entropy-Transport problems and\\ 
  a new Hellinger-Kantorovich distance\\ between positive measures}
\begin{document}

\author{
Matthias Liero 
\thanks{Weierstra\ss-Institut f\"ur Angewandte Analysis
  und Stochastik, Berlin;  email:
  \textsf{liero@wias-berlin.de}.
  Partially supported by the Einstein Stiftung Berlin via the
  ECMath/\textsc{Matheon} project SE2.} 
, Alexander Mielke \thanks{Humboldt-Universit\"at zu Berlin;
  email:
  \textsf{mielke@wias-berlin.de},
  Partially supported by DFG via project
  C5 within CRC 1114 (Scaling cascades in complex systems) and by
  ERC via AdG. 267802 AnaMultiScale.}
, Giuseppe Savar\'e\
 \thanks{Universit\`a di Pavia; email:
 \textsf{giuseppe.savare@unipv.it}. Partially supported by
 PRIN10/11 grant from MIUR for the project \emph{Calculus of
   Variations} and by IMATI-CNR.
}
 }

\date{ January 8, 2016}

\maketitle
\renewcommand{\ss}{{\mbox{\boldmath$s$}}}

\begin{abstract} 
  We develop a full theory for the new class of \emph{Optimal
    Entropy-Transport problems} between nonnegative and finite Radon
  measures in general topological spaces.

  They arise quite naturally by relaxing the marginal constraints
  typical of Optimal Transport problems: given a couple of finite
  measures (with possibly different total mass), one looks for
  minimizers of the sum of a linear transport functional and two
  convex entropy functionals, that quantify in some way the deviation
  of the marginals of the transport plan from the assigned measures.

  As a powerful application of this theory, we study the particular
  case of \emph{Logarithmic} Entropy-Transport problems and introduce
  the new \emph{Hellinger-Kantorovich distance between measures in
    metric spaces}.

  The striking connection between these two seemingly far topics
  allows for a deep analysis of the geometric properties of the new
  geodesic distance, which lies somehow between the well-known
  Hellinger-Kakutani and Kantorovich-Wasserstein distances.
\end{abstract}

{\small\tableofcontents}

\section{Introduction}
The aim of the present paper is twofold: In Part I we
develop a full theory of the new class of \emph{Optimal
  Entropy-Transport problems} between nonnegative and finite Radon
measures in general topological spaces. As a powerful application of
this theory, in Part II we study the particular case of
\emph{Logarithmic} Entropy-Transport problems and introduce the new
\emph{Hellinger-Kantorovich $(\HK)$ distance between measures in
  metric spaces}.  The striking connection between these two seemingly
far topics is our main focus, and it paves the way for a
beautiful and deep analysis of the geometric properties of the
geodesic $\HK$ distance, which (as our proposed name suggests)
can be understood as an inf-convolution of the well-known
Hellinger-Kakutani and the Kantorovich-Wasserstein distances. In fact,
our approach to the theory was opposite: in trying to characterize
$\HK$, we were first led to the Logarithmic Entropy-Transport problem,
see Section \ref{sec:Devel}.
\paragraph{From Transport to Entropy-Transport problems.}
In the classical Kantorovich formulation, 
Optimal Transport problems
\cite{Rachev-Ruschendorf98,Villani03,Ambrosio-Gigli-Savare08,Villani09}
deal with minimization of a linear cost functional
\begin{equation}
  \label{eq:193}
  \CC(\ggamma)=\int_{X_1\times
    X_2}\sfc(x_1,x_2)\,\d\ggamma(x_1,x_2),\quad
  \sfc:X_1\times X_2\to \R,
\end{equation}
among all the transport plans, i.e.~probability measures in
$\cP(X_1\times X_2)$, 
$\ggamma$ whose marginals $\mu_i=\pi^i_\sharp \ggamma\in \cP(X_i)$
are prescribed. Typically, $X_1,X_2$ are Polish spaces,
$\mu_i$ are given Borel measures (but the
case of Radon measures in Hausdorff topological spaces
has also been considered, see \cite{Kellerer84,Rachev-Ruschendorf98}),
the cost function $\sfc$ is a lower semicontinuous
(or even Borel)
function, possibly assuming the value $+\infty$, and
$\pi^i(x_1,x_2)=x_i$ are the projections on the $i$-th coordinate, so
that 
\begin{equation}
  \label{eq:197}
  \pi^i_\sharp\ggamma=\mu_i\quad\Leftrightarrow\quad
  \mu_1(A_1)=\ggamma_1(A_1\times X_2),\ 
  \mu_2(A_2)=\ggamma_1(X_1\times A_2)
  \quad\forevery A_i\in X_i.
\end{equation}
Starting from the pioneering work of Kantorovich, 
an impressive theory has been developed in the last two decades:
from one side, typical intrinsic questions of linear programming problems
concerning duality, optimality, 
uniqueness and structural properties of optimal transport plans
have been addressed and fully analyzed. 
In a parallel way, this rich general theory has been applied 
to many challenging problems in a variety of fields
(probability and statistics, functional analysis, 
PDEs, Riemannian geometry, nonsmooth analysis in metric spaces,  just to mention a few of them:
since it is impossible here to give an even 
partial account of the main contributions, 
we refer to the books \cite{Villani09,Santambrogio15}
for a more detailed overview and a complete list of references).

The class of \textbf{\em Entropy-Transport problems}, we are going to
study, arises quite naturally if one tries to relax the marginal
constraints $\pi^i_\sharp\ggamma=\mu_i$ by introducing suitable
penalizing functionals $\FF_i$, that quantify in some way the
deviation from $\mu_i$ of the marginals $\gamma_i:=\pi^i_\sharp
\ggamma$ of $\ggamma$.  In this paper we consider the general case of
integral functionals (also called \emph{Csisz\`ar $f$-divergences}
\cite{Csiszar67}) of the form
\begin{equation}
  \label{eq:195}
  \FF_i(\gamma_i|\mu_i):=
  \int_{X_i}F_i(\sigma_i(x_i))\,\d\mu_i+
  \gamma_i^\perp(X_i),\quad
  \sigma_i=\frac{\d\gamma_i}{\d\mu_i},
  \quad
  \gamma_i=\sigma_i\mu_i+\gamma_i^\perp,
\end{equation}
where $F_i:[0,+\infty)\to[0,+\infty]$ are given convex entropy
functions, like for the logarithmic or power-like entropies
\begin{equation}
  \label{eq:199}
  \begin{aligned}
    &\PE_p(s):=\frac{1}{p(p-1)}\big(s^p-p(s-1)+1\big),\quad p\in
    \R\setminus\{0,1\},\\
    &\PE_0(s):=s-1-\log s, \quad
    \PE_1(s):=s\log s-s+1,
  \end{aligned}
\end{equation}
or for the total variation functional corresponding to the nonsmooth entropy
$V(s):=|s-1|$, considered in \cite{Piccoli-Rossi14}.

Notice that the presence of the singular part $\gamma_i^\perp$ in 
the Lebesgue decomposition of $\gamma_i$ in \eqref{eq:195}
does not force $F_i(s)$ to be superlinear as $s\up+\infty$ 
and allows for all the exponents $p$ in \eqref{eq:199}.

Once a specific choice of entropies $F_i$ and of finite nonnegative
Radon measures $\mu_i\in \cM(X_i)$ is given, the Entropy-Transport
problem can be formulated as
\begin{equation}
  \label{eq:209}
  \ET(\mu_1,\mu_2):=\inf\Big\{\EE(\ggamma|\mu_1,\mu_2):\ggamma
  \in \cM(X_1\times X_2)\Big\},
\end{equation}
where $\EE$ is the convex functional
\begin{equation}
  \label{eq:207}
  \EE(\ggamma|\mu_1,\mu_2):=
  \FF_1(\gamma_1|\mu_1)+
  \FF_2(\gamma_2|\mu_2)+\int_{X_1\times X_2}\sfc(x_1,x_2)\,\d\ggamma.
\end{equation}
Notice that the entropic formulation allows for measures
$\mu_1,\mu_2$ and $\ggamma$ with possibly different total mass.

The flexibility in the choice of the entropy functions $F_i$ 
(which may also take the value $+\infty$) covers a wide spectrum of
situations (see Section \ref{ex:1} for various examples)
and in particular guarantees that \eqref{eq:209} is a real
generalization
of the classical optimal transport problem, which 
can be recovered as a particular case of \eqref{eq:207} when
$F_i(s)$ is the indicator function of $\{1\}$
(i.e.~$F_i(s)$ always
takes the value $+\infty$ with the only exception of $s=1$, where
it vanishes). 

Since we think that the structure \eqref{eq:207}
of Entropy-Transport problems will lead to new and 
interesting models and applications, we 
have tried to establish
their basic theory in the greatest generality, by pursuing
the same line of development of Transport problems:
in particular we will obtain general results concerning existence, 
duality and optimality conditions.

Considering e.g.~the Logarithmic Entropy case, where $F_i(s)=s\log
s-(s-1)$, \textbf{\em the dual formulation of \eqref{eq:209}} is given
by
\begin{equation}
  \label{eq:210}
\begin{aligned}
  &\sfD(\mu_1,\mu_2):= \sup \Big\{
  \DD(\varphi_1,\varphi_2|\mu_1,\mu_2) \,:\ 
  \varphi_i:X_i\to \R,\
  \varphi_1(x_1)+\varphi_2(x_2)\le \sfc(x_1,x_2)\Big\},
\\
&\text{where } \DD(\varphi_1,\varphi_2|\mu_1,\mu_2):= \int_{X_1}\kern-6pt\big(1-\rme^{-\varphi_1}\big)\,\d\mu_1+
  \int_{X_2}\kern-6pt\big(1-\rme^{-\varphi_2}\big)\,\d\mu_2,
\end{aligned}
\end{equation}
where one can immediately recognize the 
same convex constraint 
of Transport problems:
the couple  of dual potentials $\varphi_i$
should satisfy $\varphi_1\oplus\varphi_2\le \sfc$ on
$X_1\times X_2$.
The main difference is due to the concavity of 
the objective functional
\begin{displaymath}
  (\varphi_1,\varphi_2)\mapsto \int_{X_1}\big(1-\rme^{-\varphi_1}\big)\,\d\mu_1+
  \int_{X_2}\big(1-\rme^{-\varphi_2}\big)\,\d\mu_2,
\end{displaymath}
whose form can be explicitly calculated in terms of the 
Lagrangian conjugates $F_i^*$ of the entropy functions.
The change of variables $\psi_i:=1-\rme^{-\varphi_i}$
transforms \eqref{eq:210} in the equivalent problem of maximizing the linear
functional
\begin{equation}
  \label{eq:213}
  (\psi_1,\psi_2)\mapsto \sum_i\int_{X_1}\psi_1\,\d\mu_1+
  \int_{X_2}\psi_2\,\d\mu_2
\end{equation}
on the more complicated convex set
\begin{equation}
  \label{eq:212}
  \Big\{(\psi_1,\psi_2):\psi_i:X_i\to (-\infty,1),\quad
  (1-\psi_1(x_1))(1-\psi_2(x_2))\ge \rme^{-\sfc(x_1,x_2)}\Big\}.
\end{equation}
We will calculate the dual problem for every choice of $F_i$ and
show that its value always coincide with $\ET(\mu_1,\mu_2)$.
The dual problem also provides \textbf{\em optimality conditions,}
that involve the couple of potentials $(\varphi_1,\varphi_2)$, the
support of the optimal plan $\ggamma$
and the densities $\sigma_i$ of its marginals $\gamma_i$ w.r.t.~$\mu_i$.
For the Logarithmic Entropy Transport problem above, 
they read as
\begin{equation}
  \label{eq:214}
  \begin{gathered}
    \sigma_i>0,\ \varphi_i=-\log \sigma_i\quad\mu_i\text{ a.e.~in
    }X_i,\\
    \varphi_1\oplus\varphi_2\le \sfc\quad\text{in }X_1\times X_2,\quad
    \varphi_1\oplus\varphi_2= \sfc\quad\text{$\ggamma$-a.e.~in
    }X_1\times X_2,
  \end{gathered}
\end{equation}
and they are necessary and sufficient for optimality. 

The study of optimality conditions reveals a different behavior
between pure transport problems and the other entropic ones.  In
particular, the $\sfc$-cyclical monotonicity of the optimal plan
$\ggamma$ (which is still satisfied in the entropic case) does not
play a crucial role in the construction of the potentials $\varphi_i$.
When $F_i(0)$ are finite (as in the logarithmic case) it is possible
to obtain a general existence result of (generalized) optimal
potentials even when $\sfc$ takes the value $+\infty$.

A crucial feature of Entropy-Transport problems
(which is not shared by the pure transport ones)
concerns a \textbf{\em third ``homogeneous'' formulation}, which 
exhibits new and unexpected properties. 
It is related to the \textbf{\em $1$-homogeneous Marginal Perspective function}
\begin{equation}
  \label{eq:215}
  H(x_1,r_1;x_2,r_2):=\inf_{\theta>0}\Big(r_1F_1(\theta/r_1)+r_2F_2(\theta/r_2)+
  \theta\sfc(x_1,x_2)\Big)
\end{equation}
and to the corresponding integral functional
\begin{equation}
  \label{eq:219}
  \HH(\mu_1,\mu_2|\ggamma):=
  \int_{X_1\times     X_2}
  H(x_1,\varrho_1(x_1);x_2,\varrho_2(x_2))\,\d\ggamma+\sum_iF_i(0)\mu_i^\perp(X_i),\ 
  \varrho_i:=\frac{\d\mu_i}{\d\gamma_i},
\end{equation}
where $\mu_i=\varrho_i\gamma_i+\mu_i^\perp$
is the ``reverse'' Lebesgue decomposition of $\mu_i$ w.r.t.~the
marginals $\gamma_i$ of $\ggamma$.
We will prove that
\begin{equation}
  \label{eq:220}
  \ET(\mu_1,\mu_2)=\min\Big\{ \HH(\mu_1,\mu_2|\ggamma):\ggamma\in
  \cM(X_1\times X_2)\Big\}
\end{equation}
with a precise relation between optimal plans.
In the Logarithmic Entropy case $F_i(s)=s\log s-(s-1)$ 
the marginal perspective function $H$ takes the particular form
\begin{equation}
  \label{eq:221}
  H(x_1,r_1;x_2,r_2)=r_1+r_2-2\sqrt{r_1\,r_2}\,\rme^{-\sfc(x_1,x_2)/2},
\end{equation}
which will be the starting point for understanding
the deep connection with the Hellinger-Kantorovich distance.
Notice that in the case when $X_1=X_2$ and $\sfc$ is the singular cost
\begin{equation}
  \label{eq:38}
  \sfc(x_1,x_2):=
  \begin{cases}
    0&\text{if }x_1=x_2,\\
    +\infty&\text{otherwise},
  \end{cases}
\end{equation}
\eqref{eq:220} provides an equivalent formulation of the
Hellinger-Kakutani distance \cite{Hellinger09,Kakutani48},
see also Example E.5 in Section \ref{ex:1}.

Other choices, still in the simple class \eqref{eq:199}, give raise to
``transport'' versions of well known functionals (see
e.g.~\cite{Liese-Vajda06} for a systematic presentation): starting
from the reversed entropies $F_i(s)=s-1-\log s$ one gets
\begin{equation}
  \label{eq:32}
  H(x_1,r_1;x_2,r_2)=\s_1\log \s_1+\s
    _2\log \s _2 -(\s _1+\s_2)\log\Big(\frac {\s _1+\s _2}{2+\sfc(x_1,x_2)}\Big),
\end{equation}
which in the extreme case \eqref{eq:38} reduces to the
\emph{Jensen-Shannon divergence} \cite{Lin91}, a squared distance
between measures derived from the celebrated \emph{Kullback-Leibler
  divergence} \cite{Kullback-Leibler51}. The quadratic entropy
$F_i(s)=\frac 12(s-1)^2$ produces
\begin{equation}
  \label{eq:78}
  H(x_1,r_1;x_2,r_2)=\frac1{2(\s _1+\s _2)}\Big((\s
  _1-\s _2)^2+h(\sfc(x_1,x_2))\s _1\s _2\Big),
\end{equation}
where $h(c)=c(4-c)$ if $0\le c\le 2$ and $4 $ if $c\ge 2$:
Equation \eqref{eq:78} can be seen as the transport variant of 
the \emph{triangular discrimination} (also called symmetric
$\nchi^2$-measure), based on the \emph{Pearson} $\nchi^2$-divergence,
and still obtained by \eqref{eq:219} when $\sfc$ has the form
\eqref{eq:38}.

Also nonsmooth cases, as for $V(s)=|s-1|$ associated to the
\emph{total variation distance} (or nonsymmetric choices of $F_i$) can
be covered by the general theory. In the case of $F_i(s)=V(s)$ 
the marginal perspective function is
\begin{displaymath}
   H(x_1,r_1;x_2,r_2)=\s _1+\s _2
   -(2-\sfc(x_1,x_2))_+(\s _1\land\s _2)=
   |\s_2-\s_1|+(\sfc(x_1,x_2)\land 2) (\s _1\land\s _2);
\end{displaymath}
when $X_1=X_2=\R^d$ with $\sfc(x_1,x_2):=|x_1-x_2|$
we recover the \emph{generalized Wasserstein distance} $W^{1,1}_1$
introduced and studied by
\cite{Piccoli-Rossi14}; 
it provides an equivalent variational characterization of 
the flat metric \cite{Piccoli-Rossi14preprint}. 

However, because of our original
motivation (see Section \ref{sec:Devel}), Part II will focus on the
case of the logarithmic entropy $F_i=U_1$, where $H$ is given by
\eqref{eq:221}. We will exploit its relevant geometric applications,
reserving the other examples for future investigations.

\paragraph{From the Kantorovich-Wasserstein distance to the Hellinger-Kantorovich distance.}
From the analytic-geometric point of view,
one of the most interesting cases of transport problems
occurs when $X_1=X_2=X$ coincide and 
the cost functional $\CC$  is induced
by a distance $\sfd$ on $X$: in the quadratic case,
the minimum value of \eqref{eq:193} for
given measures $\mu_1,\mu_2$
in the space $\cP_2(X)$ of probability measures with finite quadratic moment
defines the so called $L^2$-Kantorovich-Wasserstein distance
\begin{equation}
  \label{eq:194}
  \Wd^2(\mu_1,\mu_2):=\inf\Big\{\int
  \sfd^2(x_1,x_2)\,\d\ggamma(x_1,x_2):
  \ggamma\in \cP(X\times X),\ \pi^i_\sharp\ggamma=\mu_i\Big\},
\end{equation}
which metrizes the weak convergence (with quadratic moments) of
probability measures. The metric space $(\cP_2(X),\Wd)$ inherits
many geometric features from the underlying $(X,\sfd)$ 
(as separability, completeness, length and geodesic properties,
positive curvature in the Alexandrov sense, see \cite{Ambrosio-Gigli-Savare08}).
Its dynamic characterization in terms of the continuity equation \cite{Benamou-Brenier00} 
and its dual formulation in terms of the Hopf-Lax formula 
and the corresponding (sub-)solutions of the Hamilton-Jacobi equation 
\cite{Otto-Villani00} lie at the core of the applications to gradient
flows and
partial differential equations of diffusion type
\cite{Ambrosio-Gigli-Savare08}.
Finally, the behavior of entropy functionals as \eqref{eq:195} along
geodesics in $(\cP_2(X),\Wd)$
\cite{McCann97,Otto-Villani00,Cordero-McCann-Schmuckenschlager01}
encodes a valuable geometric information,
with relevant applications to Riemannian geometry 
and to the recent theory of metric-measure spaces with
Ricci curvature bounded from below
\cite{Sturm06I,Sturm06II,Lott-Villani09,Ambrosio-Gigli-Savare14,AGS14b,AGS15,Erbar-Kuwada-Sturm13}.

It has been a challenging question to find a corresponding distance 
(enjoying analogous deep geometric properties)
between finite positive Borel measures with arbitrary mass in
$\cM(X)$.
In the present paper we will show that by
choosing the particular cost function
\begin{equation}
  \label{eq:230}
  \sfc(x_1,x_2):=\ell(\sfd(x_1,x_2)),\quad
  \text{where}\quad
  \ell(\sfd):=\begin{cases}
    -\log\big(\cos^2(\sfd)\big)
    &\text{if }\sfd
    <\pi/2,\\
    +\infty&\text{otherwise},
  \end{cases}
\end{equation}
the corresponding Logarithmic-Entropy Transport problem
\begin{equation}
  \label{eq:231}
  \LET(\mu_1,\mu_2):=\min_{\sggamma\in \cM(\sxX)} \sum_i\int_X
    \big(\sigma_i\log\sigma_i-\sigma_i+1\big)
    \,\d\mu_i+
    \int_{X^2} \ell\big(\sfd(x_1,x_2)\big)\,\d\ggamma,\quad
    \sigma_i=\frac{\d\gamma_i}{\d\mu_i},
\end{equation}
coincides with 
a (squared) distance in $\cM(X)$
(which we will call \emph{Hellinger-Kantorovich distance} and denote by 
$\HK$\nc)
that
can play the same fundamental role like the
Kantorovich-Wasserstein distance for $\cP_2(X)$. 

Here is a schematic list of our main results:
\begin{enumerate}[(i)]
\item The representation \eqref{eq:220} based on the Marginal Perspective
  function \eqref{eq:221} yields 
  \begin{equation}
    \label{eq:232}
    \LET(\mu_1,\mu_2)=
    \min\Big\{\int \Big(\varrho_1+\varrho_2-2 \varrho_1\varrho_2
    \cos(\sfd(x_1,x_2)\land \pi/2)\Big)\,\d\ggamma:
    \varrho_i=\frac{\d\mu_i}{\d\gamma_i}\Big\}.
  \end{equation}
\item By performing the rescaling $r_i\mapsto r_i^2$ 
  we realize that the function $H(x_1,r_1^2;x_2,r_2^2)$ is strictly
  related to the squared (semi)-distance
  \begin{equation}
    \label{eq:276}
    \sfdc^2(x_1,r_1;x_2,r_2):=
    r_1^2+r_2^2-2r_1r_2\cos(\sfd(x_1,x_2)\land\pi),\quad
    (x_i,r_i)\in X\times\R_+
  \end{equation}
  which is the so-called \emph{cone distance} 
  in the metric cone $\tY$ over $X$, cf.\
  \cite{Burago-Burago-Ivanov01}. The latter 
  is 
  the quotient space of $X\times \R_+$ obtained by collapsing 
  all the points $(x,0)$, $x\in X$, in a single point $\fro$, called
  the vertex of the cone. We introduce the notion of 
  ``$2$-homogeneous marginal'' 
  \begin{equation}
    \label{eq:285}
    \mu=\chm2{}\alpha:=\pi^{x}_\sharp(r^2\alpha),\quad
    \int_X \zeta(x)\,\d\mu=\int_\tY\zeta(x)r^2\,\d\alpha(x,r)
    \quad\forevery \zeta\in \rmC_b(X),
  \end{equation}
  to ``project'' measures $\alpha\in \cM(\tY)$ 
  on measures $\mu\in \cM(X)$. Conversely, there are many ways
  to ``lift'' a measure $\mu\in \cM(X)$ to $\alpha\in \cM(\tY)$ 
  (e.g.~by taking $\alpha:=\mu\otimes \delta_1$).
  The Hellinger-Kantorovich distance $\HK(\mu_1,\mu_2)$ 
  can then be defined by taking the best Kantorovich-Wasserstein
  distance between all the possible lifts of $\mu_1,\mu_2$ 
  in $\cP_2(\tY)$, i.e. 
  \begin{equation}
    \label{eq:286}
    \HK(\mu_1,\mu_2)=
    \min\Big\{\Wc(\alpha_1,\alpha_2):\alpha_i\in \cP_2(\tY),\ 
    \chm2{}\alpha_i=\mu_i\Big\}.
  \end{equation}
  It turns out that (the square of)
  \eqref{eq:286} yields an equivalent variational representation of 
  the $\LET$ functional. 
  In particular, \eqref{eq:286} shows that in the case of concentrated
  measures
  \begin{equation}
    \label{eq:344}
    \LET(a_1\delta_{x_1},a_2\delta_{x_2})=
    \HK^2(a_1\delta_{x_1},a_2\delta_{x_2})=
    \sfdc^2(x_1,a_1;x_2,a_2).
  \end{equation}
  Notice that \eqref{eq:286} resembles the very definition
  \eqref{eq:194} of the Kantorovich-Wasserstein distance, where now
  the role of the marginals $\pi^i_\sharp$ is replaced by the
  homogeneous marginals $\chm2{}$.  It is a nontrivial part of the
  equivalence statement to check that the difference between the
  cut-off thresholds ($\pi/2$ in \eqref{eq:232} and $\pi$ in
  \eqref{eq:276} does not affect the identity $\LET=\HK^2$.

\item By refining the representation formula \eqref{eq:286} by a
  suitable rescaling and gluing technique we can prove that
  $(\cM(X),\HK)$ is a geodesic metric space, a property that it is
  absolutely not obvious from the $\LET$-representation and depends on
  a subtle interplay of the entropy functions $F_i(\sigma)=\sigma
  \log\sigma - \sigma +1$ and the cost function $\sfc$ from
  \eqref{eq:230}. We show that the metric induces the weak convergence
  of measures in duality with bounded and continuous functions, thus
  it is topologically equivalent to the flat or Bounded Lipschitz
  distance \cite[Sec.~11.3]{Dudley02}, see also \cite[Thm.\ 3]{KMV15}.
  It also inherits the separability, completeness, length and geodesic
  properties from the correspondent ones of the underlying space
  $(X,\sfd)$.  On top of that, we will prove a precise superposition
  principle (in the same spirit of the Kantorovich-Wasserstein one
  \cite[Sect.8]{Ambrosio-Gigli-Savare08},\cite{Lisini07}) for general
  absolutely continuous curves in $(\cM(X),\HK)$ in terms of dynamic
  plans in $\tY$: as a byproduct, we can give a precise
  characterization of absolutely continuous curves and geodesics as
  homogeneous marginals of corresponding curves in $(\cP_2(\tY),\Wc)$.
  An interesting consequence of these results concerns the lower
  curvature bound of $(\cM(X),\HK)$ in the sense of Alexandrov: it is
  a positively curved space if and only if $(X,\sfd)$ is a geodesic
  space with curvature $\ge1$.

\item The dual formulation of the $\LET$ problem provides a dual
  characterization of $\HK$, viz.\
  \begin{equation}
      \label{eq:293}
      \frac 12\HK^2(\mu_1,\mu_2)=
       \sup \Big\{
      \int\HJ{\pi/2}1\xi\,\d\mu_2-\int\xi\,\d\mu_1:
      \xi\in \Lip_{b}(X),\ \inf_X \xi>-1/2\Big\},
  \end{equation}
  where $(\PP_t)_{0\le t\le 1}$ is given by the inf-convolution
  \begin{equation}
    \nonumber 
    \HJ{\pi/2}{t}\xi(x):=\inf_{x'\in X}
    \frac{\xi(x')}{1+2t\xi(x')}+
    \frac{\sin^2(\sfd_{\pi/2}(x,x'))}{2+4t\xi(x')}
    =\inf_{x'\in X}
    \frac 1t\Big(1-\frac{\cos^2(\sfd_{\pi/2}(x,x'))}{ 1+2t\xi(x')}\Big).
  \end{equation}

\item 
  By exploiting the Hopf-Lax representation formula for the
  Hamilton-Jacobi equation in $\tY$, we will show that for arbitrary
  initial data $\xi\in \Lip_b(X)$ with $\inf\xi>-1/2$ the function 
  $\xi_t:=\PP_t\xi$ is a subsolution (a solution, if $(X,\sfd)$ is a
  length space) of
  \begin{displaymath}
      \partial^+_t \xi_t(x)+\frac 12|\rmD_X\xi_t|^2(x)+2\xi_t^2(x)\le
      0\quad
      \text{pointwise in }X\times (0,1).
  \end{displaymath}
  If $(X,\sfd)$ is a length space we thus obtain the characterization
  \begin{equation}
      \label{eq:321}
    \begin{aligned}
          \frac 12\HK^2(\mu_0,\mu_1)=
          \sup\Big\{&\int_X\xi_1\,\d\mu_1-\int_0\xi_0\,\d\mu_0:
          \xi\in \rmC^k([0,1];\Lip_{b}(X)), \\&
          \partial_t\xi_t(x)+\frac 12|\rmD_X
          \xi_t|^2(x)+2\xi_t^2(x)\le 0 \quad \text{in
          }X\times(0,1)\Big\},
      \end{aligned}
  \end{equation}
  which reproduces, at the level of $\HK$, the nice link between $\Wd$
  and Hamilton-Jacobi equations. One of the direct applications of
  \eqref{eq:321} is a sharp contraction property w.r.t.~$\HK$ for the
  Heat flow in $\mathrm{RCD}(0,\infty)$ metric measure spaces (and
  therefore in every Riemannian manifold with nonnegative Ricci
  curvature).
\item \eqref{eq:321} clarifies that the $\HK$ distance can be
  interpreted as a sort of inf-convolution between the Hellinger (in
  duality with solutions to the ODE $\partial_t \xi+2\xi_t^2=0$) and
  the Kantorovich-Wasserstein distance
  (in duality with (sub-)solutions to\\
  $\partial_t\xi_t(x)+\frac 12|\rmD_X \xi_t|^2(x)\le 0$). The
  Hellinger distance
    \begin{displaymath}
      \Hell^2(\mu_1,\mu_2)=\int_X\big(\sqrt{\varrho_1}-\sqrt{\varrho_2}\big)^2\,\d\gamma
      ,\quad
      \mu_i=\varrho_i\gamma,
    \end{displaymath}
    corresponds to the $\HK$ functional generated by the discrete
    distance ($\sfd(x_1,x_2)=\pi/2$ if $x_1\neq x_2$). We will prove
    that 
    \begin{gather*}
      \HK(\mu_1,\mu_2)\le \Hell(\mu_1,\mu_2),\quad
      \HK(\mu_1,\mu_2)\le \Wd(\mu_1,\mu_2),\\
      \HK_{n\sfd}(\mu_1,\mu_2)\uparrow \Hell(\mu_1,\mu_2),\quad
      n\HK_{\sfd/n}\uparrow \Wd(\mu_1,\mu_2)\quad
      \text{as }n\uparrow\infty,
    \end{gather*}
    where $\HK_{n\sfd}$ (resp.~$\HK_{\sfd/n}$) is the
    $\HK$ distance induced by $n\sfd$ (resp.~$\sfd/n$).
  \item Combining the superposition principle and the duality with
    Hamilton-Jacobi equations, we eventually prove that $\HK$ admits
    an equivalent dynamic characterization ``\`a la Benamou-Brenier''
    \cite{Benamou-Brenier00,Dolbeault-Nazaret-Savare08b} (see also the
    recent \cite{KMV15}) in $X=\R^d$
      \begin{equation}
    \label{eq:177intro}
    \begin{aligned}
      \HK^2(\mu_0,\mu_1)=
      \min\Big\{&\int_0^1 \int \Big(|\vv_t|^2+\frac 14|w_t|^2\Big)\,\d\mu_t\,\d
      t:
      \mu\in \rmC([0,1];\cM(\R^d)), 
   \\
      &\mu_{t=i}=\mu_i, \ \partial_t \mu_t+\nabla{\cdot}(\vv_t\mu_t)=w_t\mu_t
      \text{ in }\DD'(\R^d\times (0,1)) \Big\}.
    \end{aligned}
  \end{equation}  
  Moreover, for the length space $X=\R^d$ a curve $[0,1]\ni
  t\mapsto \mu(t)$ is geodesic curve w.r.t.\ $\HK$ if and only if the
  coupled system
  \begin{equation}
   \label{eq:GeodSyst}
   \partial_t \mu_t+\nabla \cdot(\rmD_x \xi_t \mu_t)=4\xi_t\mu_t, \quad 
    \partial_t\xi_t + \frac12|\rmD_x\xi^2|^2 + 2 \xi_t^2=0 
  \end{equation}
  holds for a suitable solution $\xi_t=\HJ{\pi/2}{t}\xi_0$. The
  representation \eqref{eq:177intro} is the starting point for further
  investigations and examples, which we have collected in \cite{LMS15}.
\end{enumerate}

It is not superfluous to recall that the $\HK$ variational problem
is just one example in the realm of Entropy-Transport problems
and we think that other interesting applications can
arise by different choices of entropies and cost.
One of the simplest variation is to choose
the (seemingly more natural) 
quadratic cost function $\sfc(x_1,x_2):=\sfd^2(x_1,x_2)$
instead of the more ``exotic'' \eqref{eq:230}.
The resulting functional is still associated to a distance
expressed by
\begin{equation}
  \label{eq:351}
  \GHK^2(\mu_1,\mu_2):=
  \min\Big\{\int 
  \Big(\s_1^2+\s_2^2-2\s_1\s_2\exp(-\sfd^2(x_1,x_2)/2)\Big)\,\d\aalpha\Big\}
\end{equation}
where the minimum runs among  
all the plans $ \aalpha\in \cM(\tY\times \tY)$ 
such that $\chm2{}\pi^i_\sharp\aalpha=\mu_i$
(we propose the name ``Gaussian Hellinger-Kantorovich distance'').
If $(X,\sfd)$ is a complete, separable and length metric space,
$(\cM(X),\GHK)$ is a complete and separable metric space, 
inducing the weak topology as $\HK$.
However, it is not a length space in general, and we will show that
the length distance generated by $\GHK$ is precisely $\HK$.
\medskip

The plan of the paper is as follows. 

\paragraph{Part I} develops the general theory of Optimal
Entropy-Transport problems. Section \ref{sec:prel} collects some
preliminary material, in particular concerning the measure-theoretic
setting in arbitrary Hausdorff topological spaces (here we follow
\cite{Schwartz73}) and entropy functionals.  We devote some effort to
deal with general functionals (allowing a singular part in the
definition \eqref{eq:195}) in order to include entropies which may
have only linear growth. The extension to this general framework of
the duality theorem \ref{thm:duality} (well known in Polish
topologies) requires some care and the use of lower semicontinuous
test functions instead of continuous ones.

Section \ref{sec:ET} introduces the class of Entropy-Transport
problems, discussing same examples and proving a general existence
result
for optimal plans. The ``reverse'' formulation of Theorem 
\ref{thm:reverse-characterization}, though simple, 
justifies the importance to deal with the largest class of entropies
and will play a crucial role in Section \ref{sec:MP}.

Section \ref{sec:duality} is devoted to find the dual formulation,
to prove its equivalence with the primal problem 
(cf.\  Theorem \ref{thm:weak-duality}),
 to derive sharp optimality conditions
(cf.\ Theorem \ref{thm:joint-optimality})
and to prove the existence of optimal
potentials in a suitable generalized sense 
(cf.\ Theorem \ref{thm:pot-ex}).
The particular class of ``regular'' problems 
(where the results are richer) is also
studied with some details.

Section \ref{sec:MP} introduces the third formulation
\eqref{eq:219}
based on the marginal perspective function \eqref{eq:215}
and its ``homogeneous'' version (Section \ref{subsec:hom-marg}).
The proof of the equivalence with the previous formulations
is presented in Theorem \ref{thm:crucial} 
and Theorem \ref{thm:main-hom-marg}.
This part provides the crucial link for
the further development in the cone setting.

\paragraph*{Part II} 
is devoted to Logarithmic Entropy-Transport ($\LET$) problems (Section
\ref{sec:LET}) and to their applications to the Hellinger-Kantorovich
distance $\HK$ on $\cM(X)$.

The Hellinger-Kantorovich distance is introduced by the lifting
technique in the cone space in Section \ref{sec:cone}, where we try to
follow a presentation modeled on the standard one for the
Kantorovich-Wasserstein distance, independently from the results on
the $\LET$-problems.  After a brief recap on the cone geometry
(Section \ref{subsec:cone}) we discuss in some detail the crucial
notion of homogeneous marginals in Section \ref{subsec:RHM} and the
useful tightness conditions (Lemma \ref{le:compactnessH}) for plans
with prescribed homogeneous marginals.  Section \ref{subsec:HK}
introduces the definition of the $\HK$ distance and its basic
properties.  The crucial rescaling and gluing techniques are discussed
in Section \ref{subsec:triangle}: they lie at the core of the main
metric properties of $\HK$, leading to the proof of the triangle
inequality and to the characterizations of various metric and
topological properties in Section \ref{subsec:MT}.  The equivalence
with the $\LET$ formulation is the main achievement of Section
\ref{subsec:HKET} (Theorem \ref{thm:main-equivalence}), with
applications to the duality formula (Theorem \ref{thm:dualityHK}), to
the comparisons with the classical Hellinger and Kantorovich distances
(Section \ref{subsec:limiting}) and with the Gaussian
Hellinger-Kantorovich distance (Section \ref{subsec:GHK}).

The last Section of the paper collects various important properties of
$\HK$, that share a common ``dynamic'' flavor.  After a preliminary
discussion of absolutely continuous curves and geodesics in the cone
space $\tY$ in Section \ref{subsec:geodesic}, we derive the basic
superposition principle in Theorem \ref{thm:plan-representation}.
This is the cornerstone to obtain a precise characterization of
geodesics (Theorem \ref{thm:geoHK}), a sharp lower curvature bound in
the Alexandrov sense (Theorem \ref{thm:curvatureHK}) and to prove the
dynamic characterization \`a la Benamou-Brenier of Section
\ref{subsec:BB}.  The other powerful tool is provided by the duality
with subsolutions to the Hamilton-Jacobi equation (Theorem
\ref{thm:main-HKHJ}), which we derive after a preliminary
characterization of metric slopes for a suitable class of test
functions in $\tY$.  One of the most striking results of Section
\ref{subsec:HJ} is the explicit representation formula for solutions
to the Hamilton-Jacobi equation in $X$, that we obtain by a careful
reduction technique from the Hopf-Lax formula in $\tY$. In this
respect, we think that Theorem \ref{thm:main-HJ} is interesting by
itself and could find important applications in different contexts.
From the point of view of Entropy-Transport problems, Theorem
\ref{thm:main-HJ} is particularly relevant since it provides a dynamic
interpretation of the dual characterization of the $\LET$ functional.
In Section \ref{subsec:geodesicRd} we show that in the Euclidean case
$X=\R^d$ all geodesic curves are characterized by the system
\eqref{eq:GeodSyst}. 
The last Section \ref{subsec:contraction} provides various contraction
results: in particular we extend the well known contraction property
of the Heat flow in spaces with nonnegative Riemannian Ricci curvature
to $\HK$.

\paragraph*{Note during final preparation.} The earliest parts of the
work developed here were first presented at the ERC Workshop on
Optimal Transportation and Applications in Pisa in 2012. Since then
the authors developed the theory continuously further and presented
results at different workshops and seminars, \AAA see Appendix
\ref{sec:Devel} \EEE for some remarks concerning the chronological
development of our theory.  In June 2015 they became aware of the
parallel work \cite{KMV15}, which mainly concerns the dynamical
approach to the Hellinger-Kantorovich distance discussed in Section
\ref{subsec:BB} and the metric-topological properties of Section
\ref{subsec:MT} in the Euclidean case. Moreover, in mid August 2015 we
became aware of \AAA the work \cite{CPSV15?IDOT,CPSV15?UOTG}, which
starts \EEE from the dynamical formulation of the
Hellinger-Kantorovich distance in the Euclidean case, prove existence
of geodesics and sufficient optimality and uniqueness conditions
(which we state in a stronger form in Section \ref{subsec:geodesicRd})
with a precise characterization in the case of a couple of Dirac
masses, provide a detailed discussion of curvature properties
following Otto's formalism \cite{Otto01}, and study more general
dynamic costs on the cone space with their equivalent primal and dual
static formulation (leading to characterizations analogous to
\eqref{eq:353} and \eqref{eq:204} in the Hellinger-Kantorovich case).

Apart from the few above remarks, these
independent works did not influence the \AAA first  (cf.\
arXiv1508.07941v1) and the \EEE  present version of this
manuscript, \AAA which is essentially a minor modification  and
correction of the first version. In \EEE  the final Appendix \ref{sec:Devel} 
we give a brief account of the chronological development
of our theory. 

\nc\bigskip 


\centerline{\large\bfseries Main notation}

\smallskip
\halign{$#$\hfil\ &#\hfil
\cr
\cM(X)&finite positive Radon measures on a Hausdorff topological space
$X$
\cr
\cP(X),\ \cP_2(X)&Radon probability measures on $X$ (with finite
quadratic moment)
\cr
\BorelSets X&Borel subsets of $X$
\cr
T_\sharp\mu&push forward of $\mu\in \cM(X)$ by a map $T:X\to Y$:
\eqref{eq:460}
\cr 
\gamma=\sigma\mu{+}\mu^\perp,\ \mu=\varrho\gamma{+}\gamma^\perp&
Lebesgue decompositions of $\gamma$ and $\mu$, Lemma \ref{le:Lebesgue}
\cr
\rmC_b(X)&continuous and bounded real functions on $X$\cr
\Lip_b(X),\,\Lip_{bs}(X)&bounded (with bounded support) Lipschitz
real functions on $X$
\cr
\LSC_b(X),\LSC_s(X)&lower semicontinuous and bounded (or simple) real
functions on $X$
\cr
\USC_b(X),\USC_s(X) &upper semicontinuous and bounded (or simple) real
functions on $X$
\cr
\rmB(X),\rmB_b(X)&Borel (resp.~bounded Borel) real functions
\cr
{\rmL}^p(X,\mu),\ \rmL^p(X,\mu;\R^d) & Borel $\mu$-integrable real (or
$\R^d$-valued) functions
\cr
\Gamma(\R_+)&set of admissible entropy functions, see \eqref{eq:18},
\eqref{eq:205}.
\cr
 F(\r),F_i(\r)&admissible entropy functions.
\cr
\Fstar(\phi),\Fstar_i(\phi)&Legendre transform of $F,F_i$, see
\eqref{eq:22}.
\cr 
\Gstar(\varphi),\Gstar_i(\varphi_i)&concave
conjugate of an entropy function, see \eqref{eq:304}.
\cr
\FH(\s),\FH_i(\s_i)&reversed entropies, see \eqref{eq:12}.
\cr
\MP_c(\s_1,\s_2),\; H(x_1,\s_1;x_2,\s_2)&marginal perspective
function, see \eqref{eq:186}, \eqref{eq:188}, \eqref{eq:342}
\cr
\sfc(x_1,x_2)&lower semicontinuous cost function defined in
$\xX=X_1\times X_2$.
\cr 
\FF(\gamma|\mu),\FHH(\mu|\gamma)&entropy
functionals and their reverse form, 
see \eqref{eq:21} and \eqref{eq:211}
\cr
\EE(\ggamma|\mu_1,\mu_2),\ET(\mu_1,\mu_2)&general
Entropy-Transport functional and its minimum, see \eqref{eq:4} 
\cr
\DD(\vvarphi|\mu_1,\mu_2),\sfD(\mu_1,\mu_2)&dual functional and
its supremum, see \eqref{eq:44} and \eqref{eq:49}
\cr
\Cphi{},\Cpsi{}&set of admissible Entropy-Kantorovich potentials
\cr
\LET(\mu_1,\mu_2),\ \ell(\sfd)&Logarithmic Entropy Transport
functional
and its cost: Section \ref{subsec:LET1}
\cr
\Wd(\mu_1,\mu_2)&Kantorovich-Wasserstein distance in $\cP_2(X)$
\cr
\HK(\mu_1,\mu_2)&Hellinger-Kantorovich distance in $\cM(X)$:
Section \ref{subsec:HK}
\cr
\GHK(\mu_1,\mu_2)&Gaussian Hellinger-Kantorovich distance in $\cM(X)$:
Section \ref{subsec:GHK}
\cr
(\tY,\ \sfdc),\ \fro&metric cone and its vertex, see Section
\ref{subsec:cone}
\cr
\cball r&ball of radius $r$ centered at $\fro$ in $\tY$
\cr
\chm2{i}{},\ \dilp{\theta}{2}{\cdot}&homogeneous marginals and
dilations, see 
\eqref{eq:155}, \eqref{eq:147bis}
\cr
\cHM2{\mu_1}{\mu_2},\ \cHMle2{\mu_1}{\mu_2}&plans in $\tY\times\tY$ 
with constrained homogeneous marginals, see \eqref{eq:161bis}
\cr
\mathrm{AC}^p([0,1];X) & space of curves $\rmx:[0,1]\to
X$ with $p$-integrable metric speed
\cr
|\rmx'|_\sfd&metric speed of a curve $\rmx\in
\mathrm{AC}([a,b];(X,\sfd))$, Sect.~\ref{subsec:geodesic}
\cr
|\rmD_Z f|,\ \alc Z f&metric slope and asymptotic Lipschitz constant
in $Z$, see  \eqref{eq:408}
\cr}

\PART{Part I. Optimal Entropy-Transport problems}

\section{Preliminaries}
\label{sec:prel}
\subsection{Measure theoretic notation}
\paragraph{Positive Radon measures, narrow and weak convergence, tightness.}
Let $(X,\tau)$ be a 
Hausdorff topological space.
We will denote by $\BorelSets X$ the $\sigma$-algebra of its Borel sets and by
$\cM(X)$ the set of finite nonnegative \emph{Radon} measures
on $X$ \cite{Schwartz73}, i.e.~$\sigma$-additive set functions
$\mu:\BorelSets X\to
[0,\infty)$ such that 
\begin{equation}
  \label{eq:238}
  \forall\, B\in \BorelSets X,\ \forall\,\eps>0\quad 
   \exists\,K_\eps\subset B \text{ compact such that}\quad
  \mu(B\setminus K_\eps)\le \eps.
\end{equation}
Radon measures have strong continuity property with respect to
monotone convergence. For this, denote by $\LSC(X)$ the space of
all lower semicontinuous real-valued functions on $X$ and consider
a nondecreasing directed family $(f_\lambda)_{\lambda\in
  \L}\subset \LSC(X)$ (where  $\L$ is a possibly uncountable
directed set) of nonnegative and lower semicontinuous functions
$f_\lambda$ converging to $f$, we have (cf.\ \cite[Prop.\,5,
p.\,42]{Schwartz73})
\begin{equation}
  \label{eq:280}
  \lim_{\lambda\in \L}\int_X f_\lambda\,\d\mu=\int_X
  f\,\d\mu \quad \text{for all } \mu\in \cM(X).
\end{equation}
We endow $\cM(X)$ with the \emph{narrow} topology, the coarsest
(Hausdorff) topology for which all the maps $ \mu\mapsto \int_X
\varphi\,\d\mu$ are lower semicontinuous, as $\varphi:X\to \R$ varies
among the set $\LSC_b(X)$ of all bounded lower semicontinuous
functions \cite[p.~370, Def.~1]{Schwartz73}.
\begin{remark}[Radon versus Borel, narrow versus weak]
  \label{rem:particular-cases}
  \upshape
  When $(X,\tau)$ is a Radon space (in particular a Polish, or Lusin
  or Souslin space \cite[p.~122]{Schwartz73}) then every Borel measure satisfies \eqref{eq:238},
  so that $\cM(X)$ coincides with the set of all nonnegative and
  finite \emph{Borel} measures.
  Narrow topology is in general stronger than 
  the standard weak topology
  induced by the duality with continuous and bounded functions of
  $\rmC_b(X)$. 
  However, when $(X,\tau)$ is completely regular, i.e.
  \begin{equation}
  \label{eq:CR}
  \begin{gathered}
    \text{for any closed set $F\subset X$ and any $x_0\in X\setminus F$}\\
    \text{there exists $f\in C_b(X)$ with $f(x_0)>0$ and $f\equiv 0$ on $F$,} 
  \end{gathered}
\end{equation}
(in particular when $\tau$ is metrizable), narrow and weak topology
coincide \cite[p.~371]{Schwartz73}.  Therefore when $(X,\tau)$ is a
Polish space we recover the usual setting of Borel measures endowed
with the weak topology.
\end{remark}
A set $\cK\subset \cM(X)$
is \emph{bounded} if $\sup_{\mu\in \cK} \mu(X)<\infty$; it is 
\emph{equally tight} if 
\begin{equation}
  \label{eq:25}
  \forall\,\eps>0\quad
  \exists\,K_\eps\subset X\text{ compact such that}\quad
  \mu(X\setminus K_\eps)\le \eps\quad \forevery \mu\in \cK.
\end{equation}
Compactness with respect to narrow topology is 
guaranteed by an extended version of \emph{Prokho\-rov's Theorem}
\cite[Thm.~3, p.~379]{Schwartz73}. Tightness of weakly convergent
\emph{sequences}
in metrizable spaces is due to \textsc{Le Cam} \cite{LeCam57}.
\begin{theorem}
  \label{thm:Prokhorov}
  If a subset $\cK\subset \cM(X)$ is bounded and equally tight then 
  it is relatively compact with respect to
  the narrow topology. The converse is also true in the following
  cases:\\
  (i)\phantom i $(X,\tau)$ 
  is a locally compact or a Polish space; \\
  (ii) $(X,\tau)$ is metrizable and $\cK=\{\mu_n:n\in \N\}$
  for a given weakly convergent sequence $(\mu_n)$.  
\end{theorem}
If $\mu\in \cM(X)$ and $Y$ is another Hausdorff topological space, a
map $T:X\to Y$ is \emph{Lusin $\mu$-measurable} \cite[Ch.~I,
Sec.~5]{Schwartz73} if for every $\eps>0$ there exists a compact set
$K_\eps\subset X$ such that $\mu(X\setminus K_\eps)\le \eps$ and the
restriction of $T$ to $K_\eps$ is continuous.  We denote by $T_\sharp
\mu\in \cM(Y)$ the push-forward 
measure
defined by
\begin{equation}
  \label{eq:460}
  T_\sharp\mu(B):=\mu(T^{-1}(B))\ \forevery B\in \BorelSets Y.
\end{equation}
For
 $\mu\in \cM(X)$ and a Lusin $\mu$-measurable $T:X\to Y$, we
have $T_\sharp \mu \in \cM(Y)$. The linear space $\rmB(X)$
(resp.~$\rmB_b(X)$) denotes the space of real Borel
(resp.~bounded Borel) functions. If $\mu\in \cM(X)$, $p\in
[1,\infty]$, we will denote by $\rmL^p(X,\mu)$ the subspace of Borel
$p$-integrable functions w.r.t.~$\mu$, without identifying
$\mu$-almost equal functions.
\paragraph{Lebesgue decomposition.}
Given $\gamma,\mu\in \cM(X)$, we write $\gamma\ll\mu$ if
$\mu(A)=0$ yields $\gamma(A)=0$ for every $A\in \BorelSets X$.
We say that $\gamma\perp\mu$ if there exists 
$B\in \BorelSets X$ such that $\mu(B)=0=\gamma(X\setminus B)$.
\begin{lemma}[Lebesgue decomposition]
  \label{le:Lebesgue}
  For every $\gamma,\mu\in \cM(X)$ (with $(\gamma+\mu)(X)>0$), there exist
  Borel functions $\sigma,\varrho:X\to[0,\infty)$ and
  a Borel partition $(A,A_\gamma,A_\mu)$ of $X$ with the following
  properties:
  \begin{gather}
    \label{eq:243}
    A=\{x\in X:\sigma(x)>0\}=\{x\in X:\varrho(x)>0\},\quad
    \sigma\cdot\varrho\equiv 1\quad\text{in }A,\\
    \label{eq:254}
    \gamma=\sigma\mu+\gamma^\perp,\quad 
      \sigma\in \rmL^1_+(X,\mu),\quad \gamma^\perp\perp\mu,\quad
      \gamma^\perp(X\setminus A_\gamma)=\mu(A_\gamma)=0,\\
      \label{eq:310}
      \mu=\varrho\gamma+\mu^\perp,\quad 
      \varrho\in \rmL^1_+(X,\gamma),\quad \mu^\perp\perp\gamma,\quad
      \mu^\perp(X\setminus A_\mu)=\gamma(A_\mu)=0.
  \end{gather}
  Moreover, the sets $A,A_\gamma,A_\mu$ and the densities $\sigma,\varrho$ are 
  uniquely
  determined up to $(\mu+\gamma)$-negligible sets.
\end{lemma}
\begin{proof}
  Let $\theta\in \rmB(X;[0,1])$ be the Lebesgue density of $\gamma$
  w.r.t.~$\nu:=\mu+\gamma$. 
  Thus, $\theta$ is uniquely determined up
  to $\nu$-negligible sets.
  The Borel partition can be defined by setting
  $A:=\{x\in X:0<\theta(x)<1\}$, $A_\gamma:=\{x\in X:\theta(x)=1\}$
  and
  $A_\mu:=\{x\in X:\theta(x)=0\}$.
  By defining 
  $\sigma:=\theta/(1-\theta)$, $\varrho:=1/\sigma=(1-\theta)/\theta$ for
  every $x\in A$ 
  and 
  $\sigma=\varrho\equiv 0$ in $X\setminus A$, we obtain Borel functions 
  satisfying \eqref{eq:254} and \eqref{eq:310}. 

  Conversely, it is not difficult to check that starting from
  a decomposition as in \eqref{eq:243}, \eqref{eq:254}, and
  \eqref{eq:310} and defining
  $\theta\equiv 0$ in $A_\mu$, $\theta\equiv 1$ in $A_\gamma$ and
  $\theta:=\sigma/(1+\sigma)$ in $A$ we obtain
  a Borel function with values in $[0,1]$ such that
  $\gamma=\theta(\mu+\gamma)$.
\end{proof}

\subsection{Min-max and duality}\label{subsec:minmax}
We recall now a powerful form of von Neumann's Theorem, concerning
minimax properties of convex-concave functions in convex subsets of
vector spaces and refer to \cite[Prop.\,1.2+3.2,
Chap.\,VI]{Ekeland-Temam74} for a general exposition. \EEE

Let $A,B$ be nonempty convex sets of some vector spaces and let us
suppose that $A$ is endowed with a Hausdorff topology.  Let $L:A\times
B\to \R$ be a function such that
\begin{subequations}
\begin{align}
  \label{eq:28}
    a\mapsto L(a,b)\quad&\text{is convex and lower semicontinuous in
      $A$ for every $b\in B$,}\\
    \label{eq:37}
    b\mapsto L(a,b)\quad&\text{is concave in $B$ for every $a\in A$}.
\end{align}
\end{subequations}
Notice that for arbitrary functions $L$ one always has
\begin{equation}
  \label{eq:278}
     \infp_{a\in A}\sup_{b\in B}L(a,b)\ge\sup_{b\in B}\infp_{a\in A}L(a,b);
\end{equation}
so that equality holds in \eqref{eq:278} if $\sup_{b\in B} \infp_{a\in
  A}L(a,b)=+\infty$.  When $\sup_{b\in B} \infp_{a\in A}L(a,b) $ is
finite, we can still have equality thanks to the following result.

The statement has the advantage of involving a minimal set of
topological assumptions (we refer to \cite[Thm. 3.1]{Simons98} for the
proof, see also \cite[Chapter 1, Prop.~1.1]{Brezis73}).

\begin{theorem}[Minimax duality]
  \label{thm:minimax}  Assume that \eqref{eq:28} and \eqref{eq:37}
  hold. If there exists $b_\star\in B$ and $C>\sup_{b\in
    B}\infp_{a\in A}L(a,b)$ such that
  \begin{equation}
    \label{eq:29}
    \big\{a\in A:L(a,b_\star)\le C\big\}\quad\text{is compact in }A,
  \end{equation}
  then 
  \begin{equation}
    \label{eq:30}
    \infp_{a\in A}\sup_{b\in B}L(a,b)=\sup_{b\in B}\infp_{a\in A}L(a,b).
  \end{equation}
\end{theorem}
\subsection{Entropy functions and their conjugates}
\label{subsec:entropy}

\paragraph{Entropy functions in $[0,\infty)$.}
We say that 
$F:[0,\infty)\to [0,\infty]$ 
belongs to the class $\Gamma(\R_+)$ of 
\emph{admissible entropy function} if it satisfies
\begin{equation}
    \label{eq:18}
    \begin{gathered}
      F\text{ is convex and lower semicontinuous with}\quad
      \dom F\cap (0,\infty)\neq \emptyset,
    \end{gathered}
\end{equation}
where
\begin{equation}
  \label{eq:205}
    \dom F:=\{\r \ge 0: F(\r )<\infty\}, \ 
      \ldom F:=\inf \dom F,\ 
      \rdom F:=\sup\dom F>0.  
\end{equation}
The recession constant $\rec F$, the right derivative $\derzero F$ at
$0$, and the asymptotic affine coefficient $\asympt F$ are defined by
(here $\r _o\in \dom F$)
\begin{gather}
  \label{eq:227}
  \rec F :=\lim_{\r \to\infty}\frac{F(\r )}\r =\sup_{\r >0}\frac{F(\r
    )-F(\r _o)}{\r -\r _o}, 
  \quad
  \derzero F:=
  \begin{cases}
    -\infty&\text{if }F(0)=+\infty,\\
    \lim\limits_{\r \down0}\frac{F(\r )-F(0)}\r &\text{otherwise,}
  \end{cases}
  \\  \label{eq:228}
  \asympt F:=
  \begin{cases}
    +\infty&\text{if }\rec F=+\infty,\\
    \lim\limits_{\r \to\infty}\big(\rec F\,\r -F(\r )\big)&\text{otherwise.}
  \end{cases}
\end{gather}
To avoid trivial cases, we assumed in \eqref{eq:18} that the
proper domain $\dom F$ contains at least a strictly positive real
number.  By convexity, $\dom F$ is a subinterval \EEE
of $[0,\infty)$, and we will mainly focus on the case when
$\dom F$ has nonempty interior and $F$ has superlinear growth,
i.e.~$\rec F =+\infty$, but it will be useful to deal with the general
class defined by \eqref{eq:18}.
\paragraph{Legendre duality.}
As usual, 
the \emph{Legendre conjugate} function $\Fstar:\R\to (-\infty,+\infty]$
is defined by
\begin{equation}
  \label{eq:22}
  \Fstar(\phi):=\sup_{\r \ge0}\big(\r \phi-F(\r )\big),
\end{equation}
with proper domain $\dom \Fstar :=\{\phi\in \R:\Fstar (\phi)\in
\R\}$. Strictly speaking, $\Fstar$ is the conjugate
of the convex function $\tilde F:\R\to(-\infty,+\infty]$, obtained by
extending $F$ to $+\infty$ for negative arguments.  Notice that
\begin{equation}
  \label{eq:206}
  \inf\dom \Fstar=-\infty,\quad
  \sup\dom \Fstar=\rec F ,
\end{equation}
so that $\Fstar $ is finite and continuous in $(-\infty,\rec F )$, 
nondecreasing, and
satisfies 
\begin{equation}
  \label{eq:40}
  \lim_{\phi\down-\infty}\Fstar(\phi)=\inf \Fstar=
  -F(0),\quad
  \sup \Fstar =\lim_{\phi\up+\infty}\Fstar (\phi)=+\infty.
\end{equation}
Concerning the behavior of $\Fstar$ at the boundary of its
proper domain we can distinguish a few cases depending on the
behavior of $F$ at $\ldom F$ and $\rdom F$: \EEE
\begin{itemize}
\itemindent-0.5em
\item[\textbullet] If $\derzero F=-\infty$ (in particular if
  $F(0)=+\infty$) then $\Fstar $ is strictly increasing in
  $\dom\Fstar$.
\item[\textbullet] If $\derzero F $ is
  finite, then $\Fstar $ is strictly increasing in $[\derzero F,\rec
  F)$
  and takes the constant value $F(0)$ in $(-\infty,\derzero F]$.
  Thus $F(0)$ belongs to the range of $\Fstar$ only if $\derzero F>-\infty$.
\item[\textbullet] If $\rec F$ is finite, then $\lim_{\phi\up \rec
    F}\Fstar(\phi)=\asympt F$. Thus $\rec F\in \dom {\Fstar}$ 
  only if $\asympt F<\infty$.
\item[\textbullet] The
  degenerate case when $\rec F=\derzero F$ occurs only when $F$ is linear.
\end{itemize}
If $F$ is not linear, we always have
\begin{equation}
  \label{eq:225}
  \text{$\Fstar $ is an increasing homeomorphism between 
    $(\derzero F,\rec F )$ and $(-F(0),\asympt F)$}
\end{equation}
with the obvious extensions to the boundaries of the intervals when
$\derzero F$ or $\asympt F$ are finite.

By introducing the closed convex  subset $\frF$ of $\R^2$ via \EEE
\begin{equation}
  \label{eq:281}
  \frF:=
  \big\{(\phi,\psi)\in \R^2:\psi\le -\Fstar(\phi)\big\} 
 = \big\{(\phi,\psi)\in \R^2:\r \phi+\psi\le F(\r )\ \forall\,
  \r > 0\big\},
\end{equation}
the function $F$ can be recovered from $\Fstar $ and from
\EEE $\frF$ through the dual 
Fenchel-Moreau formula
\begin{equation}
  \label{eq:269}
  F(\r )=\sup_{\phi\in \R} \big(\r \phi-\Fstar (\phi)\big)=
  \sup_{(\phi,\psi)\in \frF} \r \phi+\psi.
\end{equation}
Notice that $\frF$ satisfies the obvious monotonicity property
\begin{equation}
  \label{eq:284}
  (\phi,\psi)\in \frF,\quad
  \tilde\psi\le \psi,\ \tilde\phi\le \phi\quad
  \Rightarrow\quad
  (\tilde\phi,\tilde\psi)\in \frF.
\end{equation}
If $F$ is finite in a neighborhood of $+\infty$, then $\Fstar$ is
superlinear as $\phi\up\infty$.  More precisely, its asymptotic
behavior as $\phi\to\pm\infty$ is related to the proper domain of $F$
by
\begin{equation}
  \label{eq:2}
 \r^\pm_F =\lim_{\phi\to\pm\infty}\frac {\Fstar (\phi)}\phi.
\end{equation}
The functions $F$ and $\Fstar$ are also related to the
subdifferential $\partial F:\R\to 2^{\R}$ by
\begin{equation}
  \label{eq:325}
  \phi\in \partial F(\r )\quad\Leftrightarrow\quad
  \r \in \dom F,\quad \phi\in \dom \Fstar,\quad
  F(\r )+\Fstar(\phi)=\r \phi.
\end{equation}
\begin{example}[Power-like entropies]
  \label{ex:power-like}
  \upshape
  An important class of entropy functions is provided by the power
  like functions $\PE_p:[0,\infty)\to [0,\infty]$ with $p\in \R$
  characterized by 
  \begin{equation}
    \label{eq:237}
    \PE_p\in \rmC^\infty(0,\infty),\quad
    \PE_p(1)=\PE_p'(1)=0,\quad \PE_p''(\r )=\r ^{p-2},\quad
    \PE_p(0)=\lim_{\r \down0}\PE_p(\r ).
  \end{equation}
  Equivalently, we have the explicit formulas \EEE
  \begin{equation}
    \label{eq:236}
    \PE_p(\r )=
    \begin{cases}
      \frac{1}{p(p-1)}\big(\r ^p-p(\r -1)-1\big)&\text{if }p\neq 0,1,\\
      \r \log \r -\r +1&\text{if }p=1,\\
      \r -1-\log \r &\text{if }p=0,
    \end{cases}
    \qquad \text{for }\r >0,
  \end{equation}
  with $\PE_p(0)=1/p$ if $p>0$ and $U_p(0)=+\infty$ if $p\le0$.

Using the dual exponent $q=p/(p-1)$, the corresponding Legendre
conjugates read 
\begin{displaymath}
  \PEstar_q(\phi):=
  \left\{
    \begin{aligned}
      \frac{q-1}q\Big[\big(1+\frac\phi{q-1}\big)^q_+-1\Big],& 
            \quad \dom{\PEstar_q}=\R,&&
        \text{if } p>1,\ q>1, \\
        \rme^{\phi}-1,&\quad\dom{\PEstar_q}=\R,&& \text{if }p=1,\  
        q=\infty, \\
        \frac {q-1}q\Big[\big(1+\frac\phi{q-1}\big)^q-1\Big],&\quad
        \dom{\PEstar_q}=(-\infty,1-q),&& \text{if
        }0< p<1,\ q<0,\\
        -\log(1-\phi),&\quad\dom{\PEstar_q}=(-\infty,1),&& 
         \text{if }p=0,\ q=0, \\
        \frac {q-1}q\Big[\big(1+\frac\phi{q-1}\big)^q-1\Big],&\quad
        \dom{\PEstar_q}=(-\infty,1-q],&& \text{if }p<0,\ 0<q<1.
      \end{aligned}
      \right.    
  \end{displaymath}
  \end{example}
\paragraph{Reverse entropies.}
Let us now introduce the reverse density function
$\FH:[0,\infty)\to [0,\infty]$ as 
\begin{equation}
  \label{eq:12}
  \FH(\s ):=
  \begin{cases}
    \s F(1/\s )&\text{if }\s >0,\\
    \rec F &\text{if }\s =0.
  \end{cases}
\end{equation}
It is not difficult to check that $\FH$ is a proper, convex 
and lower semicontinuous function, with
\begin{gather}
  \label{eq:183}
  \FH(0)=\rec F 
  ,\quad
  \rec\FH=
  F(0),\quad
  \asympt F=-\derzero \FH,\quad
  \asympt \FH=-\derzero F,
\end{gather}
so that $\FH\in \Gamma(\R_+)$ and the map $F\mapsto \FH$ is an
involution on $\Gamma(\R_+)$.  A further remarkable
involution property is enjoyed by the dual convex set
$\frR:=\{(\psi,\phi)\in \R^2:\FHstar(\psi)+\phi\le 0\}$ defined as
\eqref{eq:281}: it is easy to check that
\begin{equation}
  \label{eq:297}
  (\phi,\psi)\in \frF\quad\Leftrightarrow\quad
  (\psi,\phi)\in \frR.
\end{equation}
It follows that the Legendre transform of $\FH$ and $F$ are related by
\begin{equation}
  \label{eq:217}
  \psi\le -\Fstar(\phi)\quad\Leftrightarrow\quad
  \phi\le -\FHstar(\psi)\quad\Leftrightarrow
  \quad (\phi,\psi)\in \frF
  \qquad
  \text{for every }\phi,\psi\in \R.
\end{equation}
As in \eqref{eq:225} we have
\begin{equation}
  \label{eq:298}
  \text{$\FHstar $ is an increasing homeomorphism between 
    $(-\asympt F,F(0) )$ and $(-\rec F,-\derzero F).$}
\end{equation}
A last useful identity involves the subdifferentials of $F$ and $\FH$: 
for every $\r,\s >0$ with $\r\s =1$, and \EEE
$\phi,\psi\in \R$ we have 
\begin{equation}
  \label{eq:326}
  \Big(  \phi\in \partial F(r) \text{ and }  
  \psi=-\Fstar(\phi)  \Big) \quad\Longleftrightarrow\quad
   \Big(   \psi\in \partial \FH(s) 
   \text{ and }  
  \phi=-\FHstar(\psi) \Big) .\EEE
\end{equation}
It is not difficult to check that the reverse entropy associated to
$\PE_p$ is $\PE_{1-p}$.

\subsection{Relative entropy integral functionals}
 For $F\in \Gamma(\R_+)$ we consider the functional
$\FF:\cM(X)\times\cM(X)\to [0,\infty]$ defined by
\begin{equation}
  \label{eq:21}
  \FF(\gamma|\mu):=\int_X F(\sigma
  )\,\d\mu+\rec F \,\gamma^\perp(X),\quad \gamma=\sigma\mu+
  \gamma^\perp,\quad
  \gamma^\perp\perp\mu,\quad
  \sigma:=\frac{\d\gamma}{\d\mu},
\end{equation}
where $\gamma=\sigma\mu+\gamma^\perp$ is the Lebesgue decomposition of
$\gamma$ w.r.t.~$\mu$,  see \eqref{eq:254}.  Notice that
\begin{equation}
  \label{eq:19}
  \text{if $F$ is superlinear then}\quad
  \mathscr F(\gamma|\mu)=+\infty\quad\text{if }\gamma\not\ll\mu,
\end{equation}
and, whenever $\eta_0$ is the null measure,  we have 
\begin{equation}
  \label{eq:27}
  \FF(\gamma|\eta_0)=
  \rec F  \,\gamma(X),
\end{equation}
where, as usual in measure theory, we adopted the convention
$0\cdot\infty=0$.

 Because of our applications in Section~\ref{sec:ET}, our next lemma deals with
\EEE Borel functions $\varphi\in
\rmB(X;\bar \R)$ taking values in the extended real line
$\bar\R:=\R\cup\{\pm\infty\}$.  By $\bar \frF$ \EEE
we denote the closure of $\frF$ in $\bar \R\times \bar \R$,  i.e.\
\EEE  
\begin{equation}
  \label{eq:322}
  (\phi,\psi)\in \bar\frF\quad\Leftrightarrow\quad
  \begin{cases}
    \psi\le -\Fstar(\phi)&\text{if }-\infty<\phi\le\rec F,\ \phi<+\infty\\
    \psi=-\infty&\text{if }\phi=\rec F=+\infty,\\    
    \psi\in [-\infty,F(0)]&\text{if }\phi=-\infty,
  \end{cases}
\end{equation}
and, symmetrically by  \eqref{eq:183} and \eqref{eq:297}, 
\begin{equation}
  \label{eq:322bis}
  (\phi,\psi)\in \bar\frF\quad\Leftrightarrow\quad
  \begin{cases}
    \phi\le -\FHstar(\psi)&\text{if }-\infty<\psi\le F(0),\ \psi<+\infty\\
    \phi=-\infty&\text{if }\psi=F(0)=+\infty,\\    
    \phi\in [-\infty,\rec F]&\text{if }\psi=-\infty  . \EEE
  \end{cases}
\end{equation}
In particular, we have 
\begin{equation}
  \label{eq:313}
  (\phi,\psi)\in \bar\frF\quad \Longrightarrow\quad \big(\; \EEE
  \phi\le \rec F \text{ and } \psi\le F(0) \;\big) .
\end{equation}
We continue to use the notation $\phi_-$ and $\phi_+$ to denote the
negative and the positive part of a function $\phi$, where
$\phi_-(x):=\min\{\phi(x), 0\}$ and $\phi_+(x):=\max\{\phi(x),0\}$. 

\begin{lemma}
  \label{le:trivial}
  If $\gamma,\mu \in \cM(X)$ and $ (\phi,\psi) \nc \in
  \rmB(X;\bar\frF)$  satisfy 
\[
  \FF(\gamma|\mu)<\infty,
  \quad 
   \psi_-\in \rmL^1(X,\mu) 
  \text{ (resp.~$\phi_-\in \rmL^1(X,\gamma)$)}, 
\]
\EEE  
  then 
  $\phi_+\in \rmL^1(X,\gamma)$ 
  (resp. $\psi_+\in \rmL^1(X,\mu)$) and
  \begin{equation}
    \label{eq:71}
    \FF(\gamma|\mu)- \int_X \psi\,\d\mu\ge \int_X \phi\,\d\gamma.
  \end{equation}
  Whenever $\psi\in \rmL^1(X,\mu)$ or
  $\phi\in \rmL^1(X,\gamma)$, equality holds in
  \eqref{eq:71} if and only if 
  for the Lebesgue decomposition given by Lemma \ref{le:Lebesgue} one
  has
    \begin{align}
    &\phi\in \partial F(\sigma), \ \psi=-\Fstar(\phi)
        \quad
    \text{$(\mu{+}\gamma)$-a.e.~in $A$,}
    \label{eq:323}
  \\
    \label{eq:324}
    &\psi=F(0)<\infty\ \text{$\mu^\perp$-a.e.~in $A_\mu$},\quad
    \phi=\rec F<\infty\ \text{$\gamma^\perp$-a.e.~in $A_\gamma$.}
  \end{align}
Equation \eqref{eq:323} can equivalently be formulated as
$\psi\in \partial R(\varrho)$ and $\phi=-\FHstar(\psi)$. 
  \end{lemma}
  \begin{proof}
    Let us first show 
    that in both cases 
    the two integrals of \eqref{eq:71}
    are well defined (possibly taking the 
    value $-\infty$). 
    If $\psi_-\in \rmL^1(X,\mu)$ 
    (in particular $\psi>-\infty$ $\mu$-a.e.) 
    with $(\phi,\psi)\in \bar\frF$ we use
    the pointwise bound $\r \phi\le F(\r )-\psi$ that yields $\r \phi_+\le
    (F(\r )-\psi)_+\le F(\r )+\psi_-$
    obtaining 
    $\phi_+\in \rmL^1(X,\gamma)$, since $(\phi,\psi)\in \bar\frF$ yields
    $\phi_+\le \rec F $.

    If $\phi_-\in \rmL^1(X,\gamma)$ (and thus $\phi>-\infty$
    $\gamma$-a.e.)
    the analogous inequality
    $\psi_+\le F(\r )+\r \phi_-$ yields $\psi_+\in \rmL^1(X,\mu)$.
    Then, \eqref{eq:71}\EEE follows from \eqref{eq:281} and \eqref{eq:313}.

    Once $\phi\in \rmL^1(X,\mu)$ (or $\psi\in \rmL^1(X,\gamma)$),
    estimate \eqref{eq:71} can be written as 
    \begin{displaymath}
      \int_A \Big(F(\sigma)-\sigma\phi-\psi\Big)\,\d\mu+
      \int_{A_\mu}\Big(F(0)-\psi\Big)\,\d\mu^\perp+
      \int_{A_\gamma} (\rec F-\phi)\,\d\gamma^\perp \ge0,
    \end{displaymath}
    and  by \eqref{eq:281} and \eqref{eq:313} 
    the equality case immediately yields that each of the three
    integrals of the previous formula vanishes. 
    Since $(\phi,\psi)$ lies in $\bar\frF\subset \R^2$ $(\mu+\gamma)$-a.e.~in $A$,\EEE
    the vanishing of the first integrand yields
    $\psi=-\Fstar(\sigma)$ and $\phi\in \partial F(\sigma)$ 
    by \eqref{eq:325} for $\mu$ and $(\mu+\gamma)$ almost every
    point in $A$. The equivalence \eqref{eq:326} provides the reversed identities 
    $\psi\in \partial\FH(\varrho)$, $\phi=-\FHstar(\psi)$.

   The relations in  \EEE\eqref{eq:324} follow easily by the vanishing of the last two
    integrals and the fact that $\psi$ is finite $\mu$-a.e.~and $\phi$
    is
    finite $\gamma$-a.e.
\end{proof}
The next theorem gives a characterization of the relative entropy $\FF$, which is the main result of this section.
Its proof 
is a careful adaptation of
\cite[Lemma 9.4.4]{Ambrosio-Gigli-Savare08} to the present more general setting,
which includes the sublinear case when $\rec F<\infty$ and
the lack of complete regularity of the space.
This suggests to deal with lower semicontinuous functions instead of
continuous ones. 
We denote by $\LSC_s(X)$ the class of lower semicontinuous and simple
functions
(i.e.~taking a finite number of real values only\EEE) 
and introduce the notation $\varphi=-\phi$ and the concave function
\begin{equation}
  \label{eq:304}
  \Gstar(\varphi):=-\Fstar(-\varphi).
\end{equation}
\begin{theorem}[Duality and lower semicontinuity]
  \label{thm:duality}
  For every $\gamma,\mu\in \cM(X)$ we have
  \begin{align}
    \label{eq:23}
      \FF(\gamma|\mu)
      &=
      \sup\Big\{\int_X
      \psi\,\d\mu
      +\int_X \phi\,\d\gamma:
      \phi,\psi\in \LSC_s(X),\
      (\phi(x),\psi(x))\in \frF\ \forall\,x\in X
      \Big\}
      \\\label{eq:299}&=
      \sup\Big\{\int_X
      \psi\,\d\mu
      -\int_X \FHstar(\psi)\,\d\gamma:
      \psi,\FHstar(\psi) \in \LSC_s(X)
      \Big\}
      \\\label{eq:299bis}&=
      \sup\Big\{\int_X
                           \Gstar(\varphi)\,\d\mu
      -\int_X \varphi\,\d\gamma:
      \varphi,\Gstar(\varphi)\in \LSC_s(X)
      \Big\}
  \end{align}
  and the space $\LSC_s(X)$ in the supremum of \eqref{eq:23}, 
  \eqref{eq:299} and \eqref{eq:299bis} can also be replaced by
  the space $\LSC_b(X)$ (resp~$\rmB_b(X)$) 
  of bounded l.s.c.~(resp.~Borel) functions.
\end{theorem}
\begin{remark}
  \label{rem:LSC=C.1}
  \upshape
  If $(X,\tau)$ is completely regular (recall \eqref{eq:CR}), then 
  we can equivalently replace lower semicontinuous functions by
  continuous ones in \eqref{eq:23}, \eqref{eq:299} and
  \eqref{eq:299bis}).
  E.g.~in the case of \eqref{eq:23} we have
    \begin{equation}
    \label{eq:23bis}
    \begin{aligned}
      \FF(\gamma|\mu)
      &=
      \sup\Big\{\int_X
      \psi\,\d\mu
      +\int_X \phi\,\d\gamma:
      (\phi,\psi)\in \rmC_b(X;\frF)
      \Big\}.
    \end{aligned}
  \end{equation}
  In fact, considering first \eqref{eq:23}, by complete regularity 
  it is possible to express every couple $\phi,\psi$ of bounded lower semicontinuous
  functions with values in $\frF$
  as the
  supremum
  of a directed family of continuous and bounded functions
  $(\phi_\alpha,\psi_\alpha)_{\alpha\in \A}$
  which still satisfy the constraint $\frF$ due to \eqref{eq:284}.
  We can then apply
  the continuity \eqref{eq:280} of the integrals with respect to 
  the Radon measures $\mu$ and $\gamma$.

  In order to replace l.s.c.~functions with continuous ones in 
  \eqref{eq:299} we can approximate $\psi$ by an increasing directed
  family of continuous functions $(\psi_\alpha)_{\alpha\in \A}$.
  By truncation, one can always assume that
  $\max\psi\ge \sup\psi_\alpha\ge\inf\psi_\alpha\ge\min\psi$. 
  Since $\FHstar(\psi)$ is 
  bounded, it is easy to check that also $\FHstar(\psi_\alpha)$ is
  bounded
  and it is an increasing directed family converging to $\FHstar (\psi)$.
  An analogous argument works for \eqref{eq:23bis}.
\end{remark}
\begin{proof}
  Let us prove \eqref{eq:23}: 
  denoting by $\FF'$ its 
  right-hand side, Lemma \ref{le:trivial} yields 
  $\FF\ge \FF'$.
  In order to prove the opposite inequality let $B\in \BorelSets X$ a
  $\mu$-negligible Borel set where $\gamma^\perp$ is concentrated, let
  $A:=X\setminus B$ and let $\sigma:X\to [0,\infty)$ be a Borel
  density for $\gamma$ w.r.t.~$\mu$.  We consider a countable subset
  $(\phi_n,\psi_n)_{n=1}^\infty$ with $\psi_1=\phi_1=0$, which is
  dense in $\frF$ and an increasing sequence $\bar \phi_n\in
  (-\infty,\rec F)$ converging to $\rec F$, with $\bar\psi_n :=
  -\Fstar(\bar\phi_n)$.  By \eqref{eq:269} we have
  \begin{displaymath}
    F(\sigma(x))=\lim_{N\up\infty} F_N(x),
    \quad \text{where for every }x\in X\quad
    F_N(x):=\sup_{1\le n\le N}
    \psi_n+\sigma(x)\phi_n\quad.
  \end{displaymath}
 Hence,  Beppo Levi's monotone convergence theorem 
  (notice that $F_N\ge F_1=0$) implies $\FF(\gamma|\mu)=\lim_{N\up\infty} \FF_N'(\gamma|\mu)$, 
  where
  \[
\FF_N'(\gamma|\mu):=\int_A 
    F_N(x)\,\d\mu(x)+
    \bar\phi_N\gamma(B).
  \]\EEE
  It is therefore sufficient to prove that 
  \begin{equation}
    \label{eq:290}
    \FF'(\gamma|\mu) \geq  
    \FF_N'(\gamma|\mu) \quad \text{for every }N\in \N.
  \end{equation}
  We fix $N\in \N$,  set $\phi_{0}:=\bar\phi_N$,
  $\psi_{0}:=\bar\psi_N$, and recursively define the Borel sets
  $A_j$, for $j=0,\ldots,N$, with $A_0:=B$ and
  \begin{equation}
    \label{eq:291}
    \begin{split}
    A_1&:=\{x\in A: F_1(x)=F_N(x)\},\\
     A_{j}&:=\{x\in A : F_N(x)= F_j(x)>F_{j-1}(x)\}\quad\text{for }j=2,\ldots,N.
  \end{split}
  \end{equation}\EEE
  Since $F_1\leq F_2\leq\ldots\leq F_N$, the sets $A_i$ form a Borel partition of $A$.
  As $\mu$ and $\gamma$ are Radon measures, 
  for every $\eps>0$ we find disjoint compact sets $K_j\subset A_j$
  and disjoint open sets (by the Hausdorff separation property of
  $X$) $G_j\supset K_j$ such that
\[
  \sum_{j=0}^{N} \Big(\mu(A_j\setminus K_j)+
  \gamma(A_j\setminus K_j)\Big)=
  \mu\Big(X\setminus \bigcup_{j=0}^{N} K_j\Big)+\gamma\Big(X\setminus \bigcup_{j=0}^{N} K_j\Big)\le 
  \eps/S_N
\]
where 
\[
S_N:=\max_{0\le n\le N} \big[(\phi_n-\phi_{\rm min}^N)+(\psi_n-\psi_{\rm min}^N)\big],
  \quad \phi_{\rm min}^N:=\min_{0\le j\le N} \phi_j, \quad \psi_{\rm
    min}^N:=\min_{0\le j\le N} \psi_j.
\] 
Since $(\phi_{\rm min}^N,\psi_{\rm min}^N)\in \frF$ and the sets $G_n$
are disjoint, the lower semicontinuous functions
\begin{equation}
    \label{eq:292}
    \begin{aligned}
      \psi_N(x):={}\psi_{\rm min}^N+\sum_{n=0}^N(\psi_n-\psi_{\rm min}^N)
      \nchi_{G_n}(x),\quad 
      \phi_N(x):={}\phi_{\rm min}^N+\sum_{n=0}^N(\phi_n-\phi_{\rm min}^N) \nchi_{G_n}(x)
    \end{aligned}
\end{equation}
take values in $\frF$ and satisfy
  \begin{align*}
    \FF_N'(\gamma|\mu)\EEE&=\sum_{j=1}^N \int_{A_j}F_j(x)\,\d\mu(x)+
    \phi_0\gamma(A_0)
    \\&=\phi_{\rm min}^N\gamma(X)+
    \psi_{\rm min}^N\mu(X)+
    \sum_{j=0}^N \Big(\int_{A_j} (\phi_j-\phi_{\rm
      min}^N)\,\d\gamma(x)+\int_{A_j}(\psi_j-\psi_{\rm min}^N)\,\d\mu(x)\Big)
    \\&\le 
    \phi_{\rm min}^N\gamma(X)+
    \psi_{\rm min}^N\mu(X)+
    \sum_{j=0}^N \Big(\int_{K_j}  (\phi_j-\phi_{\rm
      min}^N)\,\d\gamma(x)+\int_{K_j}(\psi_j-\psi_{\rm min}^N)\,\d\mu(x)\Big)
    +\eps
    \\&\le\int_X \phi_N(x)\,\d\gamma(x)+
    \int_X\psi_N(x)\,\d\mu(x)+\eps.
  \end{align*}
Since $\eps$ is arbitrary we obtain \eqref{eq:290}.

Equation \eqref{eq:299} follows directly by \eqref{eq:23} and the
previous Lemma \ref{le:trivial}. In fact, denoting by $\FF''$ the
righthand side of \eqref{eq:299}, Lemma \ref{le:trivial} shows that
$\FF''(\gamma|\mu)\le \FF(\gamma|\mu)=\FF'(\gamma|\mu)$.  On the other
hand, if $\phi,\psi\in \LSC_s(X)$ with $(\phi,\psi)\in \frF$ then
$-\FHstar (\psi)\ge \phi$. Hence, $\FHstar(\psi)\in \LSC_s(X)$ since
$\FHstar$ is nondecreasing, does not take the value $-\infty$, and is
bounded from above by $-\phi$.  We thus get $\FF''(\gamma|\mu)\ge
\FF'(\gamma|\mu)$.

In order to show \eqref{eq:299bis} we observe that for every $\psi\in
\LSC_s(X)$ with $\FHstar (\psi)\in \LSC_s(X)$ we can set
$\varphi:=\FHstar(\psi)\in \LSC_s(X);$ since $(\psi,-\FHstar(\psi))\in
\frF$ \eqref{eq:217} yields $\psi\le
-\Fstar(-\varphi)=\Gstar(\varphi)$ so that $\int
\Gstar(\varphi)\,\d\mu-\int\varphi\,\d\gamma\ge \int \psi\,\d\mu-\int
\FHstar(\psi)\,\d\gamma$. Since $\Gstar$ cannot take the value
$+\infty$, we also have that $(-\varphi,\Gstar(\varphi))\in \frF$ so
that $\int \Gstar(\varphi)\,\d\mu-\int\varphi\,\d\gamma\le
\FF(\gamma|\mu)$ by Lemma \ref{le:trivial}.

  When one replaces $\LSC_s(X)$ with $\LSC_b(X)$ or $\rmB_b(X)$ in 
  \eqref{eq:23},
  the supremum is taken on a larger set, so that the righthand side
  of \eqref{eq:23} cannot decrease; on the other hand, Lemma
  \ref{le:trivial}
  shows that $\FF(\gamma|\mu)$ still provides an upper bound
  even if $\phi,\psi$ are in $\rmB_b(X)$, 
  thus duality also holds in this case. The same argument applies to \eqref{eq:299} or \eqref{eq:299bis}.
\end{proof}
The following result provides lower semicontinuity of the relative entropy
or of an increasing sequence of relative entropies.

\begin{corollary}
  \label{cor:limit1}
  The functional $\FF$ is jointly convex and lower semicontinuous in
  $\cM(X)\times\cM(X)$.  More generally, if $F_n\in \Gamma(\R_+)$,
  $n\in \N$, is an increasing sequence pointwise converging to $F$ and
  $(\mu,\gamma)\in \cM(X)\times\cM(X)$ is the narrow limit of a
  sequence $(\mu_n,\gamma_n)\in \cM(X)\times\cM(X)$, then the
  corresponding entropy functionals $\FF_n,\FF$ satisfy
  \begin{equation}
    \label{eq:240}
    \liminf_{n\to\infty}\FF_n(\gamma_n|\mu_n)\ge \FF(\gamma|\mu).
  \end{equation}
\end{corollary}
\begin{proof}
    The lower semicontinuity of $\FF$ follows by 
    \eqref{eq:23}, which provides a representation of $\FF$ 
  as the supremum of a family of lower semicontinuous functionals
  for the narrow topology. Using $F_n\geq F_m$ for $n\geq m$ fixed,
  we have 
  \[
\liminf_{n\to\infty} \FF_n(\gamma_n|\mu_n)\geq \liminf_{n\to\infty}
\FF_m(\gamma_n|\mu_n) \geq \FF_m(\gamma|\mu),
  \]
by the above lower semicontinuity. Hence, it suffices
  to check that 
  \begin{equation}
    \label{eq:241}
    \lim_{n\to\infty}\FF_n(\gamma|\mu)=\FF( \gamma|\mu)\quad \text{for every } \gamma,\mu\in \cM(X).
  \end{equation}
  This formula follows easily by the monotonicity of the convex sets
  $\frF_n$ (associated to $F_n$ by \eqref{eq:281}) $\frF_{n}\subset
  \frF_{n+1}$ and by the fact that $\frF=\cup_n \frF_n$, since
  $\Fstar_n$ is pointwise decreasing to $\Fstar$.  Thus for every
  couple of simple and lower semicontinuous functions $(\phi,\psi)$
  taking values in $\frF$ we have $(\psi(x),\phi(x))\in \frF_N$ for
  every $x\in X$ and a sufficiently large $N$ so that
  \begin{displaymath}
    \liminf_{n\to\infty}\FF_n(\gamma|\mu)\ge
    \int_X \psi\,\d\mu+\int_X\phi\, \d\gamma.
  \end{displaymath}
  Since $\phi,\psi$ are arbitrary we conclude applying the 
  duality formula \eqref{eq:23}.
\end{proof}
Next, we provide a compactness result for the sublevels of the
relative entropy, which will be useful in Section
\ref{subsec:existenceprimal} (see Theorem~\ref{thm:easy-but-important}
and Lemma~\ref{le:lsc}). 

\begin{proposition}[Boundedness and tightness]
  \label{prop:DLT}
  If $\cK\subset \cM(X)$ is 
  bounded and $\rec F >0$, then for every $C\ge 0$ the
  sublevels of $\FF$ 
  \begin{equation}
    \label{eq:24}
    \Xi_C:=\Big\{\gamma\in \cM(X):\FF(\gamma|\mu)\le C\text{ for some }\mu\in
    \cK\Big\},
    \end{equation}
    are bounded. If moreover $\cK$ is equally tight and $\rec F
    =\infty$, then the sets $\Xi_C$ are equally tight.
\end{proposition}
%
\begin{proof}
  Concerning the properties of $\Xi_C$, we will use 
  the inequality
  \begin{equation}
    \label{eq:26}
    \lambda\gamma(B)\le 
    \FF(\gamma|\mu)+\Fstar(\lambda)\mu(B) 
    \quad\forevery 
    \lambda\in (0,\rec F ),~\text{ and }~
    B\in \BorelSets X.
  \end{equation}
  This follows easily by integrating the Young inequality $\lambda
  \sigma\le F(\sigma)+\Fstar(\lambda)$ for $\lambda>0$ and the
  decomposition $ \gamma=\sigma\mu+\gamma^\perp$ in $B$ with respect
  to $\mu$ and by observing that
  \begin{displaymath}
    \lambda \gamma(B)=
    \lambda \int_B \sigma\,\d\mu+\lambda\gamma^\perp(B)
    \le \lambda \int_B \sigma\,\d\mu+\rec F  \gamma^\perp(B)
    \quad 
    \text{if }0<\lambda<\rec F .
  \end{displaymath}
  Choosing first 
  $B=X$ in \eqref{eq:26} and an arbitrary $\lambda$ in $(0,\rec F )$ 
  (notice that $\Fstar (\lambda)<\infty$ thanks to \eqref{eq:206})
  we immediately get a uniform bound of $\gamma(X)$ 
  for every $\gamma\in \Xi_C$. 
  
  In order to prove 
  the tightness when $\rec F =\infty$, 
  whenever $\eps>0$ is given, we can choose 
  $\lambda=2C/\eps$ and $\eta>0$ so small that 
  $\eta
  \Fstar(\lambda)/\lambda\le \eps/2$, and then a compact set $K\subset X$ such that 
  $\mu(X\setminus K)\le \eta$ for every $\mu\in \cK$.
  \eqref{eq:26} shows that $\gamma(X\setminus K)\le \eps$ for every
  $\gamma\in \Xi$.
\end{proof}
We conclude this section with a useful representation of $\FF$ 
in terms of the reverse entropy $\FH$ \eqref{eq:12}
and the corresponding functional $\FHH$. We will use the result
in Section~\ref{subsec:reverse} for the  reverse formulation of the primal
 entropy-transport problem. \EEE

\begin{lemma}
  \label{le:reverse-identity}
  For every $\gamma,\mu\in \cM(X)$ we have
  \begin{equation}
    \label{eq:211}
    \FHH(\mu|\gamma)=\int_X \FH(\varrho(x))\,\d\gamma(x)+
    \FH_\infty\,\mu^\perp(X),
  \end{equation}
  where $\mu=\varrho\gamma+\mu^\perp$ is 
  the reverse Lebesgue decomposition given by \eqref{eq:310}. In particular
  \begin{equation}
    \label{eq:208}
    \FF(\gamma|\mu)=\FHH(\mu|\gamma).
  \end{equation}
\end{lemma}
\begin{proof}
  It is an immediate consequence of the dual characterization in \eqref{eq:23} 
  and the equivalence in \eqref{eq:297}\EEE.
  \end{proof}

\section{Optimal Entropy-Transport problems}
\label{sec:ET}

The major object of Part I is the entropy-transport functional,
where two measures $\mu_1\in\cM(X_1)$ and $\mu_2\in\cM(X_2)$
are given, and one has to find a transport plan $\gamma\in\cM(X_1\times X_2)$
that minimizes the functional.

\subsection{The basic setting}
\label{subsec:setting}
Let us fix the basic set of data for Entropy-Transport problems. We
are given
\begin{itemize}
\item[-] two Hausdorff topological spaces $(X_i,\tau_i)$, $i=1,2$,
  which  define the Cartesian product $\xX:=X_1\times X_2$ and the 
  canonical projections $\pi^i:\xX\to X_i$;
\item[-] two entropy functions $ F_i\in \Gamma(\R_+)$,
  thus satisfying \eqref{eq:18}; 
\item[-] a proper 
  lower semicontinuous cost function $\sfc:\xX\to
  [0,+\infty]$; 
%
\item[-] a couple of nonnegative Radon measures $\mu_i\in
  \cM(X_i)$
  with finite mass $\mass i:=\mu_i(X_i)$ satisfying the compatibility
  condition
  \begin{equation}
        \label{eq:83}
        J:=\Big(\mass 1\, \dom {F_1}\Big)\cap
    \Big(\mass 2\, \dom {F_2}\Big)\neq \emptyset.
  \end{equation}
\end{itemize}
We will often assume that the above basic setting is also \emph{coercive}:
this means that \emph{at least one} of the following two \EEE
coercivity conditions holds:
\begin{subequations}\label{eq:coercivity}
  \begin{gather}
    \label{eq:320}
    \text{$F_1$ and $F_2$ are superlinear, i.e.~$\rec{(F_i)}=+\infty$;}\\
    \text{$\rec{(F_1)}+\rec{(F_2)}+\inf \sfc>0$ 
      and $\sfc$ has compact sublevels.}
\label{eq:319}
 \end{gather}
\end{subequations}
For every transport plan $\ggamma\in \mathcal M(\xX)$ we define the
marginals $\gamma_i:=\pi^i_\sharp \ggamma$ and, as in \eqref{eq:21},
we define the relative entropies
\begin{equation}
  \label{eq:1}
  \begin{aligned}
    \mathscr F_{i}(\ggamma|\mu_i):={}&\int_{X_i} F_i\Big(\frac{\d
      \gamma_i}{\d\mu_i}\Big)\,\d\mu_i+ \rec{(F_i)} \gamma_i^\perp(X_i),
    \ \gamma_i=\pi^i_\sharp
    \ggamma=\sigma_i\mu_i+\gamma_i^\perp,\quad
    \sigma_i:=\frac{\d\gamma_i}{\d\mu_i}.
  \end{aligned}
\end{equation}
With this, we introduce
the \emph{Entropy-Transport functional} as
\begin{align}
  \label{eq:4}
  \ETint(\ggamma|\mu_1,\mu_2)&:=
  \sum_i \FF_i(\ggamma|\mu_i)+
  \int_{\sxX}\sfc(x_1,x_2)\,\d\ggamma(x_1,x_2),
\end{align}
possibly taking the value $+\infty$.
Our basic setting is \emph{feasible} if the functional $\ETint$ is
not identically $+\infty$, i.e.~there exists at least
one plan $\ggamma$ with $\ETint(\ggamma|\mu_1,\mu_2)<\infty$.

\subsection{The primal formulation of the Optimal Entropy-Transport problem}

In the basic setting described in the previous Section \ref{subsec:setting}, 
we want to investigate the following problem.

\begin{problem}[Entropy-Transport minimization]
  \label{pr:1} 
  Given $\mu_i\in \cM(X_i)$ find $\ggamma\in 
  \cM(\xX)=\cM(X_1\times X_2)$ minimizing $\ETint(
  \ggamma|{\mu_1,\mu_2})$, i.e.
  \begin{equation}
    \label{eq:20}
    \ETint(\ggamma|\mu_1,\mu_2)=
    \ET(\mu_1,\mu_2):=\inf_{\sssigma\in \cM(\sxX)}
    \ETint(\ssigma|\mu_1,\mu_2).
  \end{equation}
  We denote by $\OptET(\mu_1,\mu_2)\subset \cM(\xX)$ the collection of
  all the minimizers of \eqref{eq:20}.
\end{problem}

\begin{remark}[Feasibility conditions]
  \label{rem:feasible}
  \upshape 
  Problem \ref{pr:1} is feasible if there exists at least
  one plan $\ggamma$ with $\ETint(\ggamma|\mu_1,\mu_2)<\infty$.
  Notice that this is always the case when 
  \begin{equation}
    F_i(0)<\infty,\quad i=1,2,
    \label{eq:261}
  \end{equation}
  since
  among the competitors one can choose the null plan $\eeta$, 
  so that
    \begin{equation}
    \label{eq:137pre}
    \ET(\mu_1,\mu_2)\le \ETint(\eeta|\mu_1,\mu_2)= F_1(0)\mu_1(X)+F_2(0)\mu_2(X).
  \end{equation}
  More generally, thanks to \eqref{eq:83}
  a sufficient condition for feasibility in the
  nondegenerate case $\mass 1\mass 2\neq 0$ is
  that there exit functions $B_1$ and $B_2$ with 
  \begin{gather}
    \label{eq:86}
    \sfc(x_1,x_2)\le B_1(x_1)+B_2(x_2),\quad B_i\in \rmL^1(X_i,\mu_i).
  \end{gather}
  In fact, the plans
  \begin{equation}
    \label{eq:229}
    \ggamma=\frac{\theta }{\mass 1\mass 2}\mu_1\otimes \mu_2\quad
    \text{with }\theta\in J\quad\text{given by }\eqref{eq:83}
  \end{equation}
  are Radon \cite[Thm.~17, p.~63]{Schwartz73}, have finite cost and provide the estimate
  \begin{equation}
    \label{eq:87}
    \ET(\mu_1,\mu_2)\le m_1 F_1(\theta/m_1)+m_2 F_2(\theta/m_2)+
    \theta \sum_i \mass i^{-1}\|B_i\|_{\rmL^1(X_i,\mu_i)},\quad
    \forevery \theta\in J.
  \end{equation}
  Notice that \eqref{eq:83} is also necessary for feasibility:
  in fact, setting $m_{i,n}:=m_i+\gamma_i^\perp(X_i)/n$,
  the convexity of
  $F_i$, the definition \eqref{eq:227} of $\rec {(F_i)}$, and Jensen's inequality provide
  \begin{align}
    \label{eq:312}
    \notag
    \FF_i(\ggamma|\mu_i)&=
    \int_{X_i}F_i(\sigma_i)\,\d\mu_i+
    \lim_{n\up\infty} \int_{X_i}F_i(n) \,\d (n^{-1}\gamma_i^\perp)
    \ge \lim_{n\to\infty} m_{i,n} F_i\big(\gamma_i(X_i)/m_{i,n}\big)
    \\&\ge  m_iF_i(m/m_i),\quad\text{ where}\quad
    m:=\gamma_i(X_i)=\ggamma(\xX). 
  \end{align}
  Thus, whenever $\EE(\ggamma|\mu_1,\mu_2)<\infty$, we have 
  \begin{equation}
    \label{eq:262}
    \EE(\ggamma|\mu_1,\mu_2)\ge m \inf\sfc+
    \mass1F_1(m/\mass1)+\mass2F_2(m/\mass2),\quad
  \end{equation}
  and therefore
  \begin{equation}
    \label{eq:263}
    m =\EEE\ggamma(\xX)\in \big(\mass 1\, \dom {F_1}\big)\cap
    \big(\mass 2\, \dom {F_2}\big)=J.
  \end{equation}
  We will often strengthen \eqref{eq:83} by assuming that at least one
  of the domains of the entropies $F_i$ has nonempty interior,
  containing a point of the other domain:  
  \begin{equation}
    \label{eq:302a}
    \Big(\interior\big(m_1{\dom{F_1}}\big)\cap m_2\dom{F_2}\Big)\cup
    \Big(m_1\dom{F_1}\cap \interior\big(m_2{\dom{F_2}}\big)\Big)\neq \emptyset.
  \end{equation}
  This condition is surely satisfied if $J$ has nonempty interior, 
  i.e.~%
  $\max (m_1 \r_1^-,m_2\r_2^-)< \min(m_1 \r_1^+,m_2 \r_2^+),$
  where $\r_i^-=\inf\dom{F_i}$, $\r_i^+:=\sup\dom{F_i}$.
\end{remark}
We also observe that whenever $\mu_i(X_i)=0$
then the null plan $\ggamma=\eeta_0$ provides the trivial solution to Problem
\ref{pr:1}.
Another trivial case occurs when $F_i(0)<\infty$ and $F_i$ are nondecreasing in
$\dom{F_i}$ (in particular when $F_i(0)=0$). 
Then it is clear that the null plan is a minimizer and $\ET(\mu_1,\mu_2)=
F_1(0)\mass 1+F_2(0)\mass 2$.
%
\subsection{Examples}
  \label{ex:1}
  Let us consider a few particular cases:
  \begin{enumerate}[\rm E.1]
  \item \textbf{Costless transport}: \nc Consider the case $\sfc\equiv0$.  
    Since $F_i$ are convex, in this case
    the minimum is attained when the marginals $\gamma_i$ have
    constant densities. Setting
    $\sigma_i\equiv \theta/m_i$ in order to have 
    $m_1\sigma_1=m_2\sigma_2$, we thus have
    \begin{equation}
      \label{eq:54}
      \ET(\mu_1,\mu_2)=\MPc {m_1}{m_2}0:=
      \min\Big\{m_1F_1(\theta/m_1)+m_2F_2(\theta/m_2):\theta\ge 0
      \Big\}.
    \end{equation}
  \item \textbf{Entropy-potential problems}: 
    If $\mu_2\equiv \eta_0$ then
    setting $V(x_1):=\inf_{x_2\in X_2}\sfc(x_1,x_2)$
    we easily get
    \begin{equation}
      \label{eq:334}
      \ET(\mu,0)=\inf_{\gamma\in
        \cM(X_1)}\FF_1(\gamma|\mu)+\int_{X_1}V\,\d\gamma+
      \rec{(F_2)}\gamma(X_1).
    \end{equation}
  \item \textbf{Pure transport problems:} We choose $F_i(r)=\rmI_1(r)=
    \begin{cases}
      0&\text{if }r=1\\
      +\infty&\text{otherwise}.
    \end{cases}$

    In this case any feasible plan $\ggamma$ should have $\mu_1$ and
    $\mu_2$ as marginals and the functional just reduces to 
    the pure transport part
    \begin{equation}
      \label{eq:59}
      \mathsf{T}(\mu_1,\mu_2)=\min\Big\{\int_{X_1\times
        X_2}\sfc\,\d\ggamma:\quad
      \pi^i_\sharp\ggamma=\mu_i\Big\}.
    \end{equation}
    As a necessary condition for feasibility we get
    $\mu_1(X_1)=\mu_2(X_2)$. 

    A situation equivalent to the optimal transport case occurs when
    \eqref{eq:302a} does not hold.
    In this case, the set $J$ defined by \eqref{eq:83} contains only
    one point $\theta$ which separates $m_1\dom {F_1}$ and
    $m_2\dom{F_2}$:
    \begin{equation}
      \theta=m_1\r _1^+=m_2\r _2^-\quad
      \text{or }\quad
      \theta=m_1 \r _1^-=m_2\r _2^+.\label{eq:294}
    \end{equation}
    It is not difficult to check that in this case
    \begin{equation}
      \label{eq:289}
      \ET(\mu_1,\mu_2)=m_1F_1(\theta/m_1)+
      m_2F_2(\theta/m_2)+\mathsf{T}(\mu_1,\mu_2).
    \end{equation}
        
    \item  \textbf{Optimal transport with density constraints:}
      We realize density constraints by introducing characteristic functions 
      of intervals $[a_i,b_i]$, viz. \EEE
      $F_i(r):=\rmI_{[a_i,b_i]}(r)$, $a_i\le 1\le b_i$. 
      E.g.~when $a_i=1$, $b_i=\infty$ we have
      \begin{equation}
      \label{eq:59bis}
      \ET(\mu_1,\mu_2)=\min\Big\{\int_{X_1\times
        X_2}\sfc\,\d\ggamma:\quad
      \pi^i_\sharp\ggamma\ge \mu_i\Big\}.
    \end{equation}
    For $[a_1,b_1]=[0,1]$ and $[a_2,b_2]=[1,+\infty]$ we get
      \begin{equation}
      \label{eq:59tris}
      \ET(\mu_1,\mu_2)=\min\Big\{\int_{X_1\times
        X_2}\sfc\,\d\ggamma:\quad
      \pi^1_\sharp\ggamma\le \mu_1,\ \pi^2_\sharp\ggamma\ge\mu_2\Big\},
    \end{equation}
    whose feasibility requires $\mu_2(X_2)\ge \mu_1(X_1)$.
    \item \textbf{Pure entropy 
        problems:} These problems arise if $X_1=X_2=X$ and transport is forbidden, i.e. 
        $\rec{(F_i)}=+\infty$, $\sfc(x_1,x_2)=
    \begin{cases}
      0&\text{if }x_1=x_2\\
      +\infty&\text{otherwise.}
    \end{cases} $

    In this case the marginals of $\ggamma$ coincide: we denote them
    by $\gamma$. We can write the
    density of $\gamma$ w.r.t.\ any measure $\mu$ such that
    $\mu_i\ll\mu$ (say, e.g., $\mu=\mu_1+\mu_2$) as $\gamma=\vartheta
    \mu$ and then $\mu_i=\vartheta_i\mu$.  Since $\gamma\ll\mu_i$ we
    have $\vartheta(x)=0$ for $\mu$-a.e.~$x$ where
    $\vartheta_1(x)\vartheta_2(x)=0$. Thus
    $\sigma_i=\vartheta/\vartheta_i$ is well defined and we have
    \begin{equation}
      \label{eq:55}
      \ETint(\ggamma|\mu_1,\mu_2)=\int_{X}
      \Big(\vartheta_1 F_1(\vartheta/\vartheta_1)
      +\vartheta_2 F_2(\vartheta/\vartheta_2)\Big)\,\d\mu,
    \end{equation}
    with the convention that $\vartheta_i
    F_i(\vartheta/\vartheta_i)=0$ if $\vartheta=\vartheta_i=0$.
    Since we expressed everything in terms of $\mu$, 
    by recalling the definition of the function $\MP_0$ given in
    \eqref{eq:54} we get 
    \begin{equation}
      \label{eq:57}
      \ET(\mu_1,\mu_2)=\int_X \MP_0\Big(\frac
      {\d\mu_1}{\d\mu},\frac{\d\mu_2}{\d\mu}\Big)\,\d\mu,\quad
      \text{whenever} \quad \mu_i\ll\mu.
    \end{equation}
    In the Hellinger case $F_i(s)=
    U_1(s)=s\log s-s+1$ a simple calculation
    yields
    \begin{equation}
      \label{eq:58}
      \MPc{\theta_1}{\theta_2}0=\theta_1+\theta_2-2\sqrt{\theta_1\theta_2}=
      \Big(\sqrt {\theta_1}-\sqrt{\theta_2}\Big)^2.
    \end{equation}
    In the Jensen-Shannon case, where $F_i(s)=U_0(s)=s-1-\log s$,
    we obtain
    \begin{displaymath}
      H_0(\theta_1;\theta_2)=\theta_1\log
      \Big(\frac{2\theta_1}{\theta_1+\theta_2} \Big)+\theta_2
      \log \Big(\frac {2\theta _2}{\theta_1+\theta_2}\Big).
    \end{displaymath}
    Two other interesting examples are provided
    by the quadratic case $F_i(s)=\frac 12(s-1)^2$ and by the
    nonsmooth ``piecewise affine'' case $F_i(s)=|s-1|$, for which we
    obtain
    \begin{displaymath}
      H_0(\theta_1,\theta_2)=
      \frac1{2(\theta _1+\theta _2)}(\theta
      _1-\theta _2)^2,\quad\text{and}\quad
      H_0(\theta_1,\theta_2)=|\theta_1-\theta_2|,\quad\text{respectively}.
    \end{displaymath}
    
  \item \textbf{Regular entropy-transport problems:} 
    These problems correspond to the choice of a couple of differentiable entropies $F_i$ with
    $\dom{F_i}\supset (0,\infty)$, as in the case of the 
    power-like entropies $\PE_p$ 
    defined in \eqref{eq:237}. When they vanish (and thus have a
    minimum) at $\r =1$,
    the Entropic Optimal Transportation can be considered as a smooth
    relaxation of the Optimal Transport case E.3. 

  \item \textbf{Squared Hellinger-Kantorovich distances:} For a metric
    space $(X,\sfd)$, set $X_1=X_2=X$ and let $\tau$ be induced by
    $\sfd$.  Further, set $F_1(\r )=F_2(\r ):=\PE_1(\r )=\r \log \r
    -\r +1$ and
      \begin{displaymath}
        \sfc(x_1,x_2):=-\log\Big(\cos^2\big(\sfd(x_1,x_2)\land \pi/2\big)\Big)
        \quad\text{or simply}\quad
        \sfc(x_1,x_2):=\sfd^2(x_1,x_2).
      \end{displaymath}
    These cases will be thoroughly studied in the second
    part of the present paper, see Section~\ref{sec:LET}.

  \item \textbf{Marginal Entropy-Transport problems:} In this case one
    of the two marginals of $\ggamma$ is fixed, say $\gamma_1$, by
    choosing $F_1(r):=\rmI_1(r)$.  Thus the functional minimizes the
    sum of the transport cost and the relative entropy of the second
    marginal $\FF_2(\gamma_2|\mu_2)$ with respect to a reference
    measure $\mu_2$, namely
    \begin{displaymath}
      \ET(\mu_1,\mu_2)=
      \min_{\gamma\in
        \cM(X_2)}\Big\{\FF_2(\gamma|\mu_2)+\sfT(\gamma,\mu_1)\Big\}. 
    \end{displaymath}
    This is the typical situation one has to solve at each iteration
    step of the Minimizing Movement scheme
    \cite{Ambrosio-Gigli-Savare08}, when $\sfT$ is a (power of a)
    transport distance induced by $\sfc$, as in the
    Jordan-Kinderlehrer-Otto approach
    \cite{Jordan-Kinderlehrer-Otto98}.

  \item \textbf{The Piccoli-Rossi ``generalized Wasserstein distance''
      \cite{Piccoli-Rossi14,Piccoli-Rossi14preprint}:} for a metric
    space $(X,\sfd)$, set $X_1=X_2=X$, let $\tau$ be induced by
    $\sfd$, and consider $F_1(\r )=F_2(\r ):=V(\r )=|\r-1|$ with
    $\sfc(x_1,x_2):=\sfd(x_1,x_2)$.
  \item \textbf{The discrete case.}
    Let $\mu_1=\sum_{i=1}^m\alpha_i\delta_{x_i}$, $\mu_2=\sum_{j=1}^N
    \beta_j\delta_{y_j}$ with $\alpha_i,\beta_j>0$,
    and let $\sfc_{i,j}:=\sfc(x_i,y_j)$.
    The Entropy-Transport problem for this discrete model
    consists in finding coefficients $\gamma_{i,j}\ge0$ 
    which minimize
    \begin{equation}
      \label{eq:79}
      \EE(\gamma_{i,j}|\alpha_i,\beta_j):=
      \sum_{i}\alpha_i F_1\Big(\frac{\sum_j
        \gamma_{i,j}}{\alpha_i}\Big)+
      \sum_{j}\beta_j F_2\Big(\frac{\sum_i
        \gamma_{i,j}}{\beta_j}\Big)+\sum_{i,j} \sfc_{i,j}\gamma_{i,j}.
    \end{equation}
  \end{enumerate}

\subsection{Existence of solutions to the primal problem}\label{subsec:existenceprimal}
The next result provides a first general existence result for 
Problem \ref{pr:1} in the basic coercive setting of Section \ref{subsec:setting}.
\begin{theorem}[Existence of minimizers]
  \label{thm:easy-but-important}
  Let us assume that Problem \ref{pr:1} is feasible (see Remark
  \ref{rem:feasible})
  and coercive, i.e.~at least one of the following conditions hold:
  \begin{enumerate}[(i)]
  \item the entropy functions $F_1$ and $F_2$ are superlinear,
    i.e.~$\rec{(F_1)}=\rec{(F_2)}=+\infty$;
  \item 
    $\sfc$ has compact sublevels in $\xX$ and
    $\rec{(F_1)}+\rec{(F_2)}+\inf \sfc>0$.
  \end{enumerate}
  Then Problem \ref{pr:1} admits at least one optimal solution.  
  In this case $\OptET(\mu_1,\mu_2)$ is a compact convex set
  of $\cM(\xX)$.
\end{theorem}
\begin{proof}
  We can apply the Direct Method of Calculus of Variations:
  since the map $\ggamma\mapsto
  \ETint(\ggamma|\mu_1,\mu_2)$ is lower semicontinuous in $\cM(X_1\times
  X_2)$ by Theorem \ref{thm:duality}, it is sufficient to show that
  its sublevels are relatively compact, thus 
  bounded and equally tight by Prokhorov Theorem \ref{thm:Prokhorov}.
  In both  cases
  boundedness follows by the coercivity assumptions and 
  the estimate \eqref{eq:262}: 
  
  in fact, by the definition \eqref{eq:227} of $\rec{(F_i)}$ we
  can find $\bar s\ge0 $ such that 
  $\frac {\mass i}m F_i(\frac m{\mass i})\ge \frac 12 \rec{(F_i)}$
  whenever
  $m\ge \bar s \, \mass i$; if $a:=\inf c+\sum_i
  \rec{(F_i)}>0$ the estimate \eqref{eq:262} yields
  \begin{displaymath}
    \ggamma(\xX)\le \frac 2a \EE(\ggamma|\mu_1,\mu_2)\quad 
    \text{for every }\ggamma\in  \cM(\xX)\text{ with }
    \ggamma(\xX)\ge \bar s \max(\mu_1(X_1),\mu_2(X_2)).  
  \end{displaymath}
  \nc%
  In case \emph{(ii)} equal tightness is a consequence of the Markov inequality
  and the nonnegativity of $F_i$:
  in fact, considering the compact sublevels $K_\lambda:=\{(x_1,x_2)\in X_1\times
  X_2:\sfc(x_1,x_2)\le \lambda\}$, we have
  \begin{displaymath}
    \ggamma(\xX\setminus K_\lambda)\le
    \lambda^{-1}\int\sfc\,\d\ggamma\le 
    \lambda^{-1}\EE(\gamma|\mu_1,\mu_2)\quad
    \forevery \lambda>0.
  \end{displaymath}
  In the case \emph{(i)}, since $\sfc\ge0$
  Proposition \ref{prop:DLT} shows that both the marginals 
  of plans in a sublevel of the energy 
  are equally tight: we thus conclude
  by \cite[Lemma 5.2.2]{Ambrosio-Gigli-Savare08}.
\end{proof}

\begin{remark}
  \upshape The assumptions \emph{(i)} and \emph{(ii)} in the previous
  Theorem are almost optimal, and it is possible to find
  counterexamples when they are not satisfied. In the case when
  $0<\rec{(F_1)}+\rec{(F_2)}<\infty$ but $\sfc$ does not have compact
  sublevels, one can just take $F_i(\r ):=U_0(\r )=\r -\log \r -1$,
  $X_i:=\R$, $\sfc(x_1,x_2):=3\rme^{-x_1^2-x_2^2}$, $\mu_i=\delta_0$.

  Any competitor is of the form
  $\ggamma:=\alpha\delta_0\otimes\delta_0+ \nu_1\otimes
  \delta_0+\delta_0\otimes \nu_2$ with $\nu_i\in \cM(\R)$ and
  $\nu_i(\{0\})=0$. Setting $n_i:=\nu_i(\R)$ we find
  \begin{displaymath}
    \EE(\gamma|\mu_1,\mu_2)=F(\alpha+n_1)+F(\alpha+n_2)+3\Big(\alpha+
    \int \rme^{-x^2}\,\d(\nu_1+\nu_2)\Big)+n_1+n_2.
  \end{displaymath}
  Since $\min_{\r} F(\r )+\r =\log 2$ is attained at $\r =1/2$, we immediately see
  that
  \begin{displaymath}
    \EE(\gamma|\mu_1,\mu_2)\ge 2\log 2+\alpha+3 \int \rme^{-x^2}\,\d(\nu_1+\nu_2)\ge 2\log 2.
  \end{displaymath}
  Moreover, $2\log 2$ is the infimum, which is reached by choosing $\alpha=0$ and
  $\nu_1=\nu_2=\frac12\delta_x$, and letting $x\to\infty$.  On the other
  hand, since $n_1+n_2+\alpha>0$, the infimum can never be attained.
  
  In the case when $\sfc$ has compact sublevels but
  $\rec{(F_1)}=\rec{(F_2)}=\min\sfc=0$, it is sufficient to take 
  $F_i(\r ):=\r ^{-1}$, $X_i=[-1,1]$, $\sfc(x_1,x_2)=x_1^2+x_2^2$, and
  $\mu_i=\delta_0$.
  Taking
  $\gamma_n:=n\delta_0\otimes\delta_0$ one easily checks that 
  $\inf\EE(\ggamma|\mu_1,\mu_2)=0$ but 
  $\EE(\ggamma|\mu_1,\mu_2)>0$ for every
  $\ggamma\in \cM(\R^2)$.  
\end{remark}

Let us briefly discuss the question of uniqueness, the first 
result only addresses the marginals $\gamma_i= \pi^i_\sharp \ggamma$.\EEE
\begin{lemma}[Uniqueness of the marginals in the superlinear strictly
  convex case]
  \label{le:uniqueness}
  Let us suppose that $F_i$ are \emph{strictly convex} functions. 
  Then 
  the $\mu_i$-absolutely continuous part $\sigma_i\mu_i$ 
  of the marginals $\gamma_i=\pi^i_\sharp \ggamma$ of any optimal
  plan
  are uniquely determined. In particular,
  if $F_i$ are also superlinear, then the marginals $\gamma_i$ are 
  uniquely determined, i.e.\ if $\ggamma',\ggamma''\in 
  \OptET(\mu_1,\mu_2)$ then $\pi^i_\sharp\ggamma'=\pi^i_\sharp\ggamma''$, $i=1,2$.
\end{lemma}
\begin{proof}
  It is sufficient to take $\ggamma=\frac 12 \ggamma'+\frac 12 \ggamma''$
  which is still optimal in $\OptET(\mu_1,\mu_2)$ since $\EE$ is 
  a convex functional w.r.t.~$\ggamma$. We have
  $\pi^i_\sharp \ggamma=\gamma_i=\frac 12 \gamma_i'+\frac 12
  \gamma_i''=
  \frac 12(\sigma_i'+\sigma_i'')\mu+\frac 12(\gamma_i')^\perp+
  \frac 12(\gamma_i'')^\perp$ and
  we observe that the minimality of $\ggamma$ and the convexity of each 
  addendum $F_i$ in the functional yield
  \begin{align*}
    \FF_i(\gamma_i|\mu_i) =\frac 12 \FF_i(\gamma'_i|\mu_i)+
    \frac 12 \FF_i(\gamma_i''|\mu_i) \quad i=1,2.
  \end{align*}
  Since $\gamma_i^\perp(X_i)= \frac 12
  (\gamma_i')^\perp(X_i)+\frac 12 (\gamma_i'')^\perp(X_i)$
  we obtain 
  \begin{displaymath}
    \quad
    \int_X \Big(F_i(\sigma_i)-\frac 12F_i(\sigma_i')-\frac 12 F_i(\sigma_i'')\Big)\,\d\mu_i=0 \quad i=1,2.
  \end{displaymath}
  Since $F_i$ is strictly convex, the above identity implies 
  $\sigma_i=\sigma_i'=\sigma_i''$ $\mu_i$-a.e.~in $X$.
\end{proof}
The next corollary reduces the uniqueness question 
of optimal couplings in $\OptET(\mu_1,\mu_2)$ 
to corresponding results for the Kantorovich 
problem associated to the 
cost $\sfc$.
\begin{corollary}
  Let us suppose that $F_i$ are \emph{superlinear strictly convex} functions and
  that for every couple of probability measures $\nu_i\in \cP(X_i)$ 
  with $\nu_i\ll\mu_i$ the optimal transport problem associated to the
  cost $\sfc$ (see Example {\rm E.3} of Section \ref{ex:1}) admits a unique solution.
  Then $\OptET(\mu_1,\mu_2)$ contains at most one plan. 
\end{corollary}
\begin{proof}
  We can assume $\mass{i}=\mu_i(X_i)> 0$ for $i=1,2$.\EEE
  It is clear that any $\ggamma\in \OptET(\mu_1,\mu_2)$ 
  is a solution of the optimal transport problem for the cost $\sfc$
  and given (possibly normalized) marginals $\gamma_i$.
  Since $\gamma_i\ll\mu_i$ and $\gamma_1$ and $\gamma_2$ are unique by Lemma \ref{le:uniqueness},
  we conclude. \EEE
\end{proof}
\begin{example}[Uniqueness in Euclidean spaces]
  \upshape 
  If $F_i$
  are \emph{superlinear strictly convex} functions,
  $\sfc(x,y)=h(x-y)$ for a strictly convex
  function $h:\R^d\to[0,\infty)$ and $\mu_1\ll\Leb d$,
  then Problem \ref{pr:1} admits at most one solution.
  It is sufficient to apply the previous corollary in conjunction with
  \cite[Theorem 6.2.4]{Ambrosio-Gigli-Savare08}
\end{example}
\begin{example}[Nonuniqueness of optimal couplings]
  \upshape
  Consider the logarithmic density functionals $F_i(\r )=
  U_1(\r )=\r \log \r -\r +1$,
  the Euclidean space $X_1=X_2=\R^2$ and any cost 
  $\sfc$ of the form $\sfc(x_1,x_2)=h(|x_1{-}x_2|)$.
  For the measures 
  \begin{displaymath}
    \mu_1=\delta_{(-1,0)}+\delta_{(1,0)}, ~\text{and}~
    \mu_2\text{ with support in $\{0\}\times \R$
      and containing at least two points},
  \end{displaymath}
  there is an infinite number of optimal plans. 
  In fact, we shall see that the first marginal $\gamma_1$ 
  of any optimal plan $\ggamma$ will have
  full support in $(-1,0),(1,0)$, i.e.~it will of the form
  $a \delta_{(-1,0)}+b\delta_{(1,0)}$ with strictly positive $a,b$, and
  the support of the second marginal 
  $\gamma_2$ will be concentrated in $\{0\}\times \R$ and will contain at least
  two points.
  In fact, any plan $\ssigma$ with marginals $\gamma_1,\gamma_2$ will then be
  optimal,
  since it can be written as the disintegration
  \begin{displaymath}
    \ssigma=\int_{\R} 
    \Big(\alpha(y)\delta_{(-1,0)}+\beta(y)\delta_{(1,0)}\Big)\,\d\gamma_2(y)
  \end{displaymath}
  with arbitrary nonnegative densities $\alpha,\beta$ with
  $\alpha+\beta=1$ and
  $\int\alpha\,\d\gamma_2(y)=a$,
  $\int\beta\,\d\gamma_2(y)=b$. 
  In fact, the cost contribution of $\ssigma$
  to the total energy is
  \begin{displaymath}
    \int_\R h(\sqrt{1+y^2})\,\d\gamma_2(y)
  \end{displaymath}
  and it is independent of the choice of $\alpha$ and $\beta$.\qed
\end{example}
We conclude this section by proving a simple lower semicontinuity
property for the energy-transport functional $\ET$. Note that in metrizable spaces any weakly convergent
sequence of Radon measures is tight.
\begin{lemma}
  \label{le:lsc}
  Let $\L$ be a directed set, $(F_i^\lambda)_{\lambda\in \L}$ and $(\sfc^\lambda)_{\lambda\in \L}$ be
  monotone nets of superlinear entropies and costs pointwise converging to $F_i$ and $\sfc$ respectively,
  and let $(\mu_i^\lambda)_{\lambda\in \L}$ be equally tight nets of
  measures narrowly converging to $\mu_i$ in $\cM(X_i)$. Denoting by
  $\ET^\lambda$ (resp.~$\ET$) the corresponding Entropy-Transport functionals 
  induced by $F_i^\lambda$ and $\sfc^\lambda$ (resp.~$F_i$ and $\sfc$) we have
  \begin{equation}
    \label{eq:256}
    \liminf_{\lambda\in \L}\ET^\lambda(\mu_1^\lambda,\mu_2^\lambda)\ge \ET(\mu_1,\mu_2).
  \end{equation}
\end{lemma}
\begin{proof}
  Let $\ggamma^\lambda \in \OptET(\mu^\lambda_1,\mu^\lambda_2)\subset \cM(\xX)$ be a corresponding net of
  optimal plans. 
  The statement follows if assuming
  that $\EE(\ggamma^\lambda|\mu_1^\lambda,\mu_2^\lambda)=\ET(\mu_1^\lambda,\mu_2^\lambda)\le
  C<\infty$
  we can prove that $\ET(\mu_1,\mu_2)\le C$.
  By applying Proposition \ref{prop:DLT}
  we obtain that the sequences of marginals $\pi^i_\sharp \ggamma^\lambda$
  are tight in $\cM(X_i)$, so that the net
  $\ggamma^\lambda$ is also tight. By extracting a suitable subnet
  (not relabeled)
  narrowly converging
  to $\ggamma$ in $\cM(\xX)$, 
  we can still apply Proposition \ref{prop:DLT} 
  and the lower semicontinuity of 
  the entropy part $\FF^\lambda$ of the functional $\EE$
  to obtain
  $\liminf_{\lambda\in \L}\FF^\lambda(\ggamma^\lambda|\mu_1^\lambda,\mu_2^\lambda)\ge
  \FF(\ggamma|\mu_1,\mu_2).$
  A completely analogous argument shows that 
  $\liminf_{\lambda\in \L}\int\sfc^\lambda\,\d\ggamma^\lambda\ge \int\sfc\,\d\ggamma$.
\end{proof}
As a simple application we prove the extremality of the class of Optimal
Transport problems (see Example E.3 in Section \ref{ex:1}) in the set of entropy-transport problems.

\begin{corollary}
  Let $F_1,F_2\in \Gamma(\R_+)$ be satisfying $F_i(r)>F_i(1)=0$ for
  every $r\in [0,\infty),\ r\neq 1$ and let $\ET^n$ be the Optimal
  Entropy Transport value \eqref{eq:20} associated to $(nF_1,nF_2)$. 
  Then for every couple of equally tight sequences
  $(\mu_{1,n},\mu_{2,n})\subset \cM(X_1)\times \cM(X_2)$, $n\in \N$, 
  narrowly converging to $(\mu_1,\mu_2)$ we have
  \begin{equation}
    \label{eq:253}
    \lim_{n\up\infty} \ET^n(\mu_{1,n},\mu_{2,n})=\mathsf{T}(\mu_1,\mu_2).
  \end{equation}
\end{corollary}

\subsection{The reverse formulation of the primal problem}
\label{subsec:reverse}
Let us introduce the reverse entropy 
functions  $\FH_i$ (see \eqref{eq:12}) via \EEE
\begin{equation}
  \label{eq:12bis}
  \FH_i(\s ):=
  \begin{cases}
    \s F_i(1/\s )&\text{if }\s >0,\\
    \rec {(F_i)} &\text{if }\s =0,
  \end{cases}
\end{equation}
and let $\FHH_i$ be the corresponding integral functionals as in
\eqref{eq:211}. 

Keeping the notation of Lemma~\ref{le:Lebesgue}
\begin{equation}
  \label{eq:257}
  \gamma_i:=\pi^i_\sharp\ggamma\in \cM(X_i),\quad
  \mu_i=\varrho_i\gamma_i+\mu_i^\perp,\quad
  \varrho_i=\frac{\d\mu_i}{\d\gamma_i},
\end{equation}
we can thus define
\begin{equation}
  \label{eq:196}
  \begin{aligned}
    \FHH(\mu_1,\mu_2|\ggamma):={}&
    \sum_{i}\FHH_i(\mu_i|\gamma_i)+\int_{\sxX}\sfc\,\d\ggamma=
    \\={}&
    \int_\sxX
    \Big(\FH_1(\varrho_1(x_1))+\FH_2(\varrho_2(x_2))+\sfc(x_1,x_2)\Big)\,\d\ggamma+
    \sum_i F_i(0)\mu_i^\perp(X_i).
  \end{aligned}
\end{equation}
By Lemma \ref{le:reverse-identity} we easily get the reverse
formulation
of the optimal Entropy-Transport Problem \ref{pr:1}.
\begin{theorem}
  \label{thm:reverse-characterization}
  For every $\ggamma\in \cM(\xX)$ and $\mu_i\in \cM(X_i)$
  \begin{equation}
    \label{eq:258}
    \EE(\ggamma|\mu_1,\mu_2)=\FHH(\mu_1,\mu_2|\ggamma).
  \end{equation}
  In particular
  \begin{equation}
    \label{eq:259}
    \ET(\mu_1,\mu_2)=\inf_{\sggamma\in \cM(\sxX)}\FHH(\mu_1,\mu_2|\ggamma),
  \end{equation}
  and $\ggamma\in \OptET(\mu_1,\mu_2)$ if and only if 
  it minimizes $\FHH(\mu_1,\mu_2|\cdot)$ in $\cM(\xX)$.
\end{theorem}
The functional $\FHH(\mu_1,\mu_2|,\cdot)$ is still a convex functional
and it will be useful in Section~\ref{sec:MP}.

\section{The dual problem}
\label{sec:duality}
In this section we want to compute and study the dual problem and the 
corresponding optimality conditions
for the Entropy-Transport Problem \ref{pr:1}
in the basic \emph{coercive} setting of Section \ref{subsec:setting}.

\subsection{The ``inf-sup'' derivation of the dual problem in the
  basic coercive setting}
\label{subsec:basic-dual}
In order to write the first formulation of the dual problem we introduce
the reverse entropy functions $\FH_i$ defined as in \eqref{eq:12} 
or Section~\ref{subsec:reverse} 
and their conjugate $\FHstar_i:\R\to (-\infty,+\infty]$ which can be
expressed by
\begin{equation}
  \label{eq:300}
  \FHstar_i(\psi):=\sup_{s>0}\big( s\psi-sF_i(1/s)\big)=
  \sup_{r>0} \big(\psi-F_i(r)\big)/r.
\end{equation}
The equivalences \eqref{eq:217} yield, for all $(\phi,\psi)\in \R^2$ 
\begin{equation}
  \label{eq:295}
  (\phi,\psi)\in \frF_i\quad\Leftrightarrow\quad
  \phi\le -\FHstar_i(\psi).
\end{equation}
As a first step we use the dual formulation of the entropy functionals
given by Theorem \ref{thm:duality} (cf.~\eqref{eq:299}) and find
\begin{align*}
  \ETint(\ggamma|\mu_1,\mu_2)&=
  \int \sfc\,\d\ggamma+\sup\Big\{\sum_i
  \Big(\int_{X_i}\psi_i\,\d\mu_i-\sum_i
  \int_{X_i}\FHstar_i(\psi_i)\,\d\gamma_i \Big)
  : \psi_i, \FHstar_i(\psi_i)\in \LSC_s(X_i) 
 \Big\}.
\end{align*}
It is natural to introduce the saddle function
$\LL(\ggamma,\ppsi)$ depending on 
$\ggamma\in \cM(\xX)$ and
$\ppsi=(\psi_1,\psi_2)$ 
%
(we omit here the dependence on the fixed measures $\mu_i\in\cM(X_i)$\EEE)
\begin{equation}
  \label{eq:42}
  \LL(\ggamma,\ppsi):=
  \int_{\sxX}
  \Big(\sfc(x_1,x_2)-\FHstar_1(\psi_1(x_1))-\FHstar_2(\psi_2(x_2))\Big)\,
  \d\ggamma
  +\sum_i
  \int_{X_i}\psi_i\,\d\mu_i.
\end{equation}
In order to guarantee that $\LL$ takes real values, 
we consider the convex set 
\begin{equation}\label{eq:45}
\rmM:=\big\{\ggamma\in \cM(\xX):
\int \sfc\,\d\ggamma<\infty\big\}.
\end{equation}
We thus have
\begin{displaymath}
  \ETint(\ggamma|\mu_1,\mu_2)=
  \sup
  _{\psi_i, \FHstar_i(\psi_i)\in \LSC_s(X_i)}
    \LL(\ggamma,\ppsi)
\end{displaymath}
and the Entropy-Transport Problem can be written as
\begin{equation}
  \label{eq:43}
  \ET(\mu_1,\mu_2)=
  \infp_{\sggamma\in \rmM } 
  \, \sup_{\psi_i, \FHstar_i(\psi_i)\in \LSC_s(X_i)}
  \LL(\ggamma,\ppsi).
\end{equation}
We can then obtain the dual problem by interchanging the order of
$\inf$ and $\sup$ as in Section~\ref{subsec:minmax}.
Let us
denote by $\varphi_1\oplus\varphi_2$ the function
$(x_1,x_2)\mapsto \varphi_1(x_1)+\varphi_2(x_2)$.
Since for every  $\ppsi=(\psi_1,\psi_2)$ with
$\psi_i, \FHstar_i(\psi_i)\in \LSC_s(X_i)$ 
\begin{displaymath}
  \inf_{\sggamma\in \rmM}
  \int\Big(\sfc(x_1,x_2)-\FHstar_1(\psi_1(x_1))-\FHstar_2(\psi_2(x_2))\Big)\,
  \d\ggamma=
  \begin{cases}
    0&\text{if }\FHstar_1(\psi_1)\oplus\FHstar_2(\psi_2)\le \sfc,  \\
    -\infty&\text{otherwise},
  \end{cases}
\end{displaymath}
we obtain \EEE
\begin{equation}
  \label{eq:47}
  \inf_{\sggamma\in \rmM}
  \LL(\ggamma,\ppsi)=
  \begin{cases}
    \displaystyle \sum_i\int_{X_i}\psi_i\,\d\mu_i&
    \text{if }\FHstar_1(\psi_1)
    \oplus\FHstar_2(\psi_2)
    \le \sfc , \\
    -\infty&\text{otherwise.}
  \end{cases}
\end{equation}
Thus, \eqref{eq:47} provides the dual formulation, that we
will study in the next section.

\subsection{Dual problem and optimality conditions}
\label{subsec:DualOpt}
\begin{problem}[$\ppsi$-formulation of the dual problem]
  \label{pr:2}
  Let $\FHstar_i$ be the convex functions defined by \eqref{eq:300}
  and let $\Cpsi{}$ be the 
  the convex set
\begin{equation}
  \label{eq:46}
  \Cpsi{} :=\Big\{\ppsi\in \LSC_s(X_1)\times\LSC_s(X_2): \ 
  \FHstar_i(\psi_i) \text{\ bounded},\ 
  \FHstar_1(\psi_1)\oplus\FHstar_2(\psi_2)
  \le \sfc
  \Big\}.
\end{equation}
  The dual Entropy-Transport problem consists in finding
  a maximizer $\ppsi\in\Cpsi{}$ for \EEE
  \begin{equation}
    \label{eq:49}
    \begin{aligned}
      \sfD(\mu_1,\mu_2)&=\sup_{\sppsi\in \cCpsi{}} 
      \int_{X_1}\psi_1\,\d\mu_1+
      \int_{X_2}\psi_2\,\d\mu_2.
    \end{aligned}
  \end{equation}
\end{problem}
As usual, by operating the change of variable
\begin{equation}
  \label{eq:303}
  \varphi_i:=-\FHstar(\psi_i), \quad
  \psi_i=\Gstar_i(\varphi_i):=-\Fstar_i(-\varphi_i),
\end{equation}
we can obtain an equivalent formulation
of the dual functional $\sfD$ 
as the supremum of the concave functionals 
\begin{equation}
  \label{eq:44}
  \DD(\vvarphi|\mu_1,\mu_2):=\sum_i\int_{X_i}\Gstar_i(\varphi_i)\,\d\mu_i,
\end{equation}
on the simpler convex set
  \begin{equation}
    \label{eq:301}
    \Cphi{}:=\Big\{\vvarphi\in \LSC_s(X_1)\times \LSC_s(X_2),\ 
    \Gstar_i(\varphi_i)
    \text{ bounded}
    ,\ 
    \varphi_1\oplus \varphi_2\le \sfc\Big\}.
  \end{equation}
\begin{problem}[$\vvarphi$-formulation of the dual problem]
  \label{pr:2bis}
  Let $\Gstar_i$ be the concave functions defined by \eqref{eq:303}
  and let $\Cphi{}$ be the 
  the convex set \eqref{eq:301}.
  The $\varphi$-formulation of the dual Entropy-Transport problem
  consists in finding a maximizer $\vvarphi\in\Cphi{}$ for
    \begin{equation}
    \label{eq:274}
    \sfD'(\mu_1,\mu_2)=\sup_{\svvarphi\in \cCphi{}}\DD(\vvarphi|\mu_1,\mu_2)
    =\sup_{\svvarphi\in \cCphi{}} 
    \sum_i\int_{X_i}\Gstar_i(\varphi_i)\,\d\mu_i. 
  \end{equation}
\end{problem}
\begin{proposition}[Equivalence of the dual formulations]
  \label{prop:equivalent-dual}
  The $\psi$- and the $\phi$- formulations of the dual problem are
  equivalent, $\sfD(\mu_1,\mu_2)=\sfD'(\mu_1,\mu_2)$.
\end{proposition}
\begin{proof}
  Since $\FHstar_i$ is 
  nondecreasing, for every
  $\ppsi\in \Cpsi{}$ the functions $\varphi_i:=\FHstar_i(\psi_i)$ belong
  to $\LSC_s(X_i)$ and satisfy $\varphi_1\oplus\varphi_2\le \sfc$,
  with $(-\varphi_i,\psi_i)\in
  \frF_i$. It then follows that 
  $\tilde\psi_i:=-\Fstar_i(-\varphi_i)=\Gstar_i(\varphi_i)\ge \psi_i$
  are bounded, so that $(\varphi_1,\varphi_2)\in \Cphi{}$
  and $\sfD'\ge \sfD$.  An analogous argument shows the converse inequality.
\end{proof}
Since ``$\inf \sup \ge \sup\inf$'' (cf.~\eqref{eq:278}), our derivation via \EEE
\eqref{eq:43} yields
\begin{equation}
  \label{eq:52}
  \ET(\mu_1,\mu_2)\ge \sfD(\mu_1,\mu_2).
\end{equation}
Using Theorem~\ref{thm:minimax} we will show in Section
\ref{sec:dualityGen} that \eqref{eq:52} is in fact an equality. Before
this, we first discuss for which class of functions $\psi_i,\varphi_i$
the dual formulations are still meaningful. Moreover, we analyze
the optimality conditions associated to the equality case in
\eqref{eq:52}.

\paragraph{Extension to Borel functions.}
It is intended that in some cases we will also consider larger 
classes
of potentials $\ppsi$ or $\vvarphi$ 
by allowing Borel functions with extended real values 
under suitable summability conditions.

First of all, recalling \eqref{eq:40} and \eqref{eq:183}, we extend
  $\FHstar$ and $\Gstar$ to $\bar \R$ by setting
  \begin{equation}
    \label{eq:327}
    \FHstar(-\infty):=-\rec F,\quad
    \FHstar(+\infty):=+\infty;\quad  
    \Gstar(-\infty):=-\infty,\quad
    \Gstar(+\infty):=F(0),
  \end{equation}
  and we observe that with the definition above 
  and according to \eqref{eq:322}--\eqref{eq:322bis}
 the couples
  \begin{equation}
  \text{$(-\varphi,\Gstar(\varphi))$ 
  and $(-\FHstar(\psi),\psi)$ 
  belong to $\bar\frF$ whenever $\psi\le F(0)$ and $\varphi\ge -\rec F$.}
\label{eq:340}
\end{equation}

We also set
\begin{equation}
    \label{eq:329}
    \zeta_1\pz\zeta_2:=\lim_{n\to\infty} (-n\lor \zeta_1\land n)+(-n\lor
    \zeta_2\land n)\quad\text{for every }\zeta_1,\zeta_2\in \bar \R.
\end{equation}
Notice that $(\pm\infty)\pz(\pm \infty)=\pm\infty$ and in the
ambiguous case $+\infty-\infty$ this definition yields
$(+\infty)\pz(-\infty)=0$. We correspondingly extend the definition of
$\oplus$ by setting
  \begin{equation}
    \label{eq:328}
    (\zeta_1\opz\zeta_2)(x_1,x_2):=\zeta_1(x_1)\pz\zeta_2(x_2)\quad
    \forevery\ \zeta_i\in \rmB(X_i;\bar\R).
  \end{equation}

The
following result is the natural extension of Lemma \ref{le:trivial}
stating that $\ETint(\ggamma|\mu_1,\mu_2)\geq\DD(\vvarphi|\mu_1,\mu_2)$
for a larger class of $\ggamma$ and $\vvarphi$ as before.

\begin{proposition}[Dual lower bound for extended real valued potentials]
  \label{prop:trivial-bis}
  Let $\ggamma$ be a feasible plan and
  let 
  $ \vvarphi \in \rmB(X_1;\bar \R)\times \rmB(X_2;\bar\R)$ with
    $\varphi_i\ge -\rec{(F_i)}$,
    $\varphi_1\opz\varphi_2\le \sfc$ with
    $(\Gstar_i\circ \varphi_i)_-\in \rmL^1(X_i,\mu_i)$
    (resp.~$    (\varphi_i)_+\in \rmL^1(X_i,\gamma_i)$).

  Then we have
  $(\varphi_i)_-\in \rmL^1(X_i;\gamma_i)$ 
  (resp.~$(\Gstar_i\circ\varphi_i)_+\in \rmL^1(X_i,\mu_i)$) and
  \begin{equation}
    \label{eq:72}
    \ETint(\ggamma|\mu_1,\mu_2)\ge
    \sum_i\int_{X_i}\Gstar_i(\varphi_i)\,\d\mu_i.
  \end{equation}
\end{proposition}
\begin{remark}
  \label{rem:trivial-bis}
  \upshape
  In a similar way, if 
  $\ppsi\in \rmB(X_1,\bar \R)\times \rmB(X_2,\bar \R)$ with $\psi_i\le
  F_i(0)$, $\FHstar_1(\psi_1)\opz\FHstar_2(\psi_2)\le \sfc$,
  and $(\psi_i)_-\in \rmL^1(X_i,\mu_i)$ 
  (resp.~$ (\FHstar_i\circ \psi_i)_+\in \rmL^1(X_i,\gamma_i)$),
  then $ (\FHstar_i\circ \psi_i)_-\in \rmL^1(X_i,\gamma_i)$
  (resp.~$(\psi_i)_+\in \rmL^1(X_i,\mu_i)$) with
  \begin{equation}
    \label{eq:72bis}
    \ETint(\ggamma|\mu_1,\mu_2)\ge
    \sum_i\int_{X_i}\psi_i\,\d\mu_i. \qedhere
  \end{equation}
\end{remark}
\begin{proof}
  Let us consider \eqref{eq:72} in the case that $(\Gstar_i\circ \varphi_i)_-\in
  \rmL^1(X_i,\mu_i)$ (the calculations in the other cases, 
  including \eqref{eq:72bis}, are
  completely analogous).
  Applying Lemma \ref{le:trivial} (with $\psi_i:=\Gstar_i\circ
  \varphi_i$ and $\phi_i:=-\varphi_i$) 
  and \eqref{eq:313} we obtain
  $(\varphi_i)_-\in \rmL^1(X_i,\gamma_i)$
  and then
  \begin{align}
    \notag
    \ETint(\ggamma|\mu_1,\mu_2)&=\sum_i \FF_i(\gamma_i|\mu_i)+
                                 \int_{\sxX} \sfc\,\d\ggamma
          \ge\sum_i  \FF_i(\gamma_i|\mu_i)+
          \int_{\sxX}
          \Big(\varphi_1(x_1)\pz\varphi_2(x_2)\Big)\,\d\ggamma\\
                               &\ge\sum_i 
                                 \FF_i(\gamma_i|\mu_i)+\int_{X_i}
                                 \varphi_i\,\d\gamma_i
                                 \topref{eq:71}\ge
                                 \sum_i
                                 \int_{X_i}\Gstar_i(\varphi_i)\,\d\mu_i.
                                 \label{eq:68}
  \end{align}
  Notice that the semi-integrability of $\varphi_i$ w.r.t.~$\gamma_i$ 
  yields $\varphi_i(\pi^i(x_1,x_2))>-\infty$ for
  $\ggamma$-a.e.~$(x_1,x_2)\in \xX$ so that $\varphi_1(x_1)\pz
  \varphi_2(x_2)=
  \varphi_1(x_1)+\varphi_2(x_2)$ and we can 
  split the integral 
  $$+\infty>\int \Big(\sum_i \varphi_i(x_i)\Big)\,\d\ggamma=
  \sum_i \int \varphi_i(x_i)\,\d\ggamma=
  \sum_i \int \varphi_i(x_i)\,\d\gamma_i.\qedhere$$
\end{proof}
\paragraph{Optimality conditions.}
If there exists a couple $\vvarphi$ as in Proposition \ref{prop:trivial-bis}
such that 
$\ETint(\ggamma|\mu_1,\mu_2)=\DD(\vvarphi|\mu_1,\mu_2)$ then
all the above inequalities (\ref{eq:68}) should be identities so that 
we have
\begin{align*}
  \FF_i(\gamma_i|\mu_i)=\int_{X_i}\Gstar_i(\varphi_i)\,\d\mu_i ,
    \quad\text{and}\quad
  \int_{\sxX}
  \Big(\sfc(x_1,x_2)-(\varphi_1(x_1)\pz\varphi_2(x_2))\Big)\,\d\ggamma=0,
\end{align*}
and the second part of Lemma \ref{le:trivial} yields
\begin{subequations}
\label{eq:OptiConds}
\begin{align}
    \label{eq:66}
  \varphi_1(x_1)\pz\varphi_2(x_2)=\sfc(x_1,x_2)\quad &\ggamma\text{-a.e.~in
  }\xX,\\
  \label{eq:67}
  -\varphi_i\in \partial F_i(\sigma_i)\quad &(\mu_i+\gamma_i)\text{-a.e.\ in }A_i\\
  \label{eq:63}
  \varphi_i=-\rec{(F_i)}\quad&\gamma_i^\perp\text{-a.e.~in
                               }A_{\gamma_i},\\
  \Gstar_i(\varphi_i)=F_i(0)\quad&\mu_i^\perp\text{-a.e.~in
                               }A_{\mu_i},
                                 \label{eq:331}
\end{align}
\end{subequations}
where $(A_i,A_{\mu_i},A_{\gamma_i})$ is a Borel partition
related to the Lebesgue decomposition of the couple $(\gamma_i,\mu_i)$
as in Lemma \ref{le:Lebesgue}.
We will show now that the existence of a couple $\vvarphi$
satisfying 
\begin{equation}
  \label{eq:216}
  \vvarphi =(\varphi_1,\varphi_2) \in \rmB(X_1;\bar \R)\times
  \rmB(X_2;\bar\R),\quad 
  \varphi_i\ge -\rec{(F_i)},\quad \varphi_1\opz\varphi_2\le \sfc,
\end{equation}
and the joint optimality conditions 
\ref{eq:OptiConds} 
is also sufficient to
prove that a feasible $\ggamma\in \cM(\xX)$ 
is optimal. We emphasize that we do not need any integrability 
assumption on $\vvarphi$.\EEE
\begin{theorem}
  \label{thm:joint-optimality}
  Let $\ggamma\in\cM(\xX)$ with $\EE(\ggamma|\mu_1,\mu_2)<\infty$; 
  if 
  there exists a couple $\vvarphi$ as in \eqref{eq:216}
  which satisfies the joint optimality
  conditions \eqref{eq:OptiConds} 
  then $\ggamma$ is optimal.
\end{theorem}
\begin{proof}
  We want to repeat the  calculations in (\ref{eq:68}) 
  of Proposition \ref{prop:trivial-bis}, but now taking care of the
  integrability issues. We use a clever truncation argument
  of \cite{Schachermeyer-Teichmann09}, based on the maps
  \begin{equation}
    \label{eq:266}
    T_n:\R\to\R,\quad T_n(\varphi):=-n\lor\varphi\land n,
  \end{equation}
  combined with a corresponding approximations of the entropies
  $F_i$ given by
  \begin{equation}
    \label{eq:332}
    F_{i,n}(r):=\max_{|\phi|\le n}\big(\phi
      r-\Fstar_{i}(\phi)\big).
  \end{equation}
  Recalling \eqref{eq:329}, it is not difficult to check that if
  $\varphi_1\pz\varphi_2\ge0$ we have $0\le
  T_n(\varphi_1)+T_n(\varphi_2)\up \varphi_1+\varphi_2$ as
  $n\up\infty$, whereas $\varphi_1\pz\varphi_2\le0$ yields $0\ge
  T_n(\varphi_1)+T_n(\varphi_2)\down \varphi_1+\varphi_2$.  In
  particular if $\vvarphi$ satisfies \eqref{eq:216} then
  $T_n(\varphi_i)\in \rmB_b(X_i)$, $T_n(\varphi_1)\oplus
  T_n(\varphi_2)\le \sfc$, and $T_n(\varphi_i)\ge -\rec{(F_i)}$ due to
  $\rec{(F_i)}\ge0$ and $\varphi_i\ge -\rec{(F_i)}$.  The boundedness
  of $T_n(\varphi_i)$ and Proposition \ref{prop:trivial-bis} yield for
  every $\tilde\ggamma\in \cM(\xX)$
  \begin{equation}
    \label{eq:268}
    \EE(\tilde\ggamma|\mu_1,\mu_2)\ge
    \sum_i\int_{X_i}\Gstar_i(T_n(\varphi_i))\,\d\mu_i.
  \end{equation}
  When $\rec{(F_i)} <\infty$, choosing $n\ge \rec{(F_i)}$ so that
  $T_n(\varphi_i)=\varphi_i=-\rec{(F_i)}$ $\gamma_i^\perp$-a.e., 
  and applying (ii) of the next Lemma
  \ref{le:auxiliary-app},
  we obtain
    \begin{align*}
      \int_{X_i}\Gstar_i(T_n(\varphi_i))\,\d\mu_i
      \stackrel{(\ref{eq:67},d)}=& \int_{X_i}
                                   \Big(F_{i,n}(\sigma_i)+\sigma_iT_n(\varphi_i)\Big)\,\d\mu_i
      \\
      \topref{eq:63}=&\int_{X_i}
                         F_{i,n}(\sigma_i)\,\d\mu_i+\rec{(F_i)}\gamma_i^\perp(X_i)+
                         \int_{X_i}T_n(\varphi_i)\,\d\gamma_i,
    \end{align*}
    and the same relation also holds when
    $\rec{(F_i)}=+\infty$
    since in this case $\gamma_i^\perp=0.$
    Summing up the two contributions 
    we get
  \begin{align*}
    \EE(\tilde\ggamma|\mu_1,\mu_2)&\ge
    \sum_{i}\Big(\int_{X_i}
                                   F_{i,n}(\sigma_i)\,\d\mu_i+
                                   \rec{(F_i)}\gamma_i^\perp(X_i)\Big)+
                                   \int_{\sxX}\Big(T_n(\varphi_1)\oplus
                                   T_n(\varphi_2)\Big)\,\d\ggamma.
  \end{align*}
  Applying Lemma \ref{le:auxiliary-app} (i)
  and the fact that $\varphi_1\opz\varphi_2=\sfc\ge0$ $\ggamma$-a.e. 
  by \eqref{eq:66},
  we can pass to the limit as $n\up\infty$ by monotone convergence 
  in the right-hand side, obtaining the desired optimality \EEE
  $\EE(\tilde\ggamma|\mu_1,\mu_2) \ge
  \EE(\ggamma|\mu_1,\mu_2) $.
\end{proof}
\begin{lemma}
  \label{le:auxiliary-app}
  Let $F_{i,n}:[0,\infty)\to [0,\infty)$ be defined by \eqref{eq:332}. Then
  \begin{enumerate}[\rm (i)]
  \item $F_{i,n}$ are Lipschitz, $F_{i,n}(\r )\le F_i(\r )$, and \EEE
    $F_{i,n}(\r )\up F_i(\r )$ as $n\up+\infty$.
  \item For every $\r \in \dom {F_i} $ and $\varphi_i\in \R\cup \{+\infty\}$
  we have \EEE
\begin{equation}
    \label{eq:333}
    \begin{aligned} 
      -\varphi_i\in \partial F_i(\r )\quad&\Rightarrow\quad
      -T_n(\varphi_i)\in \partial F_{i,n}(\r ),\\
      \varphi_i=+\infty,\ \r =0\quad&\Rightarrow\quad
      F_{i,n}(0)=\Gstar_{i}(T_n(\varphi_i))=\Gstar_i(n).
    \end{aligned}
  \end{equation}
  In particular, both cases considered in \eqref{eq:333} 
  give $\Gstar_i(T_n(\varphi_i)) = F_{i,n}(s) + sT_n(\varphi_i)$.
  \end{enumerate}
\end{lemma}
\begin{proof}
  Property (i): By \eqref{eq:269} and the definition in \eqref{eq:332}
  we get $F_{i,n}\le F_i$.  Since $-\Fstar_i(0)=\inf F_i\ge 0$ we see
  that $F_{i,n}$ are nonnegative. Recalling that $\Fstar_i$ are
  nondecreasing with $\dom {\Fstar_i}\supset (-\infty,0]$ (see Section
  \ref{subsec:entropy}) we also get the upper bound $F_{i,n}(\r )\le
  n\r-\Fstar_i(-n)$. Eventually, \eqref{eq:332} defines $F_{i,n}$ as
  the maximum of a family of $n$-Lipschitz functions, so $F_{i,n}$ is
  $n$-Lipschitz.

  Property (ii): Notice that
  $F_{i,n}=\big(\Fstar_i+\rmI_{[-n,n]}\big)^*$ so that $(F_{i,n})^*=\Fstar_i+\rmI_{[-n,n]}\ge
  \Fstar_i$. It is not difficult to check that $ F_i(\r )=F_{i,n}(\r )$ 
  if and only if 
  $\partial F_i(\r )\cap [-n,n]\neq \emptyset$.
  Therefore the set $I_n:=\{\r \ge0:
  F_i(\r )=F_{i,n}(\r )\}$ is a nonempty closed interval (possibly reduced
  to a single point)
  and it is easy to see that denoting $\r _n^+:=\max I_n$, $\r _n^-:=\min
  I_n$, $T_n'(\r ):=\r _n^-\lor \r \land \r _n^+$,
  we have
  $F_{i,n}(\r )=F_i(T_n'(\r ))+n(\r -T_n'(\r ))$. 
  In particular, whenever
  $\r \ge \r _n^+$ we have $n\in \partial F_{i,n}(\r )$ and similarly
  $-n\in \partial F_{i,n}(\r )$ if $\r \le \r _n^-$.
  If $\r $ belongs to the interior of $I_n$, then $\partial
  F_i(\r )=\partial F_{i,n}(\r )\subset [-n,n]$.
  
  Therefore, if $\phi_i=-\varphi_i\in \partial F_i(\r )$ with $\phi_i\in [-n,n]$,
  we have
  $F_i(\r )=\phi_i \r -\Fstar_i(\phi_i)=F_{i,n}(\r )$ so that
  $\phi_i\in \partial F_{i,n}(\r )$. 
  On the other hand, if
  $\partial F_i(\r )\ni\phi_i>n$, then $\r $ cannot belong to the interior of
  $I_n$, so that by monotonicity $\r \ge \r _n^+$ 
  and $\partial F_{i,n}(\r )\ni n=T_n(\phi_i)=-T_n(\varphi_i)$. 
  The case when $\partial F_i(\r )\ni\phi_i<-n$ is completely
  analogous. 

  Eventually, if $\phi_i=-\infty$ and $\r =0$ (in particular
  $F_i(0)=\Fstar_i(-\infty)<\infty$), then \eqref{eq:332} and the fact
  that $\Fstar_i$ is nondecreasing yields $F_{i,n}(0)=-\Fstar_i(-n)=
  \Gstar_i(n)=\Gstar_i(T_n(\varphi_i))$.

  For the last statement in (ii) the case $T_n(\varphi_i) = \varphi_i$
  is trivial. For $\varphi> n$ we have $-n\in\partial F_{i,n}(s)$
  implying $F_{i,n}(s)+F_i^*(-n) = -ns$. Hence, we have
  \[
  \Fstar_i(T_n(\varphi_i))=-F_i^*(-n)
  =F_{i,n}(s)+ns=F_{i,n}(s)+sT_n(\varphi_i).
  \]
  The case $\varphi_i<-n$ is similar.
\end{proof}

\subsection{A general duality result}
\label{sec:dualityGen}

The aim of this section is to show in complete generality the 
duality result
$\ET=\sfD$, by using the $\vvarphi$-formulation of the dual problem
\eqref{eq:274},
which is equivalent to \eqref{eq:46} by Proposition \ref{prop:equivalent-dual}.

We start with a simple lemma depending on a specific feature
of the entropy functions
(which fails exactly in the case of pure transport problems, see
Example
E.3 of Section \ref{ex:1}), using the strengthened feasibility condition
in \eqref{eq:302a}.
First note that the couple $\varphi_i\equiv0$ provides an obvious
lower bound for $\sfD(\mu_1,\mu_2)$, viz.
\begin{equation}
  \label{eq:307}
  \sfD(\mu_1,\mu_2)\ge \DD(0,0|\mu_1,\mu_2)=\sum_i m_i\Gstar_i(0)=
  \sum_i m_i\inf F_i.
\end{equation}

We derive an upper and lower bound for the potential
$\varphi_1$ under the assumption that $\sfc$ is bounded.

\begin{lemma}
  \label{le:entropic-estimate}
  Let $m_i=\mu_i(X_i)$  and assume 
  $\interior\big(m_1{\dom{F_1}}\big)\cap m_2\dom{F_2}\neq\emptyset$,
  so that 
  \begin{equation}
    \label{eq:306}
    \exists\,\r _1^-,\r _1^+\in \dom{F_1},\ \r _2\in \dom{F_2}:\quad
    m_1\r _1^-<m_2\r _2<m_1\r _1^+,
  \end{equation}
  and $S:=\sup\sfc<\infty$.
  Then every couple $\vvarphi=(\varphi_1,\varphi_2)\in \Cphi{}$ 
  with  $\DD(\vvarphi|\mu_1,\mu_2)\ge \sum_i m_i\inf F_i$ satisfies
  \begin{equation}
    \label{eq:245}
    \Phi_1^-
    \le \sup \varphi_1\le
    \Phi_1^+,\quad
    \Phi_1^\pm:=
    \frac{m_1(F_1(\r _1^\pm)-\inf F_1)+m_2(F_2(\r _2)-\inf F_2)+m_2\r _2 S}{m_2\r _2-m_1\r _1^\pm}.
  \end{equation}
\end{lemma}
\begin{proof}
Since $\vvarphi=(\varphi_1,\varphi_2)\in\Cphi{}$
satisfies $\sup\varphi_1+\sup\varphi_2\le S$, 
the definition of $\DD$ in \eqref{eq:44} and the  
monotonicity of $\Gstar $ yield
  \begin{align*}
    \sum_i m_i \inf F_i&\le \DD(\vvarphi|\mu_1,\mu_2)\le m_1 \Gstar_1(\sup \varphi_1)+ m_2
    \Gstar_2(S-\sup\varphi_1)
  \end{align*}
  Using  the dual bound 
  $\Gstar_i(\varphi_i)\le \varphi_i \r _i+F_i(\r _i)$ for $\r _i\in
  \dom{F_i}$ (cf.~\eqref{eq:303}) now implies
  \[
 \sum_i m_i \inf F_i\le \DD(\vvarphi|\mu_1,\mu_2)\le (m_1\r _1-m_2\r _2) \sup\varphi_1+
    m_1 F_1(\r _1)+m_2 F_2(\r _2)+m_2\r _2 S. 
  \]
  Exploiting \eqref{eq:306}, the choice $\r _1:=\r _1^-$ shows the 
  upper bound in \eqref{eq:245}; and $\r_1=\r_1^+$ the
  lower bound.
\end{proof}
We improve the previous result by showing that in the case of bounded
cost functions it is sufficient to consider bounded potentials
$\varphi_i$.  The second lemma is well known in the case of Optimal
Transport problems and will provide a useful a priori estimate in the
case of bounded cost functions used in the proof of
Theorem~\ref{thm:weak-duality}. 
\newcommand{\Phimax}{\varphi_{\rm max}}

\begin{lemma}
  \label{le:potential-bound}
  If $\sup\sfc=S<\infty$ then for every couple $\vvarphi\in \Cphi{}$
  there exists $\tilde\vvarphi\in \Cphi{}$ such that 
  $\DD(\tilde\vvarphi|\mu_1,\mu_2)\ge
  \DD(\vvarphi|\mu_1,\mu_2)$ and 
  \begin{equation}
    \label{eq:242}
    \sup\tilde\varphi_i-\infp\tilde\varphi_i\le S,\quad
    0\le \sup\tilde\varphi_1+\sup\tilde\varphi_2\le S.
  \end{equation}
  If moreover \eqref{eq:302a} holds, than there exist a constant
  $\Phimax\ge0$ only depending on $F_i,m_i,S$ such that 
  \begin{equation}
    \label{eq:308}
    -\Phimax\le \inf\tilde\varphi_i\le \sup\tilde\varphi_i\le \Phimax.
  \end{equation}
\end{lemma}
\begin{proof}
  Since $\sfc\ge 0$, possibly replacing $\varphi_1$ with
  $\tilde\varphi_1:=\varphi_1\lor (-\sup\varphi_2)$ 
  we obtain a new couple $(\tilde\varphi_1,\varphi_2)$
  with
  \begin{gather*}
    \tilde\varphi_1\ge\varphi_1,\quad
    \tilde\varphi_1(x_1)+\varphi_2(x_2)\le
    \big(\varphi_1(x_1)+\varphi_2(x_2)\big)\land 0\le \sfc(x_1,x_2)
  \end{gather*}
  so that $(\tilde\varphi_1,\varphi_2)\in \Cphi{}$ and 
  $ \DD(\tilde\varphi_1,\varphi_2|\mu_1,\mu_2)\ge
    \DD(\varphi_1,\varphi_2|\mu_1,\mu_2)$ since $\Gstar_1$ is nondecreasing.
    It is then not
  restrictive to assume that $\inf\varphi_1\ge -\sup\varphi_2$;
  a similar argument shows that we can assume $\inf\varphi_2\ge
  -\sup\varphi_1$.
  Since 
  \begin{equation}
  \sup\varphi_1+\sup\varphi_2\le
  S\label{eq:244}
\end{equation}
  we thus obtain a new couple $(\tilde\varphi_1,\tilde\varphi_2)\in
  \Cdual\rmC\sfc $ with
  \begin{equation}
    \label{eq:235}
   \DD(\tilde\varphi_1,\tilde\varphi_2|\mu_1,\mu_2)\ge
   \DD(\varphi_1,\varphi_2|\mu_1,\mu_2),\quad
   \sup\tilde\varphi_i-\inf\tilde\varphi_i\le S.
  \end{equation}
  If moreover $\sup \varphi_1+\sup\varphi_2=-\delta<0$, we could always
  add the constant $\delta$ to, e.g., $\varphi_1$, thus increasing
  the value of $\DD$ still preserving the constraint $\Cphi{}$.
  Thus, \eqref{eq:242} is established. \EEE

  When \eqref{eq:302a} holds (e.g.~in the case considered by
  \eqref{eq:306}) 
  the previous Lemma 
  \ref{le:entropic-estimate} provides constants $\varphi_1^\pm$ such that
  $\varphi_1^-\le \sup\tilde\varphi_1\le \varphi_1^+$. Now, \eqref{eq:242} 
  shows that
  $\varphi_2^-\le \sup\tilde\varphi_{2}\le \varphi_2^+$ with
  $\varphi_2^-:=-\varphi_1^+$ and $\varphi_2^+:=S-\varphi_1^-$. 
  Applying \eqref{eq:242} once again, we obtain \eqref{eq:308} 
  with $\Phimax:=S+\varphi_1^+-\varphi_1^-$.
\end{proof}
Before stating the last lemma we recall the useful notion of
$\sfc$-transforms of functions $\varphi_i:X_i\to \bar\R$ for a real
valued cost $\sfc:\xX\to[0,\infty)$, defined via
\begin{equation}
  \label{eq:74}
  \varphi_1^\sfc(x_2):=\inf_{x\in X_1}\big( \sfc(x,x_2)-\varphi_1(x)\big)
  \quad\text{ and}\quad 
  \varphi_2^\sfc(x_1):=\inf_{x\in X_2} \big(\sfc(x_1,x)-\varphi_2(x)\big).
\end{equation}
%
It is well known that if $\varphi_1\oplus\varphi_2\le \sfc$ with
$\sup\varphi_i<\infty$ then
\begin{equation}
\varphi_1^\sfc\text{ and }\varphi_2^\sfc \text{ are bounded,}\quad 
\varphi_1^{\sfc\sfc}\oplus\varphi_1^\sfc\le \sfc,\quad
\varphi_1^{\sfc\sfc}\geq \varphi_1, \text{ and } 
\varphi_1^\sfc\ge \varphi_2.\label{eq:309}
\end{equation}
Moreover, $\varphi_1=\varphi_1^{\sfc\sfc}$ if and only if
$\varphi_1=\varphi_2^\sfc$ for some function $\varphi_2$; in this case
$\varphi_1$ is called $\sfc$-concave and $(\varphi_1^{\sfc\sfc},
\varphi_1^\sfc)$ is a couple of $\sfc$-concave potentials.

Since $\Gstar_i$ are nondecreasing, it is also clear that whenever
$\varphi_1^{\sfc\sfc},\varphi_1^\sfc$ are $\mu_i$-measurable we have
the estimate
\begin{equation}
  \label{eq:81}
  \DD((\varphi_1,\varphi_2)|\mu_1,\mu_2)\le 
  \DD((\varphi_1^{\sfc\sfc},\varphi_2^\sfc)|\mu_1,\mu_2)
  \quad\forall \vvarphi\in \rmB(X_1)\times \rmB(X_2),\ 
  \varphi_1\oplus\varphi_2\le \sfc.
\end{equation}
The next lemma concerns the lower semicontinuity of $\varphi_i^\sfc$
in the case when $\sfc$ is simple (cf.~\cite{Kellerer84}), i.e.~it has
the form
\begin{equation}
  \label{eq:239}
  \sfc=\sum_{n=1}^N c_n \nchi_{A^1_n\times A^2_n},\quad 
  \text{with }c_n\geq0\text{ and } A^i_n\text{ open in }X_i.
\end{equation}

\begin{lemma}
  \label{le:technical2}
  Let us assume that $\sfc$ has the form \eqref{eq:239} and that
  $\vvarphi\in \rmB_s(X_1)\times \rmB_s(X_2)$ is a couple of simple
  functions taking values in $\dom{\Gstar_1}\times\dom{\Gstar_2}$ and
  satisfying $\varphi_1\oplus\varphi_2\le \sfc$.  Then
  $(\varphi_1^{\sfc\sfc},\varphi_1^\sfc)\in \Cphi{}$ with
  $\DD((\varphi_1^{\sfc\sfc},\varphi_1^\sfc)|\mu_1,\mu_2)\ge
  \DD(\vvarphi|\mu_1,\mu_2)$.
\end{lemma}
\begin{proof}
  It is easy to check that $\varphi_1^{\sfc\sfc},\varphi_1^\sfc$ are
  simple, since the infima in \eqref{eq:74} are taken on a finite
  number of possible values. By \eqref{eq:309} it is thus sufficient
  to check that they are lower semicontinuous functions.
  
  We do this for $\varphi_1^\sfc$, the 
  argument for $\varphi_1^{\sfc\sfc}=(\varphi_1^\sfc)^\sfc$ is
  completely analogous. 
  For this,   consider the sets
  \begin{align*}
    Z:={}&\big\{\zz=(z_n)_{n=1}^N\in \{0,1\}^N:
           \exists\,y\in X_1~\forall\,n=1,\ldots,N: \ z_n=\nchi_{A^1_n}(y)      
           \big\},\\
    Y_\szz:={}&\{y\in X_1:~\forall\,n=1,\ldots,N:\ \nchi_{A_n^1}(y)=z_n\}.
  \end{align*}
  Clearly, $(Y_\szz)_{\szz\in Z}$ defines a Borel partition of $X_1$; we define
  $\varphi_\szz:=\sup\{\varphi_1(y):y\in Y_\szz\}$. 
  
  By construction, for every $\zz\in Z$ and $y\in Y_\szz$ the map
  $f_\szz(x):=\sfc(y,x)-\varphi_\szz$ is independent of $y$ in
  $Y_\szz$ and it is lower semicontinuous w.r.t.~$x\in X_2$ since
  $\sfc$ is lower semicontinuous.  Since $\varphi_1^\sfc(x_2)$ is the
  minimum of a finite collection of lower semicontinuous functions,
  viz.
  \begin{equation}
    \label{eq:255}
    \varphi_1^\sfc(x_2)=\min\EEE\big\{f_\szz(x_2):\zz\in Z\big\}
  \end{equation}
  we obtain $\varphi_1^\sfc\in \LSC(X_1)$.
\end{proof}

With all these auxiliary results at hand, we are now ready to
prove our main result concerning the dual representation
using Theorem \ref{thm:minimax}.

\begin{theorem}
  \label{thm:weak-duality}
  In the basic coercive setting of Section \ref{subsec:setting}
  (i.e.~\eqref{eq:320} or \eqref{eq:319} hold), the Entropy-Transport
  functional \eqref{eq:4} and the dual functional \eqref{eq:44}
  satisfy
  \begin{equation}
    \label{eq:73}
    \inf_{\sggamma\in \cM(X_1\times X_2)} \ETint(\ggamma|\mu_1,\mu_2)=
    \sup_{\svvarphi\in \cCphi{} }\DD(\vvarphi|\mu_1,\mu_2)
    \quad\forevery\mu_i\in \cM(X_i),
  \end{equation}
  i.e.~$\ET(\mu_1,\mu_2)=\sfD(\mu_1,\mu_2)$
  for every $\mu_i\in \cM(X_i).$
\end{theorem}
\begin{proof}
  Since $\ET\geq \sfD$ is obvious, it suffices to show $\ET\leq\sfD$.
  In particular, it is not restrictive to assume that
  $\sfD(\mu_1,\mu_2)$ is finite.  We proceed in various steps,
  considering first the case when 
  $\sfc$ has compact sublevels.
  We will assume that $\rec{(F_i)}=+\infty$ (so that $\Gstar_i$ are
  continuous and increasing on $\R$, and $\Gstar_i\circ\varphi_i\in
  \LSC_b(X_i)$ whenever $\varphi_i\in \LSC_b(X_i)$), and we will
  remove the compactness assumption on the sublevels of $\sfc$ in the
  following steps. 

  \noindent
\textbf{Step 1:} 
\emph{The cost $\sfc$ has compact sublevels.}
We can directly apply Theorem
\ref{thm:minimax} to the saddle functional $\LL$ of \eqref{eq:42}
by choosing $A=\rmM$ given by \eqref{eq:45} endowed with the narrow
topology and $B=\Cphi{}$.
Conditions \eqref{eq:28} and \eqref{eq:37} are clearly satisfied and
the coercivity assumption $\rec{(F_1)}+\rec{(F_2)}+\min\sfc>0$ shows
that we can choose $\ppsi_*=(\bar\psi_1,\bar\psi_2)$ with constant
functions $\bar\psi_i$ and
$-\FHstar(\bar\psi_i)=-\bar\varphi_i=\bar\phi_i\in [0, \rec{(F_i)}]$
such that
\begin{displaymath}
  D=\min\Big(\sfc-(\varphi_1\oplus\varphi_2)\Big)=\phi_1+\phi_2+\min\sfc>0,\quad
  \bar\psi_i>-\infty.
\end{displaymath}
Arguing as in the proof of Theorem \ref{thm:easy-but-important} \emph{(ii)}
we immediately see that \eqref{eq:29} is satisfied, since
\begin{displaymath}
  \LL(\ggamma,\ppsi_*)=
  \int_\sxX\Big(\sfc-\min\sfc\Big)\,\d\ggamma+
  D\ggamma(\xX)+\sum_i \bar\psi_i\mu_i(X_i).
\end{displaymath}
In fact, for $C$ sufficiently big, the sublevels $\{\ggamma\in
\rmM:\LL(\ggamma,\ppsi_*)\le C\big\}$ are closed, bounded (since
$D>0$) and equally tight (by the compactness of the sublevels of
$\sfc$), thus narrowly compact.  Thus, \eqref{eq:73}, i.e.~$\ET=\sfD$,
follows from Theorem~\ref{thm:minimax}.

\noindent
\textbf{Step 2:}
\emph{The case when $\mu_i$ have compact support,
  \eqref{eq:302a} holds and
  the cost $\sfc$ is simple, i.e.~\eqref{eq:239} holds.\EEE}
Let us set $\tilde X_i:=\supp(\mu_i)$. 
Since $\rec{(F_i)}=+\infty$ the support of 
all $\ggamma$ with $\EE(\ggamma|\mu_1,\mu_2)<\infty$ 
is contained $\tilde X_1\times \tilde X_2$ so that 
the minimum of the functional $\EE(\ggamma|\mu_1,\mu_2)$
does not change 
by restricting the spaces to $\tilde X_i$. 
By applying the previous
step to the problem
stated in $\tilde X_1\times \tilde X_2$, for every
$E<\ET(\mu_1,\mu_2)$
we find $\vvarphi\in \LSC_s(\tilde X_1)\times \LSC_s(\tilde X_2)$ 
such that $\varphi_1\oplus\varphi_2\le \sfc$ in 
$\tilde X_1\times \tilde X_2$, that $\Gstar_i(\varphi_i)$ is finite,
and that $\sum_{i}\int_{\tilde X_i}\Gstar_i(\varphi_i)\,\d\mu_i\ge E$.

Extending $\varphi_i$ to $-\sup\sfc$ in $X_i\setminus \tilde X_i$ the
value of $\DD(\vvarphi|\mu_1,\mu_2)$ does not change and we obtain a
couple of simple Borel functions with $\varphi_1\oplus\varphi_2\le
\sfc$ in $\xX$. We can eventually apply Lemma \ref{le:technical2} to
find $(\varphi_1^{\sfc\sfc},\varphi_1^\sfc)\in \Cphi{}$ with
$\DD(\varphi_1^{\sfc\sfc},\varphi_1^\sfc|\mu_1,\mu_2)\ge E$. Since
$E<\ET(\mu_1,\mu_2)$ was arbitrary, we conclude that \eqref{eq:73}
holds in this case as well. 

\noindent
\textbf{Step 3:} 
\emph{We remove the assumption on the compactness of $\supp(\mu_i)$.}

    Since $\mu_i$ are Radon, we can find two
    sequences of compact sets $K_{i,n}\subset X_i$ such that
    $\eps_{i,n}:=\mu_i(X_i\setminus K_{i,n})\to0$ as $n\to\infty$, i.e.~\EEE
    $\mu_{i,n}:=\nchi_{K_{i,n}}\cdot\mu_i$ is narrowly converging to
    $\mu_i$.

    Let $E_n:=\ET(\mu_{1,n},\mu_{2,n})$ and let $E_n'<E_n$ with
    $\lim_{n\to\infty}E'_n=\liminf_{n\to \infty}E_n$. 
    Since $\mu_{i,n}$ have compact support, 
    by the previous step and Lemma \ref{le:potential-bound}
    we can find a
    sequence $\vvarphi_n\in \Cphi{}$ and a constant $\Phimax$ independent
    of $n$ such that
    \begin{displaymath}
      \DD(\vvarphi_n|\mu_{1,n},\mu_{2,n})\ge E'_n\quad\text{and}\quad
      \sup|\varphi_n^i|\le \Phimax.
    \end{displaymath}
    This yields
    \begin{displaymath}
      \DD(\vvarphi_n|\mu_{1},\mu_{2})\ge \sum_i \int_{K_{i,n}}\Gstar_i(\varphi_{i,n})\,\d\mu_i+
      \sum_i \Gstar_i(-\Phimax)\eps_{i,n}\ge E'_n+\sum_i \Gstar_i(-\Phimax)\eps_{i,n}.
    \end{displaymath}
    Using the lower semicontinuity of $\ET$ from Lemma \ref{le:lsc}
    we obtain
    \begin{displaymath}
      \sfD(\mu_1,\mu_2)\geq\liminf_{n\to\infty}\DD(\vvarphi_n|\mu_{1},\mu_{2})\ge
      \lim_{n\to\infty} E'_n =
      \liminf_{n\to\infty}\ET(\mu_{1,n},\mu_{2,n})\ge \ET(\mu_1,\mu_2).
    \end{displaymath}
    Thus, \eqref{eq:73} is established. \EEE

\noindent
\textbf{Step 4:} \emph{We remove the assumption \eqref{eq:302a} on
  $F_i$.}  It is sufficient to approximate $F_i$ by an increasing and
pointwise converging sequence $F_i^n\in \Gamma(\R_+)$. The
corresponding sequence $(F_i^n)^\circ :\varphi_i\mapsto \sup_{s\geq 0}
(F_i^n(s)+s\varphi_i)$ of conjugate concave functions is also
nondecreasing and pointwise converging to $\Gstar_i$. By the previous
step, if $E_n < \ET^n(\mu_1,\mu_2)$ with
$\lim_{n\to\infty}E^n=\lim_{n\to\infty}\ET^n (\mu_1,\mu_2)=
\ET(\mu_1,\mu_2)$ we can find $\vvarphi_n\in \Cphi{}$ such that
\begin{displaymath}
  E_n\le \sum_i\int_{X_i}(F_i^n)^\circ(\varphi_i^n)\,\d\mu_i\le 
  \sum_i\int_{X_i}\Gstar_i(\varphi_i^n)\,\d\mu_i =\DD(\vvarphi_n|\mu_1,\mu_2).
\end{displaymath}
Passing to the limit $n\to\infty$ we conclude
$\ET(\mu_1,\mu_2)\leq \sfD(\mu_1,\mu_2)$ as desired.

\noindent
\textbf{Step 5:}
\emph{the case of a general cost $\sfc$.}

Let $\sfc:\xX\to[0,\infty]$ be an arbitrary l.s.c.~cost and
let us denote by $(\sfc^\alpha)_{\alpha\in \A}$ the class of costs characterized by
\eqref{eq:239} and majorized by $\sfc$. Then, $\A$ is a directed set
with the pointwise order $\le$, 
since maxima of a finite number of cost functions in $\A$ can still be
expressed as in \eqref{eq:239}. It is not difficult to check
that $\sfc=\sup_{\alpha\in\A}\sfc^\alpha=\lim_{\alpha\in\A}\sfc^\alpha$ so that
by Lemma \ref{le:lsc}
$\ET(\mu_1,\mu_2)=\lim_{\alpha\in \A}\ET^\alpha(\mu_1,\mu_2)=
\sup_{\alpha\in \A}\ET^\alpha(\mu_1,\mu_2)$, where
$\ET^\alpha$ denotes the Entropy-Transport functional associated to
$\sfc^\alpha$.

Thus for every $E<\ET(\mu_1,\mu_2)$ we can find $\alpha\in \A$ such
that $\ET^\alpha(\mu_1,\mu_2)> E$ and therefore, by the previous step,
a couple $\vvarphi^\alpha\in \LSC_s(X_1)\times \LSC_s(X_2)$ with
$\Gstar_i(\varphi_i^\alpha)$ finite such that
$\varphi_1^\alpha\oplus\varphi_2^\alpha\le \sfc^\alpha$ in $\xX$ and
$\DD(\vvarphi^\alpha|\mu_1,\mu_2)\ge E$. Since $\sfc^\alpha\le \sfc$
we have $\vvarphi^\alpha\in \Cphi{}$ and $\ET(\mu_1,\mu_2)\leq
\sfD(\mu_1,\mu_2)$ follows.
\end{proof}
Arguing as in Remark \ref{rem:LSC=C.1} we can change the spaces 
of test potentials $\vvarphi=(\varphi_1,\varphi_2)\in\Cphi{}$, see
\eqref{eq:301}. 

\begin{corollary}
  \label{cor:LSC=C}
  The duality formula \eqref{eq:73} still holds if we replace 
  the spaces of simple lower semicontinuous functions
  $\LSC_s(X_i)$ in the definition of $\Cphi{}$ with
  the spaces of bounded lower semicontinuous functions $\LSC_b(X_i)$ 
  or with the spaces of bounded Borel functions $\rmB_b(X_i)$.
  
  If $(X_i,\tau_i)$ are completely regular spaces, then 
  we can equivalently replace lower semicontinuous functions by
  continuous ones, obtaining
  \begin{equation}
    \label{eq:252}
    \begin{aligned}
      \ET(\mu_1,\mu_2)&=
      \sup\Big\{ \sum_i\int_{X_i}
      \Gstar (\varphi_i)
      \,\d\mu_i: 
       \varphi_i,\,\Gstar_i(\varphi_i)\in
       \rmC_b(X_i),\ 
       \varphi_1\oplus\varphi_2\le \sfc\Big\}
      \\
      &=
      \sup\Big\{ \sum_i\int_{X_i}
      \psi_i
      \,\d\mu_i: 
      \psi_i ,\FHstar_i(\psi_i)\in
      \rmC_b(X_i),\ 
      \FHstar_1(\psi_1)\oplus\FHstar_2(\psi_2)\le \sfc\Big\}.
    \end{aligned}
  \end{equation}
\end{corollary}
\begin{corollary}[Subadditivity of $\ET$]
  \label{cor:subadditivity}
  The functional $\ET$
  is convex and positively $1$-homogeneous (in particular it is subadditive), i.e.
  for every $\mu_i,\mu_i'\in \cM(X)$ and $\lambda\ge0$ we have
  \begin{equation}
    \label{eq:136}
    \ET(\lambda\mu_1,\lambda\mu_2)=
    \lambda\,\ET(\mu_1,\mu_2),\quad
     \ET(\mu_1+\mu_1',\mu_2+\mu_2')\le 
     \ET(\mu_1,\mu_2)+\ET(\mu_1',\mu_2').
  \end{equation}
\end{corollary}
\begin{proof}
  By Theorem \ref{thm:weak-duality} it is sufficient to prove
  the corresponding property of $\sfD$, which follows immediately from
  its representation formula \eqref{eq:49} as a supremum of linear functionals.
\end{proof}

\subsection{Existence of optimal Entropy-Kantorovich potentials}
\label{subsec:ExiOpt}

In this section we will consider two cases, when
the dual problem admits a couple of optimal Entropy-Kantorovich potentials
$\vvarphi=(\varphi_1,\varphi_2)$.

The first case is completely analogous to the 
transport setting.

\begin{theorem}
  Consider complete metric spaces $(X_i,\sfd_i)$, $i=1,2$, and assume
  that \eqref{eq:302a} holds,
  and $\sfc$ is bounded and uniformly continuous with respect to the
  product distance $\sfd((x_1,x_2),(x_1'\,x_2')):=
  \sum_i\sfd_i(x_i,x_i')$ in $\xX =X_1\times X_2$.  Then there exists
  a couple of optimal Entropy-Kantorovich potentials $\vvarphi\in
  \rmC_b(X_1)\times\rmC_b(X_2)$ satisfying
  \begin{equation}
    \label{eq:288}
    \varphi_1\oplus\varphi_2\le \sfc,\quad
    \varphi_i\ge -\rec{(F_i)},\quad
    \ET(\mu_1,\mu_2)=\DD(\vvarphi|\mu_1,\mu_2).
  \end{equation}
\end{theorem}
\begin{proof}
  By the boundedness and uniform continuity of $\sfc$ we can find a
  continuous and concave modulus of continuity $\omega:[0,+\infty)\to
  [0,+\infty)$ with $\omega(0)=0$ such that
\begin{displaymath}
  \big|\sfc(x_1',x_2)-\sfc(x_1,x_2)\big|\le \omega(\sfd_1(x_1',x_1)),\quad
  \big|\sfc(x_1,x_2')-\sfc(x_1,x_2)\big|\le \omega(\sfd_2(x_2',x_2)).
\end{displaymath}
Possibly replacing the distances $\sfd_i$ with $\sfd_i+\omega(\sfd_i)$, we
may assume that $x_1\mapsto\sfc(x_1,x_2)$ is $1$-Lipschitz w.r.t.~$\sfd_1$
for every $x_2\in X_2$ and
$x_2\mapsto\sfc(x_1,x_2)$ is $1$-Lipschitz with respect to 
$\sfd_2$ for every $x_1\in X_1$. In particular, 
every $\sfc$-transform 
\eqref{eq:74} of a bounded function is $1$-Lipschitz
(and in particular Borel).

Let $\vvarphi_n$ be a maximizing sequence in $\Cphi{} $.  By Lemma
\ref{le:potential-bound} we can assume that $\vvarphi_n$ is uniformly
bounded; by \eqref{eq:309} and \eqref{eq:81} we can also assume that
$\vvarphi_n$ are $\sfc$-concave and thus $1$-Lipschitz.  If $K_{i,n}$
is a family of compact sets whose union $A_i$ has a full $\mu_i$
measure in $X_i$, we can thus extract a subsequence (still denoted by
$\vvarphi_n$) pointwise convergent to $\vvarphi =
(\varphi_1,\varphi_2)$ in $A_1\times A_2$.  Obviously, we have
$\varphi_1:=\lim_{n\to\infty}\varphi_{1,n}$ and
$\varphi_2:=\liminf_{n\to\infty}\varphi_{2,n}$, we obtain a family
$\varphi_i\in\rmB_b(X_i)$, and $\varphi_1\oplus\varphi_2\le\sfc$,
$\varphi_i\ge \rec{(F_i)}$ and
\begin{displaymath}
  \DD(\vvarphi|\mu_1,\mu_2)=\sum_{i}\int_{A_i}\Gstar_i(\varphi_i)\,\d\mu_i
  \ge
  \lim_{n\to\infty}\sum_{i}\int_{A_i}\Gstar_i(\varphi_{i,n})\,\d\mu_i=
  \ET(\mu_1,\mu_2), 
\end{displaymath}
thanks to Fatou's Lemma and the fact that $\Gstar_i(\varphi_{i,n})$
are uniformly bounded from above. Eventually replacing
$(\varphi_1,\varphi_2)$ with $(\varphi_1^{\sfc\sfc},\varphi_1^\sfc)$
we obtain a couple in $\rmC_b(X_1)\times\rmC_b(X_2)$ satisfying
\eqref{eq:288}.
\end{proof}
The next result is of different type, since it does not require any
boundedness nor regularity of $\sfc$ (which can also assume the value
$+\infty$ in the case $F_i(0)<\infty$).
\begin{theorem}
  \label{thm:pot-ex}
  Let us suppose that at least one of the following two conditions
  hold:

  \indent a) $\sfc$ is everywhere finite and
    \eqref{eq:302a} holds\\
  or\\
  \indent b) $F_i(0)<+\infty$.\\
  Then a plan $\ggamma\in \cM(\xX)$ with finite energy
  $\EE(\ggamma|\mu_1,\mu_2)<\infty$ is optimal if and only if there
  exists a couple $\vvarphi$ as in \eqref{eq:216} satisfying
  the optimality conditions \eqref{eq:OptiConds}.
\end{theorem}
\begin{proof}
  We already proved (Theorem 
  \ref{thm:joint-optimality}) that the existence of a couple $\vvarphi$ as in 
  \eqref{eq:216} satisfying
  \eqref{eq:OptiConds} yields the optimality of $\ggamma$.
  
  Let us now assume that $\ggamma\in \cM(\xX)$ has finite energy and is
  optimal. If $\mu_i\equiv \eta_0$ then also $\ggamma=0$ 
  and \eqref{eq:OptiConds} are always satisfied, since we can choose
  $\varphi_i\equiv0$.

  We can therefore assume that at least one of the
  measures $\mu_i$, say $\mu_2$, has positive mass. 
  Let $\ggamma\in \OptET(\mu_1,\mu_2)$,
  and let us apply Theorem \ref{thm:weak-duality} to find 
  a maximizing sequence $\vvarphi_n\in \Cphi{}$
  such that $\lim_{n\up\infty}
  \DD(\vvarphi_n|\mu_1,\mu_2)=\ET(\mu_1,\mu_2)$.

  Using the Borel partitions $(A_i,A_{\mu_i},A_{\gamma_i})$ 
  for the couples of measures $\gamma_i,\mu_i$ provided
  by Lemma \ref{le:Lebesgue} and 
  arguing as in Proposition \ref{prop:trivial-bis} we get
  \begin{align*}
    \lim_{n\to\infty}\int_{X_1\times
      X_2}\Big(\sfc(x_1,x_2)-\varphi_{1,n}(x_1)-
    \varphi_{2,n}(x_2)\Big)\,\d\ggamma&=0,\\
    \lim_{n\to\infty}\int_{A_i\cup A_{\mu_i}}\Big(F_i(\sigma_i)+\sigma_i
    \varphi_{i,n}-\Gstar_i(\varphi_{i,n})\Big)
    \,\d\mu_i&=0,\\
    \lim_{n\to\infty}\int_{A_{\gamma_i}}
    \big(\varphi_{i,n}+\rec{(F_i)}\big)\,\d\gamma_i^\perp&=0.
  \end{align*}
  Since all the integrands are nonnegative, up to selecting a suitable
  subsequence (not relabeled) 
  we can assume that the integrands are converging pointwise a.e.\ to
  $0$. We can thus find Borel sets $A_i'\subset A_i,A_{\mu_i}'\subset
  A_{\mu_i},
  A_{\gamma_i}'\subset A_{\gamma_i}$ and $A'\subset \xX$ with
  $\pi^i(A')=
  A_i'\cup
  A_{\gamma_i}'$,
  $(\mu_i+\gamma_i)\Big((A_i\setminus A_i')\cup (A_{\mu_i}\setminus
  A_{\mu_i}')
  \cup(A_{\gamma_i}\setminus A_{\gamma_i}')\big)=0$,
  and $\ggamma(\xX\setminus A')=0$ such that 
  \begin{align}
    \label{eq:282}
    \sfc(x_1,x_2)<\infty\quad
    \lim_{n\to\infty} \sfc(x_1,x_2)-\varphi_{1,n}(x_1)-
    \varphi_{2,n}(x_2)&=0\quad\text{in }A',\\
    \label{eq:265}
    F_i(\sigma_i)<\infty,\quad
    \lim_{n\to\infty}F_i(\sigma_i)+\sigma_i
    \varphi_{i,n}-\Gstar_i(\varphi_{i,n})&=0\quad\text{in }A_i'\cup A_{\mu_i}',\\
    \label{eq:283}
    \lim_{n\to\infty}\big(\varphi_{i,n}+\rec{(F_i)}\big)&=0\quad
    \text{in }A_{\gamma_i}'.
  \end{align}
    For every $x_i\in X_i$ we define the Borel functions
  $\varphi_1(x_1):=\limsup_{n\to\infty}\varphi_{1,n}(x_1)$ and \EEE
  $\varphi_2(x_2):=\liminf_{n\to\infty}\varphi_{2,n}(x_2)$, taking
  values in $ \R\cup\{\pm \infty\}$.
  It is clear that the couple $\vvarphi=(\varphi_1,\varphi_2)$ 
  complies with \eqref{eq:216}, \eqref{eq:331} and \eqref{eq:63}.
  
  If $\ggamma(\xX)=0$ then \eqref{eq:66} and \eqref{eq:67} 
  are trivially satisfied, so that it is not restrictive to assume
  $\ggamma(\xX)>0$. 

  If $\mu_1(X_1)=0$ then 
  $\rec{(F_1)}$ is finite (since 
  $\gamma_1^\perp(X_1)=\gamma_1(X_1)=\ggamma(\xX)>0$)
  and $\varphi_1\equiv \rec{(F_1)}$ on $A_{\gamma_1}'$ and on
  $A'$. It follows that $\varphi_2(x_2)=\sfc(x_1,x_2)-\rec{(F_1)}\in \R$ on
  $A'$ so that \eqref{eq:66} is satisfied. 
  Since $\varphi_2(x_2)$ is
  an accumulation point of $\varphi_{2,n}(x_2)$ 
  Lemma \ref{le:auxiliary2} below yields
  $-\varphi_2(x_2)\in \partial F_2(\sigma_2(x_2))$ in $A_2'$
  so that \eqref{eq:67} is also satisfied (in the case $i=1$ one can
  choose
  $A_1'=\emptyset$).

  We can thus assume that $\mu_i(X_i)>0$ and 
  $\ggamma(\xX)>0$.
  In order to check \eqref{eq:66} and \eqref{eq:67} 
  we distinguish two cases.
  
  \textbf{Case a:} \emph{$\sfc$ is everywhere finite
    and \eqref{eq:302a} holds.}
  Let us first prove that $\varphi_1<+\infty$ everywhere. 
  
  By contradiction, if there is a point $\bar x_1\in X_1$ such that 
  $\varphi_1(\bar x_1)=+\infty$ 
  we deduce that $\varphi_2(x_2)=-\infty$ for every $x_2\in X_2$. 
  
  Since the set $A_2'\cup A_{\mu_2}'$ has positive $\mu_2$-measure,
  it contains some point $\bar x_2$: Equation~\eqref{eq:265} and  Lemma
  \ref{le:auxiliary2} below (with $F=F_2$, $\r =\sigma_2(\bar x_2)$,
  $\phi_n:=-\varphi_{2,n}(\bar x_2)$) yield
  $\r _2^+=\max\dom{F_2}=\sigma_2(\bar x_2)<\infty$
  and $\sigma_2\equiv \r _2^+$ in $A_2'\cup A_{\mu_2}'$.
  We thus have
  $\dom{F_2}\subset [0,\r _2^+]$, $\rec{(F_2)}=+\infty$ 
  and therefore $m_2\r _2^+=\ggamma(\xX)$.

  On the other hand, if $\varphi_2=-\infty$ in $X_2$ we deduce 
  that $\varphi_1(x_1)=+\infty$ for every $x_1\in \pi^1(A')$.
  Since $\rec{(F_1)}\ge0$, it follows that $\gamma_i(A_{\gamma_i}')=0$ 
  (i.e.~$\gamma_i^\perp=0$)   
  so that there is a point $a_1$ in $A_1'$ such that
  $\varphi_1(a_1)=+\infty$.
  Arguing as before, a further application of Lemma
  \ref{le:auxiliary2} yields that
  $\sigma_1\equiv \r _1^-=\min\dom{F_1}$ $\mu_1$-a.e. 
  It follows that $m_1 \r _1^-=\gamma_1(X_1)=\ggamma(\xX)=m_2\r _2^+$,
  a situation that contradicts \eqref{eq:302a}.

  Since $\mu_1(X_1)>0$ the same argument shows that $\varphi_2<\infty$
  everywhere in $X_2$. 
  It follows that \eqref{eq:66} holds and $\varphi_i>-\infty$ 
  on $A_i'$. 
  Since $\varphi_i(x_i)$ is
  an accumulation point of $\varphi_{i,n}(x_i)$,
  Lemma \ref{le:auxiliary2} below yields
  $-\varphi_i(x_i)\in \partial F_i(\sigma_i(x_i))$ in $A_i'$
  so that \eqref{eq:67} is also satisfied.
    
  \textbf{Case b:} $F_i(0)<\infty$.
  In this case $\Gstar_i$ are bounded from above
  and
  $\varphi_{i}\ge -\rec{(F_i)}$ everywhere in $X_i$.
  By Theorem \ref{thm:weak-duality} 
  $\lim_{n\to\infty}\sum_i\int \Gstar_i(\varphi_{i,n})\,\d\mu_i>-\infty$,
  so that 
  Fatou's Lemma yields $\Gstar_1(\varphi_1)\in \rmL^1(X_1,\mu_1)$ and $\varphi_1(x_1)>-\infty$ for
  $\mu_1$-a.e.~$x_1\in X_1$, in particular for
  $(\mu_1+\gamma_1)$-a.e.~$x_1\in A_1'$.
  Applying Lemma \ref{le:auxiliary2} below \EEE, since 
  $\sigma_1(x_1)>0=\min\dom {F_1}$ in $A_1'$, we deduce that
  $-\varphi_1(x_1)\in \partial F_1(\sigma_1(x_1))$ for 
  $(\mu_1+\gamma_1)$-a.e.~$x_1\in A_1'$, i.e.~\eqref{eq:67} for $i=1$.
  Since we already checked that \eqref{eq:63}  and \eqref{eq:331} 
  hold, 
  applying Lemma \ref{le:trivial} 
  (with $\phi:=-\varphi_1$ and $\psi:=\Gstar_1(\varphi_1))$) we get 
  $\varphi_1\in \rmL^1(X_1,\gamma_1)$, in 
  particular $\varphi_1\circ \pi^1 \in \R$ $\ggamma$-a.e.~in $\xX$.
  It follows that \eqref{eq:66} holds and 
  $\varphi_2\circ \pi^2\in
  \rmL^1(\xX,\ggamma)$ so that $\varphi_2\in \R$ 
  $(\mu_2+\gamma_2)$-a.e.~in
  $A_2'$. A further application of Lemma \ref{le:auxiliary2} yields
  \eqref{eq:67} for $i=2$.
\end{proof}
\begin{corollary}
  \label{cor:smooth-case}
  Let us suppose that $\dom{F_i}\supset (0,\infty)$ and
  $F_i$ are differentiable in $(0,\infty)$.
  A plan $\ggamma\in \cM(\xX)$ with $\EE(\ggamma|\mu_1,\mu_2)<\infty$
  belongs to 
  $\OptET(\mu_1,\mu_2)$ if and only if 
  there exist Borel partitions $(A_i,A_{\mu_i},A_{\gamma_i})$ 
  and corresponding Borel densities $\sigma_i$
  associated to $\gamma_i$ and $\mu_i$ as in Lemma \ref{le:Lebesgue}
  such that setting
  \begin{equation}
    \label{eq:338}
    \varphi_i(x_i):=
    \begin{cases}
      -F_i'(\sigma_i)&\text{if }x_i\in A_i,\\
      -\derzero{(F_i)}&\text{if }x_i\in A_{\mu_i},\\
      -\rec{(F_i)}&\text{if }x_i\in X_i\setminus (A_i\cup A_{\mu_i}),
    \end{cases}
  \end{equation}
  we have
  \begin{equation}
    \label{eq:337}
    \varphi_1\opz\varphi_2\le \sfc\text{ in }X_1\times X_2,\quad
    \varphi_1\oplus\varphi_2=\sfc\text{ $\ggamma$-a.e.~in }(A_1\cup
    A_{\gamma_1})\times (A_2\cup A_{\gamma_2}).
  \end{equation}
\end{corollary}
\begin{proof}
  Since $\partial F_i(\r )=\{F_i'(\r )\}$ for every 
  $\r \in (0,\infty)$ and $\Gstar_i(\varphi_i)=F_i(0)$ if and only if 
  $\varphi_i\in [-\derzero{(F_i)},+\infty]$,
  \eqref{eq:337} is clearly a necessary condition for
  optimality, thanks to Theorem \ref{thm:pot-ex}.
  Since $\derzero{(F_i)}\le F_i'(\r )\le \rec{(F_i)}$ 
  Theorem \ref{thm:joint-optimality} shows that conditions
  \eqref{eq:338}--\eqref{eq:337} are also sufficient.
\end{proof}
The next result shows that
\eqref{eq:338}--\eqref{eq:337} take an even simpler form
when $-\derzero{(F_i)}=\rec{(F_i)}=+\infty$;
in particular,
by assuming that $\sfc$ is continuous, the support of an optimal plan $\ggamma$
cannot be too small.
\begin{corollary}[Spread of the support]
  \label{cor:spread}
  Let us suppose that 
  \begin{itemize}
  \item $\sfc:\xX\to[0,\infty]$ is continuous.
  \item $\dom{F_i}\supset (0,\infty)$, $F_i$ are differentiable in
    $(0,\infty)$, and $-\derzero{(F_i)}=\rec{(F_i)}=\infty$.
  \end{itemize}
  Then, \EEE$\ggamma$ is an optimal plan if and only if 
  $\gamma_i\ll \mu_i$, 
  for every $x_i\in \supp(\mu_i)$ we have
  $\sfc(x_1,x_2)=+\infty$ 
  if $x_1\in \supp\mu_1\setminus \supp\gamma_1$ or
  $x_2\in \supp\mu_2\setminus \supp\gamma_2$, 
  and there exist Borel sets $A_i\subset \supp\gamma_i$ 
  with $\gamma_i(X_i\setminus A_i)=0$ and
  Borel densities $\sigma_i:A_i\to
  (0,\infty)$ of $\gamma_i$ w.r.t.~$\mu_i$ such that
  \begin{equation}
    \label{eq:339}
    \begin{gathered}
        F_1'(\sigma_1)\oplus F_2'(\sigma_2)\ge -\sfc\text{ in
        }A_1\times A_2,\quad F_1'(\sigma_1)\oplus
        F_2'(\sigma_2)=-\sfc\quad\text{$\ggamma$-a.e.~in }A_1\times A_2.
    \end{gathered}
  \end{equation}
\end{corollary}
\begin{remark}
  \label{rem:optimal}
  \upshape
  Apart from the case of pure transport problems
  (Example E.3 of Section \ref{ex:1}),
  where the existence of Kantorovich potentials is well known
  (see \cite[Thm.~5.10]{Villani09}),
  Theorem \ref{thm:pot-ex} covers essentially all the interesting
  cases, at least when the cost $\sfc$ takes finite values if
  $0\not\in \dom{F_i}$.
  In fact, if the strengthened feasibility condition \eqref{eq:302a} \EEE
  does not hold,
  it is not difficult to construct an example of optimal plan
  $\ggamma$
  for which conditions \eqref{eq:216}, \eqref{eq:66}, \eqref{eq:67}
  cannot be satisfied. 
  Consider e.g.~$X_i=\R$, $\sfc(x_1,x_2):=\frac 12 |x_1-x_2|^2$, 
  $\mu_1:=\mathrm e^{-\sqrt \pi x_1^2}\Leb 1$, $\mu_2:=
  \mathrm e^{-\sqrt \pi (x_2+1)^2}\Leb 1$,
  $\dom F_1=[a,1]$, $\dom F_2=[1,b]$ with arbitrary choice of $a\in
  [0,1)$ and $b\in (1,\infty]$.
  Since $m_1=m_2=1$ the weak feasibility condition \eqref{eq:83} holds,
  but \eqref{eq:302a} is violated. We find 
  $\gamma_i=\mu_i$, $\sigma_i\equiv 1$, so that 
  the optimal plan $\ggamma$ can be obtained by solving the quadratic
  optimal transportation problem, thus $\ggamma:=\tt_\sharp \mu_1$
  where $\tt(x):=(x,x-1)$. In this case the potentials $\varphi_i$ are
  uniquely determined up to an additive constant $a\in \R$ so that we have
  $\varphi_1(x_1)=x_1+a$, $\varphi_2(x_2)=-x_2-a-\frac 12$,
  and it is clear that condition $-\varphi_i\in \partial F_i(1)$ 
  corresponding to \eqref{eq:67} cannot be satisfied, since
  $\partial F_i(1)$ are always proper subsets of $\R$. 
  We can also construct entropies such that $\partial
  F_i(1)=\emptyset$
  (e.g.~$F_1(r)=(1-r)\log(1-r)+r$, $F_2(r)=(r-1)\log(r-1)-r+2$) 
  so that \eqref{eq:67} can never hold, independently of the cost $\sfc$. 
\end{remark}
We conclude this section by proving the simple 
property on subdifferentials we used in the proof of Theorem \ref{thm:pot-ex}.
\begin{lemma}
  \label{le:auxiliary2}
  Let $F\in \Gamma(\R_+)$, $\r \in \dom F$, let 
  $\phi\in \R\cup\{\pm\infty\}$ be an accumulation point of 
  a sequence $(\phi_n)\subset \R$
  satisfying
  \begin{equation}
    \label{eq:335}
    \lim_{n\to\infty}\big(F(\r )-\r \phi_n+\Fstar(\phi_n)\big)=0.
  \end{equation}
  If $\phi\in \R$ then $\phi\in \partial
  F(\r )$, if $\phi=+\infty$ then $\r =\max \dom F$ and if $\phi=-\infty$
  then
  $\r =\min\dom F$.
  In particular, if $\r \in\interior(\dom F)$ then \EEE$\phi$ is finite.
\end{lemma}
\begin{proof}
  Up to extracting a suitable subsequence, 
  it is not restrictive to assume that $\phi$ is the limit of $\phi_n$
  as $n\to\infty$.
  For every $w\in \dom F$ the Young inequality $w\phi_n\le F(w)+\Fstar(\phi_n)$ yields
  \begin{equation}
    \label{eq:330}
    \limsup_{n\to\infty} (w-\r )\phi_{n}\le
    \limsup_{n\to\infty} F(w)-F(\r )+
    \Big(F(\r ) -\r \phi_n+\Fstar (\phi_n)\Big)
    =F(w)-F(\r )
  \end{equation}
  If $\dom F=\{\r \}$ then $\partial F(\r )=\R$ and there is nothing to
  prove; let thus assume that $\dom F$ has nonempty interior.

  If $\phi\in \R$ then $(w-\r )\phi\le F(w)-F(\r )$ for every 
  $w\in \dom F$, so that $\phi\in \partial F(\r )$.  
  Since the righthand side of \eqref{eq:330} is finite for every $w\in
  \dom F$, if $\phi=+\infty$ then $w\le \r $ 
  for every $w\in \dom F$, so that
  $\r =\max\dom F$. 
  An analogous argument holds when $\phi=-\infty$.
\end{proof}
\section{``Homogeneous'' formulations of optimal Entropy-Transport
  problems}
\label{sec:MP}
Starting from the reverse formulation of the Entropy-Transport problem
of Section \ref{subsec:reverse}
via the functional $\FHH$, see \EEE\eqref{eq:196},
in this section we will derive further equivalent representations of
the $\ET$ functional, which will also reveal new interesting
properties, in particular when we will apply these results
to the logarithmic Hellinger-Kantorovich functional.
The advantage of the reverse formulation is that it 
always admits a ``$1$-homogeneous'' representation, 
associated to a modified cost functional that can be explicitly
computed in terms of $\FH_i$ and $\sfc$.

We will always tacitly assume the basic \emph{coercive} setting of
Section \ref{subsec:setting}, see \eqref{eq:coercivity}.

\subsection{The homogeneous marginal perspective functional.}
\label{subsec:HMPf} 

First of all we introduce the marginal perspective function $\MP_c$
depending on the parameter $c\ge \inf\sfc$:
\begin{definition}[Marginal perspective function and cost]
  \label{def:MP}
  For $c\in [0,\infty)$,\EEE the marginal perspective function $\MP_c
  :[0,\infty)\times
  [0,\infty) 
  \to [0,+\infty]$
  is defined as 
  the lower semicontinuous envelope of
\begin{align}
  \label{eq:186}
  \tMPc{\s _1}{\s _2}{c}:={}&
                            \inf_{\theta>0}\theta\,
                          \big(\FH_1(\s _1/\theta)+\FH_2(\s
                              _2/\theta)+c\big)=
                              \inf_{\theta>0} \s_1
                              F_1(\theta/\s_1)+\s_2
                              F_2(\theta/\s_2)+\theta c.
\intertext{For $c=\infty$ we set}
                              \MPc {\s _1}{\s _2}{\infty}:={}&\label{eq:346}
                                F_1(0)\s _1+F_2(0)\s _2.
\end{align}
The induced marginal perspective cost is 
$\MP:(X_1\times \R_+)\times (X_2\times \R_+)\to [0,+\infty]$ with
\begin{equation}
  \label{eq:342}
  \MPH {x_1}{\s_1}{x_2}{\s_2}:=
  \MPc {\s_1}{\s_2}{\sfc(x_1,x_2)},\quad\text{for }
  x_i\in X_i \text{ and } \s_i\ge0.
\end{equation}
\end{definition}
%
The last formula \eqref{eq:346} is justified by the property
$F_i(0)=\rec{(\FH_i)}$ and the fact that 
$\MPc {\s _1}{\s _2}{c}\up \MPc {\s _1}{\s _2}{\infty}$ as $c\up\infty$ for
every $\s _1,\s _2\in [0,\infty)$, see also Lemma \ref{le:dualD} below \EEE.
\begin{example}
  \label{ex:2}
  \upshape
  Let us consider the symmetric cases associated to the entropies
  $\PE_p$
  and $V$: \nc
  \begin{enumerate}[\rm E.1]
  \item  In the ``logarithmic entropy case'', which we will
    extensively study in Part II, we have 
\begin{displaymath}
  F_i(\r ):=\PE_1(\r )=\r\log \r-(\r-1) \ \text{ and } \ 
  \FH_i(\s )=\PE_0(\s )=\s -1-\log \s .
\end{displaymath}
A direct computation shows
\begin{equation}
  \label{eq:349}
  \begin{aligned}
    \tMPc {\s _1}{\s _2}c&=\MPc {\s _1}{\s _2}c=\s _1+\s _2-2\sqrt{\s
      _1\,\s _2}\,\rme^{-c/2}\\ &=\big(\sqrt{\s _1}-\sqrt{\s
      _2}\big)^2+ 2\sqrt{\s _1\,\s _2}\,\big(1-\rme^{-c/2}\big).
  \end{aligned}
\end{equation}
\item For $p=0$, $F_i(\r)=\PE_0(\r )=\r -\log \r -1$, and $\FH_i(\s
  )=\PE_1(\s )$ we obtain
\begin{equation}
  \label{eq:349.0}
  \begin{aligned}
    \tMPc {\s _1}{\s _2}c=\MPc {\s_1}{\s_2}c&=\s_1\log \s_1+\s
    _2\log \s _2 -(\s _1+\s_2)\log\Big(\frac {\s _1+\s _2}{2+c}\Big).
  \end{aligned}
\end{equation}
\item In the power-like case with $p\in \R\setminus\{ 0,1\}$ we start from  
  \begin{displaymath}
    F_i(\r ):=\PE_p(\r )=\frac 1{p(p-1)}\big(\r ^p-p(\r-1)-1\big),\quad
    \FH_i(\s )=\PE_{1-p}(\s )
  \end{displaymath}
  and obtain, for $\s _1,\s _2>0$,
\begin{equation}
  \label{eq:349bis}
  \tMPc {\s _1}{\s _2}c=\MPc {\s _1}{\s _2}c=
  \frac 1p \Big[\big(\s_1+\s _2\big)- 
  \frac{\s _1\,\s _2}{(\s _1^{p-1}+\s
    _2^{p-1})^{1/(p-1)}} \Big(2-(p-1)c\Big)_+^q\Big],
\end{equation}
where $q=p/(p-1)$. In fact, we have 
\begin{align*}
  \theta &\big(\PE_{1-p}(\tfrac{\s_1}\theta)+
  \PE_{1-p}(\tfrac{\s_2}\theta)+c\big)=
  \frac {\s_1^{1-p}+\s_2^{1-p}}{p(p-1)}\theta^p+\frac
           1p(\s_1+\s_2)+\frac 1{p-1}((p-1)c-2)\theta)\\
         &=\frac1p(\s_1+\s_2)+    \frac1{p-1}
           \Big[\frac
           1{p}\Big((\s_1^{1-p}+\s_2^{1-p})^{1/p}\,\theta\Big)^p-
           \big(2-(p-1)c\big)\theta\Big],
\end{align*}
and \eqref{eq:349bis} follows by minimizing w.r.t.~$\theta$.
E.g.~when $p=q=2$
\begin{equation}
  \label{eq:349tris}
\MPc {\s _1}{\s _2}c=\frac12\big(\s _1+\s
  _2\big)
  -\frac 12 \frac{\s
    _1\s _2}{\s _1+\s _2}(2-c)_+^2=\frac1{2(\s _1+\s _2)}\Big((\s
  _1-\s _2)^2+h(c)\s _1\s _2\Big),
\end{equation}
where $h(c)=c(4-c)$ if $0\le c\le 2$ and $4 $ if $c\ge 2$.
For $p=-1$ and $q=1/2$ equation \eqref{eq:349bis} yields
\begin{equation}
  \label{eq:349quater}
  \tMPc {\s _1}{\s _2}c=\MPc {\s _1}{\s _2}c=
  \sqrt{(\s _1^2+\s _2^2)(2+2c)}-\big(\s _1+\s
  _2\big).
\end{equation}
\item In the case of the total variation entropy $V(s)=R(s)=|s-1|$ we
easily find 
\begin{displaymath}
  \tilde H_c(r_1,r_2)=H_c(r_1,r_2)=\s _1+\s _2
  -(2-c)_+(\s _1\land\s _2)=
   |\s_2-\s_1|+(c\land 2) (\s _1\land\s _2).
\end{displaymath}
\end{enumerate}
\end{example}

The following dual characterization of $\MP_c$  
nicely explains the crucial role of $\MP_c$.

\begin{lemma}[Dual characterization of $\MP_c$]
  \label{le:dualD}
  For every $c\ge 0$ the function $\MP_c$ admits the dual representation
    \begin{align}
    \label{eq:188}
    \MPc{\s _1}{\s _2}c &=\sup
                      \Big\{\s _1\psi_1+\s _2\psi_2:
                      \psi_i\in \dom{\FHstar_i},\
                      \FH_1^*(\psi_1)+\FH_2^*(\psi_2)\le c\Big\}\\ 
    &=\sup\Big\{\s _1\Gstar_1( \varphi_1 )+\s _2\Gstar_2(\varphi_2):
      \varphi_i\in \dom{\Gstar_i},\ \varphi_1+\varphi_2\le c\Big\}. 
   \label{eq:223}
  \end{align}
  In particular it is lower semicontinuous, convex and positively
  $1$-homogeneous (thus sublinear) with respect to $(\s_1,\s_2)$,
  nondecreasing and concave w.r.t.~$c$, and satisfies
  \begin{equation}
    \label{eq:462}
    \MPc{\s _1}{\s _2}c\le     \MPc{\s _1}{\s _2}\infty=\sum_i
    F_i(0)\s_i\quad\forevery c\ge0,\ \s_i\ge0.
  \end{equation}
  Moreover, 
  \begin{itemize}
  \item[a)] the function $\MP_c$ coincides with $\tilde\MP_c$ 
    in the interior of its domain; in particular, if 
    $F_i(0)<\infty$ then 
    $\MP_c(\s_1,\s_2)=\tilde\MP_c(s_1,\s_2)$ whenever $\s_1\s_2>0$.
  \item[b)] If $\rec{(F_1)}+\derzero{(F_2)}+c\ge 0$ and $\rec{(F_2)}+\derzero{(F_1)}+c\ge 0$, then 
      \begin{equation}
        \label{eq:374}
        \MP_c(\s_1,\s_2)=\sum_i F_i(0)\s_i\quad\text{if }\s_1\s_2=0.
      \end{equation}
  \end{itemize}
\end{lemma}
\begin{proof}
  Since $\sup\dom{\FHstar_i}=F_i(0)$ by \eqref{eq:298}, 
  one immediately gets
  \eqref{eq:188} in the case $c=+\infty$; we can thus assume
  $c<+\infty$. 

  It is not difficult to check that the function
  $(\s_1,\s_2,\theta)\mapsto \theta
  \big(\FH_1(\s_1/\theta)+\FH_2(\s_2/\theta)+c\big)$
  is jointly convex in $[0,\infty)\times [0,\infty)\times (0,\infty)$ 
  so that $\tilde\MP_c$ is a convex and positive $1$-homogeneous
  function. It is also proper (i.e.~it is not identically $+\infty$)
  thanks to \eqref{eq:83}.
  By Legendre duality \cite[Thm.12.2]{Rockafellar70}, its lower semicontinuous envelope 
  is given by
  \begin{equation}
    \label{eq:372}
    \MP_c(\s_1,\s_2)=\sup\Big\{\sum_i\psi_i\s_i:
    \MP_c^*(\psi_1,\psi_2)\le 0\Big\},    
  \end{equation}
  where 
  \begin{align*}
    \MP_c^*(\psi_1,\psi_2)&=\sup\Big\{\sum_i\psi_i\s_i-\tilde \MP_c(\s_1,\s_2):\s_i\ge0\Big\}
                            =
                                \sup_{\s_i\ge 0,\theta>0}
                                \sum_i\Big(\psi_i\s_i-\theta
                            \FH_i(\s_i/\theta)\Big)-c\theta
                            \\&=
                                \sup_{\theta>0} \theta\Big(\sum_i
                                \FHstar_i(\psi_i)-c\Big)=
                                \begin{cases}
                                  0&\text{if
                                  }\FHstar_i(\psi_i)<\infty,\quad \sum_i\FHstar_i(\psi_i)\le c\\
                                  +\infty&\text{otherwise.}
                                \end{cases}
  \end{align*}
  %
  %
  %
  In order to prove point a) it is sufficient to recall that convex
  functions are always continuous in the interior of their domain
  \cite[Thm.~10.1]{Rockafellar70}. In particular, since
  $\lim_{\theta\down0} \theta \big(\FH_1(\s_1/\theta) +
  \FH_2(\s_2/\theta)+c)= \sum_{i}\rec{(\FH_i)}\s_i=\sum_i F_i(0)\s_i$
  for every $\s_1,\s_2>0$,\EEE we have $\tilde \MP_c(\s_1,\s_2)\le
  \sum_i F_i(0)\s_i$, so that $\tilde \MP_c$ is always finite if
  $F_i(0)<\infty$.

  Concerning b), it is obvious when $\s_1=\s_2=0$. When
  $\s_1>\s_2=0$, the facts that $\sup\dom{\FHstar_i}=F_i(0)$,
  $\lim_{\s \uparrow F_i(0)}\FHstar_i(\s) =-\derzero{(F_i)}$, and
  $\inf\FHstar_i=-\rec{(F_i)}$ (see \eqref{eq:298}) yield
  \begin{displaymath}
    \MP_c(\s_1,0)=\sup \Big\{\psi_1 \s_1: \FHstar_1(\psi_1)\le
    c-\inf\FHstar_2\Big\}=
    F_1(0)\s_1.
  \end{displaymath}
  An analogous formula holds when $0=\s_1<\s_2$.
  \end{proof}
A simple consequence of Lemma \ref{le:dualD} and \eqref{eq:217} is the
lower bound
\begin{equation}
  \label{eq:314}
  \tMPc{\s_1}{\s_2}c \ge\MPc{\s_1}{\s_2}c \ge
  \sum_{i}\psi_i
  \s_i\quad\text{for }
  (-\varphi_i,\psi_i)\in \frF_i~\text{with}~
  \varphi_1+\varphi_2\le c.
\end{equation}
We now introduce 
the integral functional associated with the
marginal perspective cost~\eqref{eq:342}, which is
based on the decomposition $\mu_i=\varrho_i\gamma_i+\mu_i^\perp$:\EEE
\begin{equation}
  \label{eq:201}
  \HH(\mu_1,\mu_2|\ggamma):=\int_\sxX
  \MPH {x_1}{\varrho_1(x_1)}{x_2}{\varrho_2(x_2)}
    \,\d\ggamma+ \sum_i F_i(0)\mu_i^\perp(X_i)
\end{equation}
where we adopted the same notation as in \eqref{eq:257}.
Let us first show that $\HH$ is always greater than $\DD$.
\begin{lemma}
  \label{le:HMlower-bound}
  For every $\ggamma\in \cM(\xX)$, $\mu_i,\mu_i'\in \cM(X_i)$, $\vvarphi\in
  \Cphi{}$, $\varrho_i\in \rmL^1_+(X_i,\gamma_i)$ with
  $\mu_i=\varrho_i\gamma_i+\mu_i'$, we have
  \begin{equation}
    \label{eq:347}
    \int_{\sxX}
    \MPH {x_1}{\varrho_1(x_1)}{x_2}{\varrho_2(x_2)}\,\d\ggamma+
    \sum_i F_i(0)\mu_i'(X_i)\ge 
    \DD(\vvarphi|\mu_1,\mu_2).
  \end{equation}
\end{lemma}
\begin{proof}
  Recalling that $\Gstar_i(\varphi_i)=-\Fstar(-\varphi_i)\ge F_i(0)$
  and using \eqref{eq:314} with $r_j=\rho_j$ 
  and $\psi_j = \Gstar_j(\rho_j)$ we have
    \begin{align*}
      \int_{\sxX}&
    \MPH {x_1}{\varrho_1(x_1)}{x_2}{\varrho_2(x_2)}\,\d\ggamma+
    \sum_i F_i(0)\mu_i'(X_i)
    \\\topref{eq:314}\ge&
    \int_\sxX
    \Big(\Gstar_1(\varphi_1(x_1))\varrho_1(x_1)+\Gstar_2(\varphi_2(x_2))\varrho_2(x_2)\Big)\,\d\ggamma+
    \sum_i F_i(0)\mu_i'(X_i)
    \\=\,\,\,&\sum_i \int_{X_i} \Gstar_i(\varphi_i)\varrho_i(x_i)\,\d\gamma_i+    \sum_i
    F_i(0)\mu_i'(X_i)
    \\
    \topref{eq:40}\ge &\sum_i \int_{X_i}
    \Gstar_i(\varphi_i)\varrho_i(x_i)\,\d\gamma_i+
    \sum_i\int_{X_i}\Gstar_i(\varphi_i)\,\d\mu_i'
    =\sum_{i} \int_{X_i}\Gstar_i(\varphi_i)\,\d\mu_i=\DD(\vvarphi|\mu_1,\mu_2).
  \end{align*}
  Note that \eqref{eq:40} and \eqref{eq:304} imply
  $\Gstar_i(\varphi_i)\leq F_i(0)$. 
\end{proof}
An immediate consequence of the previous lemma is the following
important result concerning the marginal perspective cost functional
$\HH$ defined by \eqref{eq:201}. It can be nicely compared to the
Reverse Entropy-Transport functional $\FHH$ for which
Theorem~\ref{thm:reverse-characterization} stated
$\FHH(\mu_1,\mu_2|\ggamma)=\EE(\ggamma|\mu_1,\mu_2)$.
\begin{theorem}
  \label{thm:crucial}
  For every $\mu_i\in \cM(X_i)$, $\ggamma\in \cM(\xX)$ and
  $\vvarphi\in \Cphi{}$ we have
  \begin{equation}
    \label{eq:348}
    \FHH(\mu_1,\mu_2|\ggamma)\ge 
    \HH(\mu_1,\mu_2|\ggamma)\ge \DD(\vvarphi|\mu_1,\mu_2).
  \end{equation}
  In particular
  \begin{align}
    \label{eq:259bis}
    \ET(\mu_1,\mu_2)&= \sfH(\mu_1,\mu_2):=\min_{\sggamma\in
                      \cM(\sxX)}\HH(\mu_1,\mu_2|\ggamma),
  \end{align}
  and $\ggamma\in \OptET(\mu_1,\mu_2)$ if and only if 
  it minimizes $\HH(\mu_1,\mu_2|\cdot)$ in $\cM(\xX)$ and satisfies
  \begin{equation}
    \label{eq:358}
    \MPH {x_1}{\varrho_1(x_1)}{x_2}{\varrho_2(x_2)}=\sum_i
    \FH_i(\varrho_i(x_i))+\sfc(x_1,x_2)
    \quad \text{$\ggamma$-a.e.~in $\xX$,}
  \end{equation}
  where $\varrho_i$ is defined as in \eqref{eq:310}. \EEE
  If moreover the following conditions
  \begin{equation}
    \begin{aligned}
      &\text{$F_1(0)=+\infty$ or there exists $\bar x_2\in X_2$ with
        $\mu_2(\{\bar x_2\})=0$},\\
      &\text{$F_2(0)=+\infty$ or there exists $\bar x_1\in X_1$ with
        $\mu_1(\{\bar x_1\})=0$},
    \end{aligned}
\label{eq:cem}
\end{equation}
\nc are satisfied, then 
  \begin{align}
  \label{eq:363}
        \ET(\mu_1,\mu_2)&=
                          \min \Big\{
  \int_{\sxX}
  \MPH {x_1}{\varrho_1(x_1)}{x_2}{\varrho_2(x_2)}\,\d\ggamma:
  \ggamma\in \cM(\xX),\ \mu_i=\varrho_i\gamma_i\Big\}.
  \end{align}
\end{theorem}
\begin{proof}
  The inequality $\FHH(\mu_1,\mu_2|\ggamma)\ge 
  \HH(\mu_1,\mu_2|\ggamma)$ is an immediate consequence 
  of the fact that $\sum_i\FH_i(\s_1,\s_2)+c\ge \tMPc{\s_1}{\s_2}c \ge
  \MPc{\s_1}{\s_2}c$ for every 
  $\s_i,c\in [0,\infty]$, obtained by choosing $\theta=1$ in
  \eqref{eq:186}. 
  The estimate $\HH(\mu_1,\mu_2|\ggamma)\ge
  \DD(\vvarphi|\mu_1,\mu_2)$ 
  was shown in by Lemma \ref{le:HMlower-bound}.
  
  By using the ``reverse'' formulation of $\ET(\mu_1,\mu_2)$
  in terms of the functional $\FHH(\mu_1,\mu_2|\ggamma)$ 
  given by Theorem \ref{thm:reverse-characterization}
  and applying Theorem \ref{thm:weak-duality} we obtain
  \eqref{eq:259bis}
  and the characterization \eqref{eq:358}.

  To establish the identity \eqref{eq:363} we note that the difference
  to \eqref{eq:259bis} only lies in dropping the additional restriction
  $\mu_i^\perp =0$.
  When both $F_1(0)=F_2(0)=+\infty$ the 
  equivalence is obvious since the finiteness of the functional 
  $\ggamma\mapsto \HH(\mu_1,\mu_2|\ggamma)$ 
  yields $\mu_1^\perp=\mu_2^\perp=0$. \nc 

  In the general case, one immediately see that 
  the righthand side $E'$ of \eqref{eq:363} (with ``$\inf$'' instead of
  ``$\min$'')
  is larger than $\ET(\mu_1,\mu_2)$, since the infimum of 
  $\HH(\mu_1,\mu_2|\cdot)$ is constrained to the smaller set 
  of plans $\ggamma$ satisfying $\mu_i\ll\gamma_i$. 
  On the other hand, if $\bar\ggamma\in \OptET(\mu_1,\mu_2)$
  with $\mu_i=\varrho_i\bar\gamma_i+\mu_i^\perp$ and
  $\tilde m_i:=\mu_i^\perp(X_i)>0$, we can consider $\ggamma:= 
  \bar\ggamma+ \frac{1}{\tilde m_1\tilde m_2}\mu_1^\perp \otimes
  \mu_2^\perp$ which satisfies $\mu_i\ll\gamma_i$; 
  by exploiting the fact that $\MPH
  {x_1}{\s_1}{x_2}{\s_2}\le \sum_i F_i(0)\s_i$ by \eqref{eq:462},
  we obtain
  \begin{align*}
    \HH(\mu_1,\mu_2|\ggamma)&=
                              \int_{\sxX}
      \MPH {x_1}{\varrho_1(x_1)}{x_2}{\varrho_2(x_2)}\,\d\bar\ggamma+
      \frac{1}{\tilde m_1\tilde m_2}
                              \int_{\sxX}\MPH{x_1}{\tilde m_1}{x_2}{\tilde m_2}\,\d\mu_1^\perp \otimes
      \mu_2^\perp
                              \\&\le 
                                  \int_{\sxX}
      \MPH {x_1}{\varrho_1(x_1)}{x_2}{\varrho_2(x_2)}\,\d\bar\ggamma+
      \sum_i F_i(0)\tilde m_i=
                                  \HH(\mu_1,\mu_2|\bar\ggamma),
  \end{align*}
  so that we have $E'\le \ET(\mu_1,\mu_2)$. 
  The case when only one (say $\mu_2^\perp$) of the measures
  $\mu_i^\perp$ vanishes
  can be treated in the same way: since in this case
  $\tilde m_1=\mu_1^\perp(X_1)>0$ and therefore $F_1(0)<\infty$,
  \nc by 
  applying \eqref{eq:cem} we can choose
  $\ggamma:= 
  \bar\ggamma+ \frac{1}{\tilde m_1}\mu_1^\perp \otimes
  \delta_{\bar x_2}$, 
  obtaining
    \begin{align*}
    \HH(\mu_1,\mu_2|\ggamma)&=
                              \int_{\sxX}
      \MPH {x_1}{\varrho_1(x_1)}{x_2}{\varrho_2(x_2)}\,\d\bar\ggamma+
      \frac{1}{\tilde m_1}
                              \int_{X_1}\MPH{x_1}{\tilde m_1}{\bar x_2}{0}\,\d\mu_1^\perp 
                              \\&\le 
                                  \int_{\sxX}
      \MPH {x_1}{\varrho_1(x_1)}{x_2}{\varrho_2(x_2)}\,\d\bar\ggamma+
                                  F_1(0)\tilde m_1=
                                  \HH(\mu_1,\mu_2|\bar\ggamma).\qedhere
  \end{align*}
\end{proof}
\begin{remark}
  \label{rem:cemetery}
  \upshape
  Notice that \eqref{eq:cem} is always satisfied if the spaces $X_i$ are
  uncountable.
  If 
  $X_i$ is countable, one can always add 
  an isolated
  point 
  $\bar x_i$ (sometimes called ``cemetery'') to $X_i$ and consider the
  augmented space $\bar X_i=X_i\sqcup\{\bar x_i\}$ 
  obtained as the disjoint union of $X$ and
  $\bar x_i$, with augmented cost $\bar\sfc$ which extends $\sfc$ to
  $+\infty$
  on $\bar X_1\times \bar X_2\setminus (X_1\times X_2)$. 
  We can recover \eqref{eq:363} by allowing $\ggamma$
  in $\cM(\bar X_1\times\bar X_2)$. 
\end{remark}

\subsection{Entropy-transport problems with ``homogeneous'' marginal
  constraints}
\label{subsec:hom-marg}
In this section we will exploit the $1$-homogeneity
of the marginal perspective function $\HH$ in order to derive a last
representation of the functional $\ET$, related to the new
notion of \emph{homogeneous marginals}.
We will confine our presentation to the basic, still relevant, facts, 
and we will devote the second part of the paper to develop a full
theory 
for the specific case of the Logarithmic Entropy-transport case. 

In particular, the following construction (typical in the Young
measure approach to variational problems) allows us to consider
the entropy-transport problems in a setting of greater generality.  We
replace a couple \nc $(\gamma,\varrho)$, where $\gamma$ and $\varrho$
are a measure on $X$ and a nonnegative Borel function, respectively,
by a measure $\alpha\in \cM(Y)$ on the extended space $Y = X\times
[0,\infty)$.  The original couple $(\gamma,\varrho)$ corresponds to
measures $\alpha =(x,\varrho(x))_\sharp \gamma$ concentrated on the
graph of $\varrho$ in $Y$ and whose first marginal is $\gamma$. \nc

\paragraph{Homogeneous marginals.}
In the usual setting of Section \ref{subsec:setting}, 
we consider the product
spaces $\pY_i:=X_i\times [0,\infty)$ 
endowed with the product topology and denote 
the generic points in $\pY_i$ with $\py_i=(x_i,\rp_i)$, $x_i\in X_i $ and $ \rp_i\in
[0,\infty)$ for $i=1,2$. 
Projections from $\yY:=Y_1\times Y_2$ onto the various coordinates
will be denoted by $\pi^{\py_i},\ \pi^{x_i},\ \pi^{\rp_i}$ with
obvious meaning. 

\newcommand{\prd}{\mathrm{prd}}
For $p>0$ and $\yy\in \yY$ we will set $|\yy|_p^p:=\sum_i |\s_i|^p$ and
call $\cMp p (\yY)$ (resp.~$\cPp p (\yY)$)
the space of measures 
$\aalpha\in \cM(\yY)$ (resp.\ $\cP(\yY)$)
such that
\begin{equation}
  \label{eq:145}
  \int_{\syY} |\yy|_p^p\,\d\aalpha<\infty.
\end{equation}
If $\aalpha\in \cMp p(\yY)$ the measures $\rp_i^p\aalpha$ belong to
$\cM(\yY)$, which allow us to define the ``$p$-homogeneous'' marginal
$\hm pi(\aalpha)$ of $\aalpha\in \cMp p (\yY)$ as the $x_i$-marginal
of $\rp_i^p \aalpha$, namely 
\begin{equation}
  \label{eq:146}
  \hm pi (\aalpha):=\pi^{x_i}_\sharp(\rp_i^p \aalpha)\in \cM(X_i).
\end{equation}
The maps $\hm pi:\cMp p (\yY)\to \cM(X_i)$ are linear 
and invariant with respect to dilations:
if $\vartheta:\yY\to(0,\infty)$ 
is a Borel map in $\rmL^p(\yY,\aalpha)$
and 
$\prd_\vartheta(\yy):=(x_1,\rp_1/\vartheta(\yy);x_2,\rp_2/\vartheta(\yy) )$,
 we set
\begin{equation}
  \label{eq:147}
  \begin{gathered}
    \dilp\vartheta p{\aalpha}:={}
    \big(\prd_\vartheta)_\sharp\big(   \vartheta^p \aalpha\big),\quad
    \text{i.e.}\\
    \int
    \varphi(\yy)\,\d(\dilp\vartheta p{\aalpha})={}\int \varphi(x_1,\rp_1/
    \vartheta;x_2,\rp_2/\vartheta)\vartheta^p(\yy)\,\d\aalpha(\yy)\quad 
    \text{for } \varphi\in \rmB_b(\yY).
  \end{gathered}
\end{equation}
Using \eqref{eq:146} we obviously have 
\begin{equation}
  \label{eq:150}
  \hm pi(\dilp\vartheta p{\aalpha})=\hm pi(\aalpha).
\end{equation}
In particular, for $\aalpha\in \cMp p (\yY)$ with
$\aalpha(\yY)>0$, 
by choosing 
\begin{subequations}
\begin{equation}
  \label{eq:153a}
  \vartheta(\yy):= \frac{1}{r_*}
  \begin{cases}
    |\yy|_p&\text{if }|\yy|_p\neq0,\\
    1&\text{if }|\yy|_p=0,
  \end{cases}
  \qquad
  r_*\EEE:=\Big(\int_{\syY} |\yy|_p^p\,\d\aalpha+\aalpha(\{|\yy|=0\})\Big)^{1/p}
\end{equation}
we obtain a rescaled probability measure $\tilde\aalpha$
with
the same homogeneous marginals as $\aalpha$ and concentrated on 
$  \yY_{ \kern-2pt r_*, p}:=\big\{\yy\in \yY: |\yy|_p\le r_*
\nc \big\}\subset
  (X\times[0, r_* \nc ])\times(X\times [0, r_* \nc ])
$:
\begin{equation}
  \label{eq:153b}
  \tilde\aalpha=\dilp\vartheta p{\aalpha}\in \cPp p (\yY),\quad
  \hm pi(\tilde\aalpha)=\hm pi(\aalpha),\quad
  \tilde\aalpha\big(\yY\setminus \yY_{\kern-2pt r,p}\big)=0.
\end{equation}
\end{subequations}
\paragraph{Entropy-transport problems with prescribed homogeneous marginals.}
Given $\mu_1,\mu_2\in \cM(X)$ we now introduce the convex sets
%
\begin{equation}
  \begin{aligned}
        \HMle p{\mu_1}{\mu_2}:={}&\Big\{
    \taalpha\in \cMp p (\yY):\hm pi(\aalpha)\le \mu_i\Big\},
    \\
    \HM p{\mu_1}{\mu_2}:={}&\Big\{
    \taalpha\in \cMp p (\yY):\hm pi(\aalpha)= \mu_i\Big\}.
  \end{aligned}
\label{eq:161}
\end{equation}
%
Clearly $\HM p{\mu_1}{\mu_2}\subset \HMle p{\mu_1}{\mu_2}$ and they\EEE
are nonempty since plans 
of the form 
\begin{equation}
  \label{eq:103}
  \taalpha=\frac 1{a_1^p\,a_2^p} \Big(\mu_1\otimes
  \delta_{a_1}\Big)\otimes
  \Big(\mu_2\otimes \delta_{a_2}\Big),\quad\text{with } a_1,a_2>0
\end{equation}
belong to $\HM p{\mu_1}{\mu_2}$.  It is not difficult to check that
$\HMle p{\mu_1}{\mu_2}$ is also narrowly closed, while, on the
contrary, this property fails for $\HM p{\mu_1}{\mu_2}$ if
$\mu_1(X_1)\mu_2(X_2)\neq 0$. To see this, it is sufficient to
consider any $\aalpha\in \HM p{\mu_1}{\mu_2}\setminus\{0\}$ and look
at the vanishing sequence $\dilp {n^{-1}}p{\aalpha}$ for
$n\to\infty$.

There is a natural correspondence between 
$\HMle p{\mu_1}{\mu_2}$ (resp.~$\HM p{\mu_1}{\mu_2}$)
and $\HMle 1{\mu_1}{\mu_2}$ (resp.~$\HM 1{\mu_1}{\mu_2}$)
induced by the map 
$\yY\ni (x_1,\rp_1;x_2,\rp_2)\mapsto
  (x_1,\rp_1^p;x_2,\rp_2^p).$
For plans $\aalpha\in \HMle{}{\mu_1}{\mu_2}$ we can prove a result
similar to Lemma \ref{le:HMlower-bound} but now we obtain a linear functional in $\aalpha$.\EEE
\begin{lemma}
  \label{le:HMlower-bound2}
  For $p\in (0,\infty)$, $\mu_i\in \cM(X_i)$, 
  $\vvarphi\in
  \Cphi{}$, and $\aalpha\in \HMle{p}{\mu_1}{\mu_2}$ 
  we have
  \begin{equation}
    \label{eq:347bis}
    \int_{\sxX}
    \MPH {x_1}{\rp_1^p}{x_2}{\rp_2^p}\,\d\aalpha+
    \sum_i F_i(0)\mu_i'(X_i)\ge 
    \DD(\vvarphi|\mu_1,\mu_2),\quad\text{where }
  \mu_i':=\mu_i-\hm{p}i\aalpha.\EEE
  \end{equation}
\end{lemma}
\begin{proof}
  The calculations are quite similar to the proof of Lemma \ref{le:HMlower-bound}:
    \begin{align*}
      \int_{\syY}&
                   \MPH {x_1}{\rp_1^p}{x_2}{\rp_2^p}\,\d\aalpha+
                   \sum_i F_i(0)\mu_i'(X_i)
    \\\topref{eq:314}\ge&
    \int_\syY
    \Big(\Gstar_1(\varphi_1(x_1))\rp_1^p+\Gstar_2(\varphi_2(x_2))\rp_2^p\Big)\,\d\aalpha+
                          \sum_i F_i(0)\mu_i'(X_i)
    \\=\,\,\,&\sum_i \int_{X_i} \Gstar_i(\varphi_i)\,\d(\hm{p}i\aalpha)+    \sum_i
    F_i(0)\mu_i'(X_i)
    \\
    \topref{eq:40}\ge &\sum_i \int_{X_i}
                        \Gstar_i(\varphi_i) \,\d(\hm{p}i\aalpha)+
                        \sum_i\int_{X_i}\Gstar_i(\varphi_i)\,\d\mu_i'
                        =\sum_{i} \int_{X_i}\Gstar_i(\varphi_i)\,\d\mu_i=\DD(\vvarphi|\mu_1,\mu_2).\qedhere
  \end{align*}
\end{proof}

As a consequence, we can characterize the entropy-transport minimum
via measures $\aalpha\in\cM(\yY)$.
\begin{theorem}
  \label{thm:main-hom-marg}
  For every $\mu_i\in \cM(X_i)$, $p\in (0,\infty)$ we have
  \begin{align}
    \label{eq:318}
    \ET(\mu_1,\mu_2)&=
    \min_{\saalpha\in \HMle{p}{\mu_1}{\mu_2}}\int_{\syY}
                      \Big(\sum_i \FH_i(\rp_i^p)
                      +\sfc(x_1,x_2)\Big)\,\d\aalpha+
                      \sum_i F_i(0)(\mu_i-\hm{p}i(\aalpha))(X_i)\\
                    &\label{eq:343}=
    \min_{\saalpha\in \HMle{p}{\mu_1}{\mu_2}}\int_{\syY}
                      \MPH{x_1}{\rp_1^p}{x_2}{\rp_2^p}\,\d\aalpha
                      +\sum_i F_i(0)(\mu_i-\hm{p}i(\aalpha))(X_i)
                    \\  
                    &\label{eq:343bis}=
    \min_{\saalpha\in \HM{p}{\mu_1}{\mu_2}}\int_{\syY}
                      \MPH{x_1}{\rp_1^p}{x_2}{\rp_2^p}\,\d\aalpha.                         
  \end{align}
  Moreover, for every plan $\ggamma\in \OptET(\mu_1,\mu_2)$ 
  (resp.~optimal for \eqref{eq:259bis} or for \eqref{eq:363}) 
  with $\mu_i=\varrho_i\gamma_i+\mu_i^\perp$, 
  the plan $\aalpha:=(x_1,\varrho_1^{1/p}(x_1);x_2,\varrho_2^{1/p}(x_2))_\sharp
  \ggamma$
  realizes the minimum of \eqref{eq:318} (resp.~\eqref{eq:343} or \eqref{eq:343bis}).
\end{theorem}
\begin{remark}
  \label{rem:max-le}
  \upshape
  When $F_i(0)=+\infty$ \eqref{eq:318} and \eqref{eq:343}
  simply read as
    \begin{align*}
      \ET(\mu_1,\mu_2)&=
    \min_{\saalpha\in \HM{p}{\mu_1}{\mu_2}}\int_{\syY}
                      \Big(\sum_i \FH_i(\rp_i^p)
                      +\sfc(x_1,x_2)\Big)\,\d\aalpha\\
                    &
                    =\min_{\saalpha\in \HM{}{\mu_1}{\mu_2}}\int_{\syY}
      \MPH{x_1}{\rp_1^p}{x_2}{\rp_2^p}\,\d\aalpha.\qedhere
  \end{align*}
\end{remark}
\begin{proof}[Proof of Theorem \ref{thm:main-hom-marg}]
  Let us denote by $E'$ 
  (resp.~$E''$, $E'''$) the right-hand side of \eqref{eq:318}
  (resp.~of \eqref{eq:343}, \eqref{eq:343bis}),
  where ``$\min$'' has been replaced by ``$\inf$''.
  If $\ggamma\in \cM(\xX)$ 
  and $\mu_i=\varrho_i\gamma_i+\mu_i^\perp$ 
    (in the case of \eqref{eq:343bis} $\mu_i^\perp=0$) is the usual Lebesgue
  decomposition as in \eqref{eq:257}, we can consider the plan 
  $\aalpha:=(x_1,\varrho_1^{1/p}(x_1);x_2,\varrho_2^{1/p}(x_2))_\sharp
  \ggamma$.

  Since the map $(\varrho_1^{1/p},\varrho_2^{1/p}):\xX\to \R^2$ is Borel and takes
  values in a metrizable and separable space, 
  it is Lusin $\ggamma$-measurable
  \cite[Thm~5, p.~26]{Schwartz73}, so that $\aalpha$ is a Radon
  measure in $\cM(\yY)$. For every nonnegative $\phi_i\in
  \rmB_b(X_i)$ we easily get 
  \begin{align*}
    \int \phi_i(x_i)\s_i^p\,\d\aalpha&=\int
    \varrho_i(x_i)\phi_i(x_i)\,\d\ggamma
    =\int\varrho_i\phi_i\,\d\gamma_i\le \int\phi_i\,\d\mu_i,
  \end{align*}
  so that $\aalpha\in \HMle{p}{\mu_1}{\mu_2}$, $\hm{p}i\aalpha=\gamma_i$,
  and
  \begin{align*}
    \RR(\mu_1,\mu_2|\ggamma)&=
                              \int_\sxX
                              \Big(\sum_i\FH_i(\varrho_i(x_i))+\sfc(x_1,x_2)\Big)\,\d\ggamma+
                              \sum_i F_i(0)\mu_i^\perp(X_i)\\
                            &=
                              \int_\syY
                              \sum_i\FH_i(\s_i^p)+\sfc(x_1,x_2)\Big)\,\d\aalpha+
                              \sum_i
                              F_i(0)(\mu_i-\hm{p}i\aalpha)(X_i)\ge E';
  \end{align*}
  taking the infimum w.r.t.~$\ggamma$ and recalling \eqref{eq:259} we
  get $\ET(\mu_1,\mu_2)\ge E'$.
    Since $\sum_i\FH_i(\rp_i^p)+\sfc(x_1,x_2)\ge
  \MPH{x_1}{\rp_1^p}{x_2}{\rp_2^p}$
  it is also clear that $E'\ge E''$.

  On the other hand, Lemma \ref{le:HMlower-bound2} shows that 
  $E''\ge \DD(\vvarphi|\mu_1,\mu_2)$ for every $\vvarphi\in \Cphi{}$:
  applying Theorem \ref{thm:weak-duality} we get
  $\ET(\mu_1,\mu_2)=E'=E''$.

  Concerning $E'''$ it is clear that $E'''\ge E''=\ET(\mu_1,\mu_2)$;
  when \eqref{eq:cem} hold, by
  choosing $\aalpha$ induced by a minimizer of \eqref{eq:363} we get
  the opposite inequality $E'''\le \ET(\mu_1,\mu_2)$. 

  If \eqref{eq:cem} does not hold, we can still apply 
  a slight modification of the argument at the end of the proof
  of Theorem \ref{thm:crucial}. The only case to consider is when
  only one of the two measures $\mu_i^\perp$ vanishes:
  just to fix the ideas, let us suppose that $\tilde m_1=\mu_1^\perp(X_1)>0=\mu_2^\perp(X_2)$.
  If $\bar\ggamma\in \OptET(\mu_1,\mu_2)$ and 
  $\bar\aalpha$ is obtained as above, we can just set
  $\aalpha:=\bar\aalpha+(\mu_1^\perp\times
  \delta_1)\times(\nu\times \delta_0)$ for an arbitrary $\nu\in
  \cP(X_2)$.
  It is clear that $\hm pi\aalpha=\mu_i$ and
  \begin{align*}
    \int_{\syY}&
      \MPH{x_1}{\rp_1^p}{x_2}{\rp_2^p}\,\d\aalpha=
                                                   \int_{\sxX}
      \MPH {x_1}{\varrho_1(x_1)}{x_2}{\varrho_2(x_2)}\,\d\bar\ggamma+
                                                   \int_{\sxX}\MPH{x_1}{1}{x_2}{0}\,\d\mu_1^\perp \otimes
                                                   \nu
                              \\&\topref{eq:462}\le 
                                  \int_{\sxX}
      \MPH {x_1}{\varrho_1(x_1)}{x_2}{\varrho_2(x_2)}\,\d\bar\ggamma+
                                  F_1(0)\tilde m_1=
                                  \HH(\mu_1,\mu_2|\bar\ggamma)=\ET(\mu_1,\mu_2),
  \end{align*}
  which yields $E'''\le \ET(\mu_1,\mu_2)$.
\end{proof}
\begin{remark}[Rescaling invariance]
  \label{rem:rescaling1}
  \upshape
  By recalling (\ref{eq:153a},b) and exploiting the $1$-homo\-geneity of
  $H$ 
  it is not restrictive to sol\-ve the minimum problem
  \eqref{eq:343} in the smaller class of probability plans 
  concentrated in 
  \begin{displaymath}
    \yY_{\kern-2pt r,p}:=\big\{(x_1,\s_1;x_2,\s_2)\in \yY:\s_1^p+\s_2^p\le
    r^p\big\},\quad r^p=\sum_{i}\mu_i(X_i).
  \end{displaymath}
  Notice that it is not restrictive to assume that $\aalpha(\{\yy\in
  \yY:|\yy|=0\})=0$
  since $H(x_1,0;x_2,0)=0$ for every $x_i\in X_i$.
\end{remark}

\PART{Part II. The Logarithmic Entropy-Transport problem and
the Hellin\-ger-Kan\-to\-ro\-vich distance} 

\section{The Logarithmic Entropy-Transport (LET) problem}
\label{sec:LET}
Starting from this section we will study a particular
Entropy-Transport problem, whose structure reveals 
surprising properties.

\subsection{The metric setting for Logarithmic Entropy-Transport
  problems.}
\label{subsec:LET1}
Let $(X,\tau)$ be a Hausdorff topological space endowed with
an extended distance function $\sfd:X\times X\to [0,\infty]$
which is lower semicontinuous w.r.t.~$\tau$; we refer to 
$(X,\tau,\sfd)$ as an extended metric-topological space.
In the most common situations, $\sfd$ will
take finite values, $(X,\sfd)$ will be separable and complete
and $\tau$ will be the topology induced by $\sfd$;
nevertheless, there are interesting applications
where nonseparable  extended distances play an important role,
so that it will be useful to deal with an auxiliary topology,
see e.g.~\cite{Ambrosio-Gigli-Savare14,Ambrosio-Erbar-Savare15}.
 
From now on we suppose that $X_1=X_2=X$, 
we choose the logarithmic entropies 
\begin{equation}
  \label{eq:80}
  \begin{gathered}
    F_i(\r)=\PE_1(\r):=\r\log \r-\r+ 1,
  \end{gathered}
\end{equation}
and a cost $\sfc$ depending on the distance $\sfd$
through the function $\ell:[0,\infty]\to [0,\infty]$ via
\begin{equation}
  \label{eq:80bis}
  \begin{gathered}
    \sfc(x_1,x_2):=\ell\big(\sfd(x_1,x_2)\big),
    \qquad
      \ell(d):=
          \left\{\begin{aligned}
            & \log(1+\tan^2(d))
      &&\text{if }d\in [0,\pi/2),\\
      &+\infty&&\text{if }d\ge \pi/2,
    \end{aligned}
    \right.
  \end{gathered}
\end{equation}
so that
\begin{equation}
  \label{eq:39}
  \sfc(x_1,x_2)=
  \begin{cases}
    -\log\big(\cos^2(\sfd(x_1,x_2))\big)&\text{if }\sfd(x_1,x_2)< \pi/2\\
    +\infty&\text{otherwise.}
  \end{cases}
\end{equation}
Let us collect a few key properties that will be relevant in the sequel.
\begin{enumerate}[\rm LE.1]
\item $F_i$ are superlinear, regular, strictly convex, with 
  $\dom{F_i}=[0,\infty)$, \ $F_i(0)=1$, and 
  $\derzero {(F_i)}=-\infty$. \WWW For  $\r>0$ we have \EEE
  $\partial F_i(\r)=\{\log \r\}$.  
\item $\FH_i(\s)=\s F_i(1/\s)=\s-1-\log \s$, \ $\FH_i(0)=+\infty$, \
  $\rec{(\FH_i)}=1$. 
\item $\Fstar_i(\phi)=\exp(\phi)-1$,
  $\Gstar_i(\varphi)=1-\exp(-\varphi)$, \ $\dom{\Fstar_i}=\dom{\Gstar_i}=\R$.
\item $\FHstar_i(\psi)=-\log(1-\psi)$ for $\psi<1$ and 
  $\FHstar_i(\psi)=+\infty$ for $\psi\ge1$.
\item The function $\ell$ can be characterized as the unique solution
  of the differential equation
  \begin{equation}
    \label{eq:352}
    \ell''(d)=2\exp(\ell(d)),\quad \ell(0)=\ell'(0)=0,\quad 
  \end{equation}
  since it satisfies
  \begin{equation}
    \label{eq:158}
    \ell(d)=-\log\big({\cos^2(d)}\big)=
    2\int_0^d \tan(s)\,\d s,
    \quad
    d\in [0,\pi/2),
  \end{equation}
  so that
  \begin{equation}
    \label{eq:159}
    \ell(d)\ge d^2,\quad \ell'(d)=2\tan d\ge 2d,\quad \ell''(d)=2(1+\tan^2(d))
    = 2\exp(\ell(d))\ge 2.
  \end{equation}
  In particular $\ell$ is
  strictly increasing and uniformly $2$-convex.
  It is not difficult to check that $\sqrt \ell$ is also convex: this
  property is equivalent to $2\ell\ell''\ge (\ell')^2$ and a direct
  calculation shows
  \begin{align*}
    2\ell\ell''-(\ell')^2=4\log(1+\tan^2(d))(1+\tan^2(d))-4\tan^2(d)\ge0
  \end{align*}
  since $(1+r)\log (1+r)\ge r$.
\item $\MPc{\rp_1}{\rp_2}c=\rp_1+\rp_2-2\sqrt
  {\rp_1\rp_2}\exp(-c/2)$ for $c<\infty$, so that
  \begin{equation}
    \label{eq:350}
    \MPH{x_1}{\rp_1}{x_2}{\rp_2}=\rp_1+\rp_2-2\sqrt
    {\rp_1\rp_2}\cos\big(\sfdpt(x_1,x_2)\big),
  \end{equation}
  where we set
  \begin{equation}
    \label{eq:202}
    \sfd_a(x_1,x_2):=\sfd(x_1,x_2)\land a \quad \WWW 
    \text{ for } x_i\in X,\ a\ge 0.
  \end{equation}
  Since the function 
  \begin{equation}
  \MPH{x_1}{\rp_1^2}{x_2}{\rp_2^2}
  =\rp_1^2+\rp_2^2-2
  {\rp_1}{\rp_2}\cos(\sfdpt(x_1,x_2))\label{eq:387}
\end{equation}
will have an important geometric interpretation (see Section
\ref{subsec:cone}), in the following we will choose the exponent
$p=2$ \WWW in the setting of Section \ref{subsec:hom-marg}. 
\end{enumerate}
We keep the usual
notation $\xX=X\times X$, identifying $X_1$ and $X_2$ with $X$ 
and letting the index $i$ run between $1$ and $2$, 
e.g.~for $\ggamma\in \cM(\xX)$ the marginals
are denoted by $\gamma_i=(\pi^i)_\sharp\ggamma$.

\begin{problem}[The Logarithmic Entropy-Transport problem]
  \label{pr:LET} 
  Let $(X,\tau,\sfd)$ be an extended metric-topological space,
  $\ell$ and $\sfc$ be as in \eqref{eq:80bis}.
  Given $\mu_i\in \cM(X)$ find
  $\ggamma\in 
  \cM(\xX)$ 
  minimizing 
  \begin{equation}
    \label{eq:157}
\begin{aligned}    \LET(\mu_1,\mu_2)=&
    \min_{\sggamma\in \cM(\sxX)} \bigg( \sum_i\int_X \!
    \big(\sigma_i\log\sigma_i-\sigma_i+1\big)
    \,\d\mu_i+
    \int_\sxX \! \ell\big(\sfd(x_1,x_2)\big)\,\d\ggamma \bigg),\\ 
& \text{where } 
    \sigma_i=\frac{\d\gamma_i}{\d\mu_i}.
\end{aligned}
  \end{equation}
  We denote by $\OptLET(\mu_1,\mu_2)$ the set of all the
  minimizers \WWW $\ggamma$ in \eqref{eq:157}.
\end{problem}

\subsection{The Logarithmic Entropy-Transport problem: main results}
\label{subsec:LET2}
In the next theorem  we collect the main properties 
of the Logarithmic Entropy-Transport (LET) problem \WWW relying on the
reverse function $\RR$ from Section \ref{subsec:reverse}, cf.\
\eqref{eq:196}, and $\HH$
from Section \ref{subsec:HMPf}, cf.\ \eqref{eq:201}.  
\begin{theorem}[Direct formulation of the LET problem]
  \label{thm:mainHK1}
  Let $\mu_i\in \cM(X)$ be given and let $\ell,\sfdpt$
  be defined as in \eqref{eq:80bis} and \eqref{eq:202}.\\
  \noindent
  {\bfseries a) Existence of optimal plans.}
  There exists an optimal plan $\ggamma\in 
  \OptLET(\mu_1,\mu_2)$ solving Problem 
  \ref{pr:LET}.
  The set $\OptLET(\mu_1,\mu_2)$
  is convex and compact in $\cM(\xX)$, 
  $\LET$ is a convex and positively $1$-homogeneous functional
  (see \eqref{eq:136}) satisfying $0\le \LET(\mu_1,\mu_2)\le \sum_i\mu_i(X)$.
  \\
  \noindent
  {\bfseries b) Reverse formulation \WWW $(\LET=\RR_\sLE)$.}
  The functional $\LET$ has the equivalent reverse formulation as 
  \begin{align}
    \label{eq:353reverse}
         & \LET(\mu_1,\mu_2)=\WWW \min \Big\{
          \RR_\sLE(\mu_1,\mu_2|\ggamma) \: : \:
   {\ggamma\in \cM(\xX)},\ \mu_i=\varrho_i\gamma_i+\mu_i^\perp\Big\},
   \text{ where} \\
   & \nonumber \WWW 
   \RR_\sLE(\mu_1,\mu_2|\ggamma):= \sum_i \Big(\mu_i^\perp(X)+\int_X
      \big(\varrho_i-1-\log\varrho_i\big) \,\d\gamma_i\Big)+ \int_\sxX
      \ell\big(\sfd(x_1,x_2)\big)\,\d\ggamma, 
  \end{align}
  and $\bar\ggamma$ is an optimal plan in $\OptLET(\mu_1,\mu_2)$ if and
  only if it minimizes \eqref{eq:353reverse}.
    \\
  \noindent
  {\bfseries c) The homogeneous perspective formulation \WWW $(\LET=\HH_\sLE)$.}
  The functional $\LET(\mu_1,\mu_2)$ can be equivalently characterized as 
  \begin{align}
    \label{eq:354}
      &\LET(\mu_1,\mu_2)=\WWW \min \Big\{\HH_\sLE(\mu_1,\mu_2|\ggamma) \: :
      \: \ggamma  \in \cM(\xX) \Big\}, \text{ where}  \\ 
&  \WWW \nonumber
\HH_\sLE(\mu_1,\mu_2|\ggamma):= \EEE
\sum_i\mu_i(X)-
      2\max_{\sggamma\in \cM(\sxX)}
      \int_{\sxX} \sqrt{\varrho_1(x_1)\varrho_2(x_2)}
      \cos(\sfdpt(x_1,x_2))\,\d\ggamma  \\ 
& \qquad\qquad =\sum_i\mu_i^\perp(X)+ \!\! \nonumber
      \int_{\sxX}\!\!\!\! \big(\varrho_1(x_1){+}\varrho_2(x_2){-}
      2 \sqrt{\varrho_1(x_1)\varrho_2(x_2)}
       \cos(\sfdpt(x_1,x_2))\big)\,\d\ggamma
 \end{align}
and $\gamma_i=\varrho_i\mu_i+\mu_i^\perp$. Moreover, every plan
 $\bar\ggamma\in \OptLET(\mu_1,\mu_2)$ provides a solution to \eqref{eq:354}.
\end{theorem}
\begin{proof}
  The variational problem \eqref{eq:157} fits in the class 
  considered by Problem \ref{pr:1}, in the basic coercive setting 
  of Section \ref{subsec:setting} since the logarithmic entropy
  \eqref{eq:80} is superlinear with domain $[0,\infty)$.
  The problem is always feasible since $\PE_1(0)=1$ so that
  \eqref{eq:261} holds.

  a) follows by Theorem \ref{thm:easy-but-important}(i);
  the upper bound of $\LET$ is a particular case of
  \eqref{eq:137pre},
  and its convexity and $1$-homogeneity follows by
  Corollary \ref{cor:subadditivity}.

  b) is a consequence of Theorem
  \ref{thm:reverse-characterization}.

  c) is an application of Theorem \ref{thm:crucial} and \eqref{eq:350}.  
\end{proof}

We consider now the dual representation of $\LET$; recall that
$\LSC_s(X)$ denotes the space of simple (i.e.~taking a finite number
of values) lower semicontinuous functions and for a couple
$\phi_i:X\to \R$ the symbol $\phi_1\oplus \phi_2$ denotes the function
$(x_1,x_2)\mapsto \phi_1(x_1)+\phi_2(x_2)$ defined in $\xX$. \WWW In
part a) we relate to Section \ref{subsec:DualOpt}, whereas b)--d)
discusses the optimality conditions from Section
\ref{subsec:ExiOpt}.

\begin{theorem}[Dual formulation and optimality conditions]
  \label{thm:mainHK2}
  \ \\
  \noindent
  {\bfseries a) The dual problem \WWW $(\LET=\sfD_\sLE=\sfD'_\sLE)$.}
  For all $\mu_1,\mu_2 \in \cM(X)$ we have
  \begin{align}
    \label{eq:203}
    \LET(\mu_1,\mu_2)&=
                       \sup\Big\{ \WWW \DD_\sLE(\vvarphi|\mu_1,\mu_2)
                       \: : \ 
                       \varphi_i\in \LSC_s(X),\ 
                       \varphi_1\oplus \varphi_2\le
                       \ell(\sfd)\Big\},
    \\\notag
                     &=
                       \sup\Big\{ \sum_i \int_X \psi_i
                       \,\d\mu_i\ : \
                       \psi_i\in \LSC_s(X),\  \sup_X \psi_i<1,\\
    &\label{eq:204}
      \qquad\qquad
                       (1-\psi_1(x_1))(1-\psi_2(x_2))\ge 
                       \cos^2(\sfdpt(x_1,x_2))\text{ in }\xX\Big\},
  \end{align}
  \WWW where $\DD_\sLE(\vvarphi|\mu_1,\mu):=\sum_i \int_X
                       \big(1-\rme^{-\varphi_i}\big)
                       \,\d\mu_i$. 
  The same identities hold if the space $\LSC_s(X)$
  is replaced by $\LSC_b(X)$ 
  or $\rmB_b(X)$ in 
  \eqref{eq:203} and \eqref{eq:204}.
  When the topology $\tau$ is completely regular (in particular when
  $\sfd$ is a distance and $\tau$ is induced by $\sfd$)
  the space $\LSC_s(X)$ can be replaced by $\rmC_b(X)$ as well.\\
  \noindent
  {\bfseries b) Optimality conditions.}
  Let us assume that $\sfd$ is continuous. 
  A plan $\ggamma\in \cM(\xX)$ is optimal if and only if 
  its marginals $\gamma_i$ are absolutely continuous w.r.t.~$\mu_i$, 
  $\int_{\sxX} \ell(\sfd)\,\d\ggamma<\infty$,
  \begin{equation}
    \label{eq:345}
    \sfd\ge \pi/2\quad 
    \text{in}\quad
    \Big(\big(\supp\mu_1\setminus\supp\gamma_1\big)\times \supp\mu_2\Big)
    \bigcup
    \Big(\supp\mu_1\times\big(\supp\mu_2\setminus\supp\gamma_2\big)\Big),
  \end{equation}
  and there exist Borel sets $A_i\subset \supp\gamma_i$ 
  with $\gamma_i(X\setminus A_i)=0$ and Borel densities
  $\sigma_i:A_i\to(0,\infty)$ of $\gamma_i$ w.r.t.~$\mu_i$ such that
  \begin{align}
    \label{eq:355}
    \sigma_1(x_1)\sigma_2(x_2)&\ge
    \cos^2(\sfdpt(x_1,x_2))\quad\text{in }A_1\times A_2,\\
    \label{eq:356}
    \sigma_1(x_1)\sigma_2(x_2)&=
    \cos^2(\sfdpt(x_1,x_2))\quad\text{$\ggamma$-a.e.~in $A_1\times A_2$}.
  \end{align}
  {\bfseries c) $\ell(\sfd)$-cyclical monotonicity.}
  Every optimal plan $\ggamma\in \OptLET(\mu_1,\mu_2)$ is a solution
  of the optimal transport problem with cost $\ell(\sfd)$ between its
  marginals $\gamma_i$. In particular it is $\ell(\sfd)$-cyclically monotone, i.e.~it
  is concentrated on a Borel set $G\subset \xX$ ($G=\supp(\ggamma)$
  when $\sfd$ is continuous)
  such that
  for every choice of $(x_1^n,x_2^n)_{n=1}^N\subset G$ and
  every permutation $\kappa:\{1,\ldots,N\}\to \{1,\ldots,N\}$
  \begin{equation}
    \label{eq:5}
    \Pi_{n=1}^N\cos^2(\sfdpt(x_1^n,x_2^n))\ge 
    \Pi_{n=1}^N\cos^2(\sfdpt(x_1^n,x_2^{\kappa(n)})).
  \end{equation}
  \nc
  \noindent
  {\bfseries d) Generalized potentials.} 
  If $\ggamma$ is optimal and $A_i$, $\sigma_i$ are defined as in b)
  above, the Borel potentials $\varphi_i,\psi_i:X\to \bar\R$
  \begin{equation}
    \label{eq:357}
    \varphi_i:=
    \begin{cases}
      -\log\sigma_i&\text{in }A_i,\\
      -\infty&\text{in
      }X\setminus \supp\mu_i,\\
      +\infty&\text{otherwise,}
    \end{cases},\qquad
    \psi_i:=
    \begin{cases}
      1-\sigma_i&\text{in }A_i,\\
      -\infty&\text{in
      }X\setminus \supp\mu_i,\\
      1&\text{otherwise,}
    \end{cases}
  \end{equation}
  satisfy $\varphi_1\opz\varphi_2\le \ell(\sfd)$ 
  and the optimality conditions \eqref{eq:OptiConds}
  (with the analogous properties for $\psi_i$).
  Moreover $\rme^{-\varphi_i},\psi_i\in \rmL^1(X,\mu_i)$ and
  \begin{equation}
    \label{eq:359}
    \LET(\mu_1,\mu_2)=
    \sum_i\int_{X}\big(1-\rme^{-\varphi_i}\big)\,\d\mu_i=
    \sum_i\int_X \psi_i\,\d\mu_i=
    \sum_i \mu_i(X)-2\ggamma(\xX).
  \end{equation}
\end{theorem}
\begin{proof} Identity 
  \eqref{eq:203} follows by Theorem \ref{thm:weak-duality},
  recalling the definition \eqref{eq:301} of $\Cphi{}$ and
  the fact that $\Gstar_i(\varphi)=1-\exp(-\varphi)$.

  Identity \eqref{eq:204} follows from Proposition \ref{prop:equivalent-dual}
  and the fact that $\FHstar_i(\psi)=-\log(1-\psi)$. Notice that
  the definition \eqref{eq:46} of $\Cpsi{}$ ensures that 
  we can restrict the supremum in \eqref{eq:204} to functions $\psi_i$ 
  with $\sup_X \psi_i<1$. We have discussed the possibility to replace
  $\LSC_s(X)$ with $\LSC_b(X)$, $\rmB_b(X)$ or $\rmC_b(X)$ in 
  Corollary \ref{cor:LSC=C}.

  The statement of point b) follows by Corollary \ref{cor:spread};
  notice that a plan with finite energy
  $\EE(\ggamma|\mu_1,\mu_2)<\infty$ always satisfies
  $\int_{\sxX}\ell(\sfd)<\infty$. Conversely, if the latter
  integrability property holds, \eqref{eq:356} and the fact that
  $\int_{A_i} (\log\sigma_i)_-\,\d\gamma_i= \int_{A_i}
  \sigma_i(\log\sigma_i)_-\,\d\mu_i<\infty$ yields
  $\EE(\ggamma|\mu_1,\mu_2)<\infty$.
  
  Point c) is an obvious consequence of the optimality of $\ggamma$.

  Point d) can be easily deduced by b) or by applying Theorem
  \ref{thm:pot-ex}.
\end{proof}
\WWW In the one-dimensional case, the $\ell(\sfd)$-cyclic monotonicity
of part c) of the previous theorem reduces to classical monotonicity.  \EEE
\begin{corollary}[Monotonicity of optimal plans in $\R$]
  \label{cor:monotone-in-R}
  \upshape
  When $X=\R$ with the usual distance, 
  the support of every optimal plan $\ggamma$ is a monotone set, i.e.
  \begin{equation}
    \label{eq:6}
    (x_1,x_2),\ (x_1',x_2')\in \supp(\ggamma),\
    x_1<x_1'\quad\Rightarrow\quad
    x_2\le x_2'.
  \end{equation}  
\end{corollary}
\begin{proof}
  As the function $\ell$ is uniformly convex, \eqref{eq:5} is
  equivalent to monotonicity. 
\end{proof}

\WWW The next result provides a variant of the reverse formulation in
Theorem \ref{thm:mainHK1}. \EEE

\begin{corollary}\label{cor:HK-reverse}
  For all $\mu_1,\mu_2 \in \cM(X)$ we have
      \begin{align}
        \label{eq:362}
        \LET(\mu_1,\mu_2) =
                             \sum_{i}\mu_i(X)-
       2\max \Big\{&\ggamma(\xX):\ \ggamma\in \cM(\xX),\
       \gamma_i=\sigma_i\mu_i,
       \\
       &\notag 
                    \sigma_1(x_1)\sigma_2(x_2)\le 
    \cos^2(\sfdpt(x_1,x_2)) \text{ $\ggamma$-a.e.~in
                    $\xX$}\Big\}.
      \end{align}  
\end{corollary}
\begin{proof}
  Let us denote by $M'$ the right-hand side and let $\ggamma\in
  \cM(\xX)$ be a plan satisfying the conditions of \eqref{eq:362}.  If
  $A_i$ are Borel sets with $\gamma_i(X\setminus A_i)=0$ and
  $\sigma_i:X\to (0,\infty)$ are Borel densities of $\gamma_i$
  w.r.t.~$\mu_i$, we have $\varrho_i(x_i)=1/\sigma_i(x_i)$ in $A_i$ so
  that $\sigma_1(x_1)\sigma_2(x_2)\le\cos^2(\sfdpt(x_1,x_2))$ yields
  $\varrho_1(x_1)\varrho_2(x_2)\cos^2(\sfdpt(x_1,x_2))\ge 1$. Since
  $(\log\varrho_i)_+\in \rmL^1(X,\gamma_i)$ we have
  \begin{align*}
    &\sum_i \Big(\mu_i^\perp(X)+\int_X
    \big(\varrho_i-1-\log\varrho_i\big) \,\d\gamma_i\Big)+ \int_\sxX
    \ell\big(\sfd(x_1,x_2)\big)\,\d\ggamma\\
    =&
       \sum_i\big(\mu_i(X)-\gamma_i(X)\big)-
       \int_{\sxX}
       \log\big(\varrho_1(x_1)\varrho_2(x_2)\cos^2(\sfdpt(x_1,x_2))\big)\,\d\ggamma
       \le \sum_i\mu_i(X)-2\ggamma(\xX).
  \end{align*}
  By \eqref{eq:353reverse} we get $M'\ge \LET(\mu_1,\mu_2).$
  On the other hand, choosing any $\bar\ggamma\in 
  \OptLET(\mu_1,\mu_2)$ \WWW the optimality condition 
  \eqref{eq:356} shows that $\bar\ggamma$ is an admissible competitor
  for \eqref{eq:362} and \eqref{eq:359} shows that $M'=\LET(\mu_1,\mu_2)$.
\end{proof}

\WWW The nonnegative and concave functional
$ \WWW (\mu_1,\mu_2)\mapsto \sum_{i}\mu_i(X)-\LET(\mu_1,\mu_2)$ 
can be represented as \WWW in the following equivalent ways: 
\begin{align}
  \label{eq:260}
 \sum_{i}\mu_i(X)-\LET(\mu_1,\mu_2) &= 2\max_{\sggamma\in \cM(\sxX)}
    \int_{\sxX} \sqrt{\varrho_1(x_1)\varrho_2(x_2)}
    \cos(\sfdpt(x_1,x_2))\,\d\ggamma
    \\
  \label{eq:360}
  &=
     \inf\Big\{ \sum_i \int_X
                       \rme^{-\varphi_i}
                       \,\d\mu_i:
                       \varphi_i\in \LSC_s(X),\ 
                       \varphi_1\oplus \varphi_2\le
                       \ell(\sfd)\Big\}
       \\
    \notag
    &=
          \inf\Big\{ \sum_i \int_X \tilde\psi_i
                       \,\d\mu_i:
       \tilde\psi_i\in \USC_s(X),\ \inf_X\tilde\psi_i>0,
    \\\label{eq:361}
    &\qquad\qquad
          \psi_1(x_1)\psi_2(x_2)\ge 
                       \cos^2(\sfdpt(x_1,x_2))\text{ in }\xX\Big\}\\
    \notag
    &=2\max \Big\{\ggamma(\xX):\ \ggamma\in \cM(\xX),\
       \gamma_i=\sigma_i\mu_i,
       \\
        \label{eq:362bis}
    &\qquad\qquad
                    \sigma_1(x_1)\sigma_2(x_2)\le 
    \cos^2(\sfdpt(x_1,x_2)) \text{ $\ggamma$-a.e.~in
                    $\xX$}\Big\}.
\end{align}

The next result concerns uniqueness of the optimal plan $\ggamma$ in
the Euclidean case $X=\R^d$. We will use the notion of approximate
differential (denoted by $\tilde\rmD$), see
e.g.~\cite[Def.~5.5.1]{Ambrosio-Gigli-Savare08}.

\begin{theorem}[Uniqueness]Let $\mu_i\in \cM(X)$ and $\ggamma\in
      \OptLET(\mu_1,\mu_2)$.
      \begin{enumerate}[(i)]
      \item The marginals $\gamma_i=\pi^i_\sharp \ggamma$ are
        uniquely determined.
      \item If $X=\R$ with the usual distance then $\ggamma$ is the
        unique element of 
        $\OptLET(\mu_1,\mu_2)$.
      \item If $X=\R^d$ with the usual distance,
        $\mu_1\ll\Leb d$ is absolutely continuous,
        and $A_i\subset \R^d$ and $\sigma_i:A_i\to (0,\infty)$
        are as in Theorem \ref{thm:mainHK2} b),
        then $\sigma_1$ is approximately differentiable
        at $\ggamma_1$-a.e.~point of $A_1$ and $\ggamma$ is the 
        unique element of 
        $\OptLET(\mu_1,\mu_2)$;
        it is concentrated on the graph of a function
        $\tt:\R^d\to\R^d$ satisfying
        \begin{equation}
          \label{eq:7}
          \tt(x_1)=x_1+\frac{\arctan(|\xxi(x_1)|)}{|\xxi(x_1)|} \xxi(x_1),
          \quad \xxi(x_1)=-\frac 12\tilde\rmD\log\sigma_1(x_1)
          \ \text{$\gamma_1$-a.e.~in $A_1$.}
        \end{equation}
      \end{enumerate}
    \end{theorem}
    \begin{proof}
      (i) follows directly from Lemma \ref{le:uniqueness}.

      (ii) follows by Theorem \ref{thm:mainHK2}(c), since 
      whenever the marginals $\gamma_i$ are fixed there is
      only one plan with monotone support in $\R$.

      In order to prove (iii) we adapt the argument of 
      \cite[Thm.~6.2.4]{Ambrosio-Gigli-Savare08}
      to our singular setting, where the cost $\sfc$ can take
      the value $+\infty$.

      Let $A_i\subset \R^d$ and $\sigma_i:A_i\to (0,\infty)$ as in
      Theorem \ref{thm:mainHK2} b). Since $\mu_1=u\Leb d\ll \Leb d$
      \WWW with density $u \in L^1(\R^d)$, up to removing a
      $\mu_1$-negligible set (and thus $\gamma_1$-negligible) from
      $A_1$, it is not restrictive to assume that $u(x_1)>0$
      everywhere in $A_1$, so that the classes of $\Leb d$- and
      $\gamma_1$-negligible subsets of $A_1$ coincide.  For every
      $n\in \N$ we define
      \begin{equation}
        \label{eq:8}
        A_{2,n}:=\{x_2\in A_2: \sigma_2(x_2)\ge 1/n\},\quad
        s_n(x_1):=\sup_{x_2\in A_{2,n}} \cos^2(|x_1-x_2|)/\sigma_2(x_2).
      \end{equation}
      The functions $s_n$ are 
      bounded and Lipschitz
      in $\R^d$ and therefore differentiable
      $\Leb d$-a.e.~by Rademacher's Theorem. Since $\gamma_1\ll \mu_1$
      and
      $\mu_1$ is absolutely continuous w.r.t.~$\Leb d$ we deduce
      that $s_n$ are differentiable $\gamma_1$-a.e.~in $A_1$.

      By \eqref{eq:355} we have $\sigma_1(x_1)\ge s_n(x_1)$ in $A_1$.
      \WWW By \eqref{eq:356} we know that for $\gamma_1$-a.e.~$x_1\in A_1$
      there exists $x_2\in A_2$ such that $|x_1-x_2|<\pi/2$ and
      $\sigma_1(x_1)=\cos^2(|x_1-x_2|)/\sigma_2(x_2)$ so that
      $\sigma_1(x_1)=s_n(x_1)$ for $n$ sufficiently big \WWW and hence
      the  family $(B_n)_{n\in \N}$ of
      sets $B_n:=\{x_1\in A_1:\sigma_1(x_1)>s_n(x_1)\}$ is decreasing
      (since $s_n$ is increasing and dominated by $\sigma_1$) and has
      $\Leb d$-negligible intersection.

      It follows that $\gamma_1$-a.e.~$x_1\in A_1$ is a point of $\Leb
      d$-density $1$ of $\{x_1\in A_1:\sigma_1(x_1)=s_n(x_1)\}$ for
      some $n\in \N$ and $s_n$ is differentiable at $x_1$. Let us
      denote by $A_1'$ \WWW the set of all $x_1\in A_1$ such that \EEE
      $\sigma_1$ is approximately 
      differentiable at every $x_1\in A_1'$ with approximate
      differential $\tilde\rmD\sigma_1(x_1)$ equal to $\rmD s_n(x_1)$
      for $n$ sufficiently big.
      
      Suppose now that $x_1\in A_1'$ and
      $\sigma_1(x_1)=\cos^2(|x_1-x_2|)/\sigma_2(x_2)$ for some $x_2\in
      A_2$. Since \WWW by \eqref{eq:355} and \eqref{eq:356} the
      map $x_1'\mapsto \cos^2(|x_1'-x_2|)/\sigma_1(x_1')$ attains its
      maximum at $x_1'=x_1$, we deduce that
      \begin{displaymath}
        \tan(|x_1-x_2|) \frac{x_1-x_2}{|x_1-x_2|}=-\frac
        12\tilde\rmD\log\sigma_1(x_1), 
      \end{displaymath}
      so that $x_2$ is uniquely determined, and  \eqref{eq:7} follows.
\end{proof}

We conclude this section with the last representation formula for
$\LET(\mu_1,\mu_2)$ given in terms of \WWW transport plans $\aalpha$
\EEE in $\yY:=Y\times Y$ with
$Y:=X\times[0,\infty)$ with constraints on the homogeneous marginals,
keeping the notation of Section \ref{subsec:hom-marg}.  Even if it
seems the most complicated one, it will provide the natural point of
view in order to study the metric properties of the $\LET$ functional.

\begin{theorem}
      \label{thm:mainHK3}
      For every $\mu_i\in \cM(X)$ we have
      \begin{align}
      \label{eq:198}  
      &\LET(\mu_1,\mu_2)
       =
      \sum_i\mu_i(X)- 2\max_{\saalpha\in \HMle2{\mu_1}{\mu_2}}
      \int_{\sxX} \s_1\s_2
      \cos(\sfdpt(x_1,x_2))\,\d\aalpha\\
      \label{eq:84} &
      =
      \min\Big\{ 
      \int_{\syY}
      \Big(\s_1^2+\s_2^2-2\s_1\s_2\cos(\sfdpt(x_1,x_2))\Big)\,\d\aalpha+ 
      \sum_i (\mu_i-\hm 2i\aalpha)(X):
        \\
        &\hspace*{20em} {\aalpha\in \cM(\yY),\ \hm 2i\aalpha \le \mu_i
          }  \Big\} \notag
        \\
        \label{eq:84bis}
        &=
      \min\Big\{ 
      \int_{\syY}
      \Big(\s_1^2+\s_2^2-2\s_1\s_2\cos(\sfdpt(x_1,x_2))\Big)\,\d\aalpha: 
        {\aalpha\in \cM(\yY),\ \hm 2i\aalpha = \mu_i
          }
          \Big\}
  \end{align}
  Moreover, for every plan $\bar\ggamma\in \OptLET{\mu_1}{\mu_2}$ and
  every
  couple of Borel densities $\varrho_i$ as in \eqref{eq:353reverse} 
  the plan
  $\bar\aalpha:=(x_1,\sqrt{\varrho_1(x_1)};x_2,\sqrt{\varrho_2(x_2)})_\sharp
  \bar\ggamma$ is optimal for \eqref{eq:84} and \eqref{eq:198}.
\end{theorem}
\begin{proof}
      Identity \eqref{eq:84} (resp.~\eqref{eq:84bis}) 
      follows directly by \eqref{eq:343} 
      (resp.~\eqref{eq:343bis}) of Theorem
      \ref{thm:main-hom-marg}. Relation \eqref{eq:198} 
      is just a different form for \eqref{eq:84}.
\end{proof}

\section{The metric side of the LET-functional:\newline the
      Hellinger-Kan\-to\-ro\-vich distance}
\label{sec:cone}
%
In this section we want to show that the functional 
\begin{equation}
  \label{eq:353}
 \WWW (\mu_1,\mu_2) \mapsto  \sqrt{\LET(\mu_1,\mu_2)} \EEE
\end{equation}
defines a distance in $\cM(X)$, \WWW which is then called the
Hellinger-Kantorovich distance and denoted $\HK$.  This distance \EEE
property is strongly related to the \WWW property that the function
$(x_1,\s_1;x_2,\s_2)\mapsto 
\big(\MPH{x_1}{\s_1^2}{x_2}{\s_2^2}\big)^{1/2}$ is a (possibly
extended) semidistance in $Y=X\times [0,\infty)$.

In the next section we will briefly study this function and the
induced metric space, the so-called \emph{cone $\tY$ on $X$},
\cite[Sec.\,3.6]{Burago-Burago-Ivanov01} 
obtained by taking the quotient 
w.r.t.~the equivalent classes of points with distance $0$.

\subsection{The cone construction}
\label{subsec:cone}

In the extended metric-topological space 
$(X,\tau,\sfd)$ of Section \ref{subsec:LET1},
we will denote by $\sfd_a:=\sfd\land a$ the truncated distance 
and
by 
$\py=(x,\rp)$, $x\in X,\ \rp\in [0,\infty)$, 
the generic points of $\pY:=X\times[0,\infty)$.

It is not difficult to show that the function $\sfdc:Y\times Y\to [0,\infty)$ 
\begin{equation}
  \label{eq:90}
  \sfdc^2((x_1,\rp_1),(x_2,\rp_2)):=\rp_1^2+\rp_2^2-
  2{\rp_1\rp_2}\cos(\sfdp(x_1,x_2))
\end{equation}
is nonnegative, symmetric, and satisfies the
triangle inequality
(see e.g.~\cite[Prop.~3.6.13]{Burago-Burago-Ivanov01}).
We also notice that 
\begin{gather}
  \label{eq:94pre}
  \sfdc^2(\py_1,\py_2)=
                        |{\rp_1}-{\rp_2}|^2
    +4{\rp_1\rp_2}\,\sin^2\big(\sfdp(x_1,x_2)/2\big),
\intertext{\WWW which implies the useful estimates}
  \label{eq:94}
                        \max\Big(
                        |{\rp_1}-{\rp_2}|,
    \frac {2}\pi \sqrt{\rp_1\rp_2}\,
    \sfdp(x_1,x_2)\Big)\le \sfdc(\py_1,\py_2)\le 
    |{\rp_1}-{\rp_2}|+\sqrt{\rp_1\rp_2}\,\sfdp(x_1,x_2).
\end{gather}
\WWW From this it follows that 
$\sfdc$ induces a true distance in the quotient space
$\tY=\pY/\sim$ where
\begin{equation}
  \label{eq:95}
  \py_1\sim \py_2\quad\Leftrightarrow\quad
  \rp_1=\rp_2=0\quad\text{or}\quad\rp_1=\rp_2,\ x_1=x_2.
\end{equation}
Equivalence classes are usually denoted by $\ty=[y]=[x,\s]$, where the
vertex $[x,0]$ plays a distinguished role. It is denoted by $\fro$,
its complement is the open set $\tY_\fro=\tY\setminus\{\fro\}.$ On
$\tY$ we introduce a topology $\tau_\frC$, which is in general weaker
than the canonical quotient topology: $\tau_\frC$ neighborhoods of
points in $\tY_\soo$ 
coincide with neighborhoods in $\pY$, whereas the sets
\begin{equation}
\{[x,\s]:0\le \s<\eps\}=\{\ty\in\tY:\sfdc(\ty,\fro)<\eps\},\quad
\eps>0,\label{eq:463}
\end{equation}
provide a system of open neighborhoods of $\fro$. 
$\tau_\frC$ coincides with the quotient topology when $X$ is compact.

It is easy to check that $(\tY,\tau_\frC)$ is a Hausdorff topological
space and $\sfdc$ is $\tau_\frC$-lower semicontinuous. If $\tau$ is
induced by $\sfd$ then $\tau_\frC$ is induced by $\sfdc$.  If
$(X,\sfd)$ is complete (resp.~separable), then $(\tY, \sfdc)$ is also
complete (resp.~separable).

Perhaps the simplest example is provided by the unit sphere 
$X=\S^{d-1}=\{x\in \R^d:|x|=1\}$ in $\R^d$ endowed with the intrinsic Riemannian
distance:
the corresponding cone $\tY$ is precisely $\R^d$.

\renewcommand{\sfr}{\sfs}
We \WWW denote the canonical projection  by 
\begin{equation}
  \label{eq:proj}
\frp:\pY\to\tY, \quad  \frp(x,\s)=[x,\s].
\end{equation}
Clearly $\frp$ is continuous and is an homeomorphism between
$\pY\setminus (X\times\{0\})$ and $\tY_\soo$. 
\renewcommand{\sfs}{\mathsf r}
A right inverse $\frqext:\tY\to \pY$ of the map $\frp$ can be obtained by  fixing a point $\xext\in X$ 
and defining
\begin{equation}
\sfs:\tY\to[0,\infty), \  \sfs[x,\s ]=\s,\quad 
\sfx:\tY\to \Xext,\ \sfx[x,\s ]=
\begin{cases}
  x&\text{if }\s>0,\\
  \xext&\text{if }\s=0,
\end{cases}
\ \text{ and } \ 
\frqext:=(\sfx,\sfr).
\label{eq:127}
\end{equation}
Notice that $\sfr$ is continuous and $\sfx$ is continuous \WWW
restricted to $\tY_\soo$.

A continuous rescaling product from $ \tY\times [0,\infty)$ to $\tY$
can be defined by
\begin{equation}
  \label{eq:151}
  \ty\cdot \lambda:=
  \begin{cases}
    \fro&\text{if }\ty=\fro,\\
    [x,\lambda \s]&\text{if }\ty=[x,\s ],\ s>0.
  \end{cases}
\end{equation}
We conclude this introductory section by a characterization of
compact sets in $(\tY,\tau_\frC)$.

\begin{lemma}[Compact sets in $\tY$]
  \label{le:conce-compactness}
  A closed set $K$ of $\tY$ is compact if and only if there is $r_0>0$
  such that its upper sections 
  $$K(\rho):=\{x\in X:[x,\s ]\in K\text{ for
    some }\s\ge \rho\}
  $$
  are empty for $\rho>r_0$
  and compact in $X$ for $0<\rho\le r_0$.
\end{lemma}
\begin{proof}
  It is easy to check that the condition is necessary.
  
  
  In order to show the sufficiency, let $\rho=\inf_K \sfr$. If
  $\rho>0$ then $K$ is compact since it is a closed subset of the
  compact set $\frp\big(K(\rho)\times [\rho,r_0]\big)$.

  If $\rho=0$ then $\fro$ is an accumulation point of $K$ by
  \eqref{eq:463} and therefore $\fro\in K$ since $K$ is closed.  If
  $\mathscr U$ is an open covering of $K$, we can pick $U_0\in
  \mathscr U$ such that $\fro\in U_{0}$. By \eqref{eq:463} there
  exists $\eps>0$ such that $K\setminus U_{0}\subset
  \frp\big(K(\eps)\times [\eps,r_0]\big)$: since
  $\frp\big(K(\eps)\times [\eps,r_0]\big)$ is compact, we can thus
  find a finite subcover $\{U_1,\cdots, U_N\}\subset \mathscr U$ of
  $K\setminus U_{0}$. $\{U_n\}_{n=0}^N$ is therefore a finite subcover
  of $K$.
\end{proof}

\begin{remark}[Two different truncations]
  \upshape
  \label{rem:twotruncations}
  Notice that in the constitutive formula defining $\sfdc$ we used the
  truncated distance $\sfdp$ with upper threshold $\pi$,
  whereas in Theorem \ref{thm:mainHK3} an analogous formula with 
  $\sfdpt$ and threshold $\pi/2$ played a crucial role. We could then
  consider the distance
  \begin{subequations}
    \begin{align}
          \label{eq:388}
      \asfdc {\kp}^2([x_1,\rp_1],[x_2,\rp_2]):={}&\rp_1^2+\rp_2^2-
      2{\rp_1\rp_2}\cos(\sfd_{{\kp}}(x_1,x_2))
      \\    \label{eq:388bis}
      ={}& |\rp_1-\rp_2|^2+
      4{\rp_1\rp_2}\sin^2(\sfd_{\kp}(x_1,x_2)/2)
    \end{align}
  \end{subequations}
   on $\tY$, which satisfies 
  \begin{equation}
    \label{eq:389}
    \asfdc {\kp}\le \sfdc\le \sqrt2\,\tsfdc
   .
  \end{equation}
  The notation \eqref{eq:388} is justified by the fact that $\asfdc
  {\kp}$ is still a cone distance associated to the metric space
  $(X,\sfd_{\kp})$, since obviously $(\sfd_{\kp})_\pi=(\sfdpt)\land
  {\kp}=\sfd_{\kp}$.  From the geometric point of view, the choice of
  $\sfdc$ is natural, since it preserves important metric properties
  concerning geodesics (see \cite[Thm.~3.6.17]{Burago-Burago-Ivanov01}
  and the next section \ref{subsec:geodesic}) and curvature (see
  \cite[Sect.~4.7]{Burago-Burago-Ivanov01} and the next section
  \ref{subsec:curvature}).
  
  On the other hand, the choice of $\sfdpt$ is crucial for its link
  with the function $\MP$ of \eqref{eq:387}, with Entropy-Transport
  problems, and with a representation property for the Hopf-Lax
  formula that we will see in the next sections.  Notice that the
  $1$-homogeneous formula \eqref{eq:350} would not be convex in
  $(\s_1,\s_2)$ if one uses $\sfdp$ instead of $\sfdpt$.
  Nevertheless, we will prove in Section \ref{subsec:HK} the
  remarkable fact that both $\sfdp$ and $\sfdpt$ will lead to the same
  distance between positive measures.
\end{remark}

\subsection{Radon measures in the cone
  $\tY$ and homogeneous marginals}
\label{subsec:RHM}

It is clear that any measure $\tnu\in \cM(\tY)$ can be lifted to a
measure $\pnu\in \cM(\Yext)$ such that $\frp_\sharp\pnu=\tnu$: it is
sufficient to take $\pnu=\frqext_\sharp\tnu$ where $\frqext$ is a
right inverse of $\frp$ defined as in \eqref{eq:127}.

We call $\cM_2(\tY)$ (resp.~$\cP_2(\tY)$)
the space of measures 
$\tnu\in \cM(\tY)$ (resp.\ $\tnu\in \cP(\tY)$)
such that
\begin{equation}
  \label{eq:145bis}
  \int_{\tY} \sfs^2\,\d\tnu=\int_{\tY}\sfdc^2(\ty,\fro)\,\d\tnu=
  \int_{Y} \s^2\,\d\pnu<\infty,\quad
  \pnu=\frqext_\sharp\tnu.
\end{equation}
Measures in $\cM_2(\tY)$ thus
correspond to images $\frp_\sharp\pnu$ of
measures $\pnu\in \cM_2(\Yext)$ 
and have finite second moment w.r.t.~the distance $\sfdc$,
which justifies the index $2$ in $\cM_2(\tY)$.
Notice moreover that the measure $s^2\pnu$ does not charge
$\Xext\times\{0\}$ 
and it is independent of the choice of the point $\xext$ in \eqref{eq:127}.

The above considerations can be easily extended to plans in 
the product spaces $\pdyY N$ (where typically $N=2$, but
also the general case will turn out to be useful later on).
To clarify the notation, we will denote 
by $\tyy=(\ty_i)_{i=1}^N=([x_i,\rp_i])_{i=1}^N$ a point in $\pdyY N$
and we will set
$\sfr_i(\tyy)=\sfr(\ty_i)=\s_i$, 
$\sfx_i(\tyy)=\sfx(\ty_i)\in \Xext$.
Projections on the $i$-coordinate from $\pdyY N$
to $\tY$ are usually denoted by $\pi^i$ or $\pi^{\ty_i}$,
$\frpd=\tens \frp N:\tens{(\Yext)} N\to\pdyY N$, 
$\frqd=\tens\sfy N:\pdyY N\to 
\tens{(\Yext)}N
$ 
are the Cartesian products of the projections and of the lifts.

Recall that the $\rmL^2$-Kantorovich-Wasserstein (extended) distance $\Wc$ 
in $\cM_2(\tY)$ induced by $\sfdc$ is defined by
\begin{equation}
  \label{eq:149}
  \Wc^2(\tnu_1,\tnu_2):=
  \min\Big\{\int\sfdc^2(\ty_1,\ty_2)\,\d\taalpha:
  \taalpha\in \cM(\tyY),\ 
  \pi^{\ty_i}_\sharp\taalpha=\tnu_i\Big\},
\end{equation}
with the convention that $\Wc(\tnu_1,\tnu_2)=+\infty$ if
$\tnu_1(\tY)\neq \tnu_2(\tY)$ and thus the minimum in \eqref{eq:149}
is taken on an empty set.  We want to mimic the above definition,
replacing the usual marginal conditions in \eqref{eq:149} with the
homogeneous marginals $\chm2{i}{}$ which we are going to define.

Let us consider now a plan $\taalpha$ in $\cM(\pdyY N)$ with $
\paalpha=\frqd_\sharp\aalpha\in \cM(\tens YN)$: we say that $\taalpha$
lies in $\cM_2(\pdyY N)$ if
\begin{equation}
  \label{eq:99}
  \int_{\pdyY N} \sum_i \sfs_i^2\,\d\taalpha=
  \int_{\tens YN} \sum_i \s_i^2\,\d\paalpha<\infty.
\end{equation}
Its ``canonical'' marginals 
in $\cM(\tY)$ are $\taalpha_i=\pi^{\ty_i}_\sharp\taalpha$,
whereas the ``homogeneous'' marginals 
correspond to \eqref{eq:146} \WWW with $p=2$: 
\begin{equation}
  \label{eq:155}
  \frh_i^2(\taalpha):=
  (\sfx_i)_\sharp (\sfr_i^2\taalpha)= 
  \pi^{x_i}_\sharp(\s_i^2\paalpha)=
  \hm 2i{(\paalpha)}
  \in \cM(X),\quad
  \paalpha:=\frqd_\sharp\taalpha.
\end{equation}
We will omit the index $i$ when $N=1$.  Notice that $\sfr_i^2\taalpha
$ does not charge $(\pi^i)^{-1}(\fro)$ 
(similarly, $\s_i^2 \paalpha$ 
does not charge $\tens{\Yext}{i-1}\times\{(\xext,0)\}\times
\tens{\Yext}{N-i}$) so that 
\eqref{eq:155} is independent of the choice of the point $\xext$ in
\eqref{eq:127}.  

As for \eqref{eq:150}, the homogeneous marginals on
the cone are invariant with respect to dilations: if
$\vartheta:\pdyY N\to(0,\infty)$ is a Borel map in $\rmL^2(\pdyY
N,\aalpha)$ we set
\begin{equation}
  \label{eq:147bis}
  \big(\prd_{\vartheta}(\tyy)\big)_i:=\ty_i\cdot
  \big(\vartheta(\tyy)\big)^{-1} \quad \text{and} \quad  
  \dil{\vartheta,2}{\aalpha}:={}
  (\prd_\vartheta)_\sharp(\vartheta^2 \,\taalpha),
\end{equation}
so that 
\begin{equation}
  \label{eq:150bis}
  \chm 2i(\dil{\vartheta,2}{\taalpha})=\chm 2i(\aalpha)\quad
  \forevery \taalpha\in \cM_2(\pdyY N).
\end{equation}
%
As for the canonical marginals, a uniform control of the homogeneous
marginals is sufficient to get equal tightness, \WWW cf.\
\eqref{eq:25} for the definition. We state this result for an
arbitrary number of components, \WWW and we emphasize that we are
not claiming any closedness of the involved sets.

\begin{lemma}[Homogeneous marginals and tightness]
  \label{le:compactnessH}
  Let $\cK_i$, $i=1,\cdots, N$, be a finite collection of bounded and
  equally tight sets in $\cM(X)$. Then, the set
  \begin{equation}
    \label{eq:380}
    \big\{\aalpha\in \cM_2(\tY^N):
    \chm2{i}\aalpha\in \cK_i \WWW \text{ for }i=1,\ldots, N \big\}
  \end{equation}
  is equally tight in $\cM(\tY^N)$.
\end{lemma}
\begin{proof}
  By applying \cite[Lem.\,5.2.2]{Ambrosio-Gigli-Savare08}, it is
  sufficient to consider the case $N=1$: given a bounded and equally
  tight set $\cK\subset \cM(X)$ we prove that $\cH:=\big\{\aalpha\in
  \cM_2(\tY): \chm2{}\aalpha\in \cK\big\}$ is equally tight.  For
  $A\subset X$, $R\subset (0,\infty)$ we will use the short notation
  $A\timesc R$ for $\frp(A\times R)\subset \tY$. If $A$ and $R$ are
  compact, then $A\timesc R$ is compact in $\tY$.
%

  Let $M:=\sup_{\mu\in \cK}\mu(X)<\infty$; since $\cK$ is tight, we
  can find an increasing sequence of compact sets $K_n\subset X$ such
  that $\mu(X\setminus K_n)\le 8^{-n}$ for every $\mu\in \cK$.  For an
  integer $m\in \N$ we then consider the compact sets $\frK_m\subset
  \tY$ defined by
\begin{equation}
  \label{eq:105}
  \frK_m=\{\fro\}\cup
  K_m\timesc [2^{-m},2^m]\cup \Big(\bigcup_{n=1}^\infty
  K_{n+m}\timesc [2^{-n},2^{-n+1}]\Big).
\end{equation}
\WWW Setting $K_\infty=\bigcup_{n=1}^\infty K_n$, we have  
$\mu(X\setminus K_\infty)=0$ and 
\begin{displaymath}
  \tY\setminus \frK_m\subset K_m\timesc (2^m,\infty)\cup
  \Big(\bigcup_{n=1}^\infty (K_{n+m}\setminus K_{n+m-1})
  \timesc (2^{-n+1},\infty)\Big)\cup
  (X\setminus K_\infty)\timesc (0,\infty).    
\end{displaymath}
Since for every $\aalpha\in \cH$ with $\chm2{}\aalpha=\mu$ and every
$A\in \BorelSets X$ we have 
\begin{displaymath}
  \aalpha(A\timesc (s,\infty))\le s^{-2}\mu(A)\le s^{-2}M \ \text{ and
  } \  
  \aalpha\big((X\setminus K_\infty)\timesc (0,\infty)\big)=0,    
\end{displaymath}
we conclude 
  \begin{align*}
    \aalpha(\tY\setminus \frK_m)&\le 
    M\,4^{-m} +\sum_{n=1}^\infty 
                                  \aalpha\big((X \setminus K_{n+m-1})\timesc
    (2^{-n+1},\infty)\big)
    \le \\&
    M\,4^{-m}+\sum_{n=1}^\infty 4^{n-1}8^{1-n-m}\le 
    4^{-m}\Big(M+\sum_{n=1}^\infty 4^{-n}\Big) \le 
            4^{-m}\big(1+M\big)\big),
  \end{align*}
  for every $\aalpha\in \cH$.  Since all $\frK_m$ are compact, we
  obtain the \WWW desired equal tightness.
\end{proof}

\subsection{The Hellinger-Kantorovich problem }
\label{subsec:HK}
In this section we will always consider $N=2$, keeping the shorter
notation $\yY=\tens Y2$ and $\tyY=\pdyY 2$.  As for \eqref{eq:161},
for every $\mu_1,\mu_2\in \cM_2(X)$ we define the sets
\begin{equation}
  \begin{aligned}
    &\cHMle 2{\mu_1}{\mu_2}:=\Big\{
    \taalpha\in \cM_2(\yY)\ : \ \chm2i\taalpha\le \mu_i\Big\} \text{
      and } \\ 
    &\cHM 2{\mu_1}{\mu_2}:=\Big\{ \taalpha\in \cM_2(\tyY) \ : \
    \chm2i\taalpha=\mu_i\Big\}.
  \end{aligned}
\label{eq:161bis}
\end{equation}
They are the images of $\HMle2{\mu_1}{\mu_2}$ and $\HM2{\mu_1}{\mu_2}$ 
through the projections $\frpd_\sharp$; in particular 
they always contain plans $\frpd_\sharp \aalpha$, where 
$\aalpha$ is given by \eqref{eq:103}. 
The condition $\taalpha\in \cHMle 2{\mu_1}{\mu_2}$ 
is equivalent to ask that 
\begin{equation}
  \label{eq:120}
  \int
  \sfr_i^2\varphi(\sfx_i)\,\d\taalpha\le \int\varphi\,\d\mu_i\quad
  \forevery \text{nonnegative }\varphi\in \rmB_b(X).
\end{equation}
We can thus define the following minimum problem:

\begin{problem}[The Hellinger-Kantorovich problem]
  \label{pr:3}
  Given $\mu_1,\mu_2\in \cM(X)$ find 
  an optimal plan $\taalpha_{\rm opt}\in 
  \HM2{\mu_1}{\mu_2}\subset \cM_2(\tyY)$ 
  solving the minimum problem 
  \begin{equation}
    \label{eq:104}
    \begin{aligned}
      \WWW \HK(\mu_1,\mu_2)^2:= &\min\Big\{
      \int\sfdc^2(\ty_1,\ty_2)\,\d\taalpha: \taalpha\in
      \cM_2(\tyY),\ 
      \chm2i\taalpha=\mu_i\Big\}.
    \end{aligned}
  \end{equation}
  We denote by $\OptHK(\mu_1,\mu_2)\subset \cM(\tyY)$ the collection of 
  all the optimal plans $\aalpha$ realizing the minimum 
  in \eqref{eq:104} and by $\HK^2(\mu_1,\mu_2)$ the value of the minimum in
  \eqref{eq:104} (whose existence is
  guaranteed by the next Theorem \ref{thm:existenceHK}).
\end{problem}
\begin{remark}[Lifting of plans in $\pY$]
  \label{rem:lifting}
  \upshape
  Since any plan $\taalpha\in \cM(\tyY)$ 
  can be lifted to a plan $\paalpha=\frqd_\sharp\taalpha
  \in \cP(\pY\times \pY)$ 
  such that $\frpd_\sharp\paalpha=\taalpha$ 
  the previous problem \ref{pr:3} is also equivalent to
  find 
  \begin{equation}
    \label{eq:162}
    \min\Big\{
      \int\sfdc^2(\py_1,\py_2)\,\d\paalpha: \paalpha\in
      \cM(\pY\times\pY),\quad
      \hm 2i\paalpha =\mu_i\Big\}.
  \end{equation}
  The advantage to work in the quotient space $\tY$ 
  is to gain compactness, as the next Theorem \ref{thm:existenceHK}
  will show.
\end{remark}
%
\WWW An importance feature of the cone distance and the homogeneous is
an invariance under \textbf{rescaling}, which can be done by the
dilations from \eqref{eq:147bis}. 
Let us set  
  \begin{equation}
    \label{eq:378}
    \cball R:=\big\{[x,\s ]\in \tY:\s\le R\big\} \ \text{ and } \ 
    \tyY[R]:= 
    \cball R\times \cball R.
  \end{equation}
  It is not restrictive to sol\-ve the previous problem \ref{pr:3}
  by also assuming that $\taalpha$ is a probability plan in $\cP(\tyY)$
  concentrated on $\tyY[R]$ 
  with $R^2=\sum_i\mu_i(X)$,
  i.e.~
  \begin{equation}
    \label{eq:379}
    \HK^2(\mu_1,\mu_2)=\min
    _{\saalpha\in C} 
    \int
      \sfdc^2\,\d\aalpha,\qquad
      C:=\Big\{
      \aalpha\in \cP(\tyY):\chm2i\aalpha= \mu_i,\
    \aalpha\big(\tyY\setminus \tyY[R]\big)
    =0
    \Big\}.
  \end{equation}
  In fact the functional $\sfdc^2$ and 
  the constraints have a natural scaling invariance
  induced by the dilation maps
  defined by \eqref{eq:147bis}. Since
\begin{equation}
  \label{eq:121}
  \int\sfdc^2\,\d(\dil{\vartheta,2}\taalpha)=
  \int \vartheta^2\sfdc^2([x_1,\s_1/\vartheta];[x_2,\s_2/\vartheta])\,\d\taalpha=
  \int\sfdc^2\,\d\taalpha,
\end{equation}
restricting first 
$\aalpha$ to $\tyY\setminus \{(\fro,\fro)\} $ and then
choosing $\vartheta$ as in \eqref{eq:153a} with $p=2$ \nc we obtain a probability
plan $\dil{\vartheta,2}{\aalpha \res \tyY\setminus \{(\fro,\fro)\} } $ in $\HM2{\mu_1}{\mu_2}$ concentrated
in $\tyY[R]\setminus \{(\fro,\fro)\} $ with the same cost
$\int\sfdc^2\,\d\taalpha$.
%
%
In order to show that Problem \ref{pr:3} has a solution we can then
use the formulation \eqref{eq:379} and prove that the set $C$ where
the minimum will be found is narrowly compact in $\cP(\tyY)$.  Notice
that the analogous property would not be true in $\cP(\pY\times \pY)$
(unless $X$ is compact) since measures concentrated in $(X\times
\{0\})\times (X\times\{0\})$ would be out of control.
Also the constraints $\chm2i\aalpha=\mu_i$ would not be preserved by narrow
convergence, if one allows for arbitrary plans in $\cP(\tyY)$ as in
\eqref{eq:104}.

\begin{theorem}[Existence of optimal plans for the HK problem]
  \label{thm:existenceHK}
  For every $\mu_1,\mu_2\in \cM(X)$ 
  the Hellinger-Kantorovich problem
  \ref{pr:3}
  always admits a solution $\aalpha\in \cP(\tyY)$ 
  concentrated on $\tyY[R]\setminus \{(\fro,\fro)\}$
  with $R^2=\sum_i\mu_i(X)$.
\end{theorem}
\begin{proof}
  By \WWW the rescaling \eqref{eq:121} it is not restrictive 
  to look for minimizers $\aalpha$ of \eqref{eq:379}.
  Since $\tyY[R]$ is closed in $\tyY$ and the maps $\sfs_i^2$ are
  continuous and bounded in $\tyY[R]$, $C$ is clearly narrowly closed.
  By Lemma \ref{le:compactnessH}, $C$ is also equally tight in
  $\cP(\tY)$, thus narrowly compact by Theorem \ref{thm:Prokhorov}.
  Since the $\sfdc^2$ is lower semicontinuous in $\tyY$, the existence
  of a minimizer of \eqref{eq:379} then follows by the direct method
  of the calculus of variations.
\end{proof}

We can also prove an interesting characterization of $\HK$
in terms of the
$L^2$-Kantorovich-Wasserstein distance 
on $\cP_2(\tY)$ given by \eqref{eq:149}.
An even deeper connection will be discussed in the next section,
see Corollary \ref{cor:HK-W2}.

\begin{corollary}[$\HK$ and the Wasserstein distance on $\cP_2(\tY)$]
  \label{cor:HK-W1}
  For every $\mu_1,\mu_2\in \cM(X)$ we have
  \begin{equation}
    \label{eq:394}
    \HK(\mu_1,\mu_2)=\min\Big\{\Wc(\alpha_1,\alpha_2):
    \alpha_i\in \cP_2(\tY),\quad 
    \chm2{}\alpha_i=\mu_i\Big\},
  \end{equation}
  and there exist optimal measures $\bar\alpha_i$ for \eqref{eq:394} 
  concentrated on $\cball R$ 
  with $R^2=\sum_i\mu_i(X)$.
  In particular the map $\chm2{}:\cP_2(\tY)\to \cM(X)$ is 
  a contraction, i.e.
  \begin{equation}
    \label{eq:395}
    \HK(\chm2{}\alpha_1,\chm2{}\alpha_2)\le \Wc(\alpha_1,\alpha_2)
    \quad\forevery \alpha_i\in \cP_2(\tY).
  \end{equation}
\end{corollary}
\begin{proof}
  If $\alpha_i\in \cP_2(\tY)$ with $\chm2{}\alpha_i=\mu_i$ then
  any Kantorovich-Wasserstein optimal plan $\aalpha\in \cP(\tY\times \tY)$ 
  for \eqref{eq:149} with marginals $\alpha_i$ clearly belongs
  to $\cHM2{\mu_1}{\mu_2}$ and yields the bound
  $ \HK(\mu_1,\mu_2)\le \Wc(\alpha_1,\alpha_2)$.
  On the other hand, if $\aalpha\in \OptHK{\mu_1}{\mu_2}$ is 
  an optimal solution for \eqref{eq:104} and 
  $\alpha_i:=\pi^i\aalpha\in \cP_2(\tY)$ are its marginals,
  we have $ \HK(\mu_1,\mu_2)\ge \Wc(\alpha_1,\alpha_2)$,
  so that $\alpha_i$ realize the minimum for \eqref{eq:394}.
\end{proof}
We conclude this section with two simple properties of the $\HK$
functional. We denote by $\eta_0$ the null measure.

\begin{lemma}[Subadditivity of $\HK^2$]
  The functional $\HK^2$ satisfies
  \begin{equation}
    \label{eq:137}
    \HK^2(\mu,\eta_0)=\mu(X),\qquad
    \HK^2(\mu_1,\mu_2)\le \mu_1(X)+\mu_2(X)
    \quad\forevery \mu,\mu_i\in \cM(X),
  \end{equation}
  and it is subadditive, i.e.
  for every $\mu_i,\mu_i'\in \cM(X)$ we have
  \begin{equation}
    \label{eq:136bis}
    \HK^2(\mu_1+\mu_1',\mu_2+\mu_2')\le 
    \HK^2(\mu_1,\mu_2)+\HK^2(\mu_1',\mu_2').
  \end{equation}
\end{lemma}
\begin{proof}
  \WWW The relations in \eqref{eq:137} are obvious.
  If $\taalpha\in \cHM2{\mu_1}{\mu_2}$ and 
  $\taalpha'\in \cHM2{\mu_1'}{\mu_2'}$ 
  it is easy to check that 
  $\taalpha+\taalpha'\in \cHM2{\mu_1+\mu_1'}{\mu_2+\mu_2'}$.
  Since the cost functional is linear with respect to the plan,
  we get \eqref{eq:136bis}.
\end{proof}

Subsequently we will use ``$\res$'' for the restriction of
measures. 

\begin{lemma}[A formulation with relaxed constraints]
  \label{le:max-le}
  For every $\mu_1,\mu_2\in \cM(X)$ we have
  \begin{subequations}
  \begin{align}
    \label{eq:135}
    \HK^2(\mu_1,\mu_2)
     &= \min_{\saalpha\in \cHMle2{\mu_1}{\mu_2}}
          \Big\{\int \sfdc^2(\ty_1,\ty_2)\,\d\aalpha+ 
             \sum_i\big(\mu_i-\chm2i\aalpha\big)(X)\Big\} 
    \\ \label{eq:375}
                      &=\mu_1(X)+\mu_2(X)-
                      \max_{\saalpha\in 
                      \cHMle2{\mu_1}{\mu_2}}
                      \Big\{ 2\int {\sfr_1\,\sfr_2} 
                            \cos(\sfdp(\sfx_1,\sfx_2))\,\d\taalpha\Big\}.
  \end{align}
  \end{subequations}
  Moreover,
  \begin{enumerate}[(i)]
  \item equations \eqref{eq:135}--\eqref{eq:375} share the same class of optimal
    plans.  
  \item 
    A plan $\aalpha\in \cHMle2{\mu_1}{\mu_2}$ is
    optimal for \eqref{eq:135}--\eqref{eq:375} if and only if the plan
    $\aalpha_\soo:= \aalpha\res(\tY_\soo\times \tY_\soo)$ is optimal
    as well.
  \item If $\aalpha$ is optimal for \eqref{eq:135}--\eqref{eq:375}
    with $\mu_i':=\mu_i-\chm2i\aalpha$, then
    $\tilde\aalpha:=\aalpha+\aalpha'$ \WWW is an optimal plan in 
    $\OptHK(\mu_1,\mu_2)$ for all $\aalpha'\in
    \cHM2{\mu_1'}{\mu_2'}$. 
  \item
    A plan $\aalpha\in \cHM2{\mu_1}{\mu_2}$ belongs to
    $\OptHK(\mu_1,\mu_2)$ 
    if and only if
    $\aalpha_\soo:= \aalpha\res(\tY_\soo\times \tY_\soo)$ is optimal
    for \eqref{eq:135}--\eqref{eq:375}. 
  \end{enumerate}
\end{lemma}
\begin{proof}
  The formulas \eqref{eq:135} and \eqref{eq:375} are just two
  different ways to write the same functional, since for every
  $\aalpha\in \cHMle2{\mu_1}{\mu_2}$ we have
    \begin{equation}
    \label{eq:163}
    \int\sfdc^2\,\d\taalpha+\sum_i\big(\mu_i-\chm2i\aalpha\big)(X)=
    \sum_i \mu_i(X)-
     2\int {\sfr_1\,\sfr_2}\cos(\sfdp(\sfx_1,\sfx_2))\,\d\taalpha.
  \end{equation}
  Thus, to prove (i) it is sufficient to show \eqref{eq:135}.
  The inequality $\ge$ is obvious, since $\cHMle2{\mu_1}{\mu_2}
  \supset \cHM2{\mu_1}{\mu_2}$ and
  for every $\aalpha\in \cHM2{\mu_1}{\mu_2}$
  the term $\sum_i\big(\mu_i-\chm2i\aalpha\big)(X)$ vanishes.

  On the other hand, whenever $\taalpha\in \cHMle2{\mu_1}{\mu_2}$,
  setting
  $\mu_i'':=\chm2i\aalpha\in \cM(X)$, $\mu_i':=\mu_i-\mu_i''$ and observing that
  $\taalpha\in \cHM2{\mu_1''}{\mu_2''}$ we get
  \begin{align*}
    \int&
          \sfdc^2(\ty_1,\ty_2)\,\d\aalpha+\sum_i\big(\mu_i-\chm2i\aalpha\big)(X)
          \ge
        \HK^2(\mu_1'',\mu_2'')+
          \mu_1'(X)+\mu_2'(X)
        \\& \topref{eq:137}\ge 
    \HK^2(\mu_1',\mu_2')+\HK^2(\mu_1'',\mu_2'')
    \topref{eq:136bis}\ge \HK^2(\mu_1,\mu_2).
  \end{align*}
  The same calculations also prove point (iii).

  In order to check (ii) it is sufficient to observe that the
  integrand in \eqref{eq:375} vanishes on $\tyY\setminus 
  (\tY_\soo\times \tY_\soo)$.

  Finally, if $\aalpha\in \OptHK(\mu_1,\mu_2)$ is optimal for
  \eqref{eq:104},
  then by the consideration above it is optimal for \eqref{eq:375} and
  therefore (ii) shows that $\aalpha_\soo$ is optimal as well. 
  The converse implication follows by (iii).
\end{proof}

\subsection{Gluing lemma and triangle inequality}
\label{subsec:triangle}
In this section we will prove that $\HK$ satisfies the triangle
inequality and therefore is a distance on $\cM(X)$.  The main
technical step is provided by the following useful property for plans
in $\cM(\tY^{\otimes N})$ with given homogeneous marginals, which is a
simple application of the rescaling \WWW invariance in
\eqref{eq:121}. \EEE

\begin{lemma}[\WWW Normalization of lifts] 
  \label{le:multiple-rescaling}
  Let $\taalpha
  \in \cM_2(\tY^{\otimes N})$, $N\ge2,$ be a plan satisfying
  \begin{align}
    \label{eq:13}
    \chm2i\aalpha=\mu_i\in \cM(X)\text{ for } i=1,...,N,
    \  \text{ and } \ 
    a_i=\int \sfdc^2(\ty_{i-1},\ty_{i})\,\d\taalpha \text{ for } 
    i=2,..., N,
  \end{align}
  and let $j\in \{1,\ldots,N\}$ be fixed. Then, it is possible to find
  a new plan $\bar\taalpha
  \in \cM_2(\pdyY N)$ which 
  still satisfies \eqref{eq:13} 
  and \WWW additionally the normalization of the $j$th lift, 
  \begin{equation}
    \label{eq:14}
    \pi^j_\sharp(\bar\taalpha)=\delta_{\soo}+
    \frp_\sharp(\mu_j\otimes \delta_1).
  \end{equation}
\end{lemma}
\begin{proof}
  %
  By possibly adding $\otimes^N\delta_{\soo}$ 
  to $\aalpha$ (which does not modify \eqref{eq:13})
  we may suppose that 
  \[
\omega_j:=\aalpha\big(\{\tyy\in \pdyY N: \pi^j(\tyy)=\fro\}\big)\ge 1,
\]
\WWW where $j$ is fixed as in the lemma. 
 In order to find $\bar\aalpha$ 
\nc it is sufficient to rescale $\aalpha$ by the function
  \begin{equation}
    \label{eq:392}
    \vartheta(\tyy):=
    \begin{cases}
      \sfs_j(\tyy)&\text{if }\ty_j\neq \fro,\\
      \omega_j^{-1/2}&\text{otherwise.}
    \end{cases}
  \end{equation}
  With the notation of \eqref{eq:147bis} we set 
  $\bar\aalpha:=\dil{\vartheta,2}\aalpha$ and we decompose $\aalpha$
  in the sum $\aalpha=\aalpha'+\aalpha''$ 
  where $\aalpha'=\aalpha\res \{\tyy\in \pdyY N:
  \pi^j(\tyy)=\fro\} $.
  For every $\zeta\in \rmB_b(\tY)$ we have
  \begin{align*}
  & \int \zeta(\ty_j)\,\d\bar\aalpha=
                          \int \zeta(\ty_j\cdot
                            \vartheta^{-1}(\tyy))\vartheta^2(\tyy)\,\d\aalpha=
                                \int\zeta(\fro)
                                 \omega_j^{-1}\,\d\aalpha'+\!\!
                                    \int \zeta([x_j,\s_j/\vartheta(\tyy)])
                                    \vartheta^2(\tyy)\,\d\aalpha''
    \\
                                  &=\zeta(\fro)+\int \zeta([x_j,1])\sfs_j^2
                                    \d\aalpha''
                                    =
                                    \zeta(\fro)+\int \zeta([x_j,1])\sfs_j^2
                                    \d\aalpha=
                                    \zeta(\fro)+\int \zeta\circ \frp\,
                                    \d(\mu_j\otimes \delta_1)
  \end{align*}
  which yields \eqref{eq:14}.
\end{proof}
We can now prove a general form of the so-called ``gluing lemma''
that is the natural extension of the well known result for
transport problems (see e.g.~\cite[Lemma 5.3.4]{Ambrosio-Gigli-Savare08}).
Here its formulation is strongly related to the rescaling
invariance of optimal plans given by Lemma 
\ref{le:multiple-rescaling}.

\begin{lemma}[Gluing lemma]
  \label{le:gluing}
  Let us consider a finite collection
  of measures $\mu_i\in \cM(X)$ 
  for
  $i=1,\ldots,N$ with $N\ge2$. Set  
  \begin{equation}
    \Theta : =
    \sqrt{\mu_1(X)} +\sum_{i=2}^N \HK(\mu_{i-1},\mu_i)
    \nc 
    \ \text{ and } \ 
  M^2:=\sum_{i=1}^N\mu_i(X).\label{eq:367}
\end{equation}
  Then there exist plans $\taalpha_1, \,\taalpha_2
  \in \cP_2(\pdyY N)$ such that 
  \begin{gather}
    \label{eq:13bis}
   \begin{aligned}
    &\chm2i \taalpha_k=\mu_i \ \text{ for } \ 
    i=1,\ldots, N \ \text{ and } \\
   & \int
    \sfdc^2(\ty_{i-1},\ty_{i})\,\d\taalpha_k=\HK^2(\mu_{i-1},\mu_{i})
    \ \text{ for } \ 
    i=2,\ldots, N.
   \end{aligned}
  \end{gather}
 Moreover, \WWW the plans $\aalpha_k$ satisfy the following additional
 conditions:  
  \begin{align}
    \label{eq:393}
    \taalpha_1\text{ is concentrated on }&\big\{\tyy\in \pdyY
    N:\sum_i \sfs_i^2(\tyy)\le M^2\big\},\\
    \label{eq:393bis}
    \taalpha_2\text{ is concentrated on }&\big\{\tyy\in \pdyY
                                           N:\sup_i\sfs_i(\tyy)\le \Theta\big\}=
                                           \big(\cball\Theta\big)^{\otimes N}.
  \end{align}
\end{lemma}
\begin{proof}
  We first construct a plan $\aalpha$ satisfying \eqref{eq:13bis},
  then suitable rescalings will provide $\aalpha_k$ satisfying
  \eqref{eq:393} or \eqref{eq:393bis}. In order to clarify the
  argument, we consider $N$-copies $X_1,X_2,\ldots, X_N$ of $X$ (and
  for $\tY$ in a similar way) so that $\tens X N=\prod_{i=1}^N X_i$
  
  We argue by induction; the starting case $N=2$ is covered by Theorem
  \ref{thm:existenceHK} and Lemma \ref{le:multiple-rescaling}.  Let us
  now discuss the induction step, by assuming that the thesis holds
  for $N$ and proving it for $N+1$.  We can thus find an optimal plan
  $\aalpha^N$ such that \eqref{eq:13bis} 
  hold, and another optimal plan $\aalpha\in\OptHK(\mu_N,\mu_{N+1})$
  for the couple $\mu_N,\mu_{N+1}$.  Applying \WWW the normalization
  Lemma \ref{le:multiple-rescaling} to $\aalpha^N$ (with $j=N)$
  and to $\aalpha$ (with $j=1$) we can assume that
  \begin{displaymath}
    \pi^N_\sharp(\aalpha^N)=\delta_{\soo}+\frp_\sharp(\mu_N\otimes
    \delta_1)=\pi^1_\sharp(\taalpha).    
  \end{displaymath}
  Therefore we can apply the standard gluing Lemma in
  $\big(\prod_{i=1}^{N-1}\tY_i\big),\tY_N,\tY_{N+1}$ (see
  e.g.~\cite[Lemma 5.3.2]{Ambrosio-Gigli-Savare08} and \cite[Lemma
  2.2]{Ambrosio-Erbar-Savare15} in the case of arbitrary topological
  spaces) obtaining a new plan $\aalpha^{N+1}$ satisfying
  $\pi^{1,2,\cdots, N}_\sharp\aalpha^{N+1}=\aalpha^N$ and
  $\pi^{N,N+1}\aalpha^{N+1}=\aalpha$.  In particular, $\aalpha^{N+1}$
  satisfies \eqref{eq:13bis}.
  
  A further application of the \WWW rescaling \eqref{eq:121} with
  $\vartheta$ as in \eqref{eq:153a} yields a plan $\aalpha_1$
  satisfying also \eqref{eq:393}.

  In order to obtain $\aalpha_2$, we can assume 
  $\aalpha(\{|\tyy|=0\})=0$ and set
  $\aalpha_2=\dil{\vartheta,2}\aalpha$, where we use the rescaling function 
  \begin{displaymath}
    \vartheta(\tyy):=r^{-1}|\tyy|_\infty=r^{-1}\sup_i\sfs_i(\tyy) \
    \text{ with } \ 
    r^2:=\int_{\pdyY N} |\tyy|_\infty^2\,\d\aalpha. 
  \end{displaymath}
  \WWW To obtain \eqref{eq:393bis} it remains to estimate $r$. We
  consider arbitrary coefficients $\theta_i>0$ \nc and 
  use for $n=2,\ldots,N$ \nc the inequality 
  \begin{align*}
    \sfs_n&\le \sfs_1+\sum_{i=2}^n|\sfs_i-\sfs_{i-1}|
            \le \Big(\sum_{i=1}^n\theta_i^{-1}\Big)^{1/2}
            \Big(\theta_1\sfs_1^2+\sum_{i=2}^n\theta_i 
               |\sfs_i-\sfs_{i-1}|^2\Big)^{1/2}\\ 
            &\le 
              \Big(\sum_{i=1}^N\theta_i^{-1}\Big)^{1/2}
              \Big(\theta_1\sfs_1^2+\sum_{i=2}^N\theta_i 
               \sfdc^2(\ty_i,\ty_{i-1})\Big)^{1/2} ,
  \end{align*}
  which yields
  \begin{align*}
    r^2&=\int_{\pdyY N} |\tyy|_\infty^2\,\d\aalpha\le 
    \Big(\sum_{i=1}^N\theta_i^{-1}\Big)\int_{\pdyY N}
    \Big(\theta_1\sfs_1^2+\sum_{i=2}^N\theta_i\sfdc^2(\ty_i,\ty_{i-1})\Big)\,\d\aalpha
    \\& =\Big(\sum_{i=1}^N\theta_i^{-1}\Big)\cdot \Big(\theta_1\mu_1(X)+
         \sum_{i=2}^N
             \theta_i\,\HK^2(\mu_{i-1},\mu_i)\Big);
  \end{align*}
  optimizing with respect to $\theta_i>0$ we obtain 
the value of $\Theta$ given by \eqref{eq:367}. \nc
\end{proof}

\WWW The next remark gives a similar rescaling result for
probability couplings $\bbeta \in \cP_2(\pdyY {N})$.\EEE%

\begin{remark}
  \label{rem:gluing2}
  \upshape
  In a completely similar way (see \cite[Lemma
  5.3.4]{Ambrosio-Gigli-Savare08}),
  for $N\ge2$,  a finite collection
  of measures $\mu_i\in \cM(X)$, 
  and coefficients $\theta_i>0$, $i=1,\ldots,N$,
  there exists a plan $
  \bbeta
  \in \cP_2(\pdyY {N})$
  concentrated on 
  $\big\{\tyy\in \pdyY
  N:\sup_i\sfs_i(\tyy)\le
  \Xi\big\}$
  with
  \begin{equation}   
    \Xi:=\sqrt{ \mu_1(X)}+
    \sum_{i=2}^N\HK(\mu_1,\mu_i),
    \label{eq:464}
  \end{equation}
  such that 
  \begin{gather}
    \label{eq:13tris}
    \chm2i \bbeta=\mu_i \ \text{ and } \ 
    \int
    \sfdc^2(\ty_{1},\ty_{i})\,\d\bbeta=\HK^2(\mu_{1},\mu_{i})
    \ \text{ for }
    i=1,\ldots, N. \qedhere
  \end{gather}
\end{remark}
Arguing as in the proof of Corollary \ref{cor:HK-W1} 
one immediately obtains the following result, \WWW which will be
needed for the proof of Theorem \ref{thm:curvatureHK} and for the
subsequent corollary. 

\begin{corollary}
  \label{cor:HK-W2}
  For every finite collection of measures $\mu_i\in \cM(X)$,
  $i=1,\ldots, N$, there exist $\alpha_i,\beta_i\in \cP_2(\tY)$ with
  $\alpha_i$ concentrated in $\cball {r}$ where $r=\min(M,\Theta)$ is
  given as in \eqref{eq:367} and $\beta_i$ concentrated in $\cball
  \Xi$ given by \eqref{eq:464} such that 
  \begin{align*}
    \label{eq:396}
    &\chm2{} \alpha_i= \mu_i \ \text{ and } \
           \chm2{}\beta_i=\mu_i&&
                    \text{for }i=1,\ldots,N, \\
    &\WWW \HK(\mu_1,\mu_i)=\Wc(\beta_1,\beta_i) \EEE\ \text{ and } \ 
           \HK(\mu_i,\mu_{i+1})=\Wc(\alpha_i,\alpha_{i+1})      &&
           \text{for }i=2,\ldots,N.
  \end{align*}
\end{corollary}

\WWW We are now in the position to show that the functional $\HK$ is
 a true distance on $\cM(X)$, where we deduce the triangle
 inequality from that for $\Wc$ by using normalized lifts. 

\begin{corollary}[$\HK$ is a distance]
  $\HK$ is a distance on $\cM(X)$; in particular, 
  for every $\mu_1,\mu_2,\mu_3\in \cM(X)$ 
  we have the triangle inequality
  \begin{equation}
  \HK(\mu_1,\mu_3)\le \HK(\mu_1,\mu_2)+\HK(\mu_2,\mu_3).\label{eq:17}
\end{equation}
\end{corollary}
\begin{proof}
  It is immediate to check that $\HK$ is symmetric and
$\HK(\mu_1,\mu_2)=0$ if and only if $\mu_1=\mu_2$.
In order to check  \eqref{eq:17} 
it is sufficient to apply the previous corollary \ref{cor:HK-W2}
to find measures $\alpha_i\in \cP_2(\tY)$, $i=1,2,3$, such that
$\chm2{}\alpha_i=\mu_i$ and
$\HK(\mu_1,\mu_2)=\Wc(\alpha_1,\alpha_2)$ and
$\HK(\mu_2,\mu_3)=\Wc(\alpha_2,\alpha_3)$.
Applying the triangle inequality for $\Wc$ 
  we obtain
  \begin{displaymath}
    \HK(\mu_1,\mu_3)\le
                      \Wc(\alpha_1,\alpha_3)\le 
                      \Wc(\alpha_1,\alpha_2)+\Wc(\alpha_2,\alpha_3)
                      =\HK(\mu_1,\mu_2)+\HK(\mu_2,\mu_3).\qedhere
  \end{displaymath}
\end{proof}

As a consequence of the previous two results, the map 
$\chm2{}:\cP_2(\tY)\to \cM(X)$ is a metric submersion.

\subsection{Metric and topological properties}
\label{subsec:MT}
In this section we will assume that the topology $\tau$ on $X$ is
induced by $\sfd$ and that $(X,\sfd)$ is separable, so that also
$(\tY,\sfdc)$ is separable.  Notice that in this case there is no
difference between weak and narrow topology in $\cM(X)$. Moreover,
since $X$ is separable, $\cM(X)$ \WWW equipped with the weak topology
\EEE is metrizable, so that converging sequences are sufficient to
characterize the weak-narrow topology.

It turns out \cite[Chap.~7]{Ambrosio-Gigli-Savare08} 
that $(\cP_2(\tY),\Wc)$ is a separable metric space:
convergence of a sequence $(\alpha_n)_{n\in \N}$ to a limit measure
$\alpha$ 
in $(\cP_2(\tY),\Wc)$ corresponds to weak-narrow convergence
in $\cP(\tY)$ and convergence of the quadratic moments, 
or, equivalently, to convergence of 
integrals of continuous functions with quadratic growth, i.e.
\begin{equation}
  \label{eq:160}
  \lim_{n\to\infty}\int_{\tY}\varphi\,\d\alpha_n=
  \int_{\tY}\varphi\,\d\alpha
  \quad
  \forevery \varphi\in \rmC(\tY)\text{ with }
  |\varphi(\ty)|\le A+B\sfs^2(\ty),
\end{equation}
for some constants $A,B\ge0$ depending on $\varphi$. 
Recall that $\sfs^2(\ty)=\sfdc^2(\ty,\fro)$.

\begin{theorem}[$\HK$ metrizes the weak topology on $\cM(X)$]
  \label{thm:topo1}
  $\HK$ induces the weak-narrow topology on $\cM(X)$:
  a sequence $(\mu_n)_{n\in \N}\in \cM(X)$ converges to a measure
  $\mu$ in $(\cM,\HK)$ if and only if $(\mu_n)_{n\in\N}$ converges weakly to
  $\mu$
  in duality with continuous and bounded functions.

  In particular, the metric space $(\cM(X),\HK)$ is separable.
\end{theorem}
\begin{proof}
  Let us first suppose that $\lim_{n\to\infty}\HK(\mu_n,\mu)=0.$ 
  We argue by contradiction and we assume that there exists a
  function $\zeta\in \rmC_b(X)$ and a subsequence 
  (still denoted by $\mu_n$) such that 
  \begin{equation}
    \label{eq:317}
    \inf_n \Big|\int_X \zeta\,\d\mu_n-\int_X\zeta\,\d\mu\Big|>0.
  \end{equation}
  The first estimate of \eqref{eq:137} and the triangle inequality 
  show that 
  \begin{displaymath}
    \limsup_{n\to\infty}\mu_n(X)\le\limsup_{n\to\infty}
    \big(\HK(\mu_n,\mu)+\HK(\mu,\eta_0)\big)^2=\mu(X),    
  \end{displaymath}
  so that \WWW $\sup_n\mu_n(X)=M^2<\infty$. \EEE
  By Corollary \ref{cor:HK-W1} we can find measures 
  $\alpha_n,\alpha_n'\in \cP_2(\tY)$ 
  concentrated on \WWW $\cball {2M}$ such that 
  \begin{displaymath}
    \chm2{}{\alpha_n}=\mu,\quad
    \chm2{}\alpha_n'=\mu_n,\quad
    \Wc(\alpha_n,\alpha_n')=\HK(\mu,\mu_n).
  \end{displaymath}
  By Lemma \ref{le:compactnessH} the sequence $(\alpha_n)_{n\in\N}$
  is equally tight in $\cP_2(\tY)$; since it is also uniformly bounded
  there exists a subsequence $k\mapsto n_k$ such that 
  $\alpha_{n_k}$ weakly converges to a limit $\alpha\in \cP_2(\tY)$.
  Since $\alpha_n$ is concentrated on \WWW $\cball {2M}$ 
  we also have $\lim_{k\to\infty}\Wc(\alpha_{n_k},\alpha)=0$ 
  and therefore $\chm2{}\alpha=\mu$, 
  $\lim_{k\to\infty}\Wc(\alpha_{n_k}',\alpha)=0$.

  We thus have
  \begin{displaymath}
    \lim_{k\to\infty}\int_X \zeta(x)\,\d\mu_{n_k}=
    \lim_{k\to\infty}\int_{\tY} \zeta(\sfx)\sfs^2\,\d\alpha_{n_k}'=
    \int_{\tY} \zeta(\sfx)\sfs^2\,\d\alpha=
    \int_X \zeta(x)\,\d\mu
  \end{displaymath}
  which contradicts \eqref{eq:317}.

  In order to prove the converse implication, let us 
  suppose that $\mu_n$ is converging weakly to $\mu$ in $\cM(X)$.
  If $\mu$ is the null measure $\eta_0=0$, then $\lim_{n\to\infty}\mu_n(X)=0$
  so that $\lim_{n\to\infty}\HK(\mu_n,\mu)=0$ 
  by \eqref{eq:137}.

  So we can suppose that $m:=\mu(X)>0$ and \WWW have 
  $m_n:=\mu_n(X)\ge m/2>0$ for sufficiently large $n$. 
  We now consider the measures 
  $\alpha_n,\alpha\in \cP(\tY)$ given by 
  \begin{displaymath}
    \alpha_n:=\frp_\sharp\Big(m_n^{-1}\mu_n\otimes
    \delta_{\sqrt{m_n}}\Big) \ \text{ and } \ 
    \alpha:=\frp_\sharp\Big(m^{-1}\mu\otimes
    \delta_{\sqrt{m}}\Big).
  \end{displaymath}
  Since $\chm2{}\alpha_n=\mu_n$ and $\chm2{}\alpha=\mu$, 
  by \eqref{eq:395} we have
  $\HK(\mu_n,\mu)\le  \Wc(\alpha_n,\alpha).$
  Since $m_n^{-1}\mu_n$ is weakly converging to $m^{-1}\mu$ in
  $\cP(X)$
  and $m_n\to m$, it is easy to check that 
  $m_n^{-1}\mu_n\otimes
    \delta_{\sqrt{m_n}}$ weakly converges to $m^{-1}\mu\otimes
    \delta_{\sqrt{m}}$ in $\cP(Y)$ and therefore $\alpha_n$ 
    weakly converges to $\alpha$ in $\cP(\tY)$ by 
    the continuity of the projection $\frp$.
    \WWW Hence, in order to conclude that $\Wc(\alpha_n,\alpha) \to0$
    it is now sufficient to prove
  the convergence of their quadratic moments
  with respect to the vertex $\fro$. However, this is
  is immediate because of 
  \begin{displaymath}
    \lim_{n\to\infty}\int \sfdc^2(\ty,\fro)\,\d\alpha_n=
    \lim_{n\to\infty}\int \sfs^2\,\d\alpha_n=
    \lim_{n\to\infty}m_n=m=\int \sfdc^2(\ty,\fro)\,\d\alpha.    \qedhere
  \end{displaymath}
\end{proof}
\begin{corollary}[Compactness]
  \label{cor:compactness}
  If $(X,\sfd)$ is a compact metric space 
  then $(\cM(X),\HK)$ is a proper metric space, i.e.~every 
  bounded set is relatively compact.
\end{corollary}
\begin{proof}
  It is sufficient to notice that a set $\calC\subset \cM(X)$ is 
  bounded w.r.t.~$\HK$ if and only if $\sup_{\mu\in
    \calC}\mu(X)<\infty$. \WWW Then the classical weak sequential
  compactness of closed bounded sets in $\cM(X)$ gives the result. \EEE
\end{proof}

\WWW The following completeness result for $(\cM(X),\HK)$ is obtained by
suitable liftings of measures $\mu_i$ to probability measures
$\alpha_i\in \cP_2(\tY)$, supported in some $\cball\Theta$. Then the
completeness of the Wasserstein  space $(\cP_2(\tY),\Wc)$ is 
exploited. 

\begin{theorem}[Completeness of $(\cM(X),\HK)$]
  If $(X,\sfd)$ is complete than the metric space $(\cM(X),\HK)$ is complete.
\end{theorem}
\begin{proof}
  We have to prove that every Cauchy sequence $(\mu_n)_{n\in \N}$ in
  $(\cM(X),\HK)$
  admits a convergent subsequence. 
  By exploiting the Cauchy property, we can find 
  an increasing sequence of integers $k\mapsto n(k)$ 
  such that 
  $\HK(\mu_m,\mu_{m'})\le 2^{-k}$ whenever $m,m'\ge n(k)$
  and we consider the subsequence $\mu_i':=\mu_{n(i)}$,
  
  so that 
  \begin{align*}
    \sqrt{\mu_1(X)}+\sum_{i=2}^N
    \HK(\mu_{n(i)},\mu_{n(i-1)})\le \sqrt{\mu_1(X)}+1,
  \end{align*}
  \nc and by applying the Gluing Lemma \ref{le:gluing},
  for every $N>0$ 
  we can find measures $\alpha_i^N\in\cP_2(\tY)$, $i=1,\ldots,N$,
  concentrated on $\cball \Theta$ with $ \Theta:=\sqrt{\mu_1(X)}+1 \nc
  $,
  such that 
  \begin{displaymath}
   \chm2{}\alpha_i^N=\mu_i' \ \text{ and } \ 
    \Wc(\alpha_i^N,\alpha_{i-1}^N)=\HK(\mu_i',\mu_{i-1}').
  \end{displaymath}
  For every $i$ the sequence $N\mapsto \alpha_i^N\in \cP_2(\tY)$
  is tight by Lemma \ref{le:compactnessH} and concentrated
  on the bounded set $\cball\Theta$, so that by Prokhorov Theorem
  it is relatively compact in $(\cP_2(\tY),\Wc)$.

  By a standard diagonal argument, we can find a further increasing subsequence
  $m\mapsto N(m)$ and limit measures $\alpha_i\in \cP_2(\tY)$ 
  such that $\lim_{m\to\infty}\Wc(\alpha_i^{N(m)},\alpha_i)=0$.
  The convergence with respect to $\Wc$ yields that
  \begin{displaymath}
    \chm2{}\alpha_i=\mu_i,\quad
    \Wc(\alpha_i,\alpha_{i-1})=\HK(\mu_i',\mu_{i-1}')\le 2^{i-1}.
  \end{displaymath}
  It follows that $i\mapsto \alpha_i$ is a Cauchy sequence in
  $(\cP_2(\tY),\Wc)$ which is a complete metric space 
  \cite[Prop.~7.1.5]{Ambrosio-Gigli-Savare08}
  and therefore
  there exists \WWW $\alpha\in \cP_2(\tY)$ such that
  $\lim_{i\to\infty}\Wc(\alpha_i,\alpha)=0$. Setting
  $\mu:=\chm2{}\alpha\in \cM(X)$ we thus obtain
  $\lim_{i\to\infty}\HK(\mu_i',\mu)=0$.
\end{proof}

We conclude this section by proving a simple comparison estimate
for $\HK$ with the \emph{Bounded Lipschitz} metric (cf.\ 
\cite[Sec.~11.3]{Dudley02}), see also \cite[Thm.\ 3]{KMV15}. 
The Bounded Lipschitz metric is \nc defined via
\begin{equation}
  \label{eq:3}
  \BL(\mu_1,\mu_2):=\sup\Big\{\int\zeta\,\d(\mu_1-\mu_2):\zeta\in
  \Lip_b(X),\quad
  \sup_X |\zeta|+\Lip(\zeta,X)\le 1\Big\}.
\end{equation}
We do not claim that the constant $C_*$ below is optimal.

\begin{proposition}
  For every $\mu_1,\mu_2\in \cM(X)$ we have
  \begin{equation}
    \label{eq:60}
    \BL(\mu_1,\mu_2)\le
    C_*\Big(\sum_i\mu_i(X)\Big)^{1/2}\HK(\mu_1,\mu_2),\ \text{ where }
    C_*:=\sqrt{2+\pi^2/2}.
  \end{equation}
\end{proposition}
\begin{proof}
  Let $\xi\in \Lip_b(X)$ with $\sup_X |\xi|+\Lip(\xi,X)\le 1$
  and let $\aalpha\in \cP(\tyY)$ optimal for \eqref{eq:379}
  and concentrated on $\tyY[R]$ with $R^2:=\mu_1(X_1)+\mu_2(X_2)$.
  Notice that 
\begin{displaymath}
    |\xi(x_1)-\xi(x_2)|\le \max(\sfd(x_1,x_2),2)
    \le 2\sfd_2(x_1,x_2)\le 2\sfdp(x_1,x_2)
    \le 2\pi\sin(\sfdp(x_1,x_2)/2)
\end{displaymath}
We consider the function
  $\zeta:\tY\to \R$ defined by 
  $\zeta(\ty):=\xi(\sfx)\sfs^2$. Hence, 
  $\zeta$ satisfies
  \begin{align*}
    \Big|\zeta(\ty_1)-\zeta(\ty_2)\Big|&\le 
    |\xi(\sfx_1)-\xi(\sfx_2)|\sfs_1\sfs_2
    +\big(|\xi(\sfx_1)|\sfs_1+|\xi(\sfx_2)|\sfs_2\big)
                                       |\sfs_1-\sfs_2|
                                       \\
    &\le 2\pi\sin(\sfdp(\sfx_1,\sfx_2)/2)\sfs_1\sfs_2+
      (\sfs_1+\sfs_2) |\sfs_1-\sfs_2|
      \\&\topref{eq:94pre}\le 
          \sqrt{(\sfs_1+\sfs_2)^2+\pi^2\sfs_1\sfs_2} \:
          \, \sfdc(\ty_1,\ty_2)
    \le C_* 
    \sqrt{\sfs_1^2+\sfs_2^2}\: \sfdc(\ty_1,\ty_2)
  \end{align*}
  Since the optimal plan $\aalpha$ is concentrated on 
  $\{\sfs_1^2+\sfs_2^2\le R^2\}$ we obtain
  \begin{align*}
    \Big|\int_X \xi\,\d(\mu_1-\mu_2)\Big|&=
    \Big|\int \zeta(\ty_1)-\zeta(\ty_2)\,\d\aalpha\Big|
    \le 
    \int|\zeta(\ty_1)-\zeta(\ty_2)|\,\d\aalpha
    \\&
        \le C_* R\int \sfdc(\ty_1,\ty_2)\,\d\aalpha
        \le C_* R 
        \, \HK(\mu_1,\mu_2).\qedhere
  \end{align*}
\end{proof}

\subsection{Hellinger-Kantorovich distance and  Entropy-Transport functionals}
\label{subsec:HKET}

In this section we will establish our main result connecting $\HK$
with $\LET$.

It is clear that the definition of $\HK$ does not change if we replace
the distance $\sfd$ on $X$ by its truncation $\sfd_\pi=\sfd\land\pi$.
It is less obvious that we can even replace the threshold $\pi$ with
$\pi/2$ and use the distance $\tsfdc$ of Remark
\ref{rem:twotruncations} in the formulation of the
Hellinger-Kantorovich Problem \ref{pr:3}.  
This property is related to
the particular structure of the homogeneous marginals (which are not
affected by masses concentrated in the vertex $\fro$ of the cone
$\tY$);  in \cite[Sect.\,3.2]{LMS15} it is 
is called the \emph{presence of a sufficiently large reservoir}, which
shows that transport over distances larger than $\pi/2$ is never
optimal, since it is cheaper to transport into or out of the reservoir in
$\fro$).  
This
\EEE will provide an essential piece of information to connect
the $\HK$ and the $\LET$ functionals.

In order to prove \WWW that transport only occurs of distances $\leq
\pi/2$ we define the subset 
\begin{equation}
  \label{eq:tyYprime}
\tyY':=\big\{\tsfdc<\sfdc\big\}=\big\{(\ty_1,\ty_2)\in
\tY_\soo\times\tY_\soo:
\sfd(\sfx_1,\sfx_2)>\pi/2\big\}
\end{equation} 
and consider the partition $(\tyY',\tyY'')$ of $\tyY=\tY\times\tY$,
where 
$\tyY'':=\tyY\setminus \tyY'=\big\{\tsfdc=\sfdc\big\}$.
Observe that 
\begin{equation}
  \label{eq:390}
  \tyY_\soo'':=\tyY''\cap (\tY_\soo\times \tY_\soo)=\big\{(\ty_1,\ty_2)\in
  \tY_\soo\times \tY_\soo:\sfd(\sfx_1,\sfx_2)\le \pi/2\big\}.
\end{equation}
\WWW In the following lemma we show that minimizers $\aalpha \in
\OptHK(\mu_1,\mu_2)$ are concentrated on $\tyY''$, i.e.\
$\aalpha(\tyY')=0$ which holds  if and only if
$\aalpha_\soo=\aalpha\res(\tY_\soo\times \tY_\soo)$ is concentrated on
$\tyY_\soo''$. \WWW To handle the mass that is transported into or out
of $\fro$, we use the continuous projections \EEE
    \begin{equation}
    \label{eq:386}
    \frg_i:\tyY\to \tyY,\quad
    \frg_1(\ty_1,\ty_2):=(\ty_1,\fro),\quad
    \frg_2(\ty_1,\ty_2):=(\fro,\ty_2).
  \end{equation}
\begin{lemma}[Plan restriction]
  \label{le:restriction}
  For every $\aalpha\in \cM(\tyY)$ 
  the plan
  \begin{equation}
    \label{eq:402}
    \bar\aalpha:=\aalpha''+(\frg_1)_\sharp\aalpha'+(\frg_2)_\sharp
    \aalpha'\quad
    \text{with}\quad
    \aalpha':=\aalpha\res\tyY',\quad \aalpha'':=\aalpha\res\tyY'',
  \end{equation}
  is concentrated on $\tyY''$, has the same homogeneous marginals as $\aalpha$,
  i.e.~$\chm2i\bar\aalpha=\chm2i\aalpha$,
  and 
  \begin{equation}
    \label{eq:403}
    \int_{\tyY}\sfdc^2\,\d\bar\aalpha=
    \int_{\tyY}\tsfdc^2\,\d\bar\aalpha\le \int_{\tyY}\sfdc^2\,\d\aalpha,
  \end{equation}
  where the inequality is strict if $\aalpha(\tyY')>0$.
  In particular for every $\mu_1,\mu_2\in \cM(X)$
  \begin{equation}
    \label{eq:404}
    \HK^2(\mu_1,\mu_2)=\min\Big\{
    \int\tsfdc^2(\ty_1,\ty_2)\,\d\taalpha: \taalpha\in
      \cM_2(\tyY),\ 
      \chm2i\taalpha=\mu_i\Big\}.
  \end{equation}
\end{lemma}
\begin{proof}
  For every $\zeta\in \rmB_b(X)$, since $\sfs_1\circ\frg_2=0$ and $\sfs_1\circ\frg_1=\sfs_1$, we have
  \begin{align*}
    \int\zeta\,\d(\chm21\bar\aalpha)&=\int \zeta(\sfx_1)\sfs_1^2\,\d\bar\aalpha=
    \int \zeta(\sfx_1)\sfs_1^2\,\d\aalpha''+
                                    \sum_k\int \zeta(\sfx_1(\frg_k))\sfs_1(\frg_k)^2\,\d\aalpha'
                                   \\&=
    \int \zeta(\sfx_1)\sfs_1^2\,\d\aalpha''+
    \int \zeta(\sfx_1)\sfs_1^2\,\d\aalpha'
                                       =\int
                                       \zeta(\sfx_1)\sfs_1^2\,\d\aalpha=
                                       \int\zeta\,\d(\chm21\aalpha),
  \end{align*}
  so that $\chm21{\bar\aalpha}=\chm21{\aalpha}$; a similar calculation
  holds
  for $\chm22$ 
  so that $\bar\aalpha\in \cHM2{\mu_1}{\mu_2}$. Moreover,
  if $(\ty_1,\ty_2)\in \tyY'$ we easily get
  \begin{displaymath}
    \sfdc^2(\ty_1,\ty_2)>\sfs_1^2+\sfs_2^2=
    \sfdc^2(\frg_1(\ty_1,\ty_2))+\sfdc^2(\frg_2(\ty_1,\ty_2))
  \end{displaymath}
  so that whenever $\aalpha(\tyY')>0$ we get
  \begin{displaymath}
    \int\sfdc^2\,\d\bar\aalpha 
   = \int\big(\sfdc^2\circ\frg_1+\sfdc^2\circ\frg_2\big)\,\d\aalpha'+
     \int\sfdc^2\,\d\aalpha''
   < \int\sfdc^2\,\d\aalpha'+\int\sfdc^2\,\d\aalpha''
   = \int\sfdc^2\,\d\aalpha,
  \end{displaymath}
  which proves \eqref{eq:403} and characterizes the equality case.
  \eqref{eq:404} then follows by \eqref{eq:403} and 
  the fact that the homogeneous marginals of $\bar\aalpha$ and
  $\aalpha$ coincide.
\end{proof}
%

\WWW In  \eqref{eq:404} we have established that $\aalpha \in
\OptHK(\mu_1,\mu_2)$ has support in $\tyY''$. This allows us to prove
the identity $\LET = \HK^2$. For this, we \EEE
introduce the open set $\tG \subset \tyY''$ via 
\begin{displaymath}
    \tG:=\Big\{([x_1,r_1],[x_2,r_2])\in \tyY:r_1r_2\neq 0,\
    \sfd(x_1,x_2)<\pi/2\Big\}
\end{displaymath}
and note that $\sfr_1\sfr_2\cos(\sfd_{\pi/2}(\sfx_1,\sfx_2))>0$ in
$\tG$. \WWW Recall also $\boldsymbol\frp=\frp{\otimes}\frp: \yY
\to \tyY$,  where $\frp$ is defined in \eqref{eq:proj}. \EEE

\begin{theorem}[$\HK^2=\LET$]
  \label{thm:main-equivalence}
  For all $\mu_1,\,\mu_2 \in \cM(X)$ we have
      \begin{equation}
      \label{eq:383}
      \HK^2(\mu_1,\mu_2)=\LET(\mu_1,\mu_2),
    \end{equation}
    and $\aalpha(\tyY')=0$ for optimal solution $\aalpha\in \cM(\tyY)$ of
    Problem \ref{pr:3} or of {\em (\ref{eq:135},b)}.  Moreover,
  \begin{enumerate}[(i)]
  \item $\aalpha\in \cM(\tyY)$ is an optimal plan
    for {\em (\ref{eq:135},b)}
    if and only if $\aalpha(\tyY')=0$ 
    and
    $\frqd_\sharp(\aalpha\res\tY_\soo{\times} \tY_\soo)$ is an optimal
    plan for \eqref{eq:84}--\eqref{eq:198}. 
  \item $\bar\aalpha\in \cM(\yY)$ is any optimal plan for
    \eqref{eq:84bis} if and only if $\aalpha:=\frpd_\sharp\bar\aalpha$
    is an optimal plan for the Hellinger-Kantorovich Problem
    \ref{pr:3}.
  \item If $\ggamma\in \cM(X\times X)$ belongs to
    $\OptLET(\mu_1,\mu_2)$ and $\varrho_i:X\to [0,\infty)$ are Borel
    maps so that $\mu_i=\varrho_i\gamma_i+\mu_i^\perp$, then
    $\bbeta:=\big(\frpd\circ(x_1,\varrho_1^{1/2}(x_1);x_2,\varrho_2^{1/2}(x_2))\big)_\sharp
    \ggamma$ is an optimal plan for \eqref{eq:135}--\eqref{eq:375},
    and it satisfies
    $\sfr_1\sfr_2\cos(\sfd_{\pi/2}(\sfx_1,\sfx_2))=1$
    $\bbeta$-a.e.; in particular $\bbeta$ is concentrated on $\tG$.
  \item If $\aalpha \in \cM(\yY)$ is an optimal plan for Problem
    \ref{pr:3} then $\tilde\aalpha:=\aalpha\res \tG$ 
    is an optimal plan for {\rm (\ref{eq:135},b)}.
    Moreover, 
    \\ \textbullet\  
    the plan $\bbeta:=\dil{\vartheta,2}{\tilde\aalpha}$,
    with $\vartheta:=\big(\sfr_1\sfr_2\cos(\sfd_{\pi/2}(\sfx_1,\sfx_2))\big)^{1/2}$,
    is an optimal plan 
    satisfying $\sfr_1\sfr_2\cos(\sfd_{\pi/2}(\sfx_1,\sfx_2))=1$
    $\bbeta$-a.e.\\
    \textbullet\ If $(X,\tau)$ is separable and metrizable, \nc
    $\ggamma:=(\sfx_1,\sfx_2)_\sharp\bbeta$ belongs to
    $\OptLET(\mu_1,\mu_2)$,\\
    \textbullet\ If $(X,\tau)$ is separable and metrizable, \nc
    $\bbeta=\big(\frpd\circ(x_1,\varrho_1^{1/2}(x_1);x_2,
    \varrho_2^{1/2}(x_2))\big)_\sharp \ggamma$. 
  \end{enumerate}
\end{theorem}
\begin{proof}
  Identity \eqref{eq:383} and the first statement
  immediately follow by combining the previous
  Lemma \ref{le:restriction} 
  with Remark \ref{rem:lifting} and \eqref{eq:84bis}.
 
  If $\aalpha$ is an optimal plan for the formulation
  (\ref{eq:135},b) we can apply Lemma \ref{le:max-le}(iii)
  to find $\tilde\aalpha\ge\aalpha$ optimal for \eqref{eq:104},
  so that $\aalpha(\tyY')\le \tilde\aalpha(\tyY')=0$.
  
  Since all the optimal plans for $\HK$ do not charge $\tyY'$, 
  combining Lemma \ref{le:max-le}, Remark \ref{rem:lifting}
  and Theorems \ref{thm:mainHK2} and \ref{thm:mainHK3}  statements
  (i), (ii), and (iii) follow easily.

  Concerning (iv), the optimality of $\tilde\aalpha$ 
  is obvious from the formulation \eqref{eq:375} and the optimality of 
  $\bbeta =\dil{\vartheta,2}{\tilde\aalpha} $ follows from the
  invariance of \eqref{eq:375} with respect to dilations. 
  We notice that $\bbeta$-almost everywhere in $\tG$ we have
  \begin{align*}
    \sum_i\PE_0(\sfr_i^2)+\sfc(\sfx_1,\sfx_2)
    &=
      \sum_i \sfr_i^2-1-\log\sfr_i^2
      -\log(\cos^2(\sfd_{\pi/2}(\sfx_1,\sfx_2)))
    \\&
        =
        \sum_i\sfr_i^2
        -2-2\log(\sfr_1\sfr_2\cos(\sfd_{\pi/2}(\sfx_1,\sfx_2)))
    \\&
        =
        \sfr_1^2+\sfr_2^2-2\sfr_1\sfr_2\cos(\sfd_{\pi/2}(\sfx_1,\sfx_2)),
  \end{align*}
  so that by \eqref{eq:135} we arrive at
  \begin{equation}
    \label{eq:add1}
    \int
    \Big(\sum_i\PE_0(\sfr_i^2)+\sfc(\sfx_1,\sfx_2)\Big)\,\d\bbeta
    +\sum_i\big(\mu_i(X)-\chm2i\bbeta(X)\big)=
    \HK^2(\mu_1,\mu_2).
  \end{equation}
  Let us now set $\ggamma:=(\sfx_1,\sfx_2)_\sharp \bbeta\in
  \cM(X\times X)$ and $\beta_i:=\pi^i_\sharp\bbeta\in \cM(\tY)$,    
  \WWW which yield $\gamma_i:=\pi^i_\sharp\ggamma=(\sfx_i)_\sharp \bbeta=
  \sfx_\sharp\beta_i\in \cM(X)$ and 
  $\tilde\mu_i:=\chm2i\bbeta=(\sfx_i)_\sharp(r_i^2\ggamma)=
  \sfx_\sharp(r^2\beta_i)$.  
  \WWW Denoting by $(\beta_{i,x_i})_{x_i\in X}$ the disintegration
  of $\beta_i$ with respect to $\gamma_i$ 
  (here we need the metrizability and separability of
  $(X,\tau)$, see \cite[Section 5.3]{Ambrosio-Gigli-Savare08}), \nc we find
  \begin{displaymath}
    \int_X \zeta\,\d\tilde\mu_i=\int_\tY\zeta(\sfx)r^2\,\d\beta_i=
    \int_X\Big(\int_\tY\zeta(\sfx)r^2\,\d\beta_{i,x}\Big)\,\d\gamma_i=
    \int_X \zeta(x)\Big(\int_\tY r^2\,\d\beta_{i,x}\Big)\,\d\gamma_i
  \end{displaymath}
  for all $\zeta\in \rmB_b(X)$,  so that 
  \begin{displaymath}
   \WWW \tilde\mu_i=\tilde\varrho_i \gamma_i \leq \mu_i \quad
   \text{with }\  \tilde\varrho_i(x):= \int_\tY r^2\,\d\beta_{i,x}. \EEE
  \end{displaymath}
  Applying Jensen inequality we obtain
  \begin{align*}
    \int\PE_0(\sfr_i^2)\,\d\bbeta
    &=
      \int\PE_0(\sfr_i^2)\,\d\beta_i=
      \int\Big(\int \PE_0(r_i^2)\,\d\beta_{i,x_i}(r_i)\Big)\,\d\gamma_i
    \\&  \ge \int \PE_0\Big(\int
      r_i^2\,\d\beta_{i,x_i}(r_i)\Big)\,\d\gamma_i
      = \int \PE_0\big(\tilde\varrho_i(x)\big)\,\d\gamma_i.
  \end{align*}
  \WWW Now $ \int \sfc(\sfx_1,\sfx_2)\,\d\bbeta=
    \int\sfc(x_1,x_2)\,\d\ggamma$ and \eqref{eq:add1} imply \EEE
  \begin{displaymath}
    \HK^2(\mu_1,\mu_2)\ge 
    \sum_i\int_X \PE_0(\varrho_i)\,\d\gamma_i+
    \int_{X\times X}\sfc\,\d\ggamma+\sum_i\nu_i(X)
  \end{displaymath}
  with $\nu_i:=\mu_i- \WWW \tilde\mu_i \in \cM(X)$.
  \WWW Hence, $\mu_i = \tilde\varrho_i\gamma_i + \nu_i$ and the
  standard decomposition $\mu_i=\varrho_i\gamma_i+\mu_i^\perp$ (cf.\
  \eqref{eq:310}) imply  
  we get $\nu_i=
  \mu_i^\perp+(\varrho_i-\tilde\varrho_i)\gamma_i\ge
  \mu_i^\perp$. \WWW Hence, $\PE_0(s)= s-1-\log s$ and
  the monotonicity of the logarithm yield
  \begin{align*}
    \HK^2(\mu_1,\mu_2)&\geq 
                       \sum_i\Big(\int_X \PE_0(\tilde\varrho_i)\,\d\gamma_i 
                        +\nu_i(X)\Big)+\int\sfc\,\d\ggamma
                  \\&=
    \sum_i\Big(\int_X
    \Big(\PE_0(\tilde\varrho_i)+\varrho_i-\tilde\varrho_i\Big)\,\d\gamma_i+
    \mu_i^\perp(X)\Big)+\int\sfc\,\d\ggamma
    \\&\ge 
    \sum_i\Big(\int_X
    \PE_0(\varrho_i)\,\d\gamma_i+
    \mu_i^\perp(X)\Big)+\int\sfc\,\d\ggamma \ \geq \  \LET(\mu_1,\mu_2),
  \end{align*}
  \WWW where the last estimate follows from Theorem
  \ref{thm:mainHK1}(b). Above, the first inequality is strict if
  $\nu_i\neq \mu_i^\perp$ so that $\varrho_i>\tilde\varrho_i$ on some
  set with positive $\gamma_i$-measure.

  By the first statement of the Theorem it follows that $\ggamma\in
  \OptLET(\mu_1,\mu_2)$. \WWW Hence, all the inequalities are in fact
  identities, and we conclude $\tilde\varrho_i\equiv\varrho_i$. \EEE
  Since $\PE_0$ is strictly convex, the \WWW disintegration measure
  $\beta_{i,x_i}$ is a Dirac measure concentrated on
  $\sqrt{\varrho_i(x_i)}$, so that
  $\bbeta=\big(\frpd\circ(x_1,\varrho_1^{1/2}(x_1);
  x_2,\varrho_2^{1/2}(x_2))\big)_\sharp \ggamma$.
\end{proof}

We observe that the system $(\ggamma,\varrho_1,\varrho_2)$ 
provided by the previous Theorem enjoys 
a few remarkable properties, that are not obvious
from the original Hellinger-Kantorovich formulation.
\begin{enumerate}[a)]
\item First of all, the annihilated part $\mu_i^\perp$ of the measures
  $\mu_i$ is concentrated on the set
  \begin{displaymath}
    M_{i,j}:=\{x_i\in X: \sfd(x_i,\supp(\mu_j))\ge\pi/2\}
  \end{displaymath}
  When $\mu_i(M_{i,j})=0$ then $\mu_i\ll \gamma_i$.

\item As a second property, an optimal plan $\ggamma\in
  \OptLET(\mu_1,\mu_2)$ provides an optimal plan
  $\aalpha=\big(\frpd\circ(x_1,\varrho_1^{1/2}(x_1);x_2,
  \varrho_2^{1/2}(x_2))\big)_\sharp \ggamma$ which is concentrated on
  the graph of the map $(\varrho_1^{1/2}(x_1);\varrho_2^{1/2}(x_2) )$
  from $X\times X$ to $\R_+\times \R_+$, where the maps $\varrho_i$ are
  independent, in the sense that $\varrho_i$ only depends on $x_i$.

\item A third important application of Theorem
  \ref{thm:main-equivalence} is the duality formula for the $\HK$
  functional which directly follows from \eqref{eq:204} of Theorem
  \ref{thm:mainHK2}.  We will state it in a slightly different form in
  the next theorem, whose interpretation will be clearer in the light
  of Section \ref{subsec:HJ}. It is based on the inf-convolution formula
  \begin{equation}
    \label{eq:427}
    \hspace*{-1em}\HJ{\pi/2}{1}\xi(x)=\inf_{x'\in X}\!\!\left(
    \frac{\xi(x')}{1{+}2\xi(x')}+
    \frac{\sin^2(\sfd_{\pi/2}(x,x'))}{2(1+2\xi(x'))}\right)
    =\inf_{x'\in X}\frac12
    \Big(1-\frac{\cos^2(\sfd_{\pi/2}(x,x'))}{ 1+2\xi(x')}\Big).
  \end{equation}
  where $\xi\in \rmB(X)$ with $\xi>-1/2$.
\end{enumerate}

\begin{theorem}[Duality formula for $\HK$]
  \label{thm:dualityHK} \mbox{}\\[-1.8em]
\begin{enumerate}[(i)]
\item If $\xi\in \rmB_b(X)$ with $\inf_X \xi>-1/2$ then the function
  $\PP_1\xi$ defined by \eqref{eq:427} belongs to $\Lip_b(X)$,
  satisfies $\sup_X \PP_1\xi<1/2$, and admits the equivalent
  representation
  \begin{equation}
    \label{eq:61}
    \PP_1\xi(x)=\inf_{x'\in B_{\pi/2}(x)}\frac12 
    \Big(1-\frac{\cos^2(\sfd_{\pi/2}(x,x'))}{ 1+2\xi(x')}\Big).
  \end{equation}
  In particular, if $\xi$ has bounded support then $\PP_1\xi\in
  \Lip_{bs}(X)$, the space of Lipschitz functions with bounded
  support.

\item Let us suppose that $(X,\sfd)$ is a separable metric space and
  $\tau$ is induced by $\sfd$.  For every $\mu_0,\mu_1\in \cM(X)$ we
  have
  \begin{equation}
    \label{eq:406}
    \frac 12 \HK^2(\mu_0,\mu_1)=
     \sup \Big\{
    \int\HJ{\pi/2}1\xi\,\d\mu_1-\int\xi\,\d\mu_0 \ :\ 
    \xi\in \Lip_{bs}(X),\ \inf_X \xi>-1/2\Big\}.
  \end{equation}
\end{enumerate}
\end{theorem}
\begin{proof}
  Let us first observe that if
  \begin{equation}
    \label{eq:82}
    -\frac12<a\le \xi\le b\ \text{in $X$}\quad\Rightarrow\quad
    \frac a{1+2a}\le \PP_1\xi\le \frac b{1+2b} \text{ in } X,
  \end{equation}
  \WWW where the upper bound follows using $x'=x$, while the lower
  bound is easily seen from the first form of $\PP_1\xi$ in
  \eqref{eq:427} and $\sin^2\geq 0$. 
  Since $1/(1+2\xi(x'))\le 1/(1+2a)$ for every $x'\in X$, 
  the function $\PP_1\xi$ is also Lipschitz, because it is the infimum of 
  a family of uniformly Lipschitz functions.
  
  \WWW Moreover, for $ \sfd(x,x')\ge \pi/2 $ we have the estimate \EEE
  \begin{equation}
    \label{eq:85}
  \frac12\Big( 1-\frac{\cos^2(\sfd_{\pi/2}(x,x'))}{ \WWW 1+ 2  
     \xi(x')  } \Big)
  =\frac12 >\frac b{1+2b}\quad\text{if }\sfd(x,x')\ge \pi/2,
  \end{equation}
  which immediately gives \eqref{eq:61}. In particular, we have 
  \begin{equation}
    \label{eq:184}
    \xi\equiv 0\quad\text{in }X\setminus B\quad
    \Rightarrow\quad
    \PP_1\xi\equiv 0\quad\text{in }\{x\in X:\sfd(x,B)\ge \pi/2\}.
  \end{equation}

  Let us now prove statement (ii).  We denote by $E$ the the
  right-hand side of \eqref{eq:406} and by $E'$ the analogous
  expression where $\xi$ runs in $\rmC_b(X)$:
  \begin{equation}
    \label{eq:185}
    E':= 2\,\sup \Big\{
      \int\HJ{\pi/2}1\xi\,\d\mu_1-\int\xi\,\d\mu_0:
      \xi\in \rmC_b(X),\ \inf_X \xi>-1/2\Big\}.
  \end{equation}
  It is clear that $E'\ge E$.
  If $\xi\in \rmC_b(X)$ with $\inf \xi>-1/2$, setting 
  $\psi_1(x_1):=-2\xi(x_1)$, 
  $\psi_2(x_2):=2(\HJ{\pi/2}1\xi)(x_2)$, 
  we know that $\sup_X\psi_2<1$ and $\psi_2\in \Lip_b(X)$. 
  Thus, $\psi_1$ and $\psi_2$ are continuous
  and satisfy
  \begin{displaymath}
    \big(1-\psi_2(x_2)\big)\big(1-\psi_1(x_1)\big)\ge \cos^2(\sfdpt(x_1,x_2)).
  \end{displaymath}
  Hence, the couple $(\psi_1, \psi_2)$ is admissible for
  \eqref{eq:204} (with $\rmC_b(X)$ instead of $\LSC_s(X)$; note that
  $\tau$ is metrizable and thus completely regular), so that
  $\HK^2(\mu_0,\mu_1) \WWW =\LET(\mu_0,\mu_1) \ge E'$.

  On the other hand, if $(\psi_1,\psi_2)\in \rmC_b(X)\times \rmC_b(X)$
  with $\sup_X \psi_i<1$, setting $\xi_1=-\frac12\psi_1$ and 
  $\tilde\xi_2:=\PP_1(-\xi_1)$
  we see that $2\tilde\xi_2\ge \psi_2$ giving $E'\ge
  \HK^2(\mu_0,\mu_1)$, and $E=E'$ follows.

  To show that $E=E'$ in the general case, we approximate $\psi\in
  \rmC_b(X)$ with $\inf_X\psi>-1$ by a decreasing sequence of
  Lipschitz and bounded functions (e.g.~by taking $\psi_n(x):=\sup_y
  \psi(y)-n\sfd_\pi(x,y)$) and use that the supremum in \eqref{eq:185}
  does not change if we restrict it to $\Lip_b(X)$.
  
  Let now $\xi$ be Lipschitz and valued in $[a,b]$ with $-1/2<a\le
  0\le b$.  Taking the increasing sequence of nonnegative cut-off
  functions $\zeta_n(x):=0\lor \big(n-\sfd(x,\bar x))\land 1$ which
  are uniformly $1$-Lipschitz, have bounded support and satisfy
  $\zeta_n\up 1$ as $n\to\infty$, it is easy to check that
  $\xi_n:=\zeta_n\xi$ belong to $\Lip_{bs}(X)$ and take values in the
  interval $[a,b]$ so that $\frac a{1+2a}\le \PP_1\xi_n\le \frac
  b{1+2b}$ for every $n\in \N$.

  Since $\xi_n(x)= 0$ if $\sfd(x,\bar x)\ge n$ and
  $\xi_n(x)=\xi(x)$ if $\sfd(x,\bar x)\le n-1$, by \eqref{eq:61}
  we get
  \begin{equation}
    \label{eq:62}
    \PP_1\xi_n(x)=0\quad\text{if }x\ge n+\pi/2,\quad
    \PP_1\xi_n(x)=\PP_1\xi(x)\quad\text{if }x<n-1-\pi/2.
  \end{equation}
  Thus $\PP_1\xi_n\in \Lip_{bs}(X)$ and the Lebesgue Dominated
  Convergence theorem shows that 
  \begin{displaymath}
    \lim_{n\to\infty}\int_X
    \PP_1\xi_n\,\d\mu_1-\int_X\xi_n\,\d\mu_0=
    \int_X \PP_1\xi\,\d\mu_1-\int_X\xi\,\d\mu_0.\qedhere
  \end{displaymath}
\end{proof}

\subsection{Limiting cases: recovering the Hellinger--Kakutani distance
  \newline and  the
  Kantorovich--Wasserstein  distance}
\label{subsec:limiting}

In this section we will show that we can recover the
Hellinger-Kakutani and the Kantorovich-Wasserstein distance by
suitably rescaling the $\HK$ functional.

\paragraph{The Hellinger-Kakutani distance.}
As we have seen in Example E.5 of Section \ref{ex:1}, 
the Hellinger-Kakutani distance 
between two measures $\mu_1,\mu_2\in \cM(X)$
can be obtained as a limiting case when 
the space $X$ is endowed with the discrete distance 
\begin{equation}
  \label{eq:373}
  \sfd_{\sf Hell}(x_1,x_2):=
  \begin{cases}
    a&\text{if }x_1\neq x_2\\
    0&\text{if }x_1=x_2,
  \end{cases}
  \qquad
  \text{with }a\in[\pi,+\infty].
\end{equation}
The induced cone distance in this case is 
\begin{equation}
  \label{eq:369}
  \sfdc^2([x_1,\s_1],[x_2,\s_2])=
  \begin{cases}
    (\s_1-\s_2)^2&\text{if }x_1=x_2,\\
    \s_1^2+\s_2^2&\text{if }x_1\neq x_2.
  \end{cases}
\end{equation}
and the induced cost function 
for the Entropy-Transport formalism is given by
\begin{equation}
  \label{eq:377}
  \sfc_{\sf Hell}(x_1,x_2):=
  \begin{cases}
    0&\text{if }x_1=x_2,\\
    +\infty&\text{otherwise.}
  \end{cases}
\end{equation}
Recalling \eqref{eq:57}--\eqref{eq:58} we obtain
\begin{equation}
  \label{eq:376}
  \Hell^2(\mu_1,\mu_2)=\LET_{\sf Hell}(\mu_1,\mu_2)=
  \int_X \!\!\left(\sqrt{\varrho_1} -
    \sqrt{\varrho_2} 
  \right)^2\,\d\gamma  \WWW \text{ with }
   \mu_i=\varrho_i\gamma\ll\gamma \in \cM(X).
\end{equation}
Since $\sfc_{\Hell}\ge \sfc=\ell(\sfd)$ for every distance function on
$X$, we always have the upper bound
\begin{equation}
  \label{eq:401}
  \HK(\mu_1,\mu_2)\le \Hell(\mu_1,\mu_2)\quad
  \forevery \mu_1,\mu_2\in \cM(X).
\end{equation}
Applying Lemma \ref{le:lsc} we easily get
\begin{theorem}[Convergence of $\HK$ to 
  $\Hell$]
  Let $(X,\tau,\sfd)$ be an extended metric topological space
  and let $\HK_{\lambda\,\sfd}$ be the Hellinger-Kantorovich distances
  in $\cM(X)$ induced by the distances $\sfd_\lambda:=\lambda\sfd$, $\lambda>0$.
  For every couple $\mu_1,\mu_2\in \cM(X)$ we have
  \begin{equation}
    \label{eq:381}
    \HK_{\lambda\sfd}(\mu_1,\mu_2)\uparrow 
    \Hell(\mu_1,\mu_2)\quad\text{as }\lambda\up\infty.
  \end{equation}
\end{theorem}
\paragraph{The Kantorovich--Wasserstein distance.}
Let us first observe that 
whenever $\mu_1,\mu_2\in \cM(X)$ have the same mass
their $\HK$-distance is always bounded form above by the
Kantorovich-Wasserstein distance $\Wd$ 
(the upper bound is trivial when $\mu_1(X)\neq \mu_2(X)$, since in
this case
$\Wd(\mu_1,\mu_2)=+\infty$).
\begin{proposition}
  \label{prop:HKleW}
  For every couple $\mu_1,\mu_2\in \cM(X)$ we have
  \begin{equation}
    \label{eq:385}
    \HK(\mu_1,\mu_2)\le 
    \sfW_{\sfd_{\pi/2}}(\mu_1,\mu_2)
    \le \sfW_{\sfd}(\mu_1,\mu_2).
  \end{equation}
\end{proposition}
\begin{proof}
  It is not restrictive to assume that
  $\sfW_{\sfd_{\pi/2}}^2(\mu_1,\mu_2) =
  \int\sfd_{\pi/2}^2\ggamma<\infty$ for an optimal plan $\ggamma$ with
  marginals $\mu_i$.  We then define the plan $\aalpha:=\frs_\sharp
  \ggamma\in \cM(\tY\times \tY)$ where
  $\frs(x_1,x_2):=([x_1,1],[x_2,1])$, so that
  $\chm2{i}\aalpha=\mu_i$. By using \eqref{eq:404} and
  \eqref{eq:94pre} we obtain
  \begin{displaymath}
    \HK^2(\mu_1,\mu_2)\le 
    4\int_{\tyY}\sin^2(\sfdpt(\sfx_1,\sfx_2)/2)\,\d\aalpha
    \le 
    \int_{\boldsymbol X}\sfdpt^2(x_1,x_2)\,\d\ggamma
    \le \sfW^2_{\sfd_{\pi/2}}(\mu_1,\mu_2).\qedhere
  \end{displaymath}
\end{proof}
In order to recover the Kantorovich-Wasserstein distance 
we perform a simultaneous scaling,
by taking the limit of $n\HK_{\sfd/n}$ where
$\HK_{\sfd/n}$ is induced by 
the distance $\sfd/n$.
\begin{theorem}[Convergence of $\HK$ to $\sfW$]
  Let $(X,\tau,\sfd)$ be an extended metric topological space
  and let $\HK_{\sfd/\lambda}$ be the Hellinger-Kantorovich distances
  in $\cM(X)$ induced by the distances $\lambda^{-1} \sfd$ for $\lambda>0$.
  Then, for all $\mu_1,\mu_2\in \cM(X)$ we have
  \begin{equation}
    \label{eq:385bis}
    \lambda\HK_{\sfd/\lambda}(\mu_1,\mu_2)\up
    \sfW_\sfd(\mu_1,\mu_2)\quad\text{as }\lambda\up\infty.
  \end{equation}
\end{theorem}
\begin{proof}
  Let us denote by $\LET_\lambda =\HK_{\sfd/\lambda }^2$ the optimal
  value of the LET-problem associated to the distance $\sfd/\lambda $.
  Since the Kantorovich-Wasserstein distance is invariant by the
  rescaling $\lambda \sfW_{\sfd/\lambda }=\sfW_\sfd$, estimate
  \eqref{eq:385} shows that $\lambda \HK_{\sfd/\lambda } \le \sfW_\sfd$.

  Since $x\mapsto \sin(x\land \pi/2)$ is concave in $[0,\infty)$, the
  function $x\mapsto \sin(x\land \pi/2)/x$ is decreasing in
  $[0,\infty)$, so that $\alpha\sin((d/\alpha)\land \pi/2)\le
  \lambda\sin((d/\lambda)\land \pi/2)$ for every $d\ge0$ and
  $0<\alpha<\lambda $. Combining \eqref{eq:404} with
  \eqref{eq:388bis} we see that the map $\lambda \mapsto \lambda
  \HK_{\sfd/\lambda }(\mu_1,\mu_2)$ is nondecreasing.
  
  It remains to prove that
  $L:=\lim_{\lambda \to\infty}\lambda \HK_{\sfd/\lambda} (\mu_1,\mu_2)=\sup_{\lambda \ge1}
  \lambda \HK_{\sfd/\lambda }(\mu_1,\mu_2)\ge \sfW_\sfd(\mu_1,\mu_2)$.
  For this, it is not restrictive to assume that $L$ is finite.

  Let $\ggamma_\lambda $ be an optimal plan for $\HK_{\sfd/\lambda
  }(\mu_1,\mu_2)$ with marginals $\gamma_{\lambda ,i}=\pi^i_\sharp
  \ggamma_\lambda $.  We denote by $\FF$ the entropy functionals
  associated to logarithmic entropy $F(\r)=\PE_1(\r)$ and by $\GG$ the
  entropy functionals associated to $F(\r):=\rmI_1(\r)$ as in Example
  E.3 of Section \ref{ex:1}. Since the transport part of the
  LET-functional is associated to the costs
  \begin{displaymath}
    \sfc_\lambda (x_1,x_2)=\lambda ^2\ell(\sfd(x_1,x_2)/\lambda
    )\topref{eq:159}\ge \sfd^2(x_1,x_2), 
  \end{displaymath}
  we obtain the estimate
  \begin{equation}
    \label{eq:391}
    L^2\ge \lambda^2\LET_{\kern-1pt \lambda }(\mu_1,\mu_2)\ge \sum_i
    \lambda ^2\FF(\gamma_{\lambda ,i}|\mu_i)+ 
    \int_{\sxX} \sfd^2(x_1,x_2)\,\d\ggamma_\lambda .
  \end{equation}
  Proposition \ref{prop:DLT} 
  shows that 
  the family of plans $(\ggamma_{\lambda })_{\lambda\ge1} $ is relatively compact 
  with respect to narrow convergence in 
  $\cM(X\times
  X)$. 
  Since $\lambda ^2 F(\r)\up \rmI_1(\r)$, 
  passing to the limit along a suitable subnet
  $(\lambda(\alpha))_{\alpha\in \A}$ 
  parametrized by a directed set $\A$, 
  and applying Corollary \ref{cor:limit1}
  we get a limit plan $\ggamma\in \cM(X\times X)$ with marginals $\gamma_i$
  such that 
  \[
    \label{eq:398}
    \sum_i\GG(\gamma_i|\mu_i) \WWW \leq L^2 ,\quad\text{which
      implies } \quad \gamma_i=\mu_i.
  \]
  \WWW In particular, we conclude that $\mu_1(X)= \WWW \ggamma
  (X\times X)= \mu_2(X)$.
  Since $\sfd$ is lower semicontinuous, narrow convergence of 
  $\ggamma_{\lambda(\alpha)}$ and \eqref{eq:391} also yield
  \begin{displaymath}
    L^2
    \ge \liminf_{\alpha\in \A}
    \int_{\sxX} \sfd^2(x_1,x_2)\,\d\ggamma_{\lambda(\alpha)}
    \ge   \int_{\sxX} \sfd^2(x_1,x_2)\,\d\ggamma\ge
    \sfW_\sfd^2(\mu_1,\mu_2). \qedhere
  \end{displaymath}
\end{proof}

\subsection{The Gaussian Hellinger-Kantorovich distance}
\label{subsec:GHK}
We conclude this general introduction to the Hellinger-Kantorovich
distance
by discussing another interesting  example.

We consider the inverse function $\elli:\R_+\to[0,\pi/2)$ of $\sqrt{\ell}$:
\begin{equation}
  \label{eq:456}
  \elli(z):=\arccos(\rme^{-z^2/2}),\quad\text{satisfying}\quad 
  \elli(0)=0, \ 
 \elli'(0)=1, \  
 \ell(\elli(d))=d^2. 
\end{equation}
Since $\sqrt \ell$ is a convex function, $\elli$ is a concave
increasing function in $[0,\infty)$ with $\elli(z)\le z$ and
$\lim_{z\to\infty}\elli(z)=\pi/2$.

It follows that $\sfdg:=\elli\circ \sfd$ is a distance in $X$, inducing
the same topology as $\sfd$. We can now introduce 
a distance $\HK_\sfdg$ associated to $\sfdg$. The corresponding distance
on $\tY$ is given by
\begin{equation}
  \label{eq:458}
  \sfdg_{\tY}(\ty_1,\ty_2):=\sfs_1^2+\sfs_2^2-2\sfs_1\sfs_2
  \exp(-\sfd^2(\sfx_1,\sfx_2)/2).  
\end{equation}
From $g(z)\leq z$ we have  $\sfdg_\tY\le \sfdc$.

\begin{theorem}[The Gaussian Hellinger-Kantorovich distance]
  The functional
  \begin{equation}
    \label{eq:457}
    \GHK^2(\mu_1,\mu_2):=
        \HK^2_\sfdg(\mu_1,\mu_2)=\min\Big\{\int
        \sfdg_{\tY}^2(\ty_1,\ty_2)\,\d\aalpha \ : \  
    \aalpha\in \cM(\tyY),\ \chm2{i}\aalpha=\mu_i\Big\}
  \end{equation}
  defines a distance on $\cM(X)$ dominated by $\HK$. 
  If $(X,\sfd)$ is 
  separable (resp.~complete) 
  then $(\cM(X),\GHK)$ is a separable (resp.~complete) metric
  space,
  whose topology coincides with the weak convergence.
  We also have
  \begin{equation}
    \label{eq:459}
    \begin{aligned}
      \GHK^2(\mu_1,\mu_2)&=\min\Big\{\sum_i\FF(\gamma_i|\mu_i)
      +\int_{\boldsymbol X}\sfd^2(x_1,x_2)\,\d\ggamma\ :\ \ggamma\in
      \cM(\xX)\Big\}\\ 
      &=\sup\Big\{\sum_i\int \Big(1-\rme^{-\varphi_i}\Big)\,\d\mu_i\ :\
      \varphi_1\oplus\varphi_2\le \sfd^2\Big\}.
    \end{aligned}
  \end{equation}
\end{theorem}
We shall see in the next Section \ref{subsec:lifting} that
$\HK$ is the length distance induced by $\GHK$ \WWW if $\sfd$ is a
length distance on $X$. \EEE

\section{Dynamic interpretation of the\\ Hellinger-Kantorovich distance}
\label{sec:dynamic}
 As in Section \ref{subsec:MT}, in all this chapter we will suppose that $(X,\sfd)$ is a complete
and separable (possibly extended) metric space and $\tau$ coincides with the topology
induced by $\sfd$. All the results admits a natural
generalization
to the framework of extended metric-topological spaces 
\cite[Sec.~4]{Ambrosio-Erbar-Savare15}. \nc 

\subsection{Absolutely continuous curves and geodesics in the cone $\tY$}
\label{subsec:geodesic}

\paragraph{Absolutely continuous curves and metric derivative.}
If $(Z,\sfd_Z)$ is a (possibly extended) metric space and $I$ is an
interval of $\R$, a curve
$\rmz:I\to Z$ is absolutely continuous if there exists $m\in
\rmL^1(I)$ such that 
\begin{equation}
  \label{eq:365}
  \sfd_Z(\rmz(t_0),\rmz(t_1))\le \int_{t_0}^{t_1}m(t)\,\d t\quad
  \text{whenever }t_0,t_1\in I,\ t_0<t_1.
\end{equation}
Its metric derivative $|\rmz'|_{\sfd_Z}$ 
(we will omit the index $\sfd_Z$ when the choice of the metric is
clear from the context)
is the Borel function defined by
\begin{equation}
  \label{eq:366}
  |\rmz'|_{\sfd_Z}(t):=\limsup_{h\to0}\frac{\sfd_Z(\rmz(t+h),\rmz(t))}{|h|}
\end{equation}
and it is possible to show (see \cite{Ambrosio-Gigli-Savare08})
that the $\limsup$ above is in fact a limit for $\Leb 1$-a.e.~points
in $I$ and it provides the minimal (up to possible modifications in
$\Leb 1$-negligible sets) function $m$ for which
\eqref{eq:365} holds.
We will denote by $\AC^p(I;Z)$ the class of all absolutely
continuous curves $\rmz:I\to Z$ with $|\rmz'|\in \rmL^p(I)$; when $I$ is an
open set of $\R$, we will also 
consider the local space $\AC^p_{loc}(I;Z)$.
If $Z$ is complete and separable then $\AC^p([0,1];Z)$ is 
a Borel set in the space $\rmC([0,1];Z)$ endowed with the 
topology of uniform convergence. (This property
can be extended to the framework of extended metric-topological
spaces, see \cite{Ambrosio-Gigli-Savare14}.)

A curve $\rmz:[0,1]\to Z$ is a (minimal, constant speed) geodesic if
\begin{equation}
  \label{eq:364}
  \sfd_Z(\rmz(t_0),\rmz(t_1))=|t_1-t_0|\sfd_Z(\rmz(0),\rmz(1))\quad \text{for
    every }t_0,t_1\in [0,1].
\end{equation}
In particular $\rmz$ is Lipschitz and $|\rmz'|\equiv \WWW
\sfd_Z(\rmz(t_0),\rmz(t_1)) $ in $[0,1]$. 
We denote by $\Geo Z\subset \rmC([0,1];Z)$ the closed 
subset of all the geodesics.

A metric space $(Z,\sfd_Z)$ is called a length (or intrinsic) space if
the distance between arbitrary couples of points
can be obtained as the infimum of the length of 
the absolutely continuous curves connecting them.
It is called a geodesic (or strictly intrinsic) space if 
every couple of points $z_0,z_1$ at finite distance can be joined
by a geodesic.

\paragraph{Geodesics in $\tY$.}
If $(X,\sfd)$ is a geodesic (resp.~length) space, then also $\tY$ is a
geodesic (resp.~length) space, cf.\
\cite[Sec.\,3.6]{Burago-Burago-Ivanov01}. \WWW The geodesic
connecting a point $\ty=[x,\rp]$ with $\fro$ is
\begin{equation}
  \label{eq:96}
  \ty(t)=[x,t \rp]=\ty\cdot t \ \text{ for } t\in [0,1].
\end{equation}
If $x_1,x_2\in X$ with $\sfd(x_1,x_2)\ge \pi$, then 
a geodesic
between $\ty_i=[x_i,\rp_i]$ can
be easily obtained by joining two geodesics connecting
$\ty_i$ to $\fro$ 
as before; observe that in this case $\sfdc(\ty_1,\ty_2)=\rp_1+\rp_2$.

In the case when $\sfd(x_1,x_2)<\pi$ and $\rp_1,\rp_2>0$,
every geodesic $\ty:I\to\tY$ connecting
$\ty_1$ to $\ty_2$ is associated to a geodesic $\rmx$ 
in $X$ joining $x_1$ to $x_2$ and
parametrized with unit speed
in the interval $[0,\sfd(x_1,x_2)]$. \WWW To find the radius $\rp(t)$,
we use the complex plane $\C$: we write the curve
connecting $z_1={\rp_1} \in \C$ to $z_2={\rp_2}\exp(\mathrm i \,\sfd(x_1,x_2))
\in \C$ \WWW in polar coordinates, namely \EEE
\begin{equation}
  \label{eq:97}
  \rmz(t)=\rp(t)\exp (\mathrm i\,\theta(t)),\quad
  \left\{\begin{aligned}
    &\rp^2(t)=(1{-}t)^2\rp_1^2+t^2\rp_2^2
    +2t(1{-}t){\rp_1\rp_2}\cos(\sfd(x_1,x_2)),
    \\
    \!&\cos(\theta(t))=
    \frac{(1{-}t){\rp_1}+
      t{\rp_2}\cos(\sfd(x_1,x_2))}{\rp(t)},\ \ \theta(t)\in [0,\pi],
  \end{aligned}
  \right.
\end{equation}
and then \WWW the geodesic curve in $\tY$ takes the form 
\begin{equation}
  \label{eq:98}
  \ty(t)=[\rmx(\theta(t)),\rp(t)].
\end{equation}

\paragraph{Absolutely continuous curves in $\tY$.}
We want to obtain now a simple characterizations of absolutely
continuous curves in $\tY$.  If $t\mapsto \ty(t)$ is a continuous
curve in $\tY$, with $t\in [0,1]$, is clear that
$\rmr(t):=\sfr(\ty(t))$ is a continuous curve with values in
$[0,\infty)$. We can then consider the open set
$O_\rmr=\rmr^{-1}\big((0,\infty)\big)$ and the map $\rmx:[0,1]\to X$
defined by $\rmx(t):=\sfx(\ty(t))$, whose restriction to $O_\rmr$ is
also continuous.  Thus any continuous curve $\ty:I\to\tY$ can be
lifted to a couple of maps $\rmy=\frq\circ \ty=(\rmx,\rmr ):[0,1]\to
\pY$ with $\rmr$ continuous and $\rmx$ continuous on $O_\rmr$ and
constant on its complement.  Conversely, it is clear that starting
from a couple $\rmy=(\rmx,\rmr)$ as above, then $\ty=\frp\circ \rmy$
is continuous in $\tY$. We thus introduce the set
\begin{equation}
  \label{eq:114}
  \trmC([0,1];\pY):=\big\{\rmy=(\rmx,\rmr ):[0,1]\to \pY\ : \
  \rmr\in \rmC([0,1];\R_+),\ 
  \rmx\restr{O_\rmr}\text{ is continuous }\big\}
\end{equation}
and \WWW for $p\geq 1$ the analogous spaces 
\begin{equation}
  \label{eq:115}
  \begin{aligned}
    \tACp p([0,1];\pY):=\Big\{\rmy=(\rmx,\rmr )\ :\ {} &
    \rmr\in
    \AC^p([0,1];\R_+),\\
    &\rmx\restr{O_\rmr}\in \AC^p_{loc}(O_\rmr;X), \ 
    \rmr |\rmx'|\in \rmL^p(O_\rmr)\Big\}.
  \end{aligned}
\end{equation}
If $\rmy=(\rmx,\rmr )\in \tACp p([0,1];\pY)$ we define the Borel map
$|\rmy'|:[0,1]\to\R_+$ by
\begin{equation}
  \label{eq:187}
  |\rmy'|^2(t):={|\rmr'(t)|^2}
    +
    \rmr^2(t) | \rmx'|_\sfd^2(t)\quad \text{if }t\in
    O_\rmr,\quad
    |\rmy'|(t)=0\ \text{otherwise}.
\end{equation}
For absolutely continuous curves the following characterization holds: 

\renewcommand{\dot}[1]{#1'}

\begin{lemma}
  Let $\ty\in \rmC([0,1];\tY)$ be
  lifted to $\rmy=\frq\circ\ty\in \trmC([0,1];\pY)$. Then
  $\ty\in \AC^p(I;\tY)$ 
  if and only if 
  $\rmy=(\rmx,\rmr )\in \tACp p([0,1];\pY)$ and
  \begin{equation}
    \label{eq:116}
    |\ty'|_{\sfdc}(t)=|\rmy'|(t)\quad\text{for $\Leb 1$-a.e.~$t\in [0,1]$}.
  \end{equation}
\end{lemma}
\begin{proof}
  By \eqref{eq:94} one immediately sees that if
  $\ty=\frp\circ \rmy\in \AC^p([0,1];\tY)$ then $s\rmr$
  belongs to $\AC^p([0,1];\R)$ and $\rmx\in \AC^p_{\rm loc}(O_\rmr;X)$.
  Since $\ty$ is absolutely continuous, 
  we can evaluate the metric
  derivative at a.e.~$t\in O_\rmr$ where
  also $\dot \rmr$ and $|\dot \rmx|$ exist:
  starting from \eqref{eq:94pre} leads to the limit
  \begin{align*}
    \lim_{h\down0}\frac{\sfdc^2(\ty(t+h),\ty(t))}{h^2}&=
    \lim_{h\down0}\frac{|{\rmr(t+h)}-{\rmr(t)}|^2+4{\rmr(t+h)\rmr(t)}\sin^2(\frac
                                                    12 \sfdp
    (\rmx(t+h),\rmx(t)))}{h^2}
    \\&=
        {|\dot \rmr(t)|^2}
        +s(t)|\dot \rmx|_\sfd^2(t)
  \end{align*}
  which provides \eqref{eq:116}.

  Moreover, the same calculations show 
  that if the lifting $\rmy$ belongs to $\tACp p([0,1];\pY)$ then
  the restriction of $\ty$ to each connected component of 
  $O_\rmr$ is absolutely continuous with metric velocity 
  given by \eqref{eq:116} in $\rmL^p(0,1)$. 
  Since $\ty$ is globally continuous and constant
  in $[0,1]\setminus O_\rmr$, we conclude that 
  $\ty\in \AC^p([0,1];\tY)$.
\end{proof}

As a consequence, in a length space,
we get the variational representation formula
\begin{equation}
  \label{eq:117}
  \begin{aligned}
    \sfdc^2(\ty_0,\ty_1)= \inf\Big\{\int_{[0,1]\cap \{\rmr>0\}}&\Big(
    \rmr^2(t) |\dot \rmx|_\sfd^2(t)+
    {|\dot \rmr(t)|^2}
    \Big) \,\d t:\\
    &(\rmx,\rmr )\in \tACp
    2([0,1];\pY), \ [\rmx(i),\rmr(i)]=\ty_i,\ i=0,1\Big\}.
  \end{aligned}
\end{equation}

\begin{remark}[The Euclidean case]
  \upshape
  \label{rem:Euclidean1}
  Consider the case $X=\R^d$ with the usual Euclidean distance
  $\sfd(x_1,x_2):=|x_1-x_2|$. For $\ty=[\rmx,\rmr]\in
  \AC^2([0,1];\tY)$, we can define a Borel vector field
  $\ty_\tY':[0,1]\to \R^{d+1}$ by
  \begin{equation}
    \label{eq:439}
    \ty_\tY'(t):=
    \begin{cases}
      (\rmr(t)\rmx'(t),\rmr'(t))&\text{whenever }\rmr(t)\neq 0\text{ and the
        derivatives exist,}\\ 
      (0,0)&\text{otherwise.}
    \end{cases}
  \end{equation}
  Then, \eqref{eq:116} yields $|\ty'|_{\sfdc}(t)=|\ty_\tY'(t)|_{\R^{d+1}}$
  for $\Leb 1$-a.e.~$t\in (0,1)$. 

  For 
  $\psi\in \rmC^1(\R^d\times[0,1])$ we set $\zeta([x,\s ],t):=\frac 12
  \psi(x,t)r^2$ and obtain 
  $\partial_t\zeta([x,\s ],t):=\frac 12\partial_t\psi(x,t)r^2$. Now
  defining the Borel map $ \rmD_\tY\zeta:\tY\to (\R^{d+1})^*$ via 
  \begin{equation}
    \label{eq:441}
    \rmD_\tY\zeta(\ty,t):=
    \begin{cases}
      (\frac 12 \s \rmD_x\psi(x,t),\s\psi(x,t)) 
     &\text{for }\ty\neq\fro
      ,\\
      (0,0)&\text{otherwise},
    \end{cases}
  \end{equation}
  we see that the map $t\mapsto \zeta(\ty(t),t)$ is absolutely
  continuous and satisfies
  \begin{equation}
    \label{eq:440}
    \frac\d{\d t}\zeta(\ty(t),t)=\frac 12\partial_t\zeta(\ty(t),t)
    +\langle \rmD_\tY\zeta(\ty(t),t), \ty_\tY'(t)\rangle_{\R^{d+1}}
    \quad \text{$\Leb 1$-a.e.~in $(0,1)$.}\qedhere
  \end{equation}
Note that the first component of $\rmD_\tY \zeta $ contains the
factor $r$ rather than $r^2$, since $\mathfrak y'_\tY$ in (8.12)
already has one factor $r$ in its first component. \nc
\end{remark}

\subsection{Lifting of absolutely continuous curves and geodesics}
\label{subsec:lifting}
\paragraph{Dynamic plans and time-dependent marginals.}
Let $(Z,\sfd_Z)$ be a complete and separable metric space.
A dynamic plan $\ppi$ in $Z$ is a probability measure in
$\cP(\rmC(I;Z))$, and we say that $\ppi$ has finite $2$-energy if it
is concentrated on $\AC^2(I;Z)$ and
\begin{equation}
  \label{eq:169}
  \int \Big(\int_0^1 |\dot \rmz|_{\sfd_Z}^2(t)\,\d t\Big)\,\d\ppi(\rmz)<\infty.
\end{equation}
We denote by $\eval_t$ the evaluation map in $\rmC(I;Z)$ given by
$\eval_t(\rmz):=\rmz(t)$.  If $\ppi$ is a dynamic plan,
$\alpha_t=(\eval_t)_\sharp\ppi\in \cM(Z)$ is its marginal at time
$t\in I$ and the curve $t\mapsto \alpha_t$ belongs to
$\rmC(I;(\cM(Z),\sfW_{\kern-1pt \sfd_Z}))$.  If moreover $\ppi$ is a
dynamic plan with finite $2$-energy, then
$\alpha\in\AC^2(I;(\cM(Z),\sfW_{\kern-1pt \sfd_Z}))$.

We say that $\ppi$ is an \emph{optimal geodesic plan} 
between $\alpha_0,\alpha_1\in \cP(Z)$ \WWW
if $(\eval_i)_\sharp\ppi=\alpha_i$ for $i=0,1$, if 
it is a dynamic plan concentrated on $\Geo Z$, and if \EEE
\begin{equation}
  \label{eq:190}
  \int\sfd_Z^2(\rmz(0),\rmz(1))\,\d\ppi(\rmz)=
  \int \!\!\!\!\int_0^1 |\rmz'|^2\,\d t\,\d\ppi(\rmz)=
  \sfW_{\sfd_Z}^2(\alpha_0,\alpha_1).
\end{equation}

When $Z=\tY$ we will denote by $\chm2t=\chm2{}\circ(\eval_t)_\sharp$
the homogeneous marginal at time $t\in I$.  Since \WWW
$\chm2{}{}:\cP(\tY)\to \cM(X)$ is $1$-Lipschitz (cf.\ Corollary
\ref{cor:HK-W2}), it follows that
the curve $\mu_t:=\chm2{}{\alpha_t}=\chm2t\ppi$ belongs to
$\AC^2(I;(\cM(X),\HK))$ and moreover
\begin{equation}
  \label{eq:170}
  |\dot\mu_t|^2_{\HK}\le \int |\dot \ty|_{\sfdc}^2(t)\,\d\ppi(\ty)\quad
  \text{for a.e.~$t\in (0,1)$}.
\end{equation}
A simple consequence of this property is that $(\cM(X),\HK)$ 
inherits the length (or geodesic) property of $(X,\sfd)$.
\begin{proposition}
  \label{prop:HKlg}
  $(\cM(X),\HK)$ is a length (resp.~geodesic) space
  if and only if $(X,\sfd)$ is a length (resp.~geodesic) space.
\end{proposition}
\begin{proof}
  Let us first suppose that $(X,\sfd)$ is a length 
  space (the argument in the geodesic case is completely equivalent)
  and let $\mu_i\in \cM(X)$.
  By Corollary \ref{cor:HK-W1} we find 
  $\alpha_i\in \cP_2(\tY)$ such that 
  $\chm2{}\alpha_i=\mu_i$ and $\HK(\mu_1,\mu_2)=
  \Wc(\alpha_1,\alpha_2)$. Since $\tY$ is a length
  space, it is well known that $\cP_2(\tY)$ is \WWW a length space
  (see \cite{Sturm06I}), \EEE
  so that for every $\kappa>1$ there exists $\alpha\in
  \Lip([0,1];(\cP_2(\tY),\Wc))$ connecting
  $\alpha_1$ to $\alpha_2$ such that $|\alpha'|_{\Wc}\le
  \kappa\,\Wc(\alpha_1,\alpha_2)$.
  Setting $\mu_t:=\chm2{}\alpha_t$ we obtain a Lipschitz curve
  connecting $\mu_1$ to $\mu_2$ with length $ \leq \kappa\,\HK(\mu_1,\mu_2)$.
  
  The converse property is a consequence of the next representation
  Theorem
  \ref{thm:plan-representation} and the fact that if $(\cP_2(\tY),\Wd)$ is
  a length (resp.\ geodesic) space, then $\tY$ and thus $X$ are length
  (resp.\ geodesic) spaces. 
\end{proof}

We want to prove the converse representation result \WWW that every
absolutely continuous curve $\mu:[0,1]\to (\cM(X),\HK)$ can be written
via a dynamic plan $\ppi$ as $\mu_t= \chm2{t}\ppi$. The argument only
depends on the metric properties of  
the Lipschitz submersion $\frh$.

\begin{theorem}
  \label{thm:plan-representation}
  Let $(\mu_t)_{t\in [0,1]}$ be a curve
  in $ \AC^p \nc([0,1];(\cM(X),\HK))$, $ p\in [1,\infty]$,
   with 
  \begin{equation}
    \label{eq:400}
     \Theta:=\sqrt{\mu_0(X)}+\int_0^1 |\dot \mu|_\HK\,\d t.
  \end{equation}
  Then there exists a curve 
  $(\alpha_t)_{t\in [0,1]}$ 
  in $\AC^p ([0,1];(\cP_2(\tY),\Wc))$
  such that $\alpha_t $ is concentrated on $\cball
  \Theta$ for every $t\in [0,1]$ and
  \begin{equation}
    \label{eq:171a}
    \mu_t=\chm2{}\alpha_t \text{ in } [0,1], \quad
    |\dot\mu_t|_{\HK}=
    |\dot\alpha_t|_{\Wc}
    \text{ for a.e.\ }t\in (0,1).
  \end{equation}
  Moreover, when $p=2$, 
  \nc 
  there exists a dynamic plan
  $\ppi\in \cP(\AC^2([0,1];\tY))$ such that
  \begin{equation}
    \label{eq:171}
    \begin{gathered}
      \alpha_t=(\eval_t)_\sharp\ppi,\quad \nc \mu_t=\chm2t\ppi
      =\chm2{}\alpha_t\text{ in } [0,1], \\ |\dot\mu_t|_{\HK}^2=
      |\dot\alpha_t|_{\Wc}^2=\!\int\! |\dot
      \ty|_{\sfdc}^2(t)\,\d\ppi(\ty) \text{ for a.e.\ }t\in (0,1).
    \end{gathered}
  \end{equation}
\end{theorem}
\begin{proof}
  By Lisini's lifting Theorem \cite[Theorem~5]{Lisini07}
  \eqref{eq:171} is a consequence of the first part of the statement
  and \eqref{eq:171a} in the case $p=2$. \nc It is therefore
  sufficient to prove that for a given $\mu\in
  \AC([0,1];(\cM(X),\HK))$ there exists a curve $\alpha\in \AC
  ([0,1];(\cP_2(\tY),\Wc))$ such that $\mu_t=\chm2{}(\alpha_t)$ and
  $|\dot\mu_t|=|\dot\alpha_t|$ a.e.~in $(0,1)$.  By a standard
  reparametrization technique, we may assume that $\mu$ is Lipschitz
  continuous and $|\dot\mu_t|=L$.

  We divide the interval $I=[0,1]$ into $2^N$-intervals \WWW of size
  $2^{-N}$, namely $I_i^N:=[t_{i-1}^N,t_{i}^N]$ with
  $t_i^N:=i\,2^{-N}$ for $i=1,\ldots, 2^N$. Setting
  $\mu_i^N:=\mu_{t_i^N}$ we can apply the Gluing Lemma
  \ref{le:gluing} (starting from $i=0$ to $2^N$)
  to obtain measures $\alpha_i^N\in
  \cP_2(\tY)$ such that
  \begin{equation}
    \label{eq:172}
    \frh(\alpha_i^N)=\mu_i^N,\quad
    W_{\sfdc}(\alpha_i^N,\alpha_{i+1}^N)=\HK(\mu_i^N,\mu_{i+1}^N)\le 
    L 2^{-N},
  \end{equation}
  and concentrated on $\cball {\Theta_N}$ where
  \begin{displaymath}
    \Theta_N=
    \sqrt{\mu_0(X)}+\sum_{i=1}^{2^N}\HK(\mu_{i-1}^N,\mu_i^N)
    \le \Theta.
  \end{displaymath}
  Thus if $t$ is a dyadic point, we obtain a sequence of probability
  measures $\alpha^N(t)\in \cP_2(\tY)$ concentrated on $\cball \Theta$
  with $\chm2{}(\alpha^N(t))=\mu_t$ and such that
  $W_{\sfdc}(\alpha^N(t),\alpha^N(s))\le L|t-s|$ if
  $s=m2^{-N}$ and $t=n2^{-N}$ are dyadic points in the same grid.  By the
  compactness lemma \ref{le:compactnessH} and a standard diagonal
  argument, we can extract a subsequence $N(k)$ such that
  $\alpha_{N(k)}(t)$ converges to $\alpha(t)$ in $(\cP_2(\tY),\Wc)$
  for every dyadic point $t$.  Since
  $W_{\sfdc}(\alpha(s),\alpha(t))\le L \nc |t-s|$ for every dyadic $s,t$, we
  can extend $\alpha$ to a $L$-Lipschitz curve, still denoted by
  $\alpha$, which satisfies $\chm2{}(\alpha(t))=\mu_t$.  Since \WWW
  $\chm2{}$ is $1$-Lipschitz, we conclude that $|\dot\alpha|(t)=
  |\dot\mu_t|$ a.e.\ in $(0,1)$.
\end{proof}
\begin{corollary}
  Let $(\mu_t)_{t\in [0,1]}$ be a curve in $\AC^2([0,1];(\cM(X),\HK))$
  and let $\Theta$ as in \eqref{eq:400}.  Then there exists a dynamic
  plan $\tilde\ppi$ in $\cP(\trmC([0,1];\pY))$ concentrated on
  $\tAC^2([0,1];\pY)$ such that $\alpha_t=(\eval_t)_\sharp\ppi$ is
  concentrated in $X\times [0,\Theta]$, that
  $\mu_t= \hm2{} \nc ((\eval_t)_\sharp\ppi)$, and that 
  \begin{equation}
    \label{eq:173}
    |\dot\mu_t|_\HK^2=
    \int |\rmy'|^2(t) 
    \,\d\ppi(\rmy) \quad\text{for $\Leb 1$-a.e.~$t\in [0,1]$},
  \end{equation}
\WWW where $|\rmy'|$ is defined in \eqref{eq:187}. \EEE
\end{corollary}


Another important consequence of the previous representation
result is a precise characterization of the geodesics in
$(\cM(X),\HK)$.

\begin{theorem}[Geodesics in $(\cM(X),\HK)$]
  \label{thm:geoHK}
  \mbox{} \\[-1.8em]
\begin{enumerate}[(i)]

\item If $(\mu_t)_{t\in [0,1]}$ is a geodesic in $(\cM(X),\HK)$ then
  there exists an optimal geodesic plan
  $\ppi$ in $\cP(\Geo{\tY})$ (recall \eqref{eq:190}) such that \\
  (a) $\ppi$-a.e.~curve $\ty$ is a geodesic in $\tY$, \\
  (b) $[0,1]\ni t\mapsto \alpha_t:=(\eval_t)_\sharp\ppi$ is a geodesic in
  $(\cP_2(\tY),\Wc)$, where all $\alpha_t$ are\\ \mbox{}\qquad
  concentrated on $\cball\Theta$ with
  $\Theta^2=2(\mu_0(X)+\HK^2(\mu_0,\mu_1))$,\\ 
  (c) $\mu_t=\chm2t\ppi=\chm2{}\alpha_t$ for every $t\in [0,1]$, and\\
  (d) $(\eval_s,\eval_t)_\sharp \ppi\in \OptHK(\mu_s,\mu_t)$ if $0\le
  s<t\le 1$.
  \item If $(X,\sfd)$ is a geodesic space,
    for every $\mu_0,\mu_1\in \cM(X)$ and every 
    $\aalpha\in \OptHK(\mu_0,\mu_1)$ 
    there exists an optimal geodesic plan $\ppi\in \cP(\Geo{\tY})$ 
    such that $(\eval_s,\eval_t)_\sharp \ppi=\aalpha$.
  \end{enumerate}\end{theorem}
\begin{proof}
  The statement (i) is an immediate consequence of 
  Theorem \ref{thm:plan-representation}.

  Statement (ii) is a well known property of the Kantorovich-Wasserstein
  space $(\tY,\Wc)$ in the case when $\tY$ is geodesic.
\end{proof}

Theorem \ref{thm:plan-representation} also clarifies 
the relation between $\HK$ and $\GHK$ \WWW introduced in Section
\ref{subsec:GHK}. \EEE

\begin{corollary}
  \label{cor:HKlengthGHK}
  If $(X,\sfd)$ is separable and complete then
  $\AC^2([0,1];(\cM(X),\GHK))$ coincides with
  $\AC^2([0,1];(\cM(X),\HK))$ and 
  for every curve $\mu\in \AC^2([0,1];(\cM(X),\GHK))$ we have
  \begin{equation}
    \label{eq:191}
    |\mu'|_{\GHK}(t)=|\mu'|_\HK(t)
    \quad\text{for $\Leb 1$-a.e.~$t\in [0,1]$}.
  \end{equation}
  In particular if $(X,\sfd)$ is a length metric space then
  $\HK$ is the length distance generated by $\GHK$.
\end{corollary}
\begin{proof}
  Since $\GHK\le \HK$ it is clear that 
  $\AC^2([0,1];(\cM(X),\HK))\subset \AC^2([0,1];(\cM(X),\GHK))$.

  In order to prove the opposite inclusion and 
  \eqref{eq:191} it is sufficient to 
  notice that 
  the classes of absolutely continuous curves
  in $\tY$ w.r.t.~$\sfdc$ and $\sfdg_\tY$ coincide with 
  equal metric derivatives
  $|\ty'|_{\sfdc}=|\ty'|_{\sfdg_\tY}$.
  Since $\GHK=\HK_\sfdg$ is the Hellinger-Kantorovich distance
  induced by $\sfdg$, the assertion follows by \eqref{eq:171} of Theorem
  \ref{thm:plan-representation}.
\end{proof}

\subsection{Lower curvature bound in the sense of Alexandrov}
\label{subsec:curvature}
Let us first recall two possible definitions of Positively Curved (PC)
spaces in the sense of Alexandrov, referring to
\cite{Burago-Burago-Ivanov01} and to \cite{Burago-Gromov-Perelman92}
for other equivalent definitions and for the more general case of
spaces with curvature $\ge k$.

According to Sturm \cite{Sturm99}, a metric space
$(Z,\sfd_Z)$ is a Positively Curved (PC) metric space in the large if 
for every choice of points $z_0,z_1,\cdots,z_N\in Z$ and
coefficients $\lambda_1,\cdots,\lambda_N\in (0,+\infty)$ 
we have
\begin{equation}
  \label{eq:447}
  \sum_{i,j=1}^N\lambda_i\lambda_j\sfd_Z^2(z_i,z_j)
  \le 2 \sum_{i,j=1}^N\lambda_i\lambda_j\sfd_Z^2(z_0,z_j).
\end{equation}
If every point of $Z$ has a neighborhood that is PC,
then we say that $Z$ is locally positively curved.

When the space $Z$ is geodesic, the above (local and global)
definitions coincide with the corresponding one given by Alexandrov,
which is based on triangle comparison: for every choice of
$z_0,z_1,z_2\in Z$, \WWW every $t\in [0,1]$, 
and every point $z_t$ such that $\sfd_Z(z_t,z_k)=|k
{-}t|\sfd_Z(z_0,z_1)$ for $k=0,1$ we have \EEE
\begin{equation}
  \label{eq:449}
  \sfd_Z^2(z_2,z_t)\ge (1-t)\,\sfd_Z^2(z_2,z_0)+t\:\sfd_Z^2(z_2,z_1)-2t(1-t)\,\sfd^2_Z(z_0,z_1).
\end{equation}
When $Z$ is also complete, the local and the global definition are equivalent.
\WWW Next we provide conditions on $(X,\sfd)$ or
$(\tY,\sfdc)$ that guarantee that $(\cM(X),\HK)$ is a PC space.

\begin{theorem}
  \label{thm:curvatureHK} 
  Let $(X,\sfd)$ be a metric space.\vspace{-0.8em}
  \begin{enumerate}[(i)]
  \item If $X\subset \R$ is convex (i.e.~an interval)
    endowed with the standard distance, then
    $(\cM(X),\HK)$ is a PC space.
  \item If $(\tY,\sfdc)$ 
    is a PC space in the large, cf.\ \eqref{eq:447},
    then $(\cM(X),\HK(X))$ is a PC space.
  \item If $(X,\sfd)$ is separable, complete and geodesic, then
    $(\cM(X),\HK)$ is a PC space if and only if $(X,\sfd)$ has locally
    curvature $\ge1$.
  \end{enumerate}
\end{theorem}
Before we go into the proof of this result, we highlight that for
a compact convex subset $\Omega\subset \R^d$ with $d\geq 2$ equipped with the
Euclidean distance, the space $(\cM(\Omega),\HK)$ is not PC, see
\cite[Sect.\,5.6]{LMS15} for an explicit construction showing the
semiconcavity of the squared distance fails.
\begin{proof}
  Let us first prove statement (ii).  If $(\tY,\sfdc)$ is a PC space
  then also $(\cP_2(\tY),\Wc)$ is a PC space \cite{Sturm06I}.
  Applying Corollary \ref{cor:HK-W2}, for every choice of $\mu_i\in
  \cM(X)$, $i=0,\ldots,N$, we can then find measures $\beta_i\in
  \cP_2(\tY)$ such that
\begin{equation}
  \label{eq:444}
    \Wc(\beta_0,\beta_i)=\HK( \mu_0,\mu_i \nc)\ \text{ for }
    i=1,\ldots,N,
\end{equation}
where it is crucial that $\beta_0$ is the same for every $i$.  It then
follows that
\begin{displaymath}
  \sum_{i,j=1}^N \lambda_i\lambda_j
  \HK^2( \mu_i,\mu_j\nc )\le 
  \sum_{i,j=1}^N \lambda_i\lambda_j
  \Wc^2(\beta_i,\beta_j)
  \le 2
  \sum_{i,j=1}^N \lambda_i\lambda_j
  \Wc^2(\beta_0,\beta_i)=
  2
  \sum_{i,j=1}^N \lambda_i\lambda_j
  \HK^2( \mu_0,\mu_i\nc ).
\end{displaymath}

Let us now consider (iii) ``$\Rightarrow$'': If $(\cM(X),\HK)$ is PC,
we have to prove that $(X,\sfd)$ has locally curvature $\ge1$.  By
Theorem \cite[Thm.~4.7.1]{Burago-Burago-Ivanov01} it is sufficient to
prove that $\tY\setminus \{\fro\}$ is locally PC to conclude that
$(\tY,\sfd)$ has locally curvature $\geq 1$.  We thus select
points $\ty_i=[x_i,\s_i]$, $i=0,1,2$, in a sufficiently small
neighborhood of $\ty=[x,\s ]$ with $\s>0$, so that
$\sfd(x_i,x_j)<\pi/2$ for every $i,j$ and $\s_i,\s_j>0$. We also
consider a geodesic $\ty_t=[x_t,s_t]$, $t\in [0,1]$, connecting
$\ty_0$ to $\ty_1$, thus satisfying
$\sfdc(\ty_t,\ty_i)=|i-t|\sfd(\ty_0,\ty_1)$ for $i=0,1$.
  
  Setting $\mu_i:=\s_i\delta_{x_i}$, $\mu_t:=s_t\delta_{x_t}$, it is
  easy to check (cf.\ \cite[Sect.\ 3.3.1]{LMS15}) that  
  \begin{equation}
    \label{eq:450}
  \begin{aligned}  &\HK(\mu_i,\mu_j)=\sfdc(\ty_i,\ty_j) \text{ for }
    i,j\in \{0,1,2\}, \\
    &\HK(\mu_t,\mu_k)=|k-t|\HK(\mu_0,\mu_1) \text{ for } k\in \{0,1\}.
  \end{aligned}
  \end{equation}
  We can thus apply \eqref{eq:449}
  to $\mu_0,\mu_1,\mu_2,\mu_t$ and obtain the corresponding
  inequality for $\ty_0,\ty_1,$ $\ty_2,\ty_t$.
  
  (iii) ``$\Leftarrow$'': In order to prove the converse property we
  apply Remark \ref{rem:gluing2}. For $\mu_0,\mu_1,\mu_2,\mu_3=\mu_t\in
  \cM(X)$ with $t\in [0,1]$ and $\HK(\mu_3,\mu_k)= |k-t|
  \HK(\mu_0,\mu_1)$, we find a plan $\aalpha\in \cP(X_0\times
  X_1\times X_2\times X_3)$ (with the usual convention to use copies
  of $X$) such that
  \begin{equation}
    \label{eq:451}
    \chm2{i}\aalpha=\mu_i,\quad
    \int \sfdc^2(\ty_i,\ty_j)\,\d\aalpha=
    \HK^2(\mu_i,\mu_j)\quad\text{for }(i,j)\in A=\{(0,3),\,(1,3),\,(2,3)\}.
  \end{equation}
  The triangle inequality, the elementary inequality $t(1-t)(a+b)^2
  \le (1-t) a^2+t b^2$, and the very definition of $\HK$ yield for
  $t\in (0,1)$ the estimate 
  \begin{align*}
    t(1{-}t)\HK^2(\mu_0,\mu_1)
    &\le
      t(1-t)              \int \sfdc^2(\ty_0,\ty_1)\,\d\aalpha\le 
      \int
      t(1-t) \big( (\sfdc(\ty_0,\ty_3)+\sfdc(\ty_3,\ty_1)\big)^2\,\d\aalpha
    \\&\le 
        \int
        (1{-}t)
        \sfdc^2(\ty_0,\ty_3)+t\sfdc^2(\ty_3,\ty_1)\,\d\aalpha
        =(1{-}t)\HK^2(\mu_0,\mu_3)+t\HK^2(\mu_3,\mu_1)\\&
        =t(1-t)\HK^2(\mu_0,\mu_1).
  \end{align*}
  This series of inequalities shows in particular that 
  \begin{displaymath}
    (1-t)
        \sfdc^2(\ty_0,\ty_3)+t\sfdc^2(\ty_3,\ty_1)=
          t(1-t)  
          \big(\sfdc(\ty_0,\ty_3)+\sfdc(\ty_3,\ty_1)\big)^2=
          t(1-t)   \sfdc^2(\ty_0,\ty_1)\quad \text{$\aalpha$-a.e.}
  \end{displaymath}
  so that 
  \begin{displaymath}
    \sfdc(\ty_0,\ty_3)=t\sfdc(\ty_0,\ty_1) \text{ and }
    \sfdc(\ty_3,\ty_1)=(1-t)\sfdc(\ty_0,\ty_1)\quad \text{$\aalpha$-a.e.}
  \end{displaymath}
  Moreover, $\pi^{\ty_0,\ty_1}_\sharp\aalpha\in \OptHK(\mu_0,\mu_1)$,
  so that \eqref{eq:451} holds for $(i,j)\in A'=A\cup\{(0,1)\}$.
  
  By Theorem \ref{thm:main-equivalence} we deduce that 
  \begin{displaymath}
    \sfd(\sfx_i,\sfx_j)\le \pi/2\quad\text{$\aalpha$-a.e.~for
    }(i,j)\in A'.
  \end{displaymath}
  If one of the points $\ty_i$, $i=0,1,2$, is the vertex $\fro$, 
  then it is not difficult to check by a direct computation that 
  \begin{equation}
    \label{eq:452}
    \sfdc^2(\ty_2,\ty_3)\ge (1-t)\sfdc^2(\ty_2,\ty_0)+
    t\sfdc^2(\ty_2,\ty_1)-2t(1-t)\sfdc^2(\ty_0,\ty_1).
  \end{equation}
  When $\ty_i\in \tY\setminus\{\fro\}$ for every $i=0,1,2$, \WWW we
  use 
  $\sfd(\sfx_0,\sfx_1)+\sfd(\sfx_1,\sfx_2)+\sfd(\sfx_2,\sfx_0)\le
  \frac 32\pi<2\pi$, \WWW and  
  Theorem \cite[Thm.~4.7.1]{Burago-Burago-Ivanov01} yields
  \eqref{eq:452} \WWW because of the assumption that $X$ is PC. \EEE
  Integrating \eqref{eq:452} w.r.t.~$\aalpha$, by taking into account
  \eqref{eq:451}, the fact that $(\pi^0,\pi^1)_\sharp\aalpha\in
  \OptHK(\mu_0,\mu_1)$, and  that
  \begin{displaymath}
    \int \sfdc^2(\ty_2,\ty_i)\,\d\aalpha\ge
    \HK^2(\mu_2,\mu_i) \ \text{ for } i=0,1,
  \end{displaymath}
  we obtain
  \begin{displaymath}
    \HK^2(\mu_2,\mu_3)\ge 
    (1-t)\HK^2(\mu_2,\mu_0)+t\HK^2(\mu_2,\mu_1)-2t(1-t)\HK^2(\mu_0,\mu_1).
  \end{displaymath}
  Finally, statement (i) is just a particular case of (iii).
\end{proof}

As simple applications of the Theorem above we obtain that
$\cM(\R)$ and $\cM(\S^{d-1})$ endowed with $\HK$ are Positively Curved
spaces.

\subsection{Duality and Hamilton-Jacobi equation}
\label{subsec:HJ}

In this section we will show the intimate connections of the duality formula of
Theorem \ref{thm:dualityHK} with Lipschitz subsolutions of the
Hamilton-Jacobi equation in $X\times (0,1)$ given by 
\begin{equation}
  \label{eq:407}
  \partial_t \xi_t+\frac 12|\rmD_X\xi_t|^2+2\xi_t^2=0
\end{equation}
and its counterpart in the cone space
\begin{equation}
  \label{eq:405}
  \partial_t\zeta_t+\frac12 |\rmD_\tY\zeta_t|^2=0.
\end{equation}
\WWW Indeed, the first derivation of $\HK$ via $\LET$ was obtained by
solving \eqref{eq:407} for $X=\R^d$, see the remarks on the
chronological development in Section \ref{sec:Devel}. \EEE
 
At a formal level, it is not difficult to check that 
solutions to \eqref{eq:407} corresponds to the special class
of solutions to \eqref{eq:405} of the form
\begin{equation}
  \label{eq:414}
  \zeta_t([x,\s ]):=\xi_t(x)\s^2. 
\end{equation}
\WWW Indeed, still on the formal level we have  the formula 
\begin{equation}
  \label{eq:415}
  |\rmD_\tY\zeta|^2=\frac 1{\s^2}|\rmD_X \zeta|^2+| \partial_r \nc
  \zeta|^2=
  |\rmD_X \xi|^2\s^2+4\xi^2\s^2\quad\text{if }\zeta=\xi\, \s^2.
\end{equation}
Since the Kantorovich-Wasserstein distance on $\cP_2(\tY)$ can be
defined in duality with subsolutions to \eqref{eq:405} via the
Hopf-Lax formula and $2$-homogeneous marginals are modeled on test
functions as in \eqref{eq:414}, we can expect to obtain a dual
representation for the Hellinger-Kantorovich distance on $\cM(X)$ by
studying the Hopf-Lax formula for initial data of the form
$\zeta_0(x,\s )=\xi_0(x) \s^2$.

\paragraph{Slope and asymptotic Lipschitz constant.}
In order to give a metric interpretation to \eqref{eq:407} and
\eqref{eq:405}, let us first recall that for a locally Lipschitz
function $f:Z\to \R$ defined in a metric space $(Z,\sfd_Z)$ the
\emph{metric slope} $|\rmD_Z f|$ and the \emph{asymptotic Lipschitz
  constant} $\alc Zf$ are defined by
\begin{align}
  \label{eq:408}
  |\rmD_Z f|(z):=\limsup_{x\to z}\frac{|f(x)-f(z)|}{\sfd_Z(x,z)},
  \qquad
  \alc Zf(z):={}&\lim_{r\down0}\sup_{{x,y\in B_r(z)}\atop{ y\neq
      x}}\frac{|f(y)-f(x)|}{\sfd_Z(x,y)}
\end{align}
with the convention that $|\rmD_Z f|(z)=\alc Zf(z)=0$ whenever $z$ is
an isolated point. $\alc Zf$ can also be defined as the minimal
constant $L\ge0$ such that there exists a function $G_L:\WWW Z\times Z
\EEE \to [0,\infty)$ satisfying
\begin{equation}
  \label{eq:418}
  |f(x)-f(y)|\le
  G_L(x,y)\sfd_Z(x,y) 
  ,\quad
  \limsup_{x,y\to z}G_L(x,y)\le L.
\end{equation}
Note that $\alc Zf$ is \WWW always an upper semicontinuous
\EEE function. When $Z$ is a length
space, $\alc Zf$ is the upper semicontinuous envelope of the metric slope
$|\rmD_Zf|$.  We will often write $|\rmD f|, \ \alc{} f$ whenever the
space $Z$ will be clear from the context.

\begin{remark}
  \label{rem:slope-invariance}\upshape
  The notion of locally Lipschitz function and the value
  $\alc Zf$ does not change if we replace the distance
  $\sfd_Z$ with a distance $\tilde \sfd_Z$ of the form
  \begin{equation}
    \label{eq:416}
    \begin{aligned}
      &\tilde\sfd_Z(z_1,z_2):=h(\sfd_Z(z_1,z_2))\ \text{ for } z_1,z_2\in Z,\\
      & \text{with } h:[0,\infty)\to[0,\infty)\text{ concave and }
      \lim_{r\down0}\frac{h(r)}r=1.
    \end{aligned}
  \end{equation}
  In particular, the truncated distances $\sfd_Z\land \kappa$ with
  $\kappa>0$, the distances $a\sin((\sfd_Z\land \kappa)/a)$ \WWW with
  $a>0$ and $\kappa\in (0,a\pi/2]$, and the distance
  $\sfdg=\elli(\sfd)$ given by \eqref{eq:456} yield the same
  asymptotic Lipschitz constant.

  In the case of the cone space $\tY$ it is not difficult to see that
  the distance $\sfdc$ and $\asfdc {\kp}$ 
  coincide in suitably small neighborhoods of every point $\ty\in
  \tY\setminus \{\fro\}$, so that they induce the same asymptotic
  Lipschitz constants in $\tY\setminus \{\fro\}$. The same property
  holds for $\sfdg_\tY$.  In the case of the vertex $\fro$, relation 
  \eqref{eq:389} yields
  \begin{equation}
    \label{eq:422}
    \alc {\tY}f(\fro)\le 
    \alc{(\tY,\tsfdc)} f(\fro)\le 
    \sqrt2\,\alc {\tY}f(\fro).\qedhere
  \end{equation}
\end{remark}

The next result shows that the asymptotic Lipschitz constant satisfies
\WWW formula \eqref{eq:415} for $\zeta([x,r])=\xi(x)r^2$. \EEE

\begin{lemma}
  \label{le:asc-cone}
  For $\xi:X\to \R$ let 
  $\zeta:\tY\to \R$ be defined by $\zeta([x,\s ]):=\xi(x)\s^2$.

\begin{enumerate}[(i)]

\item If $\zeta$ is $\sfdc$-Lipschitz in $\cball R$, then $\xi\in
  \Lip_b(X)$ with
  \begin{equation}
    \label{eq:424}
    \sup_X |\xi|\le \frac 1{R^2}\sup_{\cball R}|\zeta| 
    \le \frac 1R \Lip(\zeta,{\cball R}) \ \text{ and } \ 
    \Lip(\xi,X)\le \frac 1R\Lip(\zeta,{\cball R}).
  \end{equation}

\item If $\xi\in \Lip_b(X)$, then $\zeta$ is $\sfdc$-Lipschitz in
  $\cball R$ for every $R>0$ with
  \begin{equation}
    \label{eq:425}
    \sup_{\cball R}|\zeta|\le  R^2 \sup_X |\xi|\ \text{ and } \
    \Lip^2(\zeta,{\cball R})\le 
    R^2\Big(\!\Lip^2(\xi,(X,\tilde\sfd)){+} 4\sup_X |\xi|^2\Big),
  \end{equation}
  where $\tilde\sfd:=2\sin(\sfd_\pi/2)$.

\item In the cases (i) or (ii) we have, for every $x\in X$ and
  $\s\ge0$, the relation
  \begin{equation}
    \label{eq:417}
    \alcs\tY\zeta([x,\s ])=
      \begin{cases}
        \Big( \alcs X \xi(x)+4\xi^2(x)\Big)\s^2&\text{for }\s>0,\\
        \mbox{}\qquad \qquad 0&\text{for }\s=0.
      \end{cases}
  \end{equation}
  \WWW The analogous formula holds for the metric slope
  $|\rmD_\tY\zeta|([x,r])$.  Moreover, equation \eqref{eq:417} \WWW
  remains true if\/ $\sfdc$ is replaced by the distance $\tsfdc$. \EEE
  \end{enumerate}
\end{lemma}
\begin{proof}
  As usual we set $\ty_i=[x_i,\s_i]$ and $\ty=[x,\s ]$.

  Let us first check statement (i). 
  If $\zeta$ is locally Lipschitz then 
  $|\xi(x)|=\frac1{R^2}|\zeta([x,R])-\zeta([x,0])|
  \le \frac 1R\Lip(\zeta;\cball R)$ for every $R$ sufficiently small,
  so that $\xi$ is uniformly bounded.
  Moreover, \WWW using \eqref{eq:94pre} for every $R>0$ we have 
  \begin{displaymath}
    R^2|\xi(x_1)-\xi(x_2)|\le |\zeta(x_1,R)-\zeta(x_2,R)|\le 
    \Lip(\zeta;\cball R) R\tilde\sfd(x_1,x_2)
    \le \Lip(\zeta;\cball R) R\sfd(x_1,x_2),
  \end{displaymath}
  so that $\xi$ is uniformly Lipschitz and \eqref{eq:424} holds.
  
  Concerning (ii), for $\xi\in \Lip_b(X)$ 
  we set $S:=\sup|\xi|$ and $L:=\Lip(\xi,(X,\tilde\sfd))$ and use the
  identity 
  \begin{align}
    \label{eq:420}
    &\zeta(\ty_1)-\zeta(\ty_2)=(\xi(x_1)-\xi(x_2))\s_1\s_2+
    \WWW 2\xi(x) \s(\s_1-\s_2)+
    \omega(\ty_1,\ty_2;\ty)(\s_1-\s_2),
  \\
   \nonumber & \text{where }
    \omega(\ty_1,\ty_2;\ty):= \WWW \s_1\xi(x_1)+\s_2\xi(x_2)-2\s \xi(x)
    \text{ with } 
    \lim_{\ty_1,\ty_2\to \ty}\omega(\ty_1,\ty_2)=0. 
  \end{align}
  Since $|\omega(\ty_1,\ty_2;0)| \WWW \leq 2RS$ if $\ty_i\in \cball R$, 
  equation \eqref{eq:420} with $\s=0$ yields
  \begin{displaymath}
    |\zeta(\ty_1)-\zeta(\ty_2)|\le
    L\tilde\sfd(x_1,x_2)\s_1\s_2 + \WWW 2RS |\s_1-\s_2|
    \le 2\big(L^2+4\s^2\big)^{1/2}R\,\sfdc(\ty_1,\ty_2).    
  \end{displaymath}
  Letting $R\down0$ the inequality above also proves \eqref{eq:417} in
  the case $\s =0$.

  In order to prove \eqref{eq:417} when $\s\neq0$ let us set $L_\tY:=
  \alcs\tY\zeta([x,\s ])$, $L_X:=\alc X\xi(x)$, and let $G_L$ be a
  function satisfying \eqref{eq:418} with respect to the distance
  $\tilde\sfd$ (see Remark \ref{rem:slope-invariance}).  Equation
  \eqref{eq:420} yields, \WWW for all $\ty=[x,r]$, the relation \EEE
  \begin{align*}
    &|\zeta(\ty_1)-\zeta(\ty_2)|
        \le G_L(x_1,x_2) \tilde\sfd(x_1,x_2) \s_1\s_2
        +\big(2|\xi(x)|\s+
          |\omega \WWW (\ty_1,\ty_2;\ty) |\big)|\s_1-\s_2|
      \\&\le
          \Big( G_L^2(x_1,x_2)\s_1\s_2+\big(2|\xi(x)|\s+ 
         |\omega \WWW (\ty_1,\ty_2;\ty) |\big)^2\Big)^{1/2} \sfdc(\ty_1,\ty_2).
  \end{align*}
  Passing to the limit $\ty_1,\ty_2\to \ty$ and using the fact that
  $x_1,x_2\to x$ due to $\s\neq0$, we obtain
  \begin{displaymath}
    L_\tY
    \le \s\Big( L_X^2+4|\xi(x)|^2 \Big)^{1/2}.
  \end{displaymath}
  In order to prove the converse inequality we observe that for every
  $L'<L_X$ there exist two sequences of points $(x_{i,n})_{n\in \N}$
  converging to $x$ w.r.t.~$\sfd$ such that \WWW
  $\xi(x_{1,n})-\xi(x_{2,n})\ge L' \delta_n$ where
  $0<\delta_n:=\tilde\sfd(x_{1,n},x_{2,n})\to0$. Choosing
  $\s_{1,n}:=\s$ and $\s_{2,n}=\s(1+\lambda\delta_n)$ for an arbitrary
  constant $\lambda\in \R$ with the same sign as $\xi(x)$, we can
  apply \eqref{eq:420} and arrive at \EEE
  \begin{align*}
    L_\tY\ge 
    \liminf_{n\to\infty}
    \frac{|\zeta(\ty_{1,n}) {-}\zeta(\ty_{2,n})|}{\sfdc(\ty_{1,n},\ty_{2,n})}
    \ge 
    \liminf_{n\to\infty}
    \frac{L'\delta_n \s^2 {+}2|\xi(x)|\s^2|\lambda| \delta_n+o(\delta_n)}
    {\sqrt {\lambda^2\s^2\delta^2_n+\s^2\delta^2_n+o(\delta_n)}}
    =\s\frac{L' {+}2|\xi(x)|\,|\lambda |}
    {\sqrt {\lambda^2+1}}.
  \end{align*}
  Optimizing with respect to $\lambda$ we obtain
  \begin{displaymath}
    L_{\tY}^2\ge \s^2\big((L')^2  +4|\xi(x)|^2\big), \ \text{ where }
    L'\le L_X \text{ is arbitrary}.
  \end{displaymath}
  \WWW This proves \eqref{eq:417} for the asymptotic Lipschitz
  constant $\alc{\tY}\zeta$. 
  The arguments for proving \eqref{eq:417} for 
  metric slopes $|\rmD_\tY \zeta|$ are completely analogous.
\end{proof}

\paragraph{Hopf-Lax formula and subsolutions to metric Hamilton--Jacobi equation in
  the cone $\tY$.} 
Whenever $f\in \Lip_b(\tY)$ 
the Hopf-Lax formula
\begin{equation}
  \label{eq:409}
  \HJC{k}{t}f(\ty):=\inf_{\ty'\in \tY} \Big(f(\ty')+
  \frac 1{2t}\sfdc^2(\ty,\ty') \Big) \quad \text{for } \ty\in \tY
 \text{ and }  t>0,
\end{equation}
provides a function $t\mapsto \HJC{}t f$ which is Lipschitz from
$[0,\infty)$ to $\rmC_b(\tY)$, satisfies the a-priori bounds
\begin{equation}
  \label{eq:410}
  \inf_{\tY} f\le \WWW \HJC{}t f \leq \sup_{\tY} f,\quad
  \Lip(\QQ_t f;\tY)\le 2\Lip (f,\tY),
\end{equation}
and solves
\begin{equation}
  \label{eq:411}
  \partial_t^+
  \HJC {k}{t} f(\frz)+\frac 12\alc
  {\tY}{\HJC {k}{t} f}^2(\frz)\le0\quad\text{for every }\frz\in {\tY},\ t>0,
\end{equation}
where $\partial_t^+$ denotes the partial right derivative w.r.t.~$t$.
It is also possible to prove that for every $\ty\in {\tY}$ the time
derivative of $\HJC{} t f(\ty)$ exists with possibly countable
exceptions and that \eqref{eq:411} is in fact an equality if
$(\tY,\sfdc)$ is a length space, a property that always holds if
$(X,\sfd)$ is a length metric space. 
This is stated in our main result:

\begin{theorem}[Metric subsolution of Hamilton-Jacobi equation in
  $X$]\label{thm:main-HJ}
  Let $\xi\in \Lip_b(X)$ satisfy the uniform lower bound $P:=\inf_X
  \xi+1/2>0$ and let us set $\zeta([x,\s ]):= \xi(x)\s^2$. Then, for
  every $t\in [0,1]$ we have
  \begin{align}
    \label{eq:423}
      &\HJC {\pi/2}{t} \zeta([x,\s ]) = \xi_t(x)\s^2 , \quad \text{where}
      \quad \xi_t(x):=\HJ{\pi}{t}\xi(x) \ \text{ and }\\
\notag &\HJ{\pi/2}{t}\xi(x):=\inf_{x'\in X}
      \Big(\frac{\xi(x')}{1 {+} 2t\xi(x')}+
      \frac{\sin^2(\sfd_{\pi/2}(x,x'))}{2t(1{+}2t\xi(x'))}\Big)
      =\inf_{x'\in X}\frac 1{2t}
      \Big(1-\frac{\cos^2(\sfd_{\pi/2}(x,x'))}{1+2t\xi(x')}\Big).
    \end{align}
  Moreover, for every $R>0$ we have
  \begin{equation}
    \label{eq:426}
     \xi_t(x)\s^2=\inf_{\ty'=[x',r']\in \cball
      R} \Big(\xi(x')(r')^2 +\frac1{2t}\sfdc^2([x,\s ];[x',r'])\Big) \
     \text{ for all }  x\in X,\ \s\le PR.
  \end{equation}
  The map $t\mapsto \xi_t$ is Lipschitz from $[0,1]$ to $\rmC_b(X)$ with
  $\xi_t\in \Lip_b(X)$ for every $t\in [0,1]$. Moreover, $\xi_t$ 
  is a subsolution to the generalized Hamilton-Jacobi equation
  \begin{subequations}
    \begin{equation}
    \label{eq:419}
    \partial_t^+\xi_t(x)+\frac 12\alcs X{\xi_t}(x)+2\xi_t^2(x)\le 0
    \quad\forevery x\in X \text{ and } t\in [0,1].
  \end{equation}
  For every $x\in X$ the map $t\mapsto \xi_t(x)$ is time differentiable
  with at most countable exceptions. 
  If $(X,\sfd)$ is a length  space,
  \eqref{eq:419} holds with equality 
  and $\alc X{\xi_t}(x)=|\rmD_X
  \xi_t|(x)$
  for every $x\in X$ and $t\in [0,1]$:
  \begin{equation}
    \label{eq:419bis}
    \partial_t^+\xi_t(x)+\frac 12\alcs X{\xi_t}(x)+2\xi_t^2(x)=
    0,\quad
    \alc X{\xi_t}(x)=|\rmD_X \xi_t|(x).
  \end{equation}
  \end{subequations}
\end{theorem}
Notice that when $\xi(x)\equiv \xi$ is constant, \eqref{eq:423}
reduces to $\HJ{} t \xi=\xi/(1+2t\xi)$ which is the solution to the
elementary differential equation $\frac \d{\d t}\xi+2\xi^2=0$.
\begin{proof}
  Let us observe that 
  $\inf_{t\in [0,1],z\in X}(1+2t\xi(z))=P>0$. 
  A simple calculation shows
  \begin{align*}
    &\xi(x')(r')^2+\frac 1{2t}\sfdc^2([x,\s ];[x',r'])
                   =\frac 1{2t}\Big((1{+}2t\xi(x'))(r')^2+\s^2-2\s\, r'
                    \cos(\sfd_\pi(x,x'))\Big)
                        \\
     &=\frac 1{2t(1{+}2t\xi(x'))}\Big[
         \Big((1{+}2t\xi(x'))r'-\cos(\sfd_\pi(x,x'))\s\Big)^2
              +\s^2\Big(2t\xi(x')+\sin^2(\sfd_\pi(x,x'))\Big)\Big].
  \end{align*}
 Hence, if we choose 
  \begin{displaymath}
    r'=r'(x,x',\s):=
    \begin{cases}
      \s\cos(\sfd_\pi(x,x'))/(1{+}2t\xi(x')) &\text{if }\sfd(x,x')\le
      \pi/2\\
      0&\text{otherwise,}
    \end{cases}
  \end{displaymath}
  we find (notice the truncation at $\pi/2$ instead of $\pi$)
  \begin{equation}
    \label{eq:445}
    \inf_{r' \ge0} \xi(x')(r')^2+\frac 1{2t}\sfdc^2([x,\s ];[x',r'])
     = \frac {\s^2}{2t(1{+}2t\xi(x'))}
       \Big(2t\xi(x')+\sin^2(\sfd_{\pi/2}(x,x'))\Big),
  \end{equation}
  which yields \eqref{eq:423} and \eqref{eq:426}. 

  Equation \eqref{eq:426} also shows that the function $\zeta_t =
  \xi_t(x)\s^2$ coincides on $\cball{PR}$ with the solution
  $\zeta^R_t$ given by the Hopf-Lax formula in the metric space
  $\cball R$.  Since the initial datum $\zeta$ is bounded and
  Lipschitz on $\cball R$ we deduce that $\zeta_t^R$ is bounded and
  Lipschitz, so that $t\mapsto \xi_t$ is bounded and Lipschitz in $X$
  by Lemma \ref{le:asc-cone}.

  Equation \eqref{eq:419} and the other regularity properties then
  follow by \eqref{eq:417} and the general properties of the Hopf-Lax
  formula in $\cball R$.
\end{proof}

\paragraph{Duality between the Hellinger-Kantorovich distance and
  subsolutions to the generalized Ha\-mil\-ton-Jacobi equation.}
We conclude this section with the main application
of the above results to the Hellinger-Kantorovich distance.

\begin{theorem}
  \label{thm:main-HKHJ}
  Let us suppose that $(X,\sfd)$ is a complete and separable
  metric space.  
\begin{enumerate}[(i)]
\item If $\mu\in \AC^2([0,1];(\cM(X),\HK))$ and $\xi:[0,1]\to
  \Lip_b(X)$ is uniformly bounded, Lipschitz w.r.t.~the uniform norm,
  and satisfies \eqref{eq:419}, then the curve $ t\mapsto
  \int\xi_t\,\d\mu_t$ is absolutely continuous and
  \begin{equation}
    \label{eq:428}
    \frac\d{\d t}\int_X\xi_t \,\d\mu_t\le  \frac12 |\mu_t'|_\HK^2
  \end{equation}

\item If $(X,\sfd)$ is a length space, then for every $\mu_0,\mu_1$ and
  $k\in \N\cup\{\infty\}$ we have 
  \begin{equation}
    \label{eq:412}
     \begin{aligned}
         \frac 12 \HK^2(\mu_0,\mu_1)=
          \sup\Big\{&\int_X\xi_1\,\d\mu_1-\int_0\xi_0\,\d\mu_0\ :
          \ \ \xi\in \rmC^k([0,1];\Lip_{b}(X)), \\&
          \partial_t\xi_t(x)+\frac 12|\rmD_X
          \xi_t|^2(x)+2\xi_t^2(x)\le 0  \text{ in } X\times(0,1)\Big\}.
     \end{aligned}
    \end{equation}
    Moreover, in the above formula we can also take the supremum over 
    functions $\xi \in\rmC^k([0,1];\Lip_{b}(X))$ with bounded support.
  \end{enumerate}
\end{theorem}
\begin{proof}
  If $\xi$ satisfies \eqref{eq:419} then setting $\zeta_t([x,\s ]):=
  \xi_t(x)\s^2$ we obtain a family of functions $t\mapsto \zeta_t$,
  $t\in [0,1]$, whose restriction to every $\cball R$ is uniformly
  bounded and Lipschitz, and it is Lipschitz continuous with respect
  to the uniform norm of $\rmC_b(\cball R)$. By Lemma
  \ref{le:asc-cone} the function $\zeta$ solves
  \begin{displaymath}
    \partial_t^+\zeta_t+\frac12\alcs{\tY}{\zeta_t}\le 0\quad
    \text{in }\tY\times (0,1).
  \end{displaymath}
  According to Theorem \ref{thm:plan-representation} we find $\theta>0$
  and a curve $\alpha\in \AC^2([0,1];(\cP_2(\cball \theta),\Wc))$ 
  satisfying \eqref{eq:171a}. 
  Applying the results of \cite[Sect. 6]{Ambrosio-Mondino-Savare15},
  the map $t\mapsto \int_{\tY}\zeta_t\,\d\alpha_t$ is absolutely
      continuous with 
  \begin{displaymath}
    \frac\d{\d t}\int_{\tY}\zeta_t\,\d\alpha_t\le \frac
    12|\alpha_t'|^2_{\Wc}\quad
    \text{$\Leb 1$-a.e.~in (0,1)}.
  \end{displaymath}
  Since $\int_{\tY}\zeta_t\,\d\alpha_t=\int_X \xi_t\,\d\mu_t$ 
  we obtain \eqref{eq:428}.
  
  Let us now prove (ii). 
  As a first step, denoting by $S$ the right-hand side of
  \eqref{eq:412}, we prove that $\HK^2(\mu_0,\mu_1)\ge S$.
  If $\xi\in \rmC^1([0,1];\Lip_b(X))$ satisfies the
  pointwise inequality
  \begin{equation}
    \label{eq:429}
    \partial_t\xi_t(x)+\frac 12|\rmD_X
          \xi_t|^2(x)+2\xi_t^2(x)\le 0, 
  \end{equation}
  then it also satisfies \eqref{eq:419}, because
  \eqref{eq:429} provides the relation
  \begin{equation}
    \label{eq:430}
    \frac 12|\rmD_X \xi_t|^2(x)\le
          -\Big(\partial_t\xi_t(x)+2\xi_t^2(x)\Big)\quad
          \text{for every }(x,t)\in X\times (0,1),
  \end{equation}
  where the right hand side is bounded and continuous in $X$.  Equation
  \eqref{eq:430} thus yields the same inequality for the upper
  semicontinuous envelope of $|\rmD_X \xi_t|$ and this function
  coincides with $\alc X{\xi_t}$ since $X$ is a length space.
  
  We can therefore apply the previous point (i) by choosing
  $\lambda>1$ and a Lipschitz curve $\mu:[0,1]\to \cM(X)$ joining
  $\mu_0$ to $\mu_1$ with metric velocity $|\mu_t'|_\HK\le
  \lambda\HK(\mu_0,\mu_1)$, whose existence is guaranteed by the
  length property of $X$ and a standard rescaling technique. Relation
  \eqref{eq:428} yields
  \begin{displaymath}
    2\int_X\xi_1\,\d\mu_1-2\int_X\xi_0\,\d\mu_0\le 
    \int_0^1 |\mu_t'|^2_\HK\,\d t\le 
    \lambda^2\HK^2(\mu_0,\mu_1).
  \end{displaymath}
  Since $\lambda>1$ is arbitrary, we get 
  $\HK^2(\mu_0,\mu_1)\ge S$.

  In order to prove the converse inequality \WWW in \eqref{eq:412}
  we fix $\eta>0$ and apply
  the duality Theorem \ref{thm:dualityHK} to get $\xi_0\in
  \Lip_{bs}(X)$ (the space of Lipschitz functions with bounded support)
   with $\inf \xi_0>-1/2$ such that
  \begin{equation}
    \label{eq:431}
    2\int_X\HJ{}1\xi_0\,\d\mu_1-2\int_X\xi_0\,\d\mu_0
    \ge \HK^2(\mu_0,\mu_1)-\eta.
  \end{equation}
  Setting $\xi_t:=\HJ{} t \xi_0$ we find a solution to \eqref{eq:419}
  which has bounded support, is uniformly bounded in $\Lip_b(X)$ and
  Lipschitz with respect to the uniform norm.  We have to show that
  $(\xi_t)_{t\in [0,1]}$ can be suitably approximated by smoother
  solutions $\xi^\eps\in \rmC^\infty([0,1];\Lip_b(X))$, $\eps>0$, in
  such a way that $\int \xi_i^\eps\,\d\mu_i\to
  \int\xi_i\,\d\mu_i$ as $\eps\down0$ for $i=0,1$.
 
  We use an argument of \cite{Ambrosio-Erbar-Savare15}, \WWW which
  relies on the scaling invariance of the generalized Hamilton--Jacobi
  equation: If $\xi$ solves \eqref{eq:429} and $\lambda>0$, then
  $\xi_t^\lambda(x):= \lambda \xi_{\lambda t+ t_0}(x)$ solves
  \eqref{eq:429} as well. Hence,  by
  approximating $\xi_t$ with $\lambda\xi(\lambda t
  {+}(1{-}\lambda)/2,x)$ with $0< \lambda<1$ and passing to the limit
  $\lambda\up 1$, it is not restrictive to assume that $\xi$ is defined
  in a larger interval $[a,b]$, with $a<0, b>1$. Now, a time
  convolution is well defined on $[0,1]$, for which we use a 
  symmetric, nonnegative kernel $\kappa\in
  \rmC_{\rmc}^\infty(\R)$ with integral $1$ defined via
  \begin{equation}
    \label{eq:433}
    \xi^\eps_t(x):=(\xi_{(\cdot)}(x)\ast \kappa_\eps)_t=\int_\R
    \xi_w(x)\kappa_\eps(t {-}w)\,\d w,\quad \text{where }
    \kappa_\eps(t):=\eps^{-1}\kappa(t/\eps),
  \end{equation}
  yields a curve $\xi^\eps\in \rmC^\infty([0,1];\Lip_b(X))$ 
  satisfying
  \begin{displaymath}
    \partial_t\xi^\eps_t+\frac 12\big(|\rmD_X
    \xi_{(\cdot)}|^2\big)\ast\kappa_\eps +2\big(\xi_{(\cdot)}^2\big)\ast
    \kappa_\eps\le 0
    \quad\text{in }X\times [0,1].
  \end{displaymath}
  By Jensen inequality $\xi_{(\cdot)}^2\ast\kappa_\eps\ge
  (\xi_{(\cdot)}\ast\kappa_\eps)^2$ and $|\rmD_X
  \xi_{(\cdot)}|^2\ast\kappa_\eps\ge (|\rmD_X
  \xi_{(\cdot)}|\ast\kappa_\eps)^2$.  Moreover, applying the following
  Lemma \ref{le:slope-commute} we also get $|\rmD_X
  \xi_{(\cdot)}|\ast\kappa_\eps\ge |\rmD_X\xi_t^\eps \xi_{(\cdot)}|$,
  so that the smooth convolution $\xi_t^\eps$ satisfies
  \eqref{eq:429}. Since $\xi_t^\eps \to\xi_t$ uniformly in $X$ for
  every $t\in [0,1]$, we easily get
  \begin{displaymath}
    S\ge\lim_{\eps\down0}2\Big(
    \int_X\xi_1^\eps\,\d\mu_1-\int_X\xi_0^\eps\,\d\mu_0\Big)
    \ge \HK^2(\mu_0,\mu_1)-\eta.
  \end{displaymath}
  Since $\eta>0$ is arbitrary the proof of (ii) is complete.  
\end{proof}

\WWW The next result shows that averaging w.r.t. a probability measure
$\pi \in \cP(\Omega)$ does not increase the metric slope nor the
asymptotic Lipschitz constant. This was used in the last proof for the
temporal smoothing and will be used for spatial smoothing in Corollary
\ref{cor:HJRd}. 

\begin{lemma}
  \label{le:slope-commute}
  Let $(X,\sfd)$ be a separable metric space, let $(\Omega,\cB,\pi)$
  be a \WWW probability space (i.e.\ $\pi(\Omega)=1$) and let
  $\xi_\omega\in \Lip_b(X)$, $\omega\in \Omega$, be a family of
  uniformly bounded functions such that $\sup_{\omega\in
    \Omega}\Lip(\xi_\omega ;X)<\infty$ and $\omega\mapsto
  \xi_\omega(x)$ is $\cB$-measurable for every $x\in X$.  Then the
  function $\WWW x\mapsto \xi(x):=\int_\Omega
  \xi_\omega(x)\,\d\pi( \omega)$ belongs 
  to $\Lip_b(X)$  and for every $x\in X$ the maps $\omega\mapsto
  |\rmD_X\xi_\omega|(x) $ and $\omega\mapsto \alc X{\xi_\omega}(x) $
  are $\cB$-measurable and satisfy
  \begin{equation}
    \label{eq:434}
    \alc X\xi(x)\le \int_X \alc
    X{\xi_\omega}(x)\,\d\pi(\omega),\quad
    |\rmD_X\xi|(x)\le \int_X |\rmD_X\xi_\omega|(x)\,\d\pi(\omega).
  \end{equation}
\end{lemma}
\begin{proof}
  The fact that $\xi_\omega\in \Lip_b(X)$ is obvious. To show
  measurability we fix $x\in
  X$ and use the expression \eqref{eq:408} for $\alc
  X\xi(x)$.  It is sufficient to prove that for every $r>0$ the map
  $\omega\mapsto s_{r,\omega}(x):=\sup_{y\neq z\in B_r(x)}
  |\xi_\omega(y)-\xi_\omega(z)|/\sfd(y,z)$ is $\cB$-measurable. This
  property follows by the continuity of $\xi_\omega$ and the
  separability of $X$, so that it is possible to restrict the supremum
  to a countable dense collection of points $\tilde B_r(x)$ in
  $B_r(x)$.  \WWW Thus, the measurability follows, because the
  pointwise supremum of countably many measurable functions is
  measurable. \EEE
  An analogous argument holds for $|\rmD_X\xi_\omega|$.
  
  \WWW Using the definition $\xi:= \int \xi_\omega \d\pi$ we have \EEE
  \begin{displaymath}
    \frac{|\xi(y)-\xi(z)|}{\sfd(y,z)}
    \le 
    \int_\Omega
    \frac{|\xi_\omega(y)-\xi_\omega(z)|}{\sfd(y,z)}\,\d\pi(\omega) \
    \text{ for } y\neq z.
  \end{displaymath}
\WWW Taking the supremum with respect to $y,z\in \tilde B_r(x)$
and $y\neq z$, we obtain 
  \begin{displaymath}
    \sup_{y\neq z\in B_r(x)}\frac{|\xi(y)-\xi(z)|}{\sfd(y,z)}
    \le \int_\Omega s_{r,\omega}(x)\,\d\pi(\omega).
  \end{displaymath}
  A further limit as $r\down0$ and the application of 
  the Lebesgue Dominated convergence Theorem yields 
  the first inequality of \eqref{eq:434}.
  The argument to prove the second inequality is completely analogous.
\end{proof}

When $X=\R^d$ the characterization \eqref{eq:412} of $\HK$  holds for an
even smoother class of subsolutions \WWW $\xi$ of the generalized
Hamilton--Jacobi equation. 

\begin{corollary}
  \label{cor:HJRd}
  Let $X=\R^d$ be endowed with the Euclidean distance. Then
  \begin{equation}
    \label{eq:443}
     \begin{aligned}
       \HK^2(\mu_0,\mu_1)= 2\,
       \sup\Big\{&\int_X\xi_1\,\d\mu_1-\int_X\xi_0\,\d\mu_0\ :\ \xi\in
       \rmC^\infty_\rmc(\R^d\times [0,1]),\ \\&
          \partial_t\xi_t(x)+\frac 12
          \big|\rmD_x \xi_t(x)\big|^2\nc + 
           2\xi_t^2(x)\le 0 \quad \text{in } X\times(0,1)\Big\}.
     \end{aligned}
  \end{equation}
\end{corollary}
\begin{proof}
  We  just have to check that the supremum of \eqref{eq:412}
  does not change if we substitute $\rmC^\infty([0,1];\Lip_{bs}(\R^d))$ with
  $\rmC^\infty_\rmc(\R^d\times [0,1])$. 
  This can be achieved by approximating any 
  subsolution $\xi\in \rmC^\infty([0,1];\Lip_{bs}(\R^d))$ 
  via convolution in space with a smooth kernel 
  with compact support, which still provides a subsolution 
  thanks to Lemma \ref{le:slope-commute}.
\end{proof}

\subsection{The dynamic interpretation of the Hellinger-Kantorovich
   distance ``\`a la Benamou-Brenier''}\label{subsec:BB}

In this section we will apply the superposition principle 
of Theorem \ref{thm:plan-representation}
and the duality result \ref{thm:main-HKHJ}
with subsolutions of the Hamilton-Jacobi equation
to quickly derive a dynamic formulation 
``\`a la Benamou-Brenier''
\cite{Benamou-Brenier00,Otto-Villani00}, \cite[Sect.~8]{Ambrosio-Gigli-Savare08}
of the Hellinger-Kantorovich distance,
which has also been considered in the recent
\cite{KMV15}.
In order to keep the exposition simpler, we will consider the
case $X=\R^d$ with the canonical Euclidean distance
$\sfd(x_1,x_2):=|x_1-x_2|$, but the result can be extended
to more general Riemannian and metric settings,
e.g.~arguing as in \cite[Sect.~6]{Ambrosio-Mondino-Savare15}.
A different approach, based on suitable representation formulae
for the continuity equation, is discussed in \WWW our companion paper
\EEE \cite{LMS15}.

Our starting point is provided by a
suitable class of linear continuity equations with reaction.
In the following we will denote by $\mu_I\in \cM(\R^d\times [0,1])$  
the measure
\begin{equation}
  \label{eq:436}
  \int \xi\,\d\mu_I:=\int_0^1\int_{\R^d} \xi_t(x)\,\d\mu_t(x)\,\d t
\end{equation}
induced by a curve $\mu\in \rmC^0([0,1];\cM(\R^d))$.

\begin{definition}
  \label{def:cea}
  Let $\mu\in \rmC^0([0,1];\cM(\R^d))$, let
  $(\vv,w):\R^d\times (0,1)\to \R^{d+1}$ be a Borel vector field
  in $\rmL^2(\R^d\times (0,1),\mu_I;\R^{d+1})$, thus 
  satisfying
  \begin{equation}
    \label{eq:432}
    \int_0^1 \int_{\R^d} \Big(|\vv_t(x)|^2+w^2_t(x)\Big)\,\d\mu_t(x)\,\d
    t=
    \int |(\vv,w)|^2\,\d\mu_I<\infty.
  \end{equation}
  We say that $\mu$ satisfies the \emph{continuity equation with reaction}
  governed by $(\vv,w)$ if
  \begin{equation}
    \label{eq:175pre}
    \partial_t \mu_t+\nabla\cdot(\vv_t\mu_t)=w_t\mu_t\quad
    \text{holds in the sense of distributions in }\R^d\times(0,1),
  \end{equation}
  i.e.~for every test function $\xi\in \rmC^\infty_\rmc(\R^d\times
  (0,1))$
  \begin{equation}
    \label{eq:435}
    \int_0^1\int_{\R^d} \Big(\partial_t\xi_t(x)+
    \rmD_x\xi_t(x)\vv_t(x)+\xi_t(x)w_t(x)\Big)\,\d\mu_t\,\d t=0.
  \end{equation}
\end{definition}
An equivalent formulation \cite[Sect.~8.1]{Ambrosio-Gigli-Savare08} of
\eqref{eq:175pre} is 
\begin{equation}
  \label{eq:438}
  \frac\d{\d t}\int_{\R^d} \xi(x)\,\d\mu_t(x)=
  \int_{\R^d} \Big(\rmD_x\xi(x)\vv_t(x)+\xi(x)w_t(x)\Big)\,\d\mu_t
  \quad\text{in }\DD'(0,1),
\end{equation}
for every $\xi\in \rmC^\infty_\rmc(\R^d)$.  We have a first
representation result for absolutely continuous curves \WWW $t\mapsto
\mu_t$, which relies in Theorem \ref{thm:plan-representation}, where
we constructed suitable lifted plans $\ppi\in \cP(\AC^2([0,1];\tY))$,
i.e.\ $\mu_t =\chm2t \ppi$, where $\tY$ is now the cone over $\R^d$. \EEE

\begin{theorem}
  \label{thm:BB1}
  Let $(\mu_t)_{t\in [0,1]}$ be a curve
  in $\AC^2([0,1];(\cM(\R^d),\HK)).$
  Then $\mu$ satisfies the continuity equation with reaction
  \eqref{eq:175pre} with a Borel vector field
  $(\vv,w)\in \rmL^2(\R^d\times (0,1),\mu_I;\R^{d+1})$
  satisfying
  \begin{equation}
    \label{eq:174pre}
    (\vv_t,w_t)\in \rmL^2(\R^d;\mu_t),\quad
    \int \Big(|\vv_t|^2+\frac14|w_t|^2\Big)\,\d\mu_t\le |\dot\mu_t|^2
    \quad\text{for $\Leb 1$-a.e.~}t\in (0,1).
  \end{equation}
\end{theorem}
\begin{proof}
  We will denote by $I$ the interval $[0,1]$ endowed with the
  Lebesgue measure $\lambda=\Leb 1\res [0,1]$. \WWW Recalling the map
  $(\sfx,\sfr):\tY \to \R^d\times [0,\infty)$ we define the maps
  $\measu \sfx:\rmC(I;\tY)\times I\to \R^d\times I$ and 
  $\sfR:\rmC(I;\tY)\times I\to \R_+$ via $\measu \sfx
  (\rmz,t):=(\sfx(\rmz(t)),t)$ and $\sfR (\rmz,t):=\sfr(\rmz(t))$.

  Let $\ppi$ be a dynamic plan in $\tY$ representing $\mu_t$ as in
  Theorem \ref{thm:plan-representation}.  We consider the deformed
  dynamic plan $\measu\ppi:=(\sfR^2\ppi)\otimes\lambda$, the measure
  \WWW $\hatmeasu\mu:=(\measu\sfx)_\sharp \measu\ppi$ and the disintegration
  $(\tilde\ppi_{x,t})_{(x,t)\in \R^d\times I}$ of $\measu\ppi$ with
  respect to $\measu\mu.$ Notice that $\tilde\ppi\le \Theta \ppi$,
  where $\Theta$ is given by \eqref{eq:400}, and that 
\begin{equation}
  \label{eq:437}
  \WWW \hatmeasu\mu =\int_0^1  (\mu_t\otimes \delta_t)
   \,\d\lambda(t),
\end{equation}
coincides with \WWW $\measu\mu$ in \eqref{eq:436}, because for
every $\xi\in \rmB_b(\R^d\times I)$ we have
\begin{align*}
  \int \xi\,\d\hatmeasu\mu 
  &=
    \int \xi(\sfx(\rmz(t)),t)\sfr^2(\rmz(t))\,\d\ppi_I(\rmz,t)=
    \int_0^1\int_{\R^d} \xi_t(x)\,\d\mu_t(x)\,\d t \WWW =\int
    \xi\,\d\measu\mu . 
\end{align*}
Let $\uu\in \rmL^2(\AC^2(I;\tY) \times I;\ppi\otimes
\lambda;\R^{d+1})$ be the Borel vector field $\uu(\ty,t):=\ty_\tY'(t)$
for every curve $\ty\in \AC^2(I;\tY)$ and $t\in I$, where $\ty_\tY'$
is defined as in \eqref{eq:439}.  By taking the density of the vector
measure $(\measu\sfx)_\sharp(\uu\measu\ppi)$ with respect to
$\measu\mu$ we obtain a Borel vector field $\measu\uu=\WWW (\vv,\hat
w) \in
\rmL^2(\R^d\times I;\measu\mu;\R^{d+1})$ which satisfies
\begin{equation}\label{eq:442}
  \measu\uu(x,t)=\int \uu\,\d\ppi_{x,t}\quad\text{for
    $\measu\mu$-a.e.\ }~(x,t)\in \R^d\times I \WWW 
   \ \text{ and } \ 
  \int \Big(|\vv_t|^2{+}\hat w_t^2\Big)\, \d \mu_t \leq |\mu'_t|^2. \EEE
\end{equation}  
  Choosing a test function $\zeta([x,\s] ,t):=
   \xi(x)\eta(t)\s^2$ with $\xi\in
  \rmC^\infty_\rmc(\R^d)$ and $\eta\in \rmC^\infty_\rmc(I)$ we
  \WWW can exploit the chain rule \eqref{eq:440} in $\R^d$ and find
  
  \begin{align*}
    -&\int_0^1\eta'\int_{\R^d} \xi\,\d\mu_t\,\d t
       =-\int_{\R^d\times I} \eta'(t)\xi(x)\,\d\mu_I
       =- \int \xi(\sfx(\ty(t))\,\sfs^2(\ty(t))\eta'(t)\, 
        \d(\ppi\otimes\lambda)
    \\&=
        -\int \partial_t \zeta(\ty(t),t)\,\d(\ppi\otimes
        \lambda)=
        \int 
        \Big(-\frac\d{\d t}\zeta(\rmy(t),t)
        +\langle \rmD_\tY\zeta(\rmy(t),t), \rmy_\tY'(t)\rangle\Big)
        \,\d(\ppi\otimes \lambda)
        \\&=
        \int\Big( \int_0^1
            -\frac\d{\d t}\zeta(\rmy(t),t)\,\dt\Big)\,\d\ppi
        +\int \langle (\rmD_x \xi(\sfx_I),2\xi(\sfx_I)),\uu\rangle
            \sfR^2
        \,\d(\ppi\otimes \lambda)
    \\&=
        \int \eta(t)\langle (\rmD_x \xi(x),2\xi(x)),\uu_I\rangle
        \,\d\mu_I
   \\&
    =
                        \int_0^1\eta(t)\int_{\R^d} 
            \Big(\langle\rmD_x \xi(x),\vv_t(x)\rangle+
           \WWW 2\xi(x) \hat w_t(x) \Big)\,\d\mu_t\,\d t . 
  \end{align*}
\WWW Setting $w_t = 2\hat w_t$ the continuity equation with reaction
\eqref{eq:438} holds. 
\end{proof}

\WWW The next result provides the opposite inequality, which will be
deduced from the duality between the solutions of the generalized
Hamilton--Jacobi equation and $\HK$ developed in Theorem
\ref{thm:main-HKHJ}. 

\begin{theorem}
  \label{thm:BB2}
  Let $(\mu_t)_{t\in [0,1]}$ be a continuous curve in $\cM(\R^d)$ that
  solves the continuity equation with reaction \eqref{eq:175pre}
  governed by the Borel vector field $(\vv,w)\in
  L^2(\R^d\times[0,1],\mu_I;\R^{d+1})$ with $\mu_I$ given by
  \eqref{eq:436}.  Then $\mu\in \AC^2([0,1];(\cM(\R^d),\HK))$ and
  \begin{equation}
    \label{eq:176}
    |\dot\mu_t|^2\le \int_{\Rd} \Big(|\vv_t|^2+\frac 14|w_t|^2\Big)\,\d\mu_t
    \quad\text{for $\Leb 1$-a.e.~}t\in (0,1).
  \end{equation}
\end{theorem}
\begin{proof}
  The simple scaling $\xi(t,x)\to (b{-}a)\xi(a{+}(b{-}a)t,x)$ transforms any
  subsolution of the Hamilton-Jacobi equation in $[0,1]$ to a
  subsolution of the same equation in $[a,b]$. Thus, 
  \begin{equation}
    \label{eq:446}
     \begin{aligned}
          \HK^2(\mu_0,\mu_1)=2(b{-}a)
          \sup\Big\{&\int_{\R^d}\xi_b\,\d\mu_1-\int_{\R^d}\xi_a\,\d\mu_0:\ \ 
          \xi\in \rmC^\infty_\rmc(\R^d\times [a,b]),\ \\&
          \partial_t\xi_t(x)+\frac 12 \big|\rmD_x \,\xi_t (x)\big|^2 
         +2 \xi_t^2(x)\le 0 \text{ in }\R^d\times(a,b)\Big\}.
        \end{aligned}
  \end{equation}
  Let $\xi\in \rmC^\infty_\rmc(\R^d\times[0,1])$
  be a subsolution to the Hamilton-Jacobi equation
  $\partial_t\xi+\frac 12|\rmD \xi|^2+2\xi^2\le 0$ in 
  $\R^d\times[0,1]$. 
  By a standard argument (see \cite[Lem.\,8.1.2]{Ambrosio-Gigli-Savare08}),
  the integrability \eqref{eq:175pre} and the weak continuity of
  $t\mapsto\mu_t$ yield 
  \begin{align*}
    2\int_{\R^d} \xi_{t_1}\,\d\mu_{t_1} -
    2\int_{\R^d}\xi_{t_0}\,\d\mu_{t_0}
     &= 2\int_{t_0}^{t_1}\int_{\R^d} \Big(\partial_t\xi_t+
        \langle\rmD_x\xi_t,\vv_t\rangle+\xi_tw_t\Big)\,\d\mu_t\,\d t
   \\
     &\le 2\int_{t_0}^{t_1}\int_{\R^d} 
          \Big(-\frac 12|\rmD_x  \xi_t|^2-2\xi_t^2
               +
    \langle\rmD_x \xi_t,\vv_t\rangle+\xi_tw_t\Big) \,\d\mu_t\,\d t
    \\
    &\le \int_{t_0}^{t_1}\int_{\R^d}  \Big(
        |\vv_t|^2+\frac 14|w_t|^2\Big) \,\d\mu_t\,\d t.   
  \end{align*}
  Applying Corollary \ref{cor:HJRd} and \eqref{eq:446} we find
  \begin{displaymath}
    \HK^2(\mu_{t_0},\mu_{t_1})\le  (t_1-t_0)\int_{t_0}^{t_1}\int_{\R^d} 
                             \Big(
                                 |\vv_t|^2+\frac 14|w_t|^2\Big)
                                 \,\d\mu_t\,\d t   
                                 \quad\forevery 0\le t_0<t_1\le 1,
  \end{displaymath}
  which yields \eqref{eq:176}. 
\end{proof}

Combining Theorems \ref{thm:BB1} and \ref{thm:BB2} with 
Theorem \ref{thm:plan-representation} and the geodesic property of 
$(\cM(\R^d),\HK)$ we immediately have \WWW the desired dynamic
representation. \EEE

\begin{theorem}[Representation of $\HK$ \`a la Benamou-Brenier]
  \label{cor:BB1+2}
  For every $\mu_0,\mu_1\in \cM(\R^d)$ we have 
  \begin{align}
  \notag    \HK^2(\mu_0,\mu_1)=
      \min\Big\{&\int_0^1 \!\!\int_\Rd\!\! \Big(|\vv_t|^2+\frac
      14|w_t|^2\Big)\,\d\mu_t\,\d 
      t\ : \ 
      \mu\in \rmC([0,1];\cM(\R^d)),\ \mu_{t=i}=\mu_i,\\
    \label{eq:177}
      &
      \mbox{}\qquad \partial_t \mu_t+\nabla\cdot(\vv_t\mu_t)=w_t\mu_t
      \text{ in }\DD'(\R^d\times (0,1))\Big\}.
  \end{align}
  The Borel vector field $(\vv,w)$ realizing the minimum in
  \eqref{eq:177}
  is uniquely determined $\mu_I$-a.e.~in $\R^d\times (0,1)$.
\end{theorem}

\WWW The discussion in \cite{LMS15} reveals however that there may be
many geodesic curves, so in general $\measu \mu$ is not
unique. Indeed, the set of all geodesics connecting
$\mu_0=a_0\delta_{x_0}$ and $\mu_1=a_1\delta_{x_1}$ with $a_0,a_1>0$
and $|x_1{-}x_0|=\pi/2$ is infinite dimensional, see
\cite[Sect.\,5.2]{LMS15}. \EEE

\subsection{Geodesics in $\cM(\R^d)$}
\label{subsec:geodesicRd}

As in the case of the Kantorovich-Wasserstein distance, one may expect
that geodesics $(\mu_t)_{t\in [0,1]}$ in $(\cM(\R^d),\HK)$ can be
characterized by the system (cf.\ \cite[Sect.\,5]{LMS15})
\begin{equation}
  \label{eq:10}
  \partial_t \mu_t+\nabla\cdot(\mu_t \,\rmD_x\xi_t)=4\xi_t\mu_t,\quad
  \partial_t \xi_t+\frac 12|\rmD_{x}\xi_t|^2+2\xi_t^2=0.
\end{equation}
In order to give a precise meaning to \eqref{eq:10} we first have to
select an appropriate regularity for $\xi_t$.  \WWW On the one hand
\EEE we cannot expect $\rmC^1$ smoothness for solutions of the
Hamilton-Jacobi equation \eqref{eq:10} (in contrast with subsolutions,
that can be regularized as in Corollary \ref{cor:HJRd}) and on the
other hand the $\Leb d$ a.e.~differentiability of Lipschitz functions
guaranteed by Rademacher's theorem is not sufficient, if we want to
consider arbitrary measures $\mu_t$ that could be singular with
respect $\Leb d$.

A convenient choice for our aims is provided by locally Lipschitz
functions which are strictly differentiable at $\mu_I$-a.e.~points,
where $\mu_I$ has been defined by \eqref{eq:436}. A
function $f:\R^d\to \R$ is \emph{strictly differentiable} at $x\in \R^d$ if
there exists $\rmD f(x)\in (\R^d)^*$ such that
\begin{equation}
  \label{eq:11}
  \lim_{x',x''\to x\atop  x'\neq x''} \frac{f(x')-f(x'')-\rmD
    f(x)(x'-x'')}{|x'-x''|}=0. 
\end{equation}
\WWW According to \cite[Prop.~2.2.4]{Clarke83} a locally
Lipschitz function $f$ is strictly differentiable at $x$ if and only
if the Clarke subgradient \cite[Sect.~2.1]{Clarke83} of $f$ at $x$
reduces to the singleton $\{\rmD f(x)\}$. In particular, denoting by
$\dD\subset \R^d$ the set where $f$ is differentiable and denoting by
$\kappa_\eps$ a smooth convolution kernel as in \eqref{eq:433},
Rademacher's theorem and \cite[Thm.~2.5.1]{Clarke83} yield
\begin{equation}
  \label{eq:36}
  \lim_{x'\to x\atop x'\in \sdD}\rmD f(x')=\rmD f(x),\quad
  \lim_{\eps\down0}\rmD (f\ast\kappa_\eps)(x)=\rmD f(x) \WWW \ \text{
    for all }x \in \dD. 
\end{equation}
 In the proofs we will also need to deal with pointwise
representatives
of the time derivative of a locally Lipschitz function $\xi:\R^d\times
(0,1)\to \R$:
if $D(\partial_t \xi)$ will denote the set (of full $\Leb{d+1}$ measure) where 
$\xi$ is differentiable w.r.t.~time and 
$\widetilde{\partial_t \xi}$ the extension of $\partial_t \xi$ to $0$
outside $D(\partial_t \xi)$, we set
\begin{equation}
  \label{eq:88}
  (\partial_t \xi_t)_-(x):=\liminf_{\eps\to0}\big(\widetilde{\partial_t
    \xi_t} \ast \kappa_\eps\big)(x),\quad
  (\partial_t \xi_t)^+(x):=\limsup_{\eps\to0}\big(\widetilde{\partial_t
    \xi_t} \ast \kappa_\eps\big)(x).
\end{equation}
It is not difficult to check that such functions are Borel; even if they
depend on the specific choice of $\kappa_\eps$, they will still be
sufficient
for our aims (a more robust definition would require the 
use of approximate limits).

\nc We are now ready to characterize the set of all geodesic
curves by giving a precise meaning to \eqref{eq:10}. The proof that
the conditions (i)--(iv) below are sufficient for geodesic follows
directly with the subsequent Lemma \ref{le:techHJ}, whereas the proof
of necessity is more involved and relies on the existence of optimal
potentials $\psi_1$ for $\LET=\HK^2$ in Theorem \ref{thm:mainHK2}(d),
 on the characterization of
subsolutions of the generalized Hamilton--Jacobi equation in Theorem
\ref{thm:main-HJ}, and on the characterization of curves $t\mapsto \mu_t$
in $\AC^2\big([0,1];(\cM(\R^d),\HK)\big)$. 

\begin{theorem}
  \label{thm:geodesicRd}
  Let $\mu\in \rmC^0([0,1];\cM(\R^d))$ be a weakly continuous curve.
  If 
  there exists
  a map $\xi \in \Lip_{\rm loc}((0,1);\rmC_b(\R^d))$ such that \nc
  \begin{enumerate}[(i)]
  \item $\xi_t\in \Lip_b(\R^d)$ for every $t\in (0,1)$ with $t\mapsto
    \Lip(\xi_t,\R^d)$ locally bounded in $(0,1)$
    (equivalently, the map $(x,t)\mapsto \xi_t(x)$ is bounded and
    Lipschitz in $\R^d
    \times [a,b]$ for every compact subinterval $[a,b]\subset (0,1)$),
    \nc
  \item $\xi$ 
    is strictly differentiable w.r.t.~$x$ at
    $\mu_I$-a.e.~$(x,t)\in \R^d\times (0,1)$,
  \item $\xi$ satisfies 
    \begin{equation}
      \label{eq:15}
      \partial_t\xi_t +\frac
      12  \big| \rmD_x \xi_t(x)\big|^2+2\xi_t^2(x)=0\quad \text{$\Leb{d+1}$-a.e.~in }\R^d\times (0,1),
    \end{equation}
    \item and the curve $(\mu_t)_{t\in [0,1]}$ solves the continuity
      equation with reaction \WWW with the vector field $(\rmD_x
      \xi,4\xi)$  in every compact
      subinterval of $(0,1)$, i.e.
      \begin{equation}
        \label{eq:16}
        \partial_t \mu_t+\nabla\cdot(\mu_t \rmD_x\xi_t)=4\xi_t\mu_t
        \quad\text{in }\DD'(\R^d\times (0,1)),
      \end{equation}
  \end{enumerate}
  
  then $\mu$ is a geodesic w.r.t.~the $\HK$ distance.
  Conversely, if $\mu$ is a geodesic then it is possible to find
  $\xi \in \Lip_{\rm loc}((0,1);\rmC_b(\R^d))$
  that satisfies 
  the properties $(i)--(iv)$ above, is right differentiable
  w.r.t.~$t$ in $\R^d\times (0,1)$, and fulfils \eqref{eq:419bis}
  everywhere in $\R^d\times (0,1).$ 
\end{theorem}
Notice that \eqref{eq:15} seems the weakest natural formulation of the
Hamilton-Jacobi equation, in view of Rademacher's Theorem. The
assumption of strict differentiability of $\xi$ at $\mu_I$-a.e.~point
provides an admissible vector field $\rmD_X\xi$ for \eqref{eq:16}.\\
\begin{proof}
  The proof splits into a sufficiency and a necessity part, the
  latter having several steps.\\
  \textbf{Sufficiency.} Let us suppose that $\mu,\xi$ satisfy
  conditions $(i),\ldots, (iv)$.

  Since $D(\partial_t \xi)$ has full $\Leb{d+1}$-measure in
  $\R^d\times (0,1)$, Fubini's Theorem shows that $N:=\{t\in (0,1): \Leb d(\{x\in
  \R^d:(x,t)\not\in D(\partial_t \xi)\})>0\}$ is $\Leb1$-negligible.
  By \eqref{eq:15} we get
  \begin{equation}
    \label{eq:89}
    (\partial_t\xi)_-(x)
    =-\limsup_{\eps\down0}
    \Big(\big(\frac 12 |\rmD_x
    \xi_t|^2+2\xi_t^2\big)\ast\kappa_\eps\Big)(x)
    \ge 
    -
    \frac 12 |\rmD
    \xi_t|_a^2(x)-2\xi_t^2(x)
  \end{equation}
  for every $x\in \R^d$ and $t\in (0,1)\setminus N$. \nc

  We apply Lemma \ref{le:techHJ} below with $\vv=\rmD_x \xi$ and 
  $w=4\xi$: observing that $|\rmD \xi_t|_a(x)=
  |\rmD_x \xi_t(x)|$ at every point $x$ of strict differentiability
  of $\xi_t$, we get, for all $0<a<b<1$,
  \begin{align*}
    &2\int_{\R^d}\xi_b\,\d\mu_b-
    2\int_{\R^d}\xi_a\,\d\mu_a
    \ge 2 \int_{\R^d\times (a,b)} \Big(
      (\partial_t\xi)_- \nc +
                      |\rmD_x\xi_t(x)|^2+4\xi_t^2(x)\Big)\,\d\mu_I
  \\
    &\topref{eq:89}=
              2\int_{\R^d\times (a,b)} \Big(
           \frac12|\rmD_x\xi_t(x)|^2+2\xi_t^2(x)\Big)\,\d\mu_I         
    \topref{eq:176}\ge  \int_a^b |\mu_t'|^2\,\d t
    \: \geq \: \frac 1{b{-}a}\HK^2(\mu_a,\mu_b).
  \end{align*}
  On the other hand, since $\R^d$ is a length space, Theorem
  \ref{thm:main-HKHJ} yields
  \begin{displaymath}
    \frac 1{b-a}\HK^2(\mu_a,\mu_b)\ge 2 \int_{\R^d}\xi_b\,\d\mu_b-
    2\int_{\R^d}\xi_a\,\d\mu_a,
  \end{displaymath}
  so that all the above inequalities are in fact identities and, hence,
  \begin{displaymath}
    \HK(\mu_a,\mu_b)= (b-a)\,|\mu_t'|\quad\text{$\Leb 1$-a.e.~in $[a,b]$}.
  \end{displaymath}
  This shows that $\mu$ is a geodesic. Passing to the limit as
  $a\down0$ and $b\up1$ we conclude the proof of the first part of the
  Theorem.\smallskip

  \noindent\textbf{Necessity.}
  Let $(\mu_t)_{t\in [0,1]}$ be a $\HK$-geodesic in 
  $\cM(\R^d)$ connecting $\mu_0$ to $\mu_1$;
  applying Theorem \ref{thm:BB1} we can find
  a Borel vector field $(\vv,w)\in
  \rmL^2(\R^d\times(0,1),\mu_I;\R^{d+1})$ 
  such that \eqref{eq:175pre} and \eqref{eq:174pre} hold.
  We also consider an optimal plan $\ggamma\in \OptLET(\mu_1,\mu_2)$.

  Let $\psi_1,\psi_2:\R^d\to [-\infty,1]$ be a pair of optimal
  potentials given by Theorem \ref{thm:mainHK2} d) and let us set
  $\xi:=-\frac 12\psi_1$ and $\xi_t:=\PP_t\xi$ for $t\in (0,1)$.
Even if we are considering more general initial data
  $\xi\in \rmB(\R^d;[-1/2,+\infty])$ in \eqref{eq:423}, it is not
  difficult to check that the same statement of Theorem
  \ref{thm:main-HJ} holds in every subinterval $[a,b]$ with $0<a<b<1$
  and
  \begin{equation}
    \label{eq:9}
    \lim_{t\down0}\PP_t\xi(x)=\sup_{t>0}\PP_t\xi(x)=\xi_*(x),\quad
    \text{where}\quad
    \xi_*(x):=\lim_{r\down0}\inf_{x'\in B_r(x)}\xi(x')
  \end{equation}
  is the lower semicontinuous envelope of $\xi$. Moreover, setting
  \begin{equation}
    \label{eq:41}
    \xi_1(x)=\PP_1\xi(x):=\lim_{t\up1}\xi_t(x)=\inf_{0<t<1}\xi_t(x),
  \end{equation}
  the function $\xi_1$ is upper semicontinuous, and optimality yields
  \begin{equation}
    \label{eq:50}
    \frac 12\psi_2(x)=\xi_1(x)\quad \text{for $\gamma_2$-a.a.\ }x\in \R^d.
  \end{equation}
  By introducing the semigroup $\bar\PP_t\xi:=-\PP_{t}(-\bar\xi)$ and
  reversing time, we can define
  \begin{equation}
    \label{eq:51}
    \bar\xi_t:=\bar\PP_{1-t}.
  \end{equation}
  By using the link with the Hopf-Lax semigroup in $\tY$ given by
  Theorem \ref{thm:main-HJ}, the optimality of $(\psi_1,\psi_2)$, and
  arguing as in \cite[Thm.\,7.36]{Villani09} it is not difficult to
  check that
  \begin{equation}
    \label{eq:53}
    \bar\xi_t\le \xi_t\quad\text{in }\R^d,\quad 
    \bar\xi_0=\xi_0= -\nc \frac 12\psi_1\quad\mu_0\text{-a.e.~in }\R^d.
  \end{equation}
  Notice that the function $x\mapsto -\cos^2(|x-x'|\land \pi/2)$ 
  has bounded first and second derivatives, so it is semiconcave.
  It follows that the map $x\mapsto \xi_t(x)$ is
  semiconcave for every $t\in (0,1)$ and 
  $x\mapsto \bar\xi_t(x)$ is semiconvex. 

  Since $t\mapsto \int \xi_t\,\d\mu_t$ and $t\mapsto
  \int\bar\xi_t\,\d\mu_t$ are absolutely continuous in $(0,1)$,
  Theorem \ref{thm:main-HKHJ}(i) yields
  \begin{equation}
    \label{eq:69}
    \frac\d{\d t}\int\xi_t\,\d\mu_t\le \frac 12|\mu'_t|^2=\frac
    12\HK^2(\mu_0,\mu_1), 
  \end{equation}
  so that
  \begin{displaymath}
    \int\xi_b\,\d\mu_b-\int\xi_a\,\d\mu_a\le \frac{b-a}2\HK^2(\mu_0,\mu_1).
  \end{displaymath}
  Passing to the limit first as $a\down0$ and then as $b\up1$ 
  by monotone convergence (notice that $\xi_t\le 
  1/2$) and using optimality once again, we obtain
  \begin{equation}
    \label{eq:70}
    \begin{aligned}
     \HK^2(\mu_0,\mu_1)&= \int \psi_1\nc \,\d\mu_0+
     \int \psi_2\nc \,\d\mu_1=
      2\int\xi_1\,\d\mu_1-
      2\int\xi_0\,\d\mu_0
      \\&=\lim_{a\down0,b\up1}2\Big(\int\xi_b\,\d\mu_b-\int\xi_a\,\d\mu_a\Big)....
    \end{aligned}
  \end{equation}
  By \eqref{eq:69} it follows that 
  \begin{equation}
    \label{eq:75}
    \frac\d{\d t}\int\xi_t\,\d\mu_t=\frac 12|\mu'_t|^2=\frac
    12\HK^2(\mu_0,\mu_1)
    \quad\text{in }(0,1).
  \end{equation}
  Reversing time, the analogous argument yields
  \begin{equation}
    \label{eq:76}
    \frac\d{\d t}\int\bar\xi_t\,\d\mu_t=\frac 12|\mu'_t|^2=\frac
    12\HK^2(\mu_0,\mu_1)
    \quad\text{in }(0,1).
  \end{equation}
  Hence, we have proved that the maps $t\mapsto \int \xi_t\,\d\mu_t$ and
  $t\mapsto \int\bar\xi_t\,\d\mu_t$ are affine in $[0,1]$ and
  coincide at $t=0$ and $t=1$, which implies 
  that 
  \begin{equation}
    \label{eq:77}
    \int \xi_t\,\d\mu_t=\int\bar\xi_t\,\d\mu_t\quad\text{for every
    }t\in [0,1].
  \end{equation}
  Recalling \eqref{eq:53}, we deduce that the complement of the set
  $Z_t:=\{x\in \R^d: \xi_t(x)=\bar\xi_t(x)\}$ is $\mu_t$-negligible.
  Since $\xi_t$ is Lipschitz and semiconcave (thus
  everywhere superdifferentiable) for $t\in (0,1)$ and since $\bar\xi_t$
  is Lipschitz and 
  semiconvex (thus everywhere subdifferentiable), we conclude that
  $\xi_t$ is strictly differentiable in $Z_t$, and thus it satisfies
  conditions (i) and (ii).

  Since (iii) is guaranteed by Theorem \ref{thm:main-HJ} ($\R^d$ is a
  length space), it remains to check \eqref{eq:16}.
  We apply the following Lemma \ref{le:techHJ} 
  by observing that \cite[Prop.\,3.2,3.3]{Ambrosio-Gigli-Savare14}
  and Theorem \ref{thm:main-HJ} yield
  \begin{displaymath}
    \limsup_{x'\to x}\partial_t^+\xi_t(x')
    \le \limsup_{x'\to x}\partial_t^-\xi_t(x')
    \le \partial_t^-\xi_t(x),\quad
    \liminf_{x'\to x}\partial_t^+\xi_t(x')
    \ge \partial_t^+\xi_t(x);
  \end{displaymath}
  
  since $\partial_t^-\xi_t(x)=\partial_t^+\xi_t(x)$
  $\mu_I\text{-a.e.}$ 
  we get 
  \begin{displaymath}
    (\partial_t \xi)^+=(\partial_t\xi)_-=\partial_t^+\xi\quad \mu_I\text{-a.e.}
  \end{displaymath}
  and therefore
  \eqref{eq:31} holds with equality. \nc

  Recalling that $|\rmD \xi_t|^2_a(x)=
  |\rmD_x\xi_t(x)|^2$ at every point of $Z_t$, 
  for every $0<a<b<1$ we
  have
  \begin{align*}
    \frac{b-a}2\HK^2(\mu_0,\mu_1)
    &=\int_{\R^d}\xi_b\,\d\mu_b-
      \int_{\R^d}\xi_a\,\d\mu_a
      \topref{eq:31}=\int_{\R^d\times (a,b)} \Big(\partial_t^+\xi+
      \rmD_x\xi\,\vv+\xi w \Big)\,\d\mu_I
    \\& =
        \int_{\R^d\times (a,b)} \Big(-\frac 12|\rmD_x\xi_t|^2-
        2\xi_t^2+
        \rmD_x\xi\,\vv+\xi w \Big)\,\d\mu_I
    \\& =
        \int_{\R^d\times (a,b)} \Big(-\frac 12|\rmD_x\xi_t-\vv|^2-
        2(\xi_t-\frac14 w)^2+
        \frac 12 |\vv|^2+\frac 18 w^2 \Big)\,\d\mu_I
        \\&\topref{eq:174pre}\le 
            - \int_{\R^d\times (a,b)} \Big(\frac 12|\rmD_x\xi_t-\vv|^2+
        2(\xi_t-\frac14 w)^2\Big)\,\d\mu_I+
            \frac 12\int_a^b|\mu_t'|^2\,\d t
  \end{align*}
  We deduce that $\vv=\rmD_x\xi$ and $w=4\xi$ holds $\mu_I$-a.e.  
\end{proof}

\WWW The following lemma provides the ``integration by parts''
formulas that where used in the sufficiency and necessity part of the
previous proof of Theorem \ref{thm:geodesicRd}. It is established by a
suitable temporal and spatial smoothing, 
involving a smooth kernel
$\kappa_\eps$ as in \eqref{eq:433}. 

\EEE  
\begin{lemma}
  \label{le:techHJ}
  Let $\mu\in \AC^2_{\rm loc}((0,1);(\cM(\R^d),\HK))$ be satisfying the 
  continuity equation with reaction \eqref{eq:175pre}
  governed by the field $(\vv,w)\in L^2(\R^d\times (a,b),\mu_I)$ 
  for every $[a,b]\subset (0,1)$.
  If $\xi\in \Lip_{\rm loc}((0,1);\rmC_b(\R^d))$ satisfies conditions
  $(i,ii)$ 
  of Theorem \ref{thm:geodesicRd}, then
  for all $0<a\le b<1$ we have 
  \begin{equation}
    \label{eq:31}
    \begin{aligned}
      \int_{\R^d\times (a,b)} \Big((\partial_t\xi)^++
      \rmD_x\xi\,\vv+\xi w \Big)\,\d\mu_I &\ge
      \int_{\R^d}\xi_b\,\d\mu_b-
      \int_{\R^d}\xi_a\,\d\mu_a
      \\&\ge\int_{\R^d\times (a,b)}
      \Big((\partial_t\xi)_-+ \rmD_x\xi\,\vv+\xi w \Big)\,\d\mu_I,
    \end{aligned}
  \end{equation}  
  where $(\partial_t\xi)^+,(\partial_t\xi)_-$ are defined
  in terms of a space convolution kernel $\kappa_\eps$ as in \eqref{eq:88}.
  
\end{lemma}
\begin{proof}
  We fix a compact subinterval $[a,b]\subset (0,1)$,
  $b'\in (b,1)$, \nc
  and set $M:=\max_{t\in [a,b']}\mu_t(\R^d)$ and
  $L:=\Lip(\xi;{\R^d\times [a,b']})+\sup_{\R^d\times [a,b']}
  |\xi|$.
  
  We regularize $\xi$ by space convolution as in \eqref{eq:433} by
  setting $\xi^\eps:=\xi\ast \kappa_\eps$ and perform a
  further regularization in time, viz.
  \begin{equation}
    \label{eq:34}
    \xi^{\eps,\tau}_t(x):=\frac 1\tau\int_0^{\tau}\xi^\eps_{t+r}(x)\,\d
    r,\quad
    0<\tau<b'-b.
  \end{equation}
  Since $\xi^{\eps,\tau}\in \rmC^1(\R^d\times [a,b])$, we can argue as
  in the proof of Theorem \ref{thm:BB2} and obtain, for every $\eps>0$ and
  $\tau\in (0,b'{-}b)$, the identity
  \begin{equation}
    \label{eq:33}
    \int_{\R^d}\xi^{\eps,\tau}_b\,\d\mu_b-
    \int_{\R^d}\xi^{\eps,\tau}_a\,\d\mu_a=\int_{\R^d\times (a,b)} \Big(\partial_t\xi^{\eps,\tau}+
    \rmD_x\xi^{\eps,\tau}\,\vv+\xi^{\eps,\tau} w \Big)\,\d\mu_I.
  \end{equation}
  We first pass to the limit as $\tau\down0$, observing that
  $\xi^{\eps,\tau}\to\xi^\eps$ uniformly because $\xi^\eps$ is bounded
  and Lipschitz. Similarly, since $\rmD
  \xi^{\eps,\tau}=(\rmD\xi^\eps)^\tau$ and $\rmD\xi^\eps$ is bounded
  and Lipschitz, we have $\rmD\xi^{\eps,\tau}\to \rmD\xi^\eps$
  uniformly.  Finally, using 
  \begin{displaymath}
    \partial_t\xi_t^{\eps,\tau}(x)=\frac
    1\tau(\xi^\eps_{t+\tau}(x)-\xi^\eps_t(x))=
    \int_{\R^d}\frac
    1\tau(\xi^\eps_{t+\tau}(x')-\xi^\eps_t(x'))\kappa_\eps(x-x')\,\d
    x',
  \end{displaymath}
  and the fact that $N:=\{t\in (0,1): \Leb d(\{x\in
  \R^d:(x,t)\not\in D(\partial_t \xi)\})>0\}$ is $\Leb1$-negligible by
  the theorems of 
  Rademacher and Fubini, 
  an application of Lebesgue's Dominated Convergence Theorem yields \nc
  \begin{equation}
    \label{eq:35}
    \lim_{\tau\down0}\partial_t\xi_t^{\eps,\tau}(x)=
    \partial_t\xi^\eps_t(x)=((\partial_t\xi)\ast \kappa_\eps)(x)
    \quad\text{for every }x\in \R^d,\quad t\in (a,b)\setminus N.
  \end{equation}
  
  Since $\R^d\times N$ is also $\mu_I$-negligible,
  a further application of Lebesgue's Dominated Convergence Theorem
  yields \nc
  \begin{equation}
    \label{eq:33bis}
    \int_{\R^d}\xi^{\eps}_b\,\d\mu_b-
    \int_{\R^d}\xi^{\eps}_a\,\d\mu_a=\int_{\R^d\times (a,b)} \Big(\partial_t\xi^{\eps}+
    \rmD_x\xi^{\eps}\,\vv+\xi^{\eps} w \Big)\,\d\mu_I.
  \end{equation}
  Now, \eqref{eq:31} will be deduced by passing to the limit
  $\eps\down0$ in \eqref{eq:33bis} as follows.  We observe that
  $\xi^\eps$ converges uniformly to $\xi$ because $\xi$ is bounded and
  Lipschitz. Moreover, since $\lim_{\eps\down0} \rmD_x\xi_t^\eps(x)=
  \rmD_x\xi_t(x)$ at every point $x\in \R^d$ where $\xi_t$ is strictly
  differentiable, we obtain
  \begin{displaymath}
    |\rmD_x\xi^\eps\,\vv|\le L|\vv|\in
    \rmL^1(\R^d\times(a,b);\mu_I) \ \text{ and } \ 
    \lim_{\eps\down0} \rmD_x\xi^\eps=
    \rmD_x\xi\quad \text{ $\mu_I$-a.e.~in } \R^d\times [a,b],
  \end{displaymath}
  so that
  \begin{displaymath}
    \lim_{\eps\down0} \int_{\R^d}\xi^{\eps}_{a,b}\,\d\mu_{a,b}=
    \int_{\R^d}\xi_{a,b}\,\d\mu_{a,b},\quad
    \int\limits_{\R^d\times (a,b)} \!\!\!\Big(\rmD_x\xi^{\eps}\,\vv+\xi^{\eps} w
    \Big)\,\d\mu_I=
    \int\limits_{\R^d\times (a,b)}\!\!\! \Big(\rmD_x\xi\,\vv+\xi w
    \Big)\,\d\mu_I.
  \end{displaymath}
  Finally, since $\partial_t\xi^\eps_t $ is also uniformly bounded,
  Fatou's Lemma yields
  \begin{displaymath}
    \limsup_{\eps\down0}\int\limits_{\R^d\times (a,b)}\!\!\! \partial_t\xi_t^\eps
    \,\d\mu_I\le
    \int\limits_{\R^d\times (a,b)}\!\!\! (\partial_t\xi_t)^+
    \,\d\mu_I,\quad
    \liminf_{\eps\down0}\int\limits_{\R^d\times (a,b)} \!\!\!\partial_t\xi_t^\eps
    \,\d\mu_I\ge
    \int\limits_{\R^d\times (a,b)}\!\!\! (\partial_t\xi_t)_-
    \,\d\mu_I.
  \end{displaymath}
  \WWW Thus, \eqref{eq:31} follows from \eqref{eq:33bis}. 
\end{proof}

\subsection{Contraction properties: convolution and Heat equation in
  $\mathrm{RCD}(0,\infty)$ metric-measure spaces.}
\label{subsec:contraction}

We conclude this paper with a few applications concerning contraction
properties of the $\HK$ distance.  The first one concerns the behavior
with respect $1$-Lipschitz maps.
\begin{lemma}
  \label{le:contraction}
  Let $(X,\sfd_X),\ (Y,\sfd_Y)$ be separable metric spaces and let
  $f:X\to Y$ be a $1$-Lipschitz map.  Then $f_\sharp:\cM(X)\to \cM(Y)$
  is $1$-Lipschitz w.r.t.~$\HK$:
  \begin{equation}
    \label{eq:448}
    \HK(f_\sharp \mu_1,f_\sharp \mu_2)\le 
    \HK(\mu_1,\mu_2).
  \end{equation}
\end{lemma}
\begin{proof}
  It is sufficient to observe that the map $\frf:\tY_X\mapsto \tY_Y$
  defined by $\frf([x,\s ]):=[f(x),s]$ satisfies
  $\sfd_{\tY_Y}(\frf([x_1,\s_1]),\frf([x_2,\s_2]))\le
  \sfd_{\tY_X}([x_1,\s_1],[x_2,\s_2])$ for every $[x_i,\s_i]\in
  \tY_X$. Thus $\frf_\sharp $ is a contraction from
  $(\cP_2(\tY_X),\sfW_{\sfd_{\tY_X}})$ to
  $(\cP_2(\tY_Y),\sfW_{\sfd_{\tY_Y}})$, and hence $f_\sharp$
  satisfies \eqref{eq:448}.
\end{proof}
A second application concerns convolutions in $\R^d$.

\begin{theorem}
  \label{thm:convolution}
  Let $X=\R^d$ with the Euclidean distance and let $\nu\in \cM(\R^d)$.
  Then the map $\mu\mapsto \mu\ast \nu$ is contractive w.r.t.~$\HK$ if
  $\nu(\R^d)=1$ and, more generally,
  \begin{equation}
    \label{eq:453}
    \HK^2(\mu_1\ast\nu,\mu_2\ast\nu)\le 
    \nu(\R^d)\HK^2(\mu_1,\mu_2)\quad\forevery \mu_1,\mu_2\in \cM(\R^d).
  \end{equation}
\end{theorem}
\begin{proof}
  The previous lemma shows that $\HK$ is invariant by isometries, in
  particular translations in $\R^d$, so that 
  \begin{displaymath}
    \HK(\mu_1\ast\delta_x,\mu_2\ast\delta_x)=\HK(\mu_1,\mu_2)\quad
    \forevery \mu_1,\mu_2\in \cM(\R^d),\ x\in \R^d.
  \end{displaymath}
  By the subadditivity property it follows that 
  if $\nu=\sum_k a_k\delta_{x_k}$ for some $a_k\ge 0$, then
  \begin{align*}
    \HK^2(\mu_1\ast\nu,\mu_2\ast\nu)&=
                                 \HK^2(\sum_k a_k\mu_1\ast \delta_{x_k},\sum_k
                                 a_k\mu_2\ast\delta_{x_k})
                     \\&\le 
                           \sum_k a_k
                             \HK^2(\mu_1\ast\delta_{x_k},\mu_2\ast\delta_{x_k})
                           =
                           \sum_k a_k
                             \HK^2(\mu_1,\mu_2)=\nu(\R^d)\HK^2(\mu_1,\mu_2).
  \end{align*}
  The general case then follows by approximating $\nu$
  by a sequence of discrete measure $\nu_n$ converging to $\nu$ in
  $\cM(\R^d)$ and observing that 
  $\mu_i\ast\nu_n\to \mu_i\ast \nu$ weakly in $\cM(\R^d)$.
  Since $\HK$ is weakly continuous we obtain \eqref{eq:453}.
\end{proof}

An easy application of the previous result is the contraction property
of the (adjoint) Heat semigroup $(P^*_t)_{t\ge0}$ in $\R^d$ with
respect to $\HK$. In fact, we can prove a much more general result for
the Heat flow in $\mathrm {RCD}(0,\infty)$ metric measure spaces
$(X,\sfd,m)$ \cite{AGS14b,AGS15}. It covers the case of the semigroups
$(P_t)_{t\ge0}$ generated by 

\paragraph{(A)} the Heat equation on a open convex domain
$\Omega\subset \R^d$  
with homogeneous Neumann conditions 
\begin{displaymath}
  \partial_t u=\Delta u\quad\text{in $\Omega\times (0,\infty)$,}\qquad
  \partial_n u=0\quad\text{ on $\partial\Omega\times
    (0,\infty)$},
\end{displaymath}

\paragraph{(B)} the Heat equation on a complete Riemannian manifold
$(\M^d,g)$ with nonnegative Ricci curvature defined by
\begin{displaymath}
  \partial_t u=\Delta_g u\quad\text{in }\M^d\times (0,\infty),
\end{displaymath}
where $\Delta_g$ is the usual Laplace-Beltrami operator, 
and 

\paragraph{(C)} the Fokker-Planck equation in $\R^d$ generated by
the gradient of a convex potentials $V:\R^d\to \R$, viz. 
\begin{displaymath}
  \partial_t u=\Delta_g u-\nabla\cdot(u\,\rmD V)\quad\text{in
  }\R^d\times (0,\infty). 
\end{displaymath}

\begin{theorem}
  \label{thm:Ricci}
  Let $(X,\sfd,m)$ be a complete and separable metric-measure space
  with nonnegative Riemannian Ricci Curvature, i.e.~satisfying the
  {\rm RCD$(0,\infty)$} condition, and let 
  $(P^*_t)_{t\ge0}:\cM(X)\to\cM(X)$ be the Heat semigroup
  in the measure setting.
  Then
  \begin{equation}
    \label{eq:454}
    \HK(P^*_t\mu_1,P^*_t\mu_2)\le \HK(\mu_1,\mu_2)\quad
    \text{for all }\mu_1,\mu_2\in \cM(X) \WWW \text{ and }t>0. \EEE
  \end{equation}
\end{theorem}
\begin{proof}
  Recall that in {\rm RCD$(0,\infty)$} metric measure spaces the
  $L^2$-gradient flow of the Cheeger energy induces a symmetric Markov
  semigroup $(P_t)_{t\ge0}$ in $L^2(X,m)$, which has a pointwise
  version satisfying the Feller regularization property
  $P_t(\rmB_b(X))\subset \Lip_b(X)$ for $t>0$ and the estimate
  \begin{equation}
    \label{eq:48}
    |\rmD_X P_t f|^2(x)\le P_t\big(|\rmD_X f|^2\big) (x)\quad
    \forevery f\in \Lip_b(X),\ x\in X,\ t\ge0.
  \end{equation}
  Its adjoint $(P_t^*)_{t\ge0}$ coincides with the
  Kantorovich-Wasserstein gradient flow in $\cP_2(X)$ of the Entropy
  Functional $\FF(\cdot|m)$ where $\FF$ is induced by $F(\r)=\PE_1(\r)
  \WWW = \r \log\r - \r +1 $ and defines a semigroup in $\cM(X)$ by
  the formula
  \begin{equation}
    \label{eq:192}
    \int_X f\,\d(P_t^*\mu)=\int_X P_t f\,\d\mu\quad
    \text{for every }f\in \rmB_b(X) \text{ and } \mu\in \cM(X).
  \end{equation}
  In order to prove \eqref{eq:454} we use \eqref{eq:412}
  ($\mathrm{RCD}$-spaces satisfy the length property)
  and apply $P_t$ to a subsolution
  $(\psi_\theta)_{\theta\in [0,1]}$ in $\rmC^1([0,1];\Lip_b(X))$ of the
  Hamilton-Jacobi equation
  \begin{equation}
    \label{eq:455}
    \partial_\theta\psi_\theta+\frac 14|\rmD_X
    \psi_\theta|^2+\psi_\theta^2\le 0\quad\text{in }X\times (0,1).
  \end{equation}
  Since $P_t$ is a linear and continuous map from $\Lip_b(X)$ to
  $\Lip_b(X)$
  the curve $\theta\mapsto \psi_{\theta,t}:=P_t(\psi_\theta)$ belongs
  to $\rmC^1([0,1];\Lip_b(X))$. Now, 
  \eqref{eq:48} and the Markov property yield
  \begin{displaymath}
    |\rmD_X P_t\psi_\theta|^2(x)\le
    P_t\big(|\rmD_X\psi_\theta|^2\big)(x),\ 
    (P_t\psi_\theta)^2(x)\le P_t(\psi_\theta^2)(x)\
    \text{ for }x\in X,\ \theta\in [0,1],\ t\ge0.
  \end{displaymath}
  Thus, for every $t\ge0$ we obtain 
  \begin{displaymath}
    \partial_\theta\psi_{\theta,t}+\frac 14|\rmD_X
    \psi_{\theta,t}|^2+\psi_{\theta,t}^2\le 0\quad\text{in }X\times (0,1),
  \end{displaymath}
  and therefore
  \begin{align*}
    \int_X \psi_{1}\,\d(P^*_t)\mu_1-
    \int_X \psi_{0}\,\d(P^*_t)\mu_0&=
                                   \int_X P_t\psi_{1}\,\d\mu_1-
                                   \int_X P_t\psi_{0}\,\d\mu_0
                                   \le \HK^2(\mu_1,\mu_0).
  \end{align*}
  We conclude by taking the supremum with respect to all the subsolutions 
  of \eqref{eq:455} in $\rmC^1([0,1];\Lip_b(X))$ and applying \eqref{eq:412}.
\end{proof}

\appendix
\section{On the chronological development of our theory}
\label{sec:Devel}
\def\bfxi{\boldsymbol\xi}
\def\bfXi{\boldsymbol\Xi}

In this section we give a brief account of the order in which we
developed the different parts of the theory. The beginning was the mostly
formal work 
in \cite{LMS15} on reaction-diffusion systems, where a
distance on vectors $\uu$ of densities over a domain $\Omega\subset
\R^d$ was formally defined in the Benamou-Brenier sense via 
\[
\sfd(\uu_0,\uu_1)^2=\inf\int_0^1 \int_\Omega \bfXi_t:
\M_\text{diff}(\uu_t) \bfXi_t + \bfxi_t \cdot \K_\text{react}(\uu_t) \bfxi_t
\d x \d t 
\]
under the constraint of the continuity equation 
$ 
\partial_t\uu_t + \nabla\cdot\big( \M_\text{diff}(\uu_t) \bfXi_t \big)
=\K_\text{react}(\uu_t) \bfxi_t .
$ 
The central question was and still is the understanding of diffusion
equations with reactions in the gradient-flow form 
$ 
\partial_t \uu = \nabla\cdot \big( \M_\text{diff}(\uu) \nabla
\delta\cF(\uu)\Big) - \K_\text{react}(\uu)\delta\cF(\uu),
$ 
see \cite[Sect.\,5.1]{LMS15}.

It was natural to treat the scalar case first and to restrict to the
case where both mobility operator $\M_\text{diff}(u)$ and
$\K_\text{react}(u)$ are linear in $u$.  Only in that case the
formally derived system \eqref{eq:GeodSyst} for the geodesics
$(u_t,\xi_t)$ decouples in the sense that $\xi_t$ solves an
Hamilton-Jacobi equation that does not depend on $u$.  Choosing
$\M_\text{diff}(u)=\alpha u$ and $\K_\text{react}(u) =\beta u$ with
$\alpha, \beta\geq 0$, the relevant Hamilton-Jacobi equation reads
\[
\partial_t \xi_t + \frac\alpha2 |\rmD_x\xi_t|^2 + \frac{\beta}2\xi_t^2 =0.
\]
As in the other parts of this paper, we restrict to the case
$\alpha=1$ and $\beta=4$ subsequently, but refer to \cite{LMS15} for
the general case. Thus, the conjectured characterization
\eqref{eq:412} was first presented in Pisa at the Workshop ``Optimal
Transportation and Applications'' in November 2012.

During a visit of the second author in Pavia, the generalized Hopf-Lax
formula via the nonlinear convolution $\HJ{\pi/2}t$ (cf.\
\eqref{eq:423}) was derived via the classical method of
characteristics. This led to the unsymmetric representation
\eqref{eq:293} for $\HK$. To symmetrize this relation we used that
$\HJ{}1\xi(x)=\inf \Phi(\xi(y),|y{-}x|)$ with $\Phi(z,R)= \frac12\big(
1{-}\frac{A(R)}{1+2z}\big)$, where $A(R)=\cos^2\big(R\wedge
({\pi}/{2})\big)$. Setting $\psi_0=-2\xi_0$ and $\psi_1=2\xi_1
=2\HJ{}1$, we have the equivalence
\[
\xi_1 = \HJ{}1\xi_0 \quad \Longleftrightarrow \quad
(1{-}\psi_0(x_0))(1{-}\psi_1(x_1)) \geq A(|x_0{-}x_1|) \text{ for all }x_i.
\]
Setting $\varphi_i=- \log(1{-}\psi_i)$
we arrived at the cost function 
\[
\sfc(x_0,x_1) = - \log A(|x_0{-}x_1|) =\begin{cases} - 2\log\big(
  \cos|x_0{-}x_1|\big) &\text{for }|x_0{-}x_1|<\pi/2,\\ \infty&
  \text{otherwise}, \end{cases}
\]
for the first time and obtained the characterization \eqref{eq:210},
namely 
\[
\HK(\mu_0,\mu_1)^2=\sfD(\mu_0,\mu_1)=\sup\big\{
\DD(\varphi_0,\varphi_1|\mu_0,\mu_1)\,:\ \varphi_0\oplus\varphi_1 \leq
\sfc\big\}.
\]
It was then easy to dualize $\DD$, and the Logarithmic Entropy
functional $\LET$ in \eqref{eq:231} was derived in July 2013.  

While the existence of minimizers for $\LET(\mu_0,\mu_1)= \min
\EE(\gamma|\mu_0,\mu_1) $ was easily obtained, it was not clear at
all, why and how $\HK$ defined via $\HK^2(\mu_0,\mu_1)=\min
\EE(\cdot|\mu_0,\mu_1)$ generates a geodesic distance. The only thing
which could easily be checked was that the minimum was consistent with
the distance between two Dirac masses, which could easily be
calculated via the dynamic formulation.

So, in parallel we tried to develop the dynamic approach, which was
not too successful at the early stages. Only after realizing and
exploiting the connection to the cone distance in Summer and Autumn of
2013 we were able to connect $\LET$ systematically with the dynamic
approach. The crucial and surprising observation was that optimal
plans for $\EE$ and lifts of measures $\mu\in \cM(X)$ to measures
$\lambda$ on the cone $\tY$ could be identified by
exploiting the optimality conditions systematically. Corresponding results
were presented in workshops on Optimal Transport in Banff (June 2014)
and Pisa (November 2014).

Already at the Banff workshop, the general structure of the primal and dual
Entropy-Transport problem as well as the homogeneous perspective
formulation were presented. Several examples and refinements where
developed afterwards. The most recent part from Summer 2015 concerns
our Hamilton-Jacobi equation in general metric spaces $(X,\sfd)$ and
the induced cone $\tY$ (cf.\ Section \ref{subsec:HJ}) and
the derivation of the geodesic equations in $\R^d$ (cf.\ Section
\ref{subsec:geodesicRd}). This last achievement 
now closes the circle, by showing
that all the initial steps, which were done on a formal level in 2012
and the first half of 2013, have indeed a rigorous interpretation.
 
\footnotesize
\bibliographystyle{siam}
\bibliography{bibexport_ET4}
\end{document}